\documentclass[11pt,a4paper]{article}

\usepackage[T1]{fontenc}
\usepackage{lmodern}
\usepackage{microtype}
\usepackage[a4paper,top=22mm,bottom=23mm,left=23mm,right=23mm,headheight=14pt,headsep=9mm]{geometry}
\usepackage{amsmath,amssymb,mathtools,amsthm}
\usepackage{tikz}
\usepackage{fancyhdr}
\usepackage{needspace}
\usepackage{tabularx}
\usepackage[hidelinks]{hyperref}

\hypersetup{
  pdftitle={Finite-time blowup of the Euler equation without forcing},
  pdfsubject={Amplification, exact packets, stage induction, and the limiting initial datum},
  pdfauthor={Fan Zheng}
}
\numberwithin{equation}{section}
\newtheorem{introlemma}[equation]{Lemma}
\newtheorem{maintheorem}{Theorem}
\newcommand{\diag}{\mathrm{diag}}
\newcommand{\av}[1]{\langle #1\rangle}

\newcommand{\lead}{\mathrm{lead}}
\newcommand{\eps}{\varepsilon}

\newcommand{\R}{\mathbb{R}}

\begin{document}
\thispagestyle{plain}
\begin{center}
  {\LARGE\bfseries Finite-time blowup of the Euler equation\\[3pt]
  without forcing}\par
  \vspace{6pt}
  {\large An exposition of the construction in \cite{euler}}\par
  \vspace{6pt}
  {\large Fan Zheng}\par
  \vspace{3pt}
  {\small Instituto de Ciencias Matem\'aticas}\par
  {\small\href{mailto:fan.zheng@icmat.es}{\texttt{fan.zheng@icmat.es}}}
\end{center}
\vspace{3pt}

\section*{Overview and main result}
\phantomsection\addcontentsline{toc}{section}{Overview and main result}
These notes present the construction in \cite{euler} of finite-time blowup
for the three-dimensional incompressible, unforced Euler equations on $\R^3$:
\begin{equation}\tag{E}\label{eq:euler}
 \partial_tu+(u\cdot\nabla)u+\nabla p=0,\qquad \nabla\cdot u=0.
\end{equation}
The construction successively adds oscillatory layers to smooth
parent solutions. Each layer is constructed by a finite expansion and
an exact correction, so that the resulting child solves Euler.
The initial increments are then summed to obtain a single initial datum.
The conclusion established by the argument is the following.

\begin{maintheorem}[Finite maximal smooth lifespan]\label{thm:main}
There exists a divergence-free $u_0\in C_c^\infty(\R^3;\R^3)$ whose
unique maximal smooth Euler solution has lifespan $T_*=T_*(u_0)$ with
\[
 0<T_*\le T<\infty,
\]
where $T$ is the accumulation time of the construction's target times.
Here the solution belongs to $C([0,S];H^m(\R^3))$ for every fixed
integer $m\ge0$ and every $S<T_*$. At its maximal time,
\[
 \limsup_{t\uparrow T_*}\|\nabla u(t)\|_{L^\infty}=\infty,
 \qquad
 \int_0^{T_*}\|\operatorname{curl}u(t)\|_{L^\infty}\,dt=\infty.
\]
\end{maintheorem}

The proof gives an upper bound $T_*\le T$, not an identification of
these times or a location of the limiting singularity. Only the initial
velocity series is summed in all Sobolev norms. Stability relative to a
hypothetical smooth solution through $T$ then contradicts the growing
finite-stage gradients. Convergence of the full velocity and pressure
series through $T$ is not required.
Stability of the blowup under perturbations of the initial datum is not
established here.

Section~\ref{sec:wave-derivation} derives useful identities and
introduces the stage parameters and clocks. Sections~\ref{sec:amplification}--\ref{sec:quantitative-amplification}
prove the wave-model comparison, amplification, and target geometry.
Sections~\ref{sec:history-activation}--\ref{sec:mean-inverse}
construct the history and forced inverses. Sections~\ref{sec:packet-expansion}--\ref{sec:complete-packet}
construct an increment yielding an exact child solution and establish
its flow-map and initial-data estimates.
Section~\ref{sec:wave-packet-interface} connects the wave estimates to
that construction. Sections~\ref{sec:stage-parameters}--\ref{sec:stage-induction}
choose the parameters, construct the seed, and close the induction.
Finally, Propositions~\ref{prop:limiting-initial-datum},
\ref{prop:limiting-finite-lifespan}, and~\ref{prop:endpoint-divergence}
in Sections~\ref{sec:limiting-initial-datum}--\ref{sec:endpoint-conclusions}
prove Theorem~\ref{thm:main}.

The construction and proof strategy are those of \cite{euler}.
These notes reorganize their presentation and expand the simultaneous
velocity--pressure derivation, variational inverses, fixed-scale Gevrey
estimates, and interfaces between the stages. All equation numbers are
local. The bibliography identifies the supplied source file by its hash.

\clearpage
\begingroup
\setlength{\parskip}{0pt}
\tableofcontents
\endgroup
\clearpage

\section*{Notation guide}
\phantomsection\addcontentsline{toc}{section}{Notation guide}
Time $t$ is physical time throughout; primes in the wave-model comparison
denote differentiation in $\tau$. The following guide records the
principal local uses of symbols; each is defined at its first use.
\begin{center}
\renewcommand{\arraystretch}{1.1}
\begin{tabularx}{\linewidth}{@{}lX@{}}
\hline
Symbol & Meaning and location \\
\hline
$a,\ell y$ & Physical parent label; $a_j$ separately denotes the coupling at the origin in \eqref{eq:amplification-scales}. \\
$F,M,H$ & Parent deformation, velocity gradient, and pressure Hessian, pulled back to labels; \eqref{intro:eq:parent}. \\
$m_j,v_j$ & New wave normal and velocity amplitude; $p,q,n$ are the preceding wave's frame at the origin; \eqref{eq:inherited-frame}. \\
$Z=(P,Q,N),\ Y=(U,V)$ & Scaled normal and velocity components; \eqref{eq:scaled-components}. The third velocity component is reconstructed by transversality. \\
$W,A_p,B_p,C_p$ & Lifted velocity increment and its oscillatory, mean, and curl-correction coefficients; Section~\ref{sec:wave-derivation}. \\
$G,K$ & $G=FR_\perp$ in the transverse inverse; $G=F^{-1}F^{-T}$ and $K=G^{-1}$ in the exact correction. $K_h,\mathcal K$ are regularity bounds. \\
$P,R,\varrho$ & Input budget, Gevrey derivative scale, and Gevrey radius; Sections~\ref{sec:packet-expansion}--\ref{sec:exact-correction}. This $P$ is not the normal component. \\
$j,p,N$ & Stage, coefficient order, and finite truncation order; $N$ in $Z$ instead denotes the third normal component. \\
$t_{j-1},t_j,S_j,\Theta_j$ & Activation, target, retained physical endpoint, and scaled comparison endpoint; Subsection~\ref{intro:subsec:time-intervals}. \\
$T,T_*$ & Target accumulation time and actual maximal smooth lifespan, with \mbox{$T_*\le T$}. \\
\hline
\end{tabularx}
\end{center}

\section{Introduction}\label{sec:wave-derivation}
We first derive identities for one child layer over a smooth parent:
the leading velocity and pressure, the successive corrections enforcing
the Euler equation and incompressibility, and the residual left by a
finite expansion. We then explain how to correct this residual and
recover the physical solution and its flow. Finally, we introduce the
parameters of the successive layers, explain the amplification and
turning mechanisms that prepare each layer to generate the next, and
describe the activation, target, and retained time intervals.

\subsection{The increment equation in Lagrangian and lifted coordinates}

Let $(u,p)$ solve the incompressible Euler equations \eqref{eq:euler}.
Write the proposed child as
$u^{\rm new}=u+w$, $p^{\rm new}=p+\pi$ and set
$D_t=\partial_t+u\cdot\nabla_x$, $M_{ij}=\partial_{x_j}u_i$.
Subtracting the parent equation gives the exact increment equation
\begin{equation}\label{intro:eq:physical-increment}
 D_tw+Mw+(w\cdot\nabla_x)w+\nabla_x\pi=0,
 \qquad \operatorname{div}_xw=0.
\end{equation}
The term $Mw=(w\cdot\nabla_x)u$ is the interaction with the parent.

Let $a$ be the parent initial label and introduce its flow:
\begin{equation}
 X_t(t,a)=u(t,X(t,a)),\qquad X(0,a)=a.
\end{equation}
Fix a localization length $\ell>0$, a frequency $k>0$, and a unit vector
$m_0$. Put $a=\ell y$ and, from now on, use pulled-back matrices
\begin{equation}
 F(t,y)=\nabla_aX(t,\ell y),\qquad
 M(t,y)=\nabla_xu(t,X(t,\ell y)),\qquad
 H(t,y)=\nabla_x^2p(t,X(t,\ell y)).
\end{equation}
They satisfy
\begin{equation}\label{intro:eq:parent}
 F_t=MF,\qquad F_{tt}=-HF,\qquad F(0)=I,\qquad \det F=1.
\end{equation}
For a physical scalar $h$ with pullback
$\widehat h(t,y)=h(t,X(t,\ell y))$, the spatial chain rule gives
\[
 \nabla_y\widehat h=\ell F^T(\nabla_xh)\circ X,
 \qquad (\nabla_xh)\circ X=\ell^{-1}d\widehat h,
 \qquad d:=F^{-T}\nabla_y.
\]
Thus $d$ represents $\ell\nabla_x$ in parent label coordinates.

Choose the physical phase to be constant along parent trajectories:
\[
 \theta(t,X(t,\ell y))=km_0\cdot y,\qquad D_t\theta=0.
\]
Its gradient gives the local oscillation direction and frequency.
Define the normal by removing the fixed factor $k/\ell$:
\[
 m(t,y):=\frac{\ell}{k}(\nabla_x\theta)(t,X(t,\ell y)).
\]
Applying the preceding chain rule to this phase, then differentiating
in time using $F_t=MF$, gives
\begin{equation}\label{intro:eq:normal}
 m=k^{-1}d(km_0\cdot y)=F^{-T}m_0,\qquad
 m_t=-F^{-T}(F_t)^Tm=-M^Tm.
\end{equation}
Here $\partial_t$ holds $y$ fixed, so it is material differentiation
along the parent trajectory.

We will choose the velocity increment $w$, the pressure increment $\pi$,
and their successive corrections as periodic profiles evaluated at the
transported phase. For the full increments, write
\begin{equation}\label{intro:eq:normalized-increment}
 w(t,X(t,\ell y))=\ell W(t,y,km_0\cdot y),\qquad
 \pi(t,X(t,\ell y))=\ell^2q(t,y,km_0\cdot y).
\end{equation}
To construct these profiles, now reuse $\theta\in\mathbb T$ as an
independent angle, rather than as the physical function $\theta(t,x)$.
For any profile $A(t,y,\theta)$, let
$A^\diag(t,y)=A(t,y,km_0\cdot y)$ and set
\begin{equation}\label{intro:eq:lifted-derivative}
 \tilde d=d+km\partial_\theta.
\end{equation}
The chain rules are
\begin{equation}\label{intro:eq:chain}
 D_t[A^\diag(t,y)]\big|_{x=X(t,\ell y)}=(\partial_tA)^\diag,
 \qquad
 \nabla_x[A^\diag(t,y)]\big|_{x=X(t,\ell y)}
 =\ell^{-1}(\tilde dA)^\diag.
\end{equation}
These formulas mean that $A^\diag$ is viewed as a physical field via
$x=X(t,\ell y)$; for a vector, apply the gradient identity componentwise.
In particular, $\partial_\theta$ alone brings no factor $k$.

Substituting \eqref{intro:eq:normalized-increment} and using
\eqref{intro:eq:chain}, a sufficient lifted equation for
\eqref{intro:eq:physical-increment} is
\begin{equation}\label{intro:eq:lifted-W}
 W_t+MW+W\cdot\tilde dW+\tilde dq=0,
 \qquad \tilde d\cdot W=0.
\end{equation}
We impose it for every independent $\theta$, then restrict to the graph.

\subsection{The leading order velocity and the pressure coefficients}\label{intro:subsec:leading-coefficients}

Let $f_\delta$ be the smooth odd, mean-zero periodic profile defined by
\eqref{eq:seed-profile-definition}, so that $f_\delta'(0)=\delta^{-1}$ and
$f_\delta'\ge-C$. For a cutoff $\chi_1(y)$, set
\begin{equation}\label{intro:eq:leading-coefficients}
 A_1=\alpha\chi_1v f_\delta(\theta),\qquad
 W_{\rm lead}=k^{-1}A_1,\qquad
 q_{\rm lead}=k^{-2}\pi_1.
\end{equation}
The coefficient of $k^0$ in the divergence equation is
$m\cdot\partial_\theta A_1=0$. We impose $m\cdot v=0$, hence
$m\cdot A_1=0$. The phase self-interaction then vanishes identically:
\begin{equation}
 (A_1\cdot m)\partial_\theta A_1=0.
\end{equation}
Matching the coefficient of $k^{-1}$ in \eqref{intro:eq:lifted-W} gives
\begin{equation}\label{intro:eq:first-order}
 (A_1)_t+MA_1+m\partial_\theta\pi_1=0.
\end{equation}
Substituting $A_1$ into \eqref{intro:eq:first-order} gives
\[
 \alpha\chi_1f_\delta(\theta)(v_t+Mv)+m\partial_\theta\pi_1=0.
\]
The pressure term is parallel to $m$. Wherever $\chi_1\ne0$, matching
the profile therefore requires $v_t+Mv$ to be parallel to $m$.
We impose this relation globally through the coefficient equation
\[
 v_t=-Mv+\lambda m,\qquad
 \partial_\theta\pi_1=-\alpha\chi_1\lambda f_\delta,
\]
where $\lambda=\lambda(t,y)$ is a scalar. To determine it, differentiate
the constraint using $m_t=-M^Tm$:
\[
 0=\partial_t(m\cdot v)
 =-m\cdot Mv+m\cdot(-Mv+\lambda m)
 =-2m\cdot Mv+\lambda|m|^2.
\]
Thus $\lambda=2(m\cdot Mv)/|m|^2$. Substitution and a zero-mean
primitive in $\theta$ give
\begin{equation}\label{intro:eq:wave}
 \begin{aligned}
 v_t&=-Mv+2m\frac{m\cdot Mv}{|m|^2},\qquad m\cdot v=0,\\
 \pi_1&=-2\alpha\chi_1\frac{m\cdot Mv}{|m|^2}
                  \partial_\theta^{-1}f_\delta.
 \end{aligned}
\end{equation}
Here $\partial_\theta^{-1}$ is the zero-mean periodic primitive.
These equations preserve $m\cdot v=0$ from its initial value.

Combining \eqref{intro:eq:normalized-increment} and
\eqref{intro:eq:wave} gives the physical leading increments:
\begin{equation}\label{eq:leading-velocity}
 \begin{aligned}
 w_{\lead}(t,x)&=\frac\ell k\alpha\chi_1v f_\delta(\theta),\\
 \pi_{\lead}(t,x)&=-2\left(\frac\ell k\right)^2\alpha\chi_1
       \frac{m\cdot Mv}{|m|^2}\varphi_\delta(\theta),
 \qquad \varphi_\delta=\partial_\theta^{-1}f_\delta.
 \end{aligned}
\end{equation}
Here $\theta=km_0\cdot y$ and $x=X(t,\ell y)$.
Differentiating the phase gives the leading gradient and pressure Hessian:
\begin{align}
 \nabla_x w_{\lead}
 &=\alpha\chi_1v\otimes m\,f_\delta'(\theta)
   +\frac\ell k f_\delta(\theta)\nabla_x(\alpha\chi_1v),
 \label{eq:gradient}\\
 \nabla_x^2\pi_{\lead}
 &=-2\alpha\chi_1(m\cdot Mv)\frac{mm^T}{|m|^2}
     f_\delta'(\theta)
   +\text{terms with derivatives of slow coefficients}.
 \label{eq:leading-pressure-hessian}
\end{align}
These display the rank-one shear and pressure Hessian used below.
Coefficient matching means $k\to\infty$ with $\ell,\alpha,\delta$ and
smooth coefficient bounds fixed. Under this convention the omitted
pressure-Hessian terms are $O(\ell/k)$ in amplitude. The leading increment
\eqref{eq:leading-velocity} still has slow divergence, pressure, and nonlinear residuals; the next
subsections explain the corrections. Matching the coefficient equations
never requires division by the profile at its zeros.

\subsection{Enforcing exact divergence-freeness}\label{intro:subsec:divergence}

Let $A_p$ be the coefficient of the $p$th-order oscillatory velocity
correction $k^{-p}A_p$, with $A_1$ as in Subsection~\ref{intro:subsec:leading-coefficients}.
Require transversality, $m\cdot A_p=0$, and zero angle mean
so that its periodic primitive exists. Transversality alone does not imply divergence freedom. Define the curl potential and divergence correction by
\begin{equation}\label{intro:eq:curl-correction}
 Q_p=-\partial_\theta^{-1}\frac{m\times A_p}{|m|^2},\qquad
 C_p=d\times Q_p,\qquad C_0=0.
\end{equation}
Since $m$ is angle independent and $m\cdot A_p=0$,
\begin{equation}
 m\times\partial_\theta Q_p=A_p,\qquad
 A_p+k^{-1}C_p=k^{-1}\tilde d\times Q_p.
\end{equation}
To justify divergence of this curl, fix $t$. The operators $d_i$ are
the pullbacks of $\ell\partial_{x_i}$, so $[d_i,d_j]=0$.
Also $m=d(m_0\cdot y)$; hence
\[
 d_im_j=d_id_j(m_0\cdot y)=d_jd_i(m_0\cdot y)=d_jm_i.
\]
Both the coefficients of $d_i$ and $m_i$ are independent of $\theta$.
Expanding the commutator therefore gives
\[
 [\tilde d_i,\tilde d_j]
 =[d_i,d_j]+k(d_im_j-d_jm_i)\partial_\theta=0.
\]
For the alternating symbol $\epsilon_{ijr}$, antisymmetry in $i,j$ now gives
\[
 \tilde d\cdot(\tilde d\times Q_p)
 =\sum_{i,j,r}\epsilon_{ijr}\tilde d_i\tilde d_j(Q_p)_r
 =\frac12\sum_{i,j,r}\epsilon_{ijr}
       [\tilde d_i,\tilde d_j](Q_p)_r=0.
\]
Thus $\tilde d\cdot(A_p+k^{-1}C_p)=0$ for independent $(y,\theta)$,
and consequently also after restriction to the physical phase graph.
\subsection{The order-$p$ equations: mean and oscillatory parts}\label{intro:subsec:order-expansion}

Products of oscillations produce an angle-independent part, so introduce
mean velocity and pressure coefficients $B_p,\bar q_p$ in addition to
$A_p,\pi_p$. We define
\[
 \av{f}:=\frac1{2\pi}\int_0^{2\pi}f\,d\theta.
\]
Then we impose
\[
 \av{A_p}=0,\qquad \av{\pi_p}=0,\qquad d\cdot B_p=0.
\]
The coefficient $\pi_p$ collects only the zero-mean part of the pressure;
the mean part is recorded separately in $\bar q_p$.
The volume-preserving parent map gives
\begin{equation}\label{intro:eq:piola}
 d\cdot B=\operatorname{div}_y(F^{-1}B).
\end{equation}
To see this geometrically, fix $t$ and set
$\mathcal X(y)=\ell^{-1}X(t,\ell y)$, so
$D\mathcal X=F$ and $\det D\mathcal X=1$.
Let $\omega$ be the standard volume form and define a vector field $v$
in the target coordinates by $v\circ\mathcal X=B$.
Its pullback is $\mathcal X^*v=F^{-1}B$, while
$\mathcal X^*\omega=\omega$. Using
$\mathrm d(\iota_v\omega)=(\operatorname{div}v)\omega$
and commutation of pullback with exterior differentiation gives
\[
 \operatorname{div}_y(F^{-1}B)\,\omega
 =\mathrm d\bigl(\iota_{\mathcal X^*v}\omega\bigr)
 =\mathcal X^*\bigl(\mathrm d(\iota_v\omega)\bigr)
 =((\operatorname{div}v)\circ\mathcal X)\,\omega.
\]
Here $\mathrm d$ is exterior differentiation, distinct from the vector
operator $d$. The chain rule completes the proof of \eqref{intro:eq:piola}:
\[
 (\operatorname{div}v)\circ\mathcal X
 =\operatorname{tr}\bigl((\nabla_yB)F^{-1}\bigr)
 =\sum_{i,j}(F^{-1})_{ji}\partial_{y_j}B_i=d\cdot B,
\]

Since $W_{\rm lead}$ and $q_{\rm lead}$ have zero angular mean,
take $B_1=0$ and $\bar q_1=0$. For a finite order $N$, set
\begin{equation}\label{intro:eq:finite-packet}
 \begin{aligned}
 W^a&=\sum_{p=1}^N\bigl[k^{-p}(A_p+B_p)+k^{-(p+1)}C_p\bigr],\\
 q^a&=\sum_{p=1}^N\bigl[k^{-p}\bar q_p+k^{-(p+1)}\pi_p\bigr].
 \end{aligned}
\end{equation}
Then $\tilde d\cdot W^a=0$ exactly, at every truncation order.

The whole velocity coefficient at order $k^{-p}$ is
\begin{equation}\label{intro:eq:total-velocity-coefficient}
 V_p=A_p+B_p+C_{p-1}\quad(1\le p\le N),\qquad
 V_{N+1}=C_N,\qquad W^a=\sum_{p=1}^{N+1}k^{-p}V_p.
\end{equation}
At step $p$, all lower orders and $C_{p-1}$ are already known.
The new unknowns are $(B_p,\bar q_p)$ and $(A_p,\pi_p)$.
Label derivatives and angle derivatives enter different powers:
\begin{equation}
 W^a\cdot dW^a=\sum_{i,j}k^{-(i+j)}V_i\cdot dV_j,
 \qquad
 k(W^a\cdot m)\partial_\theta W^a
 =\sum_{i,j}k^{-(i+j-1)}(V_i\cdot m)\partial_\theta V_j.
\end{equation}
Consequently, for $2\le p\le N$, matching order $k^{-p}$ requires
\begin{equation}\label{intro:eq:pth-equation}
 \partial_t(A_p+B_p)+M(A_p+B_p)+d\bar q_p
        +m\partial_\theta\pi_p=f_p,
\end{equation}
where
\begin{equation}\label{intro:eq:pth-force}
 \begin{aligned}
 f_p={}&-\partial_tC_{p-1}-MC_{p-1}-d\pi_{p-1}\\
 &-\sum_{i+j=p}V_i\cdot dV_j
 -\sum_{i+j=p+1}(V_i\cdot m)\partial_\theta V_j.
 \end{aligned}
\end{equation}
All indices are positive. In the phase sum the $(1,p)$ term vanishes,
since $m\cdot V_1=m\cdot A_1=0$. The $(p,1)$ term contains $B_p$, but
not $A_p$, because $A_p\cdot m=0$. Hence
\begin{equation}\label{intro:eq:old-force}
 \begin{aligned}
 f_p&=f_p^{\rm old}-(B_p\cdot m)\partial_\theta A_1,\\
 f_p^{\rm old}&=-\partial_tC_{p-1}-MC_{p-1}-d\pi_{p-1}
 -\sum_{i+j=p}V_i\cdot dV_j\\
 &\quad -(C_{p-1}\cdot m)\partial_\theta A_1
 -\sum_{\substack{i+j=p+1\\2\le i,j\le p-1}}
                 (V_i\cdot m)\partial_\theta V_j.
 \end{aligned}
\end{equation}
Every term in $f_p^{\rm old}$ is known. The only term containing a
current-order unknown, $-(B_p\cdot m)\partial_\theta A_1$, has zero angle
average because $B_p$ and $m$ are angle independent. Taking the mean and then subtracting it gives the
linear triangular sequence
\begin{align}
 \partial_tB_p+MB_p+d\bar q_p&=\av{f_p^{\rm old}},
 &d\cdot B_p&=0,\label{intro:eq:mean-step}\\
 \partial_tA_p+MA_p+m\partial_\theta\pi_p&=f_p^\perp,
 &m\cdot A_p&=0,\label{intro:eq:wave-step}\\
 f_p^\perp&=f_p^{\rm old}-\av{f_p^{\rm old}}
                  -(B_p\cdot m)\partial_\theta A_1.\notag
\end{align}
First determine $(B_p,\bar q_p)$; then its value is part of the known
forcing for $(A_p,\pi_p)$; finally form $(Q_p,C_p)$ by
\eqref{intro:eq:curl-correction}. The selection of history or initial data is
additional to these coefficient identities.

We follow OpenAI~\cite{euler}. Related constructions include the
triangular mean--oscillation hierarchy in \cite[\S4.2]{cheverry2006}
and increasing-order approximations that keep the source uniformly
smooth in the IPM layer cascade of \cite{cordoba-martinez-zoroa2024}.

To solve the oscillatory equation \eqref{intro:eq:wave-step}, consider its
general form
\[
 A_t+MA+m\partial_\theta\pi=f,\qquad m\cdot A=0,
 \qquad \av{A}=\av{\pi}=\av{f}=0.
\]
At order $p$, take $A=A_p$, $\pi=\pi_p$, and $f=f_p^\perp$,
which is known after solving the mean equation \eqref{intro:eq:mean-step}.
As in the leading-order calculation, $m_t=-M^Tm$ gives
\[
 0=\partial_t(m\cdot A)=-m\cdot MA+m\cdot A_t,
 \qquad m\cdot A_t=m\cdot MA.
\]
Taking the scalar product of the momentum equation with $m$ therefore yields
\[
 2m\cdot MA+|m|^2\partial_\theta\pi=m\cdot f.
\]
Solving for the pressure derivative and substituting back into the
momentum equation gives
\begin{equation}\label{intro:eq:forced-transverse}
 \begin{aligned}
 \partial_\theta\pi&=\frac{m\cdot f-2m\cdot MA}{|m|^2},\\
 A_t&=-MA+2m\frac{m\cdot MA}{|m|^2}
          +\left(I-\frac{mm^T}{|m|^2}\right)f.
 \end{aligned}
\end{equation}
The second line is a forced linear ODE for $A$ at each $(y,\theta)$.
Given transverse, zero-mean initial data for $A$ and zero-mean forcing $f$,
the ODE preserves both conditions on $A$:
$\partial_t(m\cdot A)=0$, and averaging its right side gives a
homogeneous equation for $\av{A}$. Recover pressure by inverting the
first line of \eqref{intro:eq:forced-transverse}. Its right-hand side has zero
angle mean because $m,M$ are angle independent and $\av{A}=\av{f}=0$,
so it has a unique zero-mean periodic primitive. Substitution recovers the original
forced transverse system, in particular \eqref{intro:eq:wave-step}.

For the mean equation,
$B=F\beta$ gives the coordinate form
\begin{equation}
 \beta_t+2F^{-1}F_t\beta+G\nabla_y\bar q=F^{-1}\av{f_p^{\rm old}},
 \qquad \operatorname{div}_y\beta=0,\qquad G=F^{-1}F^{-T}.
\end{equation}

For the estimates, let $t_a$ denote activation and $[0,S]$ the construction
interval (see Figure~\ref{fig:stage-timeline}). The leading coefficient $A_1$ uses the coercive variational history
construction of Section~\ref{sec:history-activation}.
Lemma~\ref{lem:history-endpoint-choice} selects a history that reaches
activation with $v_q=1$ and
$-C\le v_p\le0$, the signs needed for subsequent amplification.
This choice retains the polynomial history bounds without requiring
an inverse for arbitrary terminal velocities.
For $p\ge2$,
Section~\ref{sec:forced-transverse} estimates $A_p$ on $[0,t_a]$ by a
coercive variational problem with homogeneous endpoint conditions,
then on $[t_a,S]$ by variation of constants and the growing-wave
propagator. After factoring out
the growth profile on $[t_a,S]$, which includes the exponential amplification,
the remaining losses are polynomial in the parent bounds. The mean coefficient $B_p$ is
instead estimated on the whole interval $[0,S]$ by the coercive
variational inverse of Section~\ref{sec:mean-inverse}, which also
localizes its velocity at time zero; no separate growth estimate on
$[t_a,S]$ is needed. Both inverses have Gevrey-2 derivative bounds, combined in
Section~\ref{sec:packet-expansion} to control the coefficient recursion.

\subsection{The remaining residual and its correction}\label{intro:subsec:exact-identities}

Define the residual exactly by
\begin{equation}
 \mathcal R^a=W_t^a+MW^a+W^a\cdot\tilde dW^a+\tilde dq^a.
\end{equation}
If the coefficient equations hold through order $N$, then
\begin{equation}\label{intro:eq:residual}
 \begin{aligned}
 \mathcal R^a={}&k^{-(N+1)}(\partial_tC_N+MC_N+d\pi_N)\\
 &+\sum_{\substack{1\le i,j\le N+1\\i+j>N}}
           k^{-(i+j)}V_i\cdot dV_j\\
 &+\sum_{\substack{1\le i,j\le N+1\\i+j-1>N}}
           k^{-(i+j-1)}(V_i\cdot m)\partial_\theta V_j.
 \end{aligned}
\end{equation}
The first line is the linear tail; the other lines are the unmatched
nonlinear orders. To explain why this residual can be made small, each
step of the coefficient recursion loses a fixed number of derivatives,
so truncation at order $N$ incurs a factorial shift of order $N$.
In Gevrey class $2$, this gives bounds of the form
\[
 \|\text{coefficient at order }p\|
 \le (P^C N^C)^{p+1},
\]
in a Gevrey norm at a radius independent of $N$. Here $P$ collects the input bounds,
including fixed powers of the old frequency $k_{j-1}$. Each new order also brings an
inverse power of the new frequency $k=k_j$. Choosing
\begin{equation}\label{intro:eq:frequency-separation}
 N=\lfloor k^\vartheta\rfloor,\qquad P^Q\le k^\vartheta,
\end{equation}
with $\vartheta>0$ sufficiently small and $Q$ sufficiently large makes
$P^CN^C$ a small power of $k$. The inverse powers of $k$ therefore
dominate the derivative losses, leaving a normalized residual of size
\[
 \|r^a\|\le k^{-cN}=e^{-cN\log k},\qquad
 r^a=kF^{-1}\mathcal R^a.
\]
The exact-correction estimate introduces a Gronwall factor
$e^{O(P^C)}$ on a time interval of length $O(1)$. The same separation
\eqref{intro:eq:frequency-separation} ensures $P^C\ll N$, so this factor
is absorbed by the residual smallness. Thus the finite expansion
remains accurate enough to construct an exact solution despite the
accumulated derivative losses.

To find the exact correction, fix the divergence operator by putting
\begin{equation}\label{intro:eq:fixed-div}
 D=k^{-1}\nabla_y+m_0\partial_\theta,\qquad
 W=k^{-1}Fz,\qquad \mathfrak p=k^2Dq,\qquad K=G^{-1}=F^TF.
\end{equation}
The vector $\mathfrak p$ is distinguished from the parent scalar
pressure $p$. We have
\begin{equation}\label{intro:eq:fixed-div-identity}
 \tilde d=kF^{-T}D,\qquad
 \tilde d\cdot(k^{-1}Fz)=\operatorname{div}_Dz.
\end{equation}
The first identity follows directly from the definitions and
\eqref{intro:eq:normal}:
$\tilde d=d+km\partial_\theta
=F^{-T}(\nabla_y+km_0\partial_\theta)=kF^{-T}D$.
For the second identity, the label part is
$d\cdot(k^{-1}Fz)=k^{-1}\operatorname{div}_y(F^{-1}Fz)
=k^{-1}\operatorname{div}_y z$ by \eqref{intro:eq:piola}.
The angle part is $km\cdot\partial_\theta(k^{-1}Fz)
=m\cdot F\partial_\theta z=m_0\cdot\partial_\theta z$,
since $F$ is angle independent and $F^Tm=m_0$ by \eqref{intro:eq:normal}.
Their sum is $\operatorname{div}_Dz$.

Substituting $W=k^{-1}Fz$ into \eqref{intro:eq:lifted-W} and multiplying
by $kF^{-1}$ gives
\begin{equation}\label{intro:eq:fixed-momentum}
 z_t+2F^{-1}F_tz+z\cdot Dz
 +k^{-1}F^{-1}(z\cdot\nabla_yF)z+G\mathfrak p=0,
 \qquad\operatorname{div}_Dz=0.
\end{equation}
Both differentiating $F$ and the parent term $MW$ contribute
$F^{-1}F_tz$, explaining the factor two.
The quadratic terms come from $kF^{-1}(W\cdot\tilde dW)
=F^{-1}(z\cdot D)(Fz)$: differentiating $z$ gives $z\cdot Dz$,
while differentiating $F$ gives $k^{-1}F^{-1}(z\cdot\nabla_yF)z$,
because $DF=k^{-1}\nabla_yF$ as $F$ is independent of $\theta$.

Set $z^a=kF^{-1}W^a$, $r^a=kF^{-1}\mathcal R^a$, and
$\mathfrak p^a=k^2Dq^a$. The approximate equation is
\begin{equation}\label{intro:eq:approximate-fixed-momentum}
 \begin{aligned}
 z_t^a+2F^{-1}F_tz^a+z^a\cdot Dz^a
 +k^{-1}F^{-1}(z^a\cdot\nabla_yF)z^a+G\mathfrak p^a&=r^a,\\
 \operatorname{div}_Dz^a&=0.
 \end{aligned}
\end{equation}
Write $z=z^a+e$ and $\mathfrak p=\mathfrak p^a+\mathfrak p^e$.
Subtracting \eqref{intro:eq:approximate-fixed-momentum} from
\eqref{intro:eq:fixed-momentum} gives
\begin{equation}\label{intro:eq:correction}
 e_t+b\cdot\nabla_{y,\theta}e+Z(e)+G\mathfrak p^e=-r^a,
 \qquad \operatorname{div}_De=0,\qquad e(0)=0,
\end{equation}
where
\begin{equation}\label{intro:eq:b-Z}
 \begin{aligned}
 b&=(k^{-1}(z^a+e),m_0\cdot(z^a+e)),\\
 Z(e)&=e\cdot Dz^a+2F^{-1}F_te\\
 &\quad+k^{-1}F^{-1}
 \bigl[((z^a+e)\cdot\nabla_yF)(z^a+e)-(z^a\cdot\nabla_yF)z^a\bigr].
 \end{aligned}
\end{equation}
In particular, $\operatorname{div}_{y,\theta}b=0$ and $Z(e)$ has
no derivative of $e$.

To find the pressure, we use the closed gradient space
$\mathcal G=\overline{\{D\varphi\}}^{L^2}$ and its orthogonal projection
$\Pi$. Then $\mathcal G^\perp=\ker\operatorname{div}_D$.
Since $D$ is independent of time and $\operatorname{div}_De=0$,
we have $e_t\in\mathcal G^\perp$ and hence $\Pi e_t=0$.
Applying $\Pi$ to \eqref{intro:eq:correction} gives
\[
 \Pi G\mathfrak p^e=-\Pi\bigl(b\cdot\nabla e+Z(e)+r^a\bigr).
\]
Since $G=F^{-1}F^{-T}$ is uniformly positive definite,
$\Pi G|_{\mathcal G}$ is coercive and invertible. Hence
\begin{equation}\label{intro:eq:correction-pressure}
 \mathcal T=(\Pi G|_{\mathcal G})^{-1}\Pi,\qquad
 \mathfrak p^e=-\mathcal T\bigl(b\cdot\nabla e+Z(e)+r^a\bigr).
\end{equation}
Thus \eqref{intro:eq:correction} is equivalently
$e_t=-(I-G\mathcal T)(b\cdot\nabla e+Z(e)+r^a)$.
An auxiliary viscous construction adds $\nu\Delta_{y,\theta}e$ to
both versions; it does not change the pressure identity
\eqref{intro:eq:correction-pressure} because
$\Pi\Delta e=0$.

\subsection{The physical flow from the exact solution}\label{intro:subsec:physical-child}

If $\mathfrak p^e\in\mathcal G$, then $D\times\mathfrak p^e=0$.
The graph chain rule gives the curl-free identity
\[
 \nabla_y\times(\mathfrak p^e)^\diag
 =k(D\times\mathfrak p^e)^\diag=0.
\]
Choose a scalar $q^e$ on $\R^3$ with
\begin{equation}\label{intro:eq:scalar-pressure-recovery}
 \nabla_yq^e=k^{-1}(\mathfrak p^e)^\diag.
\end{equation}
The physical fields are
\begin{equation}\label{intro:eq:physical-child}
 \begin{aligned}
 u^{\rm new}(t,X(t,\ell y))
 &=u(t,X(t,\ell y))+\ell(W^a+k^{-1}Fe)^\diag,\\
 p^{\rm new}(t,X(t,\ell y))
 &=p(t,X(t,\ell y))+\ell^2\bigl((q^a)^\diag+q^e\bigr).
 \end{aligned}
\end{equation}
Their increment obeys \eqref{intro:eq:physical-increment} whenever the
correction equation holds. Since $e(0)=0$, the initial velocity increment
is exactly $\ell(W^a)^\diag(0,y)$.

With $z=z^a+e$, the flow of the four-component field $b$ is
\begin{equation}\label{intro:eq:lifted-flow}
 \dot y=k^{-1}z(t,y,\theta),\qquad
 \dot\theta=m_0\cdot z(t,y,\theta),\qquad
 \Phi(0,y_0,\theta_0)=(y_0,\theta_0).
\end{equation}
It preserves the graph because
\begin{equation}
 \frac{d}{dt}(\theta-km_0\cdot y)=0.
\end{equation}
For a fixed initial child label $a_0$, define its changing parent label
and physical position by
\begin{equation}\label{intro:eq:relative-flow}
 \Psi(t,a_0)=\ell\pi_y\Phi\left(t,\frac{a_0}{\ell},
                      \frac{km_0\cdot a_0}{\ell}\right),\qquad
 X^{\rm new}(t,a_0)=X(t,\Psi(t,a_0)).
\end{equation}
Indeed, with $y=\Psi/\ell$ and $\theta=km_0\cdot y$,
\begin{equation}
 \begin{aligned}
 \Psi_t&=\ell k^{-1}z(t,y,\theta),\\
 \partial_t X(t,\Psi)
 &=X_t(t,\Psi)+\nabla_aX(t,\Psi)\Psi_t\\
 &=u(t,X(t,\Psi))+\ell k^{-1}F(t,y)z(t,y,\theta)\\
 &=u^{\rm new}(t,X(t,\Psi)),\qquad X^{\rm new}(0,a_0)=a_0.
 \end{aligned}
\end{equation}
The last equality uses $W=k^{-1}Fz$ from \eqref{intro:eq:fixed-div} and
$w(t,X(t,\ell y))=\ell W^\diag(t,y)$ from
\eqref{intro:eq:normalized-increment}: on the graph,
$\ell k^{-1}Fz=\ell W=w(t,X(t,\ell y))$, and $u^{\rm new}=u+w$.
Finally,
\begin{equation}
 \operatorname{div}_y[k^{-1}z^\diag]
 =(k^{-1}\operatorname{div}_yz+m_0\cdot\partial_\theta z)^\diag
 =(\operatorname{div}_Dz)^\diag=0.
\end{equation}
Thus the relative flow $\Psi$, and hence $X^{\rm new}=X\circ\Psi$,
preserve physical volume. The term $\ell k^{-1}Fz$ supplies the velocity
increment; the time derivative $X_t$ of $X$ supplies the parent velocity.

\subsection{Increments at each stage and their parameters}\label{intro:subsec:stage-parameters}

Start with a base Euler flow $(U^{J-2},p^{J-2})=(U_B,p_B)$.
At stage $j$, add an exact velocity increment $w_j$ and pressure
increment $\pi_j$ as in \eqref{intro:eq:physical-child}:
\[
 U^j=U^{j-1}+w_j,\qquad p^j=p^{j-1}+\pi_j.
\]
Thus the finite-stage solutions telescope as
\[
 U^L=U_B+\sum_{j=J-1}^{L}w_j,\qquad
 p^L=p_B+\sum_{j=J-1}^{L}\pi_j.
\]
The corresponding infinite series are formally
\[
 U=U_B+\sum_{j=J-1}^{\infty}w_j,\qquad
 p=p_B+\sum_{j=J-1}^{\infty}\pi_j,
 \qquad
 u_0=U_B(0)+\sum_{j=J-1}^{\infty}w_j(0).
\]
The finite sums solve Euler wherever the stages are defined; the
infinite sums require convergence, and the pressure requires a common
choice of its time-dependent additive constants. The smooth initial
datum is obtained from the last series.

For stage $j$, use the parent particle map $X^{j-1}$ and introduce
\[
 x=X^{j-1}(t,\ell_jy),\qquad
 F_j(t,y)=\nabla_aX^{j-1}(t,\ell_jy),\qquad
 \theta_j=k_jm_{0,j}\cdot y,\qquad m_j=F_j^{-T}m_{0,j}.
\]
The length $\ell_j$ sets the spatial localization scale of the leading
increment at time zero; $k_j$ is its frequency
in $y$, so the initial physical phase gradient is
$(k_j/\ell_j)m_{0,j}$. The unit vector $m_{0,j}$ selects the phase
direction. The profile parameter $\delta_j$ fixes the shape of the
mean-zero periodic function $f_{\delta_j}$, normalized by
$f_{\delta_j}'(0)=\delta_j^{-1}$. As in \eqref{intro:eq:wave}, the pressure
profile is the zero-mean periodic primitive
$\varphi_{\delta_j}=\partial_\theta^{-1}f_{\delta_j}$, so
$\partial_\theta\varphi_{\delta_j}=f_{\delta_j}$.
Combining the physical prefactors $\ell,\ell^2$ in
\eqref{intro:eq:normalized-increment} with $k^{-1},k^{-2}$ in
\eqref{intro:eq:leading-coefficients} gives the leading increments
\[
 \begin{aligned}
 w_{j,\mathrm{lead}}(t,X^{j-1}(t,\ell_jy))
 &=\frac{\ell_j}{k_j}\alpha_j\chi_1(y)v_j(t,y)f_{\delta_j}(\theta_j),\\
 \pi_{j,\mathrm{lead}}(t,X^{j-1}(t,\ell_jy))
 &=-2\left(\frac{\ell_j}{k_j}\right)^2
 \alpha_j\chi_1(y)\frac{m_j\cdot M^{j-1}v_j}{|m_j|^2}
 \varphi_{\delta_j}(\theta_j),
 \end{aligned}
\]
where $M^{j-1}=(\nabla U^{j-1})\circ X^{j-1}$ is evaluated at $\ell_jy$.
The wave normal and velocity satisfy $m_j\cdot v_j=0$.
The exact increments also contain the divergence correction, higher-order
terms in the expansion \eqref{intro:eq:finite-packet}, and the final residual
correction described above.

Starting from a fixed large $x_{J-1}$ (denoted $x_0$ in
Section~\ref{sec:stage-parameters}), define $x_j$ recursively:
\[
 x_j=j^2x_{j-1}\quad(j\ge J).
\]
For stage $j\ge J$, let the physical localization length be
\[
 \ell_j=\exp\!\left(-\frac{x_{j-1}}{j^{7/2}}\right),
\]
the frequency in the scaled label be
\[
 k_j=\exp\!\left(\frac{x_{j-1}}{j^2}\right),
\]
the angular profile parameter be
\[
 \delta_j=\exp\!\left(-\frac{x_{j-1}}{j^3}\right),
\]
and the target shear strength be
\[
 h_j=\exp\!\left(\frac{x_{j-1}}{j^5}\right).
\]
The expansion \eqref{intro:eq:finite-packet} is truncated at order
$N_j=\lfloor k_j^\vartheta\rfloor$, with $\vartheta$ chosen as in
\eqref{intro:eq:frequency-separation}.
Their quantitative compatibility is verified in
Section~\ref{sec:stage-parameters}.

The seed $j=J-1$ is prescribed at time $t_{J-1}=0$, with $D_0=1000$ and
\begin{equation}\label{intro:eq:seed-scales}
 \ell_{J-1}=x_{J-1}^{-D_0},\quad k_{J-1}=x_{J-1}^{D_k},\quad
 \delta_{J-1}=x_{J-1}^{-(D_0+10)},\quad h_{J-1}=x_{J-1}^{D_0}.
\end{equation}
Its initial normal and velocity amplitudes have unit length, so choose
\[
 \alpha_{J-1}=\delta_{J-1}h_{J-1}=x_{J-1}^{-10}
\]
to make its initial leading shear equal to $h_{J-1}$.
The fixed constants $D_k,J,x_{J-1}$ are chosen in that order as
specified in Section~\ref{sec:stage-parameters}.

For $j\ge J$, activation occurs at $t_{j-1}$ and the target at $t_j$.
Figure~\ref{fig:stage-timeline} in Subsection~\ref{intro:subsec:time-intervals}
shows these times and the intervals used in the construction.
The prescribed target shear $h_j$ determines the amplitude by
\begin{equation}\label{intro:eq:amplitude-normalization}
 \alpha_j=\frac{\delta_jh_j}{|m_j(t_j,0)|\,|v_j(t_j,0)|}.
\end{equation}
The parent geometry determines $a_j,\beta_j$ in the following way:
the preceding layer's frame is
\[
 p_{j-1}=\frac{m_{j-1}}{|m_{j-1}|},\qquad
 q_{j-1}=\frac{v_{j-1}}{|v_{j-1}|},\qquad
 n_{j-1}=p_{j-1}\times q_{j-1},
\]
evaluated at the origin and at the new layer's activation time $t_{j-1}$.
Transversality $m_{j-1}\cdot v_{j-1}=0$ makes this an orthonormal frame:
$p_{j-1}$ is the
phase-normal direction and $q_{j-1}$ the velocity direction. The preceding
layer's leading gradient is the shear $h_{j-1}q_{j-1}p_{j-1}^T$,
which maps the $p_{j-1}$-direction into the $q_{j-1}$-direction.
Measure the older gradient in this frame by setting
\[
 B_j=\nabla U^{j-2}(t_{j-1},0),\qquad
 a_j=p_{j-1}\cdot B_jq_{j-1},\qquad
 \beta_j=\frac{n_{j-1}\cdot B_jq_{j-1}}{a_j}.
\]
The construction maintains the activation ranges
\begin{equation}\label{intro:eq:activation-ranges}
 \frac12\le a_j\le2,\qquad
 \frac{1}{2x_{j-1}^2}\le\beta_j\le\frac{2}{x_{j-1}^2}
 \qquad(j\ge J).
\end{equation}
In particular, both coefficients are positive. These are induction
requirements: Proposition~\ref{prop:seed-packet} supplies the seed values.
The target-frame estimate \eqref{eq:target-a-transfer} gives summable
relative changes in $a_j$, keeping it within its prescribed range.
For $\beta_j$, only the stated two-sided range of $\beta_jx_{j-1}^2$
is needed; the sharper and more delicate asymptotic \eqref{eq:target-beta-transfer}
restores this range at the next stage. These estimates preserve the
maintained activation bounds \eqref{eq:induction-activation}, as proved
in Proposition~\ref{prop:stage-induction}.
\subsection{The clocks and amplification}\label{intro:subsec:amplification-clocks}

The preceding layer's shear amplifies a more localized new layer.
Related layer cascades appear in forced Euler
\cite[\S1.2]{cordoba-martinez-zoroa2023}, forced hypodissipative
Navier--Stokes with sufficiently small dissipation order
\cite[\S1.2]{cordoba-martinez-zoroa-zheng2026}, and unforced Euler
with $C^{1,\alpha}\cap L^2$ initial velocity smooth away from the origin
\cite[\S1]{cordoba-martinez-zoroa-zheng2025}.

\textbf{The clock for velocity amplification.}
At the origin, temporarily suppress the layer subscripts
on the moving frame and the new wave. Write
\[
 M=B+hqp^T+E,\qquad B(t)=\nabla U^{j-2}(t,0),\qquad
 h(t_{j-1})=h_{j-1},\qquad B_{pq}(t_{j-1})=a_j.
\]
Here $E$ is the inherited gradient error and $B_{rs}=r\cdot Bs$ for
frame vectors $r,s$. The old wave equations give the frame identities
(derived in \eqref{eq:frame-components} below)
\[
 \dot p=-B_{pq}q-B_{pn}n,\qquad
 \dot q=B_{pq}p-B_{nq}n.
\]
For $v_p=p\cdot v$, $v_q=q\cdot v$, and
$\gamma=(m\cdot Mv)/|m|^2$, the product rule and \eqref{intro:eq:wave} yield
\begin{equation}\label{intro:eq:stage-component-motivation}
 \begin{aligned}
 \dot v_p&=-2B_{pq}v_q-B_{pp}v_p-2B_{pn}v_n
                   +2\gamma m_p-p\cdot Ev,\\
 \dot v_q&=-hv_p-(B_{qp}-B_{pq})v_p-B_{qq}v_q
                   -(B_{qn}+B_{nq})v_n+2\gamma m_q-q\cdot Ev.
 \end{aligned}
\end{equation}
There are two copies of $-B_{pq}v_q$ in the first equation: one from
$p\cdot\dot v$ and one from $\dot p\cdot v$. The term $-hv_p$ in the
second comes from the parent's strong shear.

At activation, the new normal is chosen so that
$m(t_{j-1},0)=s_0n$, with $s_0>0$. Thus $m_p=m_q=0$, and
transversality $m\cdot v=0$ gives $v_n=0$ there. In particular, the
pressure terms $2\gamma m_p,2\gamma m_q$ in these two equations vanish
initially. The pressure vector $2\gamma m$ itself need not vanish:
it points along $n$. These statements concern the origin; neighboring
labels have the activation errors controlled in Section~\ref{sec:history-activation}.

For the initial amplification balance, put
$\varepsilon_j=\sqrt{a_j/h_{j-1}}$ and view $v_p$ as order
$\varepsilon_j$ and $v_q$ as order $1$, relative to a common amplitude,
anticipating the growing-mode alignment explained below.
Freeze $h,B_{pq}$ at their activation values. Setting aside the inherited
error $E$ and the pressure and $v_n$ terms discussed above, the remaining
older-gradient terms in \eqref{intro:eq:stage-component-motivation} give
\[
 \dot v_p=-2a_jv_q+O(G_*\varepsilon_j),\qquad
 \dot v_q=-h_{j-1}v_p+O(G_*),\qquad |B|\le CG_*.
\]
Indeed, $B_{pp}v_p$ has size $O(G_*\varepsilon_j)$, while
$B_{qq}v_q$ and $(B_{qp}-B_{pq})v_p$ have sizes $O(G_*)$ and
$O(G_*\varepsilon_j)$, respectively. The retained terms have scales
$a_j$ and $h_{j-1}\varepsilon_j=a_j/\varepsilon_j$. After division
by these respective scales, both errors are
\[
 O\!\left(\frac{\varepsilon_jG_*}{a_j}\right)
 =O\!\left(\frac{G_*}{\sqrt{a_jh_{j-1}}}\right)\ll1.
\]
To see this from the parameter choices, take $G_*\lesssim h_{j-2}$
and use $a_j\in[1/2,2]$ from \eqref{intro:eq:activation-ranges}.
For $j\ge J+2$,
\[
 \log\frac{h_{j-2}}{\sqrt{h_{j-1}}}
 =-\frac{x_{j-2}}{2(j-1)^5}
   \left(1-\frac{2(j-1)^5}{(j-2)^7}\right)
 \le-\frac{x_{j-2}}{4(j-1)^5},
\]
for sufficiently large $J$. Thus the relative errors are exponentially
small. At the first two stages, the seed choices \eqref{intro:eq:seed-scales} give
\begin{equation}\label{intro:eq:seed-clock-errors}
 \frac{\varepsilon_jG_*}{a_j}\lesssim
 \begin{cases}
 x_{J-1}^{-500},&j=J,\\
 x_{J-1}^{1000}\exp\!\left(-\dfrac{x_{J-1}}{2J^5}\right),&j=J+1,
 \end{cases}
\end{equation}
which are also small. It is the
relative error that is small; the unscaled $O(G_*)$ need not be.

Dropping these terms gives the initial simplified system
\[
 \dot v_p\approx-2a_jv_q,\qquad
 \dot v_q\approx-h_{j-1}v_p,
 \qquad\Longrightarrow\qquad
 \ddot v_q\approx2a_jh_{j-1}v_q.
\]
\textbf{Exponential velocity growth from the two couplings.}
With the selected signs $v_q>0$ and $v_p\le0$, the term $-2a_jv_q$
makes $v_p$ more negative. The strong shear then contributes
$-h_{j-1}v_p>0$ to $\dot v_q$, increasing $v_q$ and strengthening the
first effect. If $v_p=0$ initially, the first equation immediately
creates the negative component. Shear alone would keep $v_p$ constant
and produce only linear growth of $v_q$ in this model; the return
coupling through $a_j$ produces exponential growth proportional to
$\exp\!\bigl(\sqrt{2a_jh_{j-1}}(t-t_{j-1})\bigr)$, motivating the clock
\begin{equation}\label{intro:eq:stage-clock}
 \tau_j(t)=\sqrt{a_jh_{j-1}}(t-t_{j-1})
 =\frac{a_j}{\varepsilon_j}(t-t_{j-1}).
\end{equation}
With $v_p=\varepsilon_jU$, $v_q=V$, the main equations become
$U'=-2V$, $V'=-U$, where primes denote $\tau_j$ derivatives.
Using the signs $v_q>0$, $v_p\le0$ above, normalize the activation data as
$v_q(t_{j-1})=1$ and $v_p(t_{j-1})=-\varepsilon_j\lambda$, with
$\lambda\ge0$. Since $\lambda$ may be large, the activation velocity
need not be nearly aligned with $q$.
Amplification in the simplified system also aligns the velocity,
however. Its explicit solution gives
\begin{equation}\label{intro:eq:velocity-alignment}
 \begin{aligned}
 \frac{v_p}{v_q}
 &=-\varepsilon_j
 \frac{\lambda\cosh(\sqrt2\tau_j)+\sqrt2\sinh(\sqrt2\tau_j)}
      {\cosh(\sqrt2\tau_j)+(\lambda/\sqrt2)\sinh(\sqrt2\tau_j)}\\
 &=-\sqrt2\,\varepsilon_j\bigl(1+O(e^{-2\sqrt2\tau_j})\bigr),
 \qquad \tau_j\ge1,
 \end{aligned}
\end{equation}
uniformly for $\lambda\ge0$. Thus, after an initial transient, the
$p$-component is small relative to the $q$-component, even though both
may grow in absolute size. This explains the scaling
$v_p=O(\varepsilon_jv_q)$ used above during the initial amplification
regime, before the turning of the normal substantially changes the system.

\textbf{The clock for turning the normal.}
The same moving-frame product rule, now applied to
$\dot m=-M^Tm$ from \eqref{intro:eq:normal}, gives
\begin{equation}\label{intro:eq:stage-normal-components}
 \begin{aligned}
 \dot m_p&=-hm_q-B_{pp}m_p-(B_{qp}+B_{pq})m_q
                  -(B_{np}+B_{pn})m_n-(Ep)\cdot m,\\
 \dot m_q&=-2B_{nq}m_n-B_{qq}m_q-(Eq)\cdot m,\\
 \dot m_n&=(B_{nq}-B_{qn})m_q-B_{nn}m_n-(En)\cdot m.
 \end{aligned}
\end{equation}
To compare the terms, write $m=s_0(Pp+\varepsilon_jQq+Nn)$,
with $(P,Q,N)=(0,0,1)$ at activation. On the $\tau_j$ clock, freezing $h$ and
$B_{nq}$ at $h_{j-1}$ and $a_j\beta_j$, and setting aside $E$, gives
\[
 \begin{aligned}
 P'&=-Q+O\!\left(\frac{\varepsilon_jG_*}{a_j}
                     (|P|+\varepsilon_j|Q|+|N|)\right),\\
 Q'&=-2\beta_jN+O\!\left(\frac{\varepsilon_jG_*}{a_j}|Q|\right),\\
 N'&=O\!\left(\frac{\varepsilon_jG_*}{a_j}
                     (\varepsilon_j|Q|+|N|)\right).
 \end{aligned}
\]
The inherited range \eqref{intro:eq:activation-ranges} gives
\[
 \frac12\le\beta_jx_{j-1}^2\le2,
 \qquad\text{so}\qquad \beta_j\asymp x_{j-1}^{-2}.
\]
The error coefficient in $Q'$ is small even relative to $\beta_j$:
using the parameter calculation above, for $j\ge J+2$,
\[
 \frac{\varepsilon_jG_*}{a_j\beta_j}
 \lesssim x_{j-1}^2\frac{h_{j-2}}{\sqrt{h_{j-1}}}
 \le Cx_{j-1}^2
       \exp\!\left(-\frac{x_{j-2}}{4(j-1)^5}\right)\ll1.
\]
Here $x_{j-1}=(j-1)^2x_{j-2}$, so the exponential beats the polynomial
factor. At the first two stages, multiplying the bounds in
\eqref{intro:eq:seed-clock-errors}
by $x_{j-1}^2$ still gives a quantity tending to zero.
Thus, while $Q$ is bounded and $N$ stays comparable to $1$,
the error in $Q'$ is $o(\beta_j)$ and $-2\beta_jN$ is the main term.
Controlling the errors until the normal turns also requires their
accumulation to be bounded. The
comparison bounds in Sections~\ref{sec:actual-to-ideal}--\ref{sec:quantitative-amplification} control their
accumulation, as well as $E$ and the variation of the frozen coefficients.
The prescribed scales supply that stronger smallness.

Dropping the errors displayed above gives
\[
 P'=-Q,\qquad Q'=-2\beta_jN,\qquad N'=0,
 \qquad (P,Q,N)(0)=(0,0,1).
\]
Equivalently, the retained physical-time terms are
$\dot m_q\approx-2a_j\beta_jm_n$ and
$\dot m_p\approx-h_{j-1}m_q$: the older gradient creates a
$q$-component, and the strong shear turns it into a $p$-component.
Integration gives
\begin{equation}\label{intro:eq:leading-normal-components}
 \frac{m_p}{s_0}\approx\beta_j\tau_j^2,\qquad
 \frac{m_q}{s_0}\approx-2\varepsilon_j\beta_j\tau_j,\qquad
 \frac{m_n}{s_0}\approx1.
\end{equation}
Thus the $p$-component becomes comparable to the original $n$-component
when $\beta_j\tau_j^2$ reaches order one. This motivates
\begin{equation}\label{intro:eq:stage-turning-clock}
 x=\sqrt{\beta_j}\,\tau_j(t),\qquad \frac{m_p}{s_0}\approx x^2.
\end{equation}
For $x\ll1$ the normal remains close to $n$; near $x=1$ it turns
substantially, and the pressure projection changes the velocity growth
regime.

\textbf{The frame supplied to the next layer.}
The leading order geometry explains how the construction can repeat. Relative
to the preceding moving frame $(p,q,n)$, the new normal starts along
$n$. By \eqref{intro:eq:leading-normal-components}, at the large target
$x=x_j$ (see Figure~\ref{fig:stage-timeline}) its $p$-component dominates, so its unit direction $p_j$ is
close to $p$. Meanwhile, the velocity amplification in
\eqref{intro:eq:velocity-alignment} aligns $v_j$ toward $q$ in the
initial growth regime. Together, these mechanisms suggest that the
target frame $(p_j,q_j,n_j)$ is close to $(p,q,n)$, with
$q_j=v_j/|v_j|$ and $n_j=p_j\times q_j$.
The new layer thus supplies a frame close to the preceding one, together with the larger
shear $h_jq_jp_j^T$. The next layer starts its normal along $n_j$ and
repeats the turn toward $p_j$ and the velocity alignment toward $q_j$.

The later turning prepares the frame for the next layer. The small deviations
from alignment must still be controlled to keep $a_j$ and $\beta_j$ in
their prescribed ranges. In particular, obtaining a small positive
$\beta_{j+1}$, and hence another long amplification interval before the
next turn, relies on subtle cancellations too complicated to outline here.
Proposition~\ref{prop:target-frame} supplies the
quantitative estimates \eqref{eq:target-a-transfer} and
\eqref{eq:target-beta-transfer}; Proposition~\ref{prop:stage-induction}
uses them to close the induction.

\subsection{The stage timeline and interval containment}\label{intro:subsec:time-intervals}
\refstepcounter{figure}\label{fig:stage-timeline}
\begin{center}
\begin{tikzpicture}[x=0.9cm,y=0.65cm,font=\small]
 \draw[->] (0,0)--(10.5,0) node[right] {$t$};
 \foreach \x/\lab in {0/0,1.8/t_{j-1},4.4/\widehat t_j,9.6/S_{j-1}} {
   \draw (\x,-0.1)--(\x,0.1);
   \node[above] at (\x,0.12) {$\lab$};
   \draw[gray!45,densely dotted] (\x,-0.2)--(\x,-5.3);
 }
 \foreach \x in {3.1,3.3} {
   \draw (\x,-0.1)--(\x,0.1);
   \draw[gray!45,densely dotted] (\x,-0.2)--(\x,-5.3);
 }
 \node[above left] at (3.1,0.12) {$t_j$};
 \node[above right] at (3.3,0.12) {$S_j$};
 \draw[thick,dashed] (3.2,0.8)--(3.2,-5.3);
 \node[above] at (3.2,0.8) {$T$};
 \node[anchor=east] at (-0.3,-1) {Wave history};
 \draw[line width=1.3pt] (0,-1)--(1.8,-1);
 \fill (0,-1) circle (1.6pt); \fill (1.8,-1) circle (1.6pt);
 \node[anchor=east] at (-0.3,-2) {Activation to target of $U^j$};
 \draw[line width=1.3pt] (1.8,-2)--(3.1,-2);
 \fill (1.8,-2) circle (1.6pt); \fill (3.1,-2) circle (1.6pt);
 \node[anchor=east] at (-0.3,-3) {Child $U^j$};
 \draw[line width=1.3pt] (0,-3)--(3.3,-3);
 \fill (0,-3) circle (1.6pt); \fill (3.3,-3) circle (1.6pt);
 \node[anchor=east] at (-0.3,-4) {Wave-model comparison $I_j$};
 \draw[line width=1.3pt] (1.8,-4)--(4.4,-4);
 \fill (1.8,-4) circle (1.6pt); \fill (4.4,-4) circle (1.6pt);
 \node[anchor=east] at (-0.3,-5) {Parent $U^{j-1}$};
 \draw[line width=1.3pt] (0,-5)--(9.6,-5);
 \fill (0,-5) circle (1.6pt); \fill (9.6,-5) circle (1.6pt);
\end{tikzpicture}
\end{center}
\noindent\textbf{Figure~\thefigure.} All rows use the same axis. For $a_j\approx1$ and
$\beta_jx_{j-1}^2\approx1$, the positions of
$t_j,\widehat t_j,S_{j-1}$ are approximately $1,2,6$ in units of
$L_j=x_jx_{j-1}/\sqrt{a_jh_{j-1}}$ from activation at $t_{j-1}$.
The history interval is compressed, and the much smaller gaps between
$t_j,T,S_j$ are enlarged for visibility.
The intervals are explained in the row order of the diagram.

\emph{Wave history.} The child's history $[0,t_{j-1}]$ is controlled
by the history construction. At the first normal stage, $t_{J-1}=0$
and this interval reduces to a point.

\emph{Activation to target.} The target value of this clock is prescribed to be $x_j$, so
\begin{equation}
 \sqrt{\beta_j}\,\tau_j(t_j)=x_j,\qquad
 t_j-t_{j-1}=\frac{x_j}{\sqrt{\beta_ja_jh_{j-1}}}.
 \label{eq:actual-target-time}
\end{equation}
Here $x$ and $x_j$ are dimensionless clock values, not spatial coordinates.

\emph{Child $U^j$.} Let
\[
 S_j=t_j+2W_{j+1},\qquad
 W_j=\frac{3x_jx_{j-1}}{\sqrt{h_{j-1}}}
\]
be the physical endpoint through which the child solution is constructed
to serve as the parent of the next layer. Its interval is $[0,S_j]$.

\emph{Wave-model comparison $I_j$.} Let
\[
 \Theta_j=2x_jx_{j-1}
\]
be the bound on the scaled time on which the wave ODEs are compared.
The comparison interval is
$I_j=[t_{j-1},\widehat t_j]$, where
$\widehat t_j=t_{j-1}+\varepsilon_j\Theta_j/a_j$.
The bound $\Theta_j$ is not another target: we require
$\tau_j(S_j)\le\Theta_j\le\tau_j(S_{j-1})$, so the comparison
covers the child's retained interval while staying within the interval
on which the parent is available. These two inequalities are proved in
Lemmas~\ref{intro:lem:clock-child} and~\ref{intro:lem:clock-parent}, respectively.
The bound $\Theta_j$ also controls the time factors
in the wave-error estimates.

\begin{introlemma}[The comparison covers the child interval]\label{intro:lem:clock-child}
With the parameter choices of Subsection~\ref{intro:subsec:stage-parameters}, assume
$1/2\le a_j\le2$ and $\beta_jx_{j-1}^2\ge1/2$, as supplied by
\eqref{intro:eq:activation-ranges}.
For sufficiently large $J$ and then sufficiently large $x_{J-1}$,
\[
 \tau_j(S_j)\le\Theta_j,\qquad
 0\le\tau_j(S_j)-\tau_j(t_j)\le1\qquad(j\ge J).
\]
\end{introlemma}
\begin{proof}
The target relation and $\beta_jx_{j-1}^2\ge1/2$ give
\[
 \tau_j(t_j)=\frac{x_j}{\sqrt{\beta_j}}
 \le\sqrt2\,x_jx_{j-1}.
\]
The additional scaled time after the target is
\[
 \tau_j(S_j)-\tau_j(t_j)
 =2\sqrt{a_jh_{j-1}}W_{j+1}
 =6\sqrt{a_j}\,x_{j+1}x_j\sqrt{\frac{h_{j-1}}{h_j}}.
\]
For $j>J$, the logarithm of the square-root factor is
\[
 \frac12\log\frac{h_{j-1}}{h_j}
 =\frac{x_{j-1}}{2(j-1)^7}-\frac{x_{j-1}}{2j^5}
 \le-\frac{x_{j-1}}{4j^5}
\]
for sufficiently large $J$. Its exponential dominates the factor
$x_{j+1}x_j$, so the additional scaled time is at most $1$.
For $j=J$, \eqref{intro:eq:seed-scales} gives instead
$(D_0/2)\log x_{J-1}-x_{J-1}/(2J^5)$, with the same conclusion
for sufficiently large $x_{J-1}$. Therefore
\[
 \tau_j(S_j)\le\sqrt2\,x_jx_{j-1}+1
 \le2x_jx_{j-1}=\Theta_j.
\]
\end{proof}

\emph{Parent $U^{j-1}$.} The parent is already constructed on
$[0,S_{j-1}]$, with $S_{j-1}=t_{j-1}+2W_j$.
For the seed this gives
\[
 S_{J-1}=2W_J=6J^2x_{J-1}^{2-D_0/2}.
\]
Although the comparison endpoint $\widehat t_j$ may exceed $S_j$,
the wave ODEs use the already constructed parent $U^{j-1}$, which
exists through $S_{j-1}$. The next lemma ensures that this parent
is available throughout $I_j$.

\begin{introlemma}[The comparison stays within the parent interval]\label{intro:lem:clock-parent}
If $a_j\ge1/2$, as in \eqref{intro:eq:activation-ranges}, then
\[
 \Theta_j\le\tau_j(S_{j-1}).
\]
\end{introlemma}
\begin{proof}
Since $a_j\ge1/2$, the definitions give
\[
 \frac{\varepsilon_j\Theta_j}{a_j}
 =\frac{2x_jx_{j-1}}{\sqrt{a_jh_{j-1}}}
  \le\frac{2\sqrt2}{3}W_j<2W_j=S_{j-1}-t_{j-1}.
\]
Multiplying by $a_j/\varepsilon_j$ gives
$\Theta_j<\tau_j(S_{j-1})$, which implies the stated bound.
\end{proof}

The dashed line marks the accumulation time $T=\lim_{i\to\infty}t_i$.
The scale choices give $W_{i+1}\le W_i/4$ and
$t_i-t_{i-1}\le2W_i/3$, so
\[
 0<T-t_j\le\frac23\sum_{i=j+1}^{\infty}W_i
 \le\frac89W_{j+1}<2W_{j+1}=S_j-t_j.
\]
Thus $T<S_j$: each child's retained interval $[0,S_j]$ reaches beyond
the accumulation time. To compare the remaining time with the preceding
stage duration, use \eqref{eq:actual-target-time} and the ranges in
\eqref{intro:eq:activation-ranges},
\[
 t_j-t_{j-1}
 =\frac{W_j}{3\sqrt{a_j(\beta_jx_{j-1}^2)}}
 \in\left[\frac{W_j}{6},\frac{2W_j}{3}\right].
\]
Consequently,
\[
 \frac{T-t_j}{t_j-t_{j-1}}
 \le\frac{16}{3}\frac{W_{j+1}}{W_j}\ll1.
\]
Thus the remaining time is much shorter than the stage just completed:
$T$ lies close to $t_j$. The eventual blowup argument
gives a maximal smooth existence time no greater than $T$; it need not
equal $T$.

Since $\tau_j$ is increasing, Lemmas~\ref{intro:lem:clock-child}
and~\ref{intro:lem:clock-parent} give the second and third containments below;
together with $t_j\le S_j$, they imply
\[
 [t_{j-1},t_j]\subseteq[t_{j-1},S_j]\subseteq I_j
 \subseteq[t_{j-1},S_{j-1}],\qquad
 [0,S_j]\subseteq[0,S_{j-1}].
\]

\section{From the wave equations to amplification}
\label{sec:amplification}

We now continue from \eqref{intro:eq:normal} and \eqref{intro:eq:wave} to the
\emph{ideal} amplification system \eqref{eq:ideal-amplification-system}.
The route is to differentiate moving-frame components, rescale, and retain the
leading coefficients. The moving-frame calculation is exact; passing to the
ideal coefficients is an approximation controlled in
Subsection~\ref{subsec:ideal-errors} and
Section~\ref{sec:comparison-proofs}.

\textbf{Notation in this section.}
We suppress the stage index in the introductory clocks:
$a=a_j$, $\beta=\beta_j$, $\eps=\varepsilon_j$, and $t_0=t_{j-1}$.
Thus $\eps=\sqrt{a/h_*}$ is the amplification scale. The letters $U,V,W$ are scalar velocity components;
the vector profile $W$ of the Introduction will not be used here.
A dot denotes physical-time differentiation at the origin;
a prime denotes rescaled-time differentiation.

\subsection{The parent gradient and its moving frame}

Let $U^{j-1}$ be the full parent Euler velocity and $U^{j-2}$ the older
Euler solution. The increment is $U^{j-1}-U^{j-2}$.
Let $m_{j-1}$ and $v_{j-1}$ denote that increment's leading phase normal and
velocity-amplitude coefficient. Since the parent $U^{j-1}$ is odd, the origin stays still: $X(t,0)=0$. On the interval where the increment supplies the
parent shear, define
\begin{equation}
\begin{aligned}
  M(t)&=\nabla U^{j-1}(t,0),&
  B(t)&=\nabla U^{j-2}(t,0),\\
  p(t)&=\frac{m_{j-1}(t,0)}{|m_{j-1}(t,0)|},&
  q(t)&=\frac{v_{j-1}(t,0)}{|v_{j-1}(t,0)|}.
\end{aligned}
\label{eq:inherited-frame}
\end{equation}
The vectors $p,q$ are the unit phase-normal and transverse-velocity
directions of the previous leading increment already contained in the parent.
The variables $m=m_j$ and $v=v_j$ are the phase normal and velocity
amplitude of the new leading increment.

Thus the leading gradient of the increment at the origin is
\begin{equation}
  \frac{\alpha_{j-1}}{\delta_{j-1}}\,
     v_{j-1}\otimes m_{j-1}=h(t)q(t)p(t)^T,
  \qquad
  h(t)=\frac{\alpha_{j-1}}{\delta_{j-1}}
        |m_{j-1}|\,|v_{j-1}|.
  \label{eq:inherited-shear-size}
\end{equation}
We can then write
\begin{equation}
  M=B+hqp^T+E,\qquad h(t_0)=h_*.
  \label{eq:parent-shear}
\end{equation}

Here $h$ is the strength of that inherited shear,
$B$ is the older flow's gradient, and $E$ is the inherited gradient
remainder; see \eqref{eq:gradient} and
Section~\ref{sec:actual-to-ideal}.
All quantities in \eqref{eq:inherited-shear-size} are evaluated at $(t,0)$,
where $\chi_1(0)=1$ and $f'_{\delta_{j-1}}(0)=\delta_{j-1}^{-1}$.
Although the wave coefficients depend on time and particle label, we
always evaluate them at the stationary trajectory at the origin, so we
suppress their arguments $(t,0)$, for both the previous and the new
leading increments.
For the first stage, the seed supplies this shear; see
Subsection~\ref{intro:subsec:stage-parameters}.

\textbf{Evolution of the inherited frame and shear.}
At the origin, the previous leading increment's wave coefficients
satisfy \eqref{intro:eq:normal} and \eqref{intro:eq:wave} with the older
gradient $B(t)=\nabla U^{j-2}(t,0)$ in place of $M(t)$.
Their normalized directions obey $p\cdot q=0$. With $n=p\times q$, write
$B_{ij}=i\cdot Bj$ for $i,j\in\{p,q,n\}$. In particular, $p$ is a unit vector,
rather than scalar pressure. Taking the scalar product of each old wave equation
\eqref{intro:eq:normal}, \eqref{intro:eq:wave} (with $M=B$)
with its own coefficient gives
\begin{equation}
  \frac{d}{dt}\log|m_{j-1}|=-B_{pp},\qquad
  \frac{d}{dt}\log|v_{j-1}|=-B_{qq},\qquad
  \frac{\dot h}{h}=-B_{pp}-B_{qq}.
  \label{eq:old-wave-norms}
\end{equation}
The third identity follows by adding the first two, since
$h=(\alpha_{j-1}/\delta_{j-1})|m_{j-1}|\,|v_{j-1}|$ and the prefactor
is constant. It will control the variation of $h$ when we freeze the
shear coefficient in \eqref{eq:frozen-coefficients}.
Differentiating the normalizations in \eqref{eq:inherited-frame} therefore
gives the frame equations:
\begin{equation}
  \dot p=-B^Tp+B_{pp}p,\qquad
  \dot q=-Bq+2B_{pq}p+B_{qq}q.
\label{eq:inherited-frame-evolution}
\end{equation}
Expanding in the frame, and differentiating $n=p\times q$, gives
\begin{equation}
  \dot p=-B_{pq}q-B_{pn}n,\qquad
  \dot q=B_{pq}p-B_{nq}n,\qquad
  \dot n=B_{pn}p+B_{nq}q.
  \label{eq:frame-components}
\end{equation}

The new leading increment's wave coefficients $(m,v)=(m_j,v_j)$ evolve under the
full parent gradient $M=B+hqp^T+E$. At activation,
$m(t_0,0)=s_0n(t_0)$.

For use at stage $j+1$, the new leading increment supplies the next directions
\begin{equation}
  p_{\mathrm{next}}(t)=\frac{m_j(t,0)}{|m_j(t,0)|},\qquad
  q_{\mathrm{next}}(t)=\frac{v_j(t,0)}{|v_j(t,0)|}.
  \label{eq:next-layer-frame}
\end{equation}
On the retained interval after its target, these become the next inherited
frame, with older gradient $B_{\mathrm{next}}(t)=\nabla U^{j-1}(t,0)=M(t)$ and
parent gradient $M_{\mathrm{next}}(t)=\nabla U^j(t,0)$
as in Subsection~\ref{intro:subsec:stage-parameters}; the quantitative
frame transfer is proved in Proposition~\ref{prop:target-frame}. Thus $p,q$ are produced by the stage-$(j-1)$ leading increment;
normalizing the present $m,v$ gives the frame for the
\emph{next} calculation.

The exact moving-frame velocity identities, including $E$, are
\eqref{intro:eq:stage-component-motivation}. They follow by applying
$\dot v_i=\dot i\cdot v+i\cdot\dot v$ and
\eqref{eq:frame-components}; in particular the two copies of
$-B_{pq}v_q$ arise from the moving frame and from $-Mv$.
We now set $E=0$ while deriving the ideal model. Its contribution is
restored in the comparison estimates.

\subsection{Rescale time, the ray, and the velocity}

Recall the parameters of Subsection~\ref{intro:subsec:stage-parameters}
and the clock \eqref{intro:eq:stage-clock}, suppressing the stage index.
Their positivity follows from \eqref{intro:eq:activation-ranges}:
\begin{equation}
  a=B_{pq}(t_0)>0,\qquad
  \beta=\frac{B_{nq}(t_0)}{a}>0,\qquad
  \eps=\sqrt{\frac{a}{h_*}},\qquad
  \tau=\frac{a}{\eps}(t-t_0).
  \label{eq:amplification-scales}
\end{equation}
Thus $d/dt=(a/\eps)d/d\tau$ and $\eps^2h_*/a=1$.
With the constant $s_0=|m(t_0,0)|$, write
\begin{equation}
  m=s_0(Pp+\eps Qq+Nn),\qquad
  v=\eps Up+Vq+\eps Wn.
  \label{eq:scaled-components}
\end{equation}
Transversality becomes
\begin{equation}
  PU+QV+NW=0,\qquad
  W=-\frac{PU+QV}{N}\quad(N\ne0).
  \label{eq:scaled-transversality}
\end{equation}
Consequently there are only two independent velocity components.
The scalars $P,Q,N$ and $U,V,W$ describe the \emph{new} layer in the
\emph{previous} layer's moving frame; $p,q$ retain the meanings in
\eqref{eq:inherited-frame} throughout this rescaling.
At activation ($\tau=0$, hence physical time $t=t_0$) and the origin
($y=0$), the complete scaled data are
\begin{equation}
  (P,Q,N)(0)=(0,0,1),\qquad
  (U,V,W)(0)=(-\lambda,1,0),\qquad \lambda\ge0.
\label{eq:scaled-activation}
\end{equation}
The normal is chosen by \eqref{eq:history-normal-choice};
Lemma~\ref{lem:history-endpoint-choice} supplies the velocity signs
and normalization. Then \eqref{eq:scaled-components} gives these data,
with $\lambda=-v_p(t_0,0)/\eps\ge0$ and $W(0)=0$ by transversality.


The corresponding normal-component identities are
\eqref{intro:eq:stage-normal-components}, obtained from
$\dot m_i=\dot i\cdot m-(Mi)\cdot m$. Set $E=0$ in those identities
and insert \eqref{eq:scaled-components}.
For example, the shear contribution to $P'$ is
$-(\eps^2h/a)Q$, and the leading contribution to $Q'$ is
$-2(B_{nq}/a)N$. The remaining coefficients carry small scale factors or
controlled errors, see Subsection~\ref{subsec:ideal-errors} below,
which also bounds the error in freezing the principal coefficients.
The leading coefficient replacements are
\begin{equation}
  \frac{\eps^2h}{a}\ \rightsquigarrow\ 1,\qquad
  \frac{B_{pq}}a\ \rightsquigarrow\ 1,\qquad
  \frac{B_{nq}}a\ \rightsquigarrow\ \beta.
  \label{eq:frozen-coefficients}
\end{equation}
Here $\rightsquigarrow$ means ``replace by the ideal coefficient,'' not an
exact equality.

The resulting ideal equations are
\begin{equation}
  P_0'=-Q_0,\qquad Q_0'=-2\beta N_0,\qquad N_0'=0,
  \qquad (P_0,Q_0,N_0)(0)=(0,0,1).
  \label{eq:ideal-ray-system}
\end{equation}
Integration gives
\begin{equation}
  P_0=\beta\tau^2,\qquad Q_0=-2\beta\tau,\qquad N_0=1.
\label{eq:ideal-ray}
\end{equation}
The comparison bound \eqref{eq:actual-ray-comparison} keeps the actual
$N\ge1/2$, as required in
\eqref{eq:scaled-transversality}.

\subsection{Retain the pressure projection and the shear coupling}

The pressure-related term in the amplitude equation cannot be discarded.
With $\gamma=(m\cdot Mv)/|m|^2$, introduce the exact scaled quantities
\begin{equation}
  D_a=\frac{|m|^2}{s_0^2}=P^2+\eps^2Q^2+N^2,\qquad
  J=\frac{m\cdot Mv}{s_0a},\qquad
  \gamma=\frac{aJ}{s_0D_a}.
  \label{eq:scaled-pressure-quantities}
\end{equation}
Expanding $m\cdot Mv$ using \eqref{eq:scaled-components}, the three terms that
survive the scaling are
\begin{equation}
  J=
  \underbrace{\frac{B_{pq}}aPV}_{\text{$p$-row of $Bv$}}
  +\underbrace{\frac{\eps^2h}{a}QU}_{\text{parent shear}}
  +\underbrace{\frac{B_{nq}}aNV}_{\text{$n$-row of $Bv$}}
  +\mathcal R_J.
  \label{eq:three-pressure-contributions}
\end{equation}
The remainder $\mathcal R_J$ contains the other coefficient terms, including
$E$ when restored. After eliminating $W$ by transversality
\eqref{eq:scaled-transversality}, the explicit remainder
\eqref{eq:short-J-remainder} and its bound \eqref{eq:short-J-bound}
control $\mathcal R_J$ on short scaled intervals.
Section~\ref{sec:comparison-proofs} treats the full comparison interval,
including the normal error \eqref{eq:actual-ray-comparison} and the
resulting pressure and velocity coefficient errors.
Thus the ideal replacements are
\begin{equation}
  J\ \rightsquigarrow\ J_0:=(P_0+\beta)V+Q_0U,\qquad
  D_a\ \rightsquigarrow\ D:=1+P_0^2.
  \label{eq:ideal-pressure-quantities}
\end{equation}

It remains to collect the contributions to the two velocity derivatives.
Since $v_p=\eps U$ and $v_q=V$, their physical derivatives are
$\dot v_p=aU'$ and $\dot v_q=(a/\eps)V'$.
Dividing \eqref{intro:eq:stage-component-motivation} by $a$, the retained nonprojected term is
\[
  -\frac{2B_{pq}v_q}{a}\ \rightsquigarrow\ -2V,
\]
whereas the projected term is
\begin{equation}
  \frac{2\gamma m_p}{a}=\frac{2PJ}{D_a}
  \ \rightsquigarrow\
  \frac{2P_0\bigl[(P_0+\beta)V+Q_0U\bigr]}{1+P_0^2}.
  \label{eq:p-projection-survives}
\end{equation}
The omitted nonprojected terms have coefficients of size $O(\eta)$;
Subsection~\ref{subsec:ideal-errors} defines $\eta$ and proves this bound
on fixed scaled intervals.
Multiplying \eqref{intro:eq:stage-component-motivation} by $\eps/a$, the retained shear term is
\[
  -\frac{\eps h v_p}{a}=-\frac{\eps^2h}{a}U
  \ \rightsquigarrow\ -U.
\]
In contrast, the projected term in that component is
\begin{equation}
  \frac{\eps}{a}\,2\gamma m_q
      =\frac{2\eps^2QJ}{D_a}\ \rightsquigarrow\ 0.
  \label{eq:q-projection-lower-order}
\end{equation}
The extra factor $\eps^2$ places this term among the controlled coefficient
errors. Subsection~\ref{subsec:ideal-errors} below gives explicit short-time
bounds for this projected term and all the remaining background terms
omitted from the two equations in
\eqref{intro:eq:stage-component-motivation} with $E=0$.

\needspace{10\baselineskip}
Denoting the ideal velocity components again by $U,V$, collecting only these
retained contributions gives the ideal amplification system
\begin{equation}
\begin{aligned}
  U'&=-2V+
    \frac{2P_0\bigl[(P_0+\beta)V+Q_0U\bigr]}{1+P_0^2},\\
  V'&=-U,
\end{aligned}
\qquad
\begin{gathered}
  P_0=\beta\tau^2,\quad Q_0=-2\beta\tau,\\
  U(0)=-\lambda,\quad V(0)=1.
\end{gathered}
\label{eq:ideal-amplification-system}
\end{equation}
The equalities in
\eqref{eq:ideal-amplification-system} are exact for the ideal model; comparing
its solutions with the physical scaled components is the subject of
Section~\ref{sec:comparison-proofs}.

\textbf{The third component is reconstructed, not set to zero for all time.}
Once $U,V$ are known, transversality determines $W$. In the ideal model,
$N_0=1$, so
\begin{equation}
  W=-P_0U-Q_0V=-\beta\tau^2U+2\beta\tau V.
  \label{eq:ideal-W-reconstruction}
\end{equation}
In particular, $W(0)=0$ but $W'(0)=2\beta V(0)=2\beta>0$ in this model.
A nonzero component along the inherited $n$ can therefore develop while
$v$ remains perpendicular to the evolving new normal $m$.

\subsection{Eliminate one component to expose the amplification equation}

Set $D=1+P_0^2$. Since $U=-V'$, the first equation of
\eqref{eq:ideal-amplification-system} gives
\begin{align*}
  -V''&=-2V+
    \frac{2P_0\bigl[(P_0+\beta)V-Q_0V'\bigr]}{D},\\
  DV''&=2(1-\beta P_0)V+2P_0Q_0V'.
\end{align*}
But $P_0'=-Q_0$, so $D'=2P_0P_0'=-2P_0Q_0$. Therefore
\begin{equation}
(DV')'=2(1-\beta P_0)V.
\label{eq:scalar-amplification-compact}
\end{equation}
Equivalently, inserting the explicit value of $P_0$,
\begin{equation}
  \frac{d}{d\tau}\left[(1+\beta^2\tau^4)\frac{dV}{d\tau}\right]
     =2(1-\beta^2\tau^2)V,
  \qquad V(0)=1,\quad V'(0)=\lambda\ge0.
\label{eq:scalar-amplification}
\end{equation}
This scalar equation is used for the growth estimates in
Section~\ref{sec:quantitative-amplification}. The two essential ingredients in reaching it
are the \emph{frame rotation}, which supplies the second copy in the $-2V$
coupling, and the \emph{incompressibility projection}, which supplies the rational term in \eqref{eq:ideal-amplification-system}.

\subsection{Errors in the ideal approximation}
\label{subsec:ideal-errors}

The component calculation above identifies the ideal system. We now
record the coefficient estimates that justify the discarded terms.
The scale separation is motivated in
Subsection~\ref{intro:subsec:amplification-clocks}; here we track its
consequences for the pressure numerator and the velocity equations.

At stage $j$, let $h_*=h_{j-1}$ and assume the older background satisfies
$\|B(t)\|\lesssim G_*$ throughout the interval considered, with
$G_*=1+h_{j-2}$ and $a\asymp1$. Define
\begin{equation}
 \eta=\frac{\eps G_*}{a}
      =\frac{G_*}{\sqrt{ah_*}}
      \asymp\frac{1+h_{j-2}}{\sqrt{h_{j-1}}}.
 \label{eq:short-eta}
\end{equation}
The parameter choices in Subsection~\ref{intro:subsec:stage-parameters}
make this ratio small, as calculated in
Subsection~\ref{intro:subsec:amplification-clocks}.
The unspecified entries of $B$ have scaled coefficients
$\eps B_{ij}/a=O(\eta)$ or $\eps^2B_{ij}/a=O(\eps\eta)$,
whereas the principal shear coefficient is $\eps^2h_*/a=1$.

On a fixed scaled interval $0\le\tau\le T$, with $T$ independent
of the parameters, first set $E=0$ and freeze the principal coefficients:
$B_{pq}/a=1$, $B_{nq}/a=\beta$, and $\eps^2h/a=1$.
Constants $C_T$ may depend on $T$; the other coefficients are bounded by
$CG_*$. The cost of freezing is estimated below.

Recall $J=(m\cdot Mv)/(s_0a)$ from \eqref{eq:scaled-pressure-quantities}. Expanding by rows of $M=B+hqp^T$ gives
\begin{equation}
 \begin{aligned}
 J={}&\frac Pa(\eps B_{pp}U+B_{pq}V+\eps B_{pn}W)\\
 &+\frac{\eps Q}{a}
       \bigl(\eps(h+B_{qp})U+B_{qq}V+\eps B_{qn}W\bigr)\\
 &+\frac Na(\eps B_{np}U+B_{nq}V+\eps B_{nn}W).
 \end{aligned}
 \label{eq:short-J-full}
\end{equation}
Consequently $J=PV+QU+\beta NV+\mathcal R_J$, where
\begin{equation}
 \begin{aligned}
 \mathcal R_J={}&\frac\eps a
  \bigl(PB_{pp}U+PB_{pn}W+QB_{qq}V
                       +NB_{np}U+NB_{nn}W\bigr)\\
 &+\frac{\eps^2}{a}Q(B_{qp}U+B_{qn}W).
 \end{aligned}
 \label{eq:short-J-remainder}
\end{equation}
Every term in this remainder has a small coefficient. On a fixed scaled
interval, the normal comparison \eqref{eq:actual-ray-comparison},
under the hypotheses of Section~\ref{sec:comparison-proofs}, gives $|P|+|Q|+|N|\le C_T$ and $N\ge1/2$
for sufficiently small errors. Transversality then implies
$|W|\le C_T(|U|+|V|)$, and hence
\begin{equation}
 |\mathcal R_J|\le C_T\eta(|U|+|V|),\qquad
 |J|\le C_T(|U|+|V|),\qquad D_a\ge1/4.
 \label{eq:short-J-bound}
\end{equation}
These are bounds on coefficients of the linear velocity equation. They
do not assume that $U,V$ are order one; in particular they allow the
history-selected slope $\lambda$ to be large.

Before dropping terms, the scaled versions of the two equations in
\eqref{intro:eq:stage-component-motivation}, with $E=0$, are
\begin{equation}
 \begin{aligned}
 U'={}&-2V+\frac{2PJ}{D_a}
       -\frac\eps aB_{pp}U-\frac{2\eps}{a}B_{pn}W,\\
 V'={}&-U+\frac{2\eps^2QJ}{D_a}
       -\frac{\eps^2}{a}(B_{qp}-B_{pq})U\\
      &-\frac\eps aB_{qq}V
       -\frac{\eps^2}{a}(B_{qn}+B_{nq})W.
 \end{aligned}
 \label{eq:short-velocity-full}
\end{equation}
The nonprojected background terms are bounded by
$C_T\eta(|U|+|V|)$. The projected term in $V'$ has the additional
factor $\eps^2$:
\[
 \left|\frac{2\eps^2QJ}{D_a}\right|
       \le C_T\eps^2(|U|+|V|).
\]
There is no such factor in $2PJ/D_a$, so the projection in $U'$ must
remain. Substituting \eqref{eq:short-J-bound} yields
\begin{equation}
 \begin{aligned}
 U'&=-2V+\frac{2P(PV+QU+\beta NV)}{D_a}+r_U,\\
 V'&=-U+r_V,\\
 |r_U|&\le C_T\eta(|U|+|V|),\qquad
 |r_V|\le C_T(\eta+\eps^2)(|U|+|V|).
 \end{aligned}
 \label{eq:short-velocity-errors}
\end{equation}
Replacing the ray by its ideal values and $D_a$ by $1+P_0^2$
then gives \eqref{eq:ideal-amplification-system}.

For the actual time-dependent coefficients, measure the freezing error by
\begin{equation}
 d_f=\sup_{0\le\tau\le T}
 \left(\left|\frac{B_{pq}}a-1\right|
       +\left|\frac{B_{nq}}a-\beta\right|
       +\left|\frac{h}{h_*}-1\right|\right).
 \label{eq:short-freezing-error}
\end{equation}
The moving entries of $B$ have physical-time derivatives bounded by
$CG_*^2$. The shear equation \eqref{eq:old-wave-norms} gives a
logarithmic derivative bounded by $CG_*$. Thus, when the latter
variation is small,
\begin{equation}
 d_f\le C_T\left(\frac{\eps G_*^2}{a^2}
                         +\frac{\eps G_*}{a}\right).
 \label{eq:short-freezing-bound}
\end{equation}
Small $\eta$ controls the explicitly suppressed coefficients; small
$d_f$ is an additional requirement for replacing the varying principal
coefficients by constants. The parameter estimates in Section~\ref{sec:stage-parameters}
ensure both requirements.

The coefficient estimates above explain the ideal approximation on a
fixed scaled interval. The normal comparison, replacement of $D_a$
by $1+P_0^2$, and velocity solution comparison on the full amplification
interval are proved in Section~\ref{sec:comparison-proofs}, including
the inherited correction $E$ and the spatial errors introduced in
Section~\ref{sec:actual-to-ideal}.

These are absolute coefficient estimates, not pointwise relative
estimates for $J$. At activation the ideal value is $J=\beta V$,
which is small. Preserving its positive sign requires errors small
relative to $\beta$, beyond merely controlling the absolute error;
see Section~\ref{sec:quantitative-amplification}.

\section{Quantitative inputs for the wave-model comparison}
\label{sec:actual-to-ideal}
\fancyhead[R]{\small Inputs for the wave-model comparison}

Section~\ref{sec:amplification} identifies the ideal equations and their
coefficient errors. We now supply the uniform parent and activation-data
bounds needed to compare the actual wave with that model on the full
amplification interval. The parent bounds are conditional inputs from
the preceding stages; the lemmas below quantify their consequences.

\subsection{Parent inputs and comparison interval}
\label{subsec:wave-parent-inputs}

Fix stage $j$ with exact parent $u=U^{j-1}$. Use the frame and
shear decomposition \eqref{eq:inherited-frame}--\eqref{eq:parent-shear}
and the frame and shear identities \eqref{eq:inherited-frame-evolution}
and \eqref{eq:old-wave-norms}, all evaluated at the origin. In the notation of
Subsections~\ref{intro:subsec:stage-parameters}--\ref{intro:subsec:time-intervals},
set $t_a=t_{j-1}$, $h_*=h_{j-1}$, $a=a_j\in[1/2,2]$,
$\beta=\beta_j\in(0,1]$, and $\varepsilon=\varepsilon_j=\sqrt{a/h_*}\le1$.
The stated ranges of $a,\beta$ follow from
\eqref{intro:eq:activation-ranges} with $x_{j-1}\ge\sqrt2$.
The clocks are \eqref{eq:amplification-scales} and
\eqref{intro:eq:stage-turning-clock}. Use the already chosen
$\Theta=\Theta_j=2x_jx_{j-1}$ and the comparison interval
\[
 I=I_j=[t_a,t_a+\varepsilon\Theta/a],\qquad 0\le\tau\le\Theta.
\]
Lemmas~\ref{intro:lem:clock-child}--\ref{intro:lem:clock-parent}
ensure that this interval covers the child's retained interval and
lies within the parent's interval of existence.
The later comparison estimates may also be applied with any $\Theta$
for which their stated hypotheses hold.

Throughout $I$, assume the older background and inherited correction obey
\begin{equation}
 \|B\|\le CG_*,\qquad \|\dot B\|\le CG_*^2,\qquad |E(t)|\le E_*.
 \label{eq:actual-Gstar}
\end{equation}
At stage $j$, the inputs are $G_*=1+h_{j-2}$ and
$E_*=k_{j-1}^{-1/4}$. The latter is a uniform-in-time induction input,
not merely an activation bound; no time derivative of $E$ is needed.
The background derivative bound is consistent with
$\dot B=-B^2-H^{j-2}(t,0)$, where
$H^{j-2}=\nabla^2p^{j-2}$ is the older solution's pressure Hessian.
The parameter estimates in Section~\ref{sec:stage-parameters} give
$\eta\Theta\ll1$, with $\eta$ from \eqref{eq:short-eta}; hence
$\varepsilon\Theta=(a/G_*)\eta\Theta\le2\eta\Theta\ll1$.
Thus the physical interval is short even though $\Theta$ is large.

Write $K_h\ge2$ for the parent particle-map bound: derivatives of
$X-\mathrm{id},X_t,X_{tt}$ in physical particle labels obey Gevrey-2
Sobolev bounds $K_h^{r+1}(r!)^2$ at order $r$, as in \eqref{eq:child-parent-map}.
Sobolev embedding and the Jacobian identities give first-label-derivative
bounds for the coefficients used here. Fix once and for all an input
exponent $\kappa=7$ such that these coefficient bounds (including
their first physical-label derivatives) and the history inputs in
Section~\ref{sec:history-parent-inputs} are at most $CK_h^\kappa$.
The value of $\kappa$ is independent of the stage; it is not enlarged
in subsequent estimates. Products add exponents. We will obtain
$K_h^{19\kappa}$ as a common activation budget in
Lemma~\ref{lem:history-label-bounds}. To verify this choice, Sobolev
embedding in the fixed $H^s$ norm ($s>3/2$) gives
$|\partial_b^r(F,F_t,F_{tt})|\le C_rK_h^{r+2}$.
Since $\det F=1$, the cofactor formula, followed by
$M=F_tF^{-1}$ and $H=-F_{tt}F^{-1}$, gives
\begin{equation}
 \begin{aligned}
 |\partial_b^r F^{-1}|&\le C_rK_h^{r+4},\\
 |\partial_b^r(M,H)|&\le C_rK_h^{r+6}.
 \end{aligned}
 \label{eq:parent-exponent-budget}
\end{equation}
Here each derivative of a quadratic cofactor adds one to the exponent,
and multiplication by $F_t$ or $F_{tt}$ adds two. Thus $r\le1$
costs at most $K_h^7$. In the history interval the remaining input
$t_a^{-1}$ is also covered: \eqref{eq:history-size-input} and
\eqref{eq:history-activation-geometry}, with $\zeta<1/2$, imply
$t_a^{-1}\le h_*\le2|M(t_a,0)|\le CK_h^6$.

The physical label width is $\ell$. The spatial lemma below converts
these derivative bounds into a gradient error of size $\ell K_h^\kappa$.
All other powers denoted by $c$ below are fixed independently of the stage;
explicit powers of $K_h$ are tracked using $\kappa$.

\subsection{Coefficient errors}
\label{subsec:three-errors}

\refstepcounter{equation}
\label{lem:central-coefficients}
\textbf{Lemma \theequation\ (Temporal error bounds at the origin).}
Set $t_a=t_{j-1}$ and $I=[t_a,t_a+\varepsilon\Theta/a]$.
Assume $a\in[1/2,2]$, $G_*,\Theta\ge1$, $0<\varepsilon\le1$,
and $\varepsilon\Theta G_*^2\le c_0$ for a sufficiently small fixed
constant $c_0$. Suppose throughout $I$, in a (time-dependent) orthonormal
frame $(p,q,n)$ at the origin,
\[
 M(t,0)=B(t)+h(t)q(t)p(t)^T+E(t),\qquad
 |E(t)|\le E_*,\qquad h(t)>0.
\]
The background satisfies
\[
 |B|\le CG_*,\qquad |\dot B|\le CG_*^2.
\]
The frame $(p,q,n)$ and shear magnitude $h$ obey the identities \eqref{eq:inherited-frame-evolution} and \eqref{eq:old-wave-norms}, respectively.
At activation, $B_{pq}(t_a)=a$, $B_{nq}(t_a)=a\beta$, and
$h(t_a)=h_*$. Here $B_{ik}=i\cdot Bk$ denotes a moving-frame component.
Then, uniformly for $t\in I$,
\begin{equation}
 \begin{aligned}
 \left|\frac{B_{pq}(t)}a-1\right|
 +\left|\frac{B_{nq}(t)}a-\beta\right|
 &\le C\varepsilon\Theta G_*^2,\\
 \left|\frac{h(t)}{h_*}-1\right|&\le C\varepsilon\Theta G_*.
 \end{aligned}
 \label{eq:budget-principal-coefficients}
\end{equation}
The other coefficients obey
\begin{equation}
 \frac\varepsilon a|B_{ik}(t)|\le C\varepsilon G_*,\qquad
 \frac{|E(t)|}{a}\le CE_*,\qquad i,k\in\{p,q,n\}.
 \label{eq:budget-other-coefficients}
\end{equation}
The constants depend only on the fixed constants in the hypotheses,
not on the stage.

\textbf{Remark.} The shear estimate is relative; its corresponding absolute bound is
$|h-h_*|\le Ch_*\varepsilon\Theta G_*$. The bound on $E$ is an
input, not a consequence of the background ODE estimates.

\emph{Proof.} The frame equations \eqref{eq:inherited-frame-evolution} and
orthonormality give
\[
 |\dot p|\le2|B|,\qquad |\dot q|\le4|B|,\qquad
 |\dot n|=|\dot p\times q+p\times\dot q|
 \le|\dot p|+|\dot q|.
\]
Consequently $|\dot p|+|\dot q|+|\dot n|\le CG_*$.
Likewise, the shear identity gives
$|\dot h|/h=|B_{pp}+B_{qq}|\le2|B|\le CG_*$.
Thus these derivative bounds follow from the identities and the size
bound on $B$; they are not additional assumptions.

Now differentiate the moving-frame component:
\[
 \frac{d}{dt}(p\cdot Bq)
 =\dot p\cdot Bq+p\cdot\dot Bq+p\cdot B\dot q.
\]
Each term is bounded by $CG_*^2$. The same argument with $n$ in
place of $p$ applies to $B_{nq}$. Integrating from $t_a$ and using
$|I|=\varepsilon\Theta/a$ gives
\[
 |B_{pq}/a-1|+|B_{nq}/a-\beta|
 \le C\varepsilon\Theta G_*^2.
\]
Next the shear identity gives
\[
 \left|\log\frac{h(t)}{h_*}\right|
 \le CG_*|t-t_a|\le C\varepsilon\Theta G_*.
\]
By the assumed smallness, $|e^r-1|\le C|r|$ applies to this logarithm.
Thus $|h/h_*-1|\le C\varepsilon\Theta G_*$, proving
\eqref{eq:budget-principal-coefficients}. Also
$\varepsilon|B_{ik}|/a\le C\varepsilon G_*$, and $|E|/a\le CE_*$
follows from the assumption and $a\asymp1$.
\hfill$\square$

The spatial region for the comparison is $\Omega=\{|y|\le1/2\}$.
It contains the leading oscillatory increment because
$\operatorname{supp}\chi_1\subset\{|y|<1/2\}$; its physical image is
$X(t,\ell\Omega)$. The mean and exact corrections have separate global
estimates and need not be supported here. We use the frame evaluated
at the origin for every label.

\refstepcounter{equation}
\label{lem:spatial-coefficients}
\textbf{Lemma \theequation\ (Parent-gradient bounds on the support of the leading increment).}
Let $I$ be the retained interval, with $\Omega$ and the frame as above.
Suppose the decomposition
\begin{equation}\label{eq:spatial-origin-decomposition}
 M(t,0)=B(t)+h(t)q(t)p(t)^T+E(t),\qquad |E(t)|\le E_*
\end{equation}
holds at the origin throughout $I$, and $a\in[1/2,2]$. In physical particle labels
$b$, define
\[
 \widehat M(t,b)=\nabla U^{j-1}(t,X(t,b)),\qquad
 M(t,y)=\widehat M(t,\ell y).
\]
Assume the full parent gradient has the uniform spatial bound
\[
 \sup_{t\in I,\ |b|\le\ell/2}|\nabla_b\widehat M(t,b)|\le CK_h^{\kappa}.
\]
Then the remainder
$\mathfrak E(t,y)=M(t,y)-B(t)-h(t)q(t)p(t)^T$ satisfies
\begin{equation}
 \begin{aligned}
 \sup_{t\in I,\ y\in\Omega}|M(t,y)-M(t,0)|&\le C\ell K_h^{\kappa},\\
 \sup_{t\in I,\ y\in\Omega}\frac{|\mathfrak E(t,y)|}{a}
 &\le C(E_*+\ell K_h^{\kappa}).
 \end{aligned}
 \label{eq:budget-spatial-coefficients}
\end{equation}

\textbf{Remark.} Only the bound on $E(t)$ at the origin is needed in this proof; no spatial
or time derivative bound on $E$ is required. The spatial derivative
assumption concerns the full parent gradient. It follows from the
particle-map estimates \eqref{eq:parent-exponent-budget}:
the identity
\[
 \widehat M=(\partial_bX_t)(\partial_bX)^{-1}
\]
and its label derivative express it in terms of the controlled map
derivatives and inverse Jacobian. All these inputs must hold throughout $I$.
Unlike Lemma~\ref{lem:central-coefficients}, Lemma~\ref{lem:spatial-coefficients} does not
require a small time interval: it applies on any interval on which its
two uniform inputs hold.

\emph{Proof.} The decomposition \eqref{eq:spatial-origin-decomposition} gives the exact identity
\[
 \mathfrak E(t,y)=E(t)+\widehat M(t,\ell y)-\widehat M(t,0).
\]
The segment from $0$ to $\ell y$ stays in $|b|\le\ell/2$. Hence
\[
 \widehat M(t,\ell y)-\widehat M(t,0)
 =\int_0^1\nabla_b\widehat M(t,s\ell y)[\ell y]\,ds,
 \qquad
 |\mathfrak E(t,y)|\le E_*+C\ell K_h^{\kappa}.
\]
Since $a\asymp1$, these estimates prove
\eqref{eq:budget-spatial-coefficients}.\hfill$\square$

To insert these bounds into the equations, use the scaled components
\eqref{eq:scaled-components} at each label. Restore the time-dependent
principal coefficients $B_{pq}/a$, $B_{nq}/a$, and $h/h_*$ in
\eqref{eq:short-J-full} and \eqref{eq:short-velocity-full}; their
variations are bounded by \eqref{eq:budget-principal-coefficients}.
With $\mathfrak E$ from Lemma~\ref{lem:spatial-coefficients}, the parent
gradient is $M=B+hqp^T+\mathfrak E$. Relative to the part $B+hqp^T$,
the correction adds $-\mathfrak E^Tm$ to the normal equation,
$-\mathfrak Ev$ to the nonprojected velocity equation, and
$m\cdot\mathfrak Ev/(s_0a)$ to $J$ and hence to the pressure projection.
Transversality still gives $W=-(PU+QV)/N$ when $N\ne0$.

The gradient of the parent's newest increment at the origin contains the
remainder $E$ in \eqref{eq:parent-shear}. To see how it enters the scaled
equations, consider the term $-Ev$ in the velocity equation. Since
$v_p=\varepsilon U$, we have $\dot v_p=aU'$, so its contribution to $U'$ is
\begin{equation}
 -\frac1a p\cdot Ev
 =-\frac{E_{pq}}aV-\frac{\varepsilon E_{pp}}aU
                        -\frac{\varepsilon E_{pn}}aW.
 \label{eq:actual-E-example}
\end{equation}
The first coefficient is $O(E_*)$, not $O(E_*/h_*)$.
Thus an error merely small relative to the enormous shear can still
destroy the order-one coupling in the model. The pressure contribution
of $E$ is estimated in the same coefficient accounting, with polynomial
losses from the ray components. Increasing the preceding frequency gives
the absolute bound $E_*=k_{j-1}^{-1/4}$ needed here.

\subsection{Activation data and the combined error budget}
\label{subsec:wave-activation-budget}

Write $Z=(P,Q,N)^T$ and $Y=(U,V)^T$ for the scaled components in
\eqref{eq:scaled-components}. At activation $t=t_{j-1}$, or $\tau=0$,
the values at the origin are $Z(0,0)=(0,0,1)^T$ and
$Y(0,0)=(-\lambda,1)^T$. Phase transport and the history construction
supply the neighboring data. Their physical-label derivative estimates
are additional inputs, beyond the coefficient lemmas; they are proved
under the history hypotheses in Section~\ref{sec:history-activation}.
By the mean-value
theorem, these give physical differences of size $C\ell K_h^{19\kappa}$ on
$\Omega$. Scaling in the common frame at the origin introduces
$\varepsilon^{-1}$: for example,
\[
 |v(t_{j-1},\ell y)-v(t_{j-1},0)|\le C\ell K_h^{19\kappa}
 \quad\Longrightarrow\quad
 |U(0,y)-U(0,0)|\le C\ell K_h^{19\kappa}/\varepsilon.
\]
Here the arguments of $v$ are physical particle labels, and
$U=v_p/\varepsilon$. Similarly, $Q=m_q/(s_0\varepsilon)$ incurs this
loss. More precisely, the history estimates give
$|\nabla_bm|+s_0^{-1}\le CK_h^\kappa$, so the scaled normal costs
$C\ell K_h^{2\kappa}/\varepsilon$; the velocity costs
$C\ell K_h^{19\kappa}/\varepsilon$. Thus the activation bounds
used in Proposition~\ref{prop:conditional-wave-control} are
\begin{equation}
 \begin{aligned}
 \sup_{y\in\Omega}|Z(0,y)-(0,0,1)^T|&\le C\ell K_h^{19\kappa}/\varepsilon,\\
 \sup_{y\in\Omega}|Y(0,y)-(-\lambda,1)^T|&\le C\ell K_h^{19\kappa}/\varepsilon.
 \end{aligned}
 \label{eq:budget-activation-data}
\end{equation}
The spatial gradient error of $M$
in Lemma~\ref{lem:spatial-coefficients} costs only $\ell K_h^\kappa$;
the larger budget $\ell K_h^{19\kappa}/\varepsilon$ also covers these scaled data $Z$ and $Y$.

For the comparison estimates, combine the coefficient and activation
errors into the common budget
\begin{equation}
 e_*=C\left(
 \varepsilon\Theta G_*^2+E_*+
                  \frac{\ell K_h^{19\kappa}}{\varepsilon}\right).
 \label{eq:actual-error-budget}
\end{equation}
Choose the prefactor $C$ once to dominate the constants in the
coefficient and activation estimates above, including their dependence
on the fixed parent size bounds. Thus each normalized input error is
bounded by $e_*$ with a universal numerical prefactor. Since
$G_*,\Theta\ge1$ and $\varepsilon\le1$, also $\varepsilon\le e_*$.
All dependence on those parent constants is carried by this choice of
budget; the subsequent finite-dimensional comparison constants and its
smallness threshold can then be taken universal. These inputs do not yet bound the accumulated solution
error or the full two-component matrix error after eliminating $W$ and
inserting the evolving ray into the pressure projection. Section~\ref{sec:comparison-proofs}
derives $|Z-Z_0|\le Ce_*\Theta^5$, $\|\mathcal E\|\le Ce_*\Theta^9$,
and the relative velocity error $Ce_*\Theta^{24}$ for $\tau\ge1$.

The comparison estimates in Section~\ref{sec:comparison-proofs} below
lose only fixed powers of $\Theta$.
Consequently, if $e_*\Theta^N$ is sufficiently small for a sufficiently
large fixed exponent $N$, all the errors in the wave-model comparison
remain controlled throughout the amplification interval. The parameter
choices provide this smallness, with additional margins for the pressure
sign and transfer to the next layer. Section~\ref{sec:comparison-proofs}
gives the detailed bounds and the precise common smallness condition.

Increasing $h_*$, and hence decreasing $\varepsilon$,
improves time freezing but makes spatial accuracy in the scaled components
more demanding.
The new layer must be narrow enough that
$\ell K_h^{19\kappa}\Theta^N\ll\varepsilon$, with $N$ as above.
Large parent derivatives enter through $K_h$, so
small support must be chosen to compensate for them.

\section{Wave-model comparison bounds}
\label{sec:comparison-proofs}
\fancyhead[R]{\small Wave-model comparison bounds}

We apply the coefficient and activation-data bounds of
Section~\ref{sec:actual-to-ideal} to the wave ODEs
\eqref{intro:eq:normal} and \eqref{intro:eq:wave}. Let $m_j,v_j$
solve these equations with $m_j\cdot v_j=0$ at activation;
the constraint is then preserved. Use the scaled decomposition
\eqref{eq:scaled-components}, with $Z=(P,Q,N)^T$ and $Y=(U,V)^T$
as in Subsection~\ref{subsec:wave-activation-budget}.
The frame depends only on time; the components also depend on the spatial label $y$.

\refstepcounter{equation}
\label{prop:conditional-wave-control}
\textbf{Proposition \theequation\ (Conditional control of $m_j,v_j$).}
Fix a stage $j$ and an interval
\[
 I_j=[t_{j-1},t_{j-1}+\varepsilon\Theta/a],\qquad
 \tau=\frac a\varepsilon(t-t_{j-1}),\qquad \Theta\ge1.
\]
All spatial bounds below hold on $\Omega=\{|y|\le1/2\}$, on which
the leading part of the stage-$j$ wave is supported. Assume the following inputs.

\textit{Parent and frame.}
The parent flow is smooth on $I_j$ and satisfies the hypotheses of
Lemmas~\ref{lem:central-coefficients} and~\ref{lem:spatial-coefficients}
on this whole interval, with the same decomposition and frame at the origin.
In addition, require $\beta\in(0,1]$ and $\varepsilon=\sqrt{a/h_*}$.
Let $e_*$ be the common budget in \eqref{eq:actual-error-budget}, and assume
\[
 e_*\Theta^{26}\le c,
\]
where $c>0$ is a universal sufficiently small constant after the
normalization of \eqref{eq:actual-error-budget}. The constants in the
conclusions below are universal as well.\footnote{The
exponent $26$ is not optimal but suffices for all conclusions below.}

\textit{Wave and activation data.}
Assume the activation-data bounds \eqref{eq:budget-activation-data},
including the exact $y=0$ values specified there, with
$\lambda\ge0$. They are at most $Ce_*$ by \eqref{eq:actual-error-budget}.

Then the wave ODEs remain well posed on $I_j$,
and have the following bounds, uniformly over the support.

\textit{(i) Phase control and nondegeneracy.}
For $0\le\tau\le\Theta$, set
$Z_0=(\beta\tau^2,-2\beta\tau,1)^T=(P_0,Q_0,1)^T$. Then
\[
 |Z-Z_0|\le Ce_*\Theta^5,\qquad N\ge1/2,
 \qquad D_a=P^2+\varepsilon^2Q^2+N^2\ge1/4.
\]
In particular, with $m_{j,0}=s_0(P_0p+\varepsilon Q_0q+n)$,
\[
 |m_j-m_{j,0}|\le Cs_0e_*\Theta^5,\qquad
 \tfrac12s_0\le|m_j|\le Cs_0\Theta^2.
\]

\textit{(ii) Velocity propagation and comparison.}
Let $V_\lambda$ solve \eqref{eq:scalar-amplification} with $V_\lambda(0)=1$, $V_\lambda'(0)=\lambda$, and put
$Y_\lambda=(-V_\lambda',V_\lambda)^T$.
At each fixed spatial label $y$, let $\Phi(\tau,s)$ be the propagator
of the actual system \eqref{eq:actual-matrix-comparison} for $Y=(U,V)^T$,
so that $Y(\tau,y)=\Phi(\tau,s)Y(s,y)$, with the dependence of
$\Phi$ on $y$ suppressed. Uniformly in $y\in\Omega$, it satisfies
\begin{equation}
 \|\Phi(\tau,s)\|\le C\Theta^7\frac{V_0(\tau)}{V_0(s)}
 \quad(0\le s\le\tau\le\Theta).
 \label{eq:conditional-velocity-propagator}
\end{equation}
Moreover,
\begin{equation}
 |Y-Y_\lambda|\le Ce_*\Theta^{24}(1+\lambda)V_0
 \quad(0\le\tau\le\Theta),
 \label{eq:conditional-Y-absolute}
\end{equation}
while for $1\le\tau\le\Theta$,
\begin{equation}
 \frac{|Y-Y_\lambda|}{V_\lambda}\le Ce_*\Theta^{24},
 \qquad V\ge\tfrac12V_\lambda>0.
 \label{eq:conditional-Y-relative}
\end{equation}
All third components are recovered from $W=-(PU+QV)/N$.
If $W_\lambda=-P_0(-V_\lambda')-Q_0V_\lambda$ and
$v_{j,0}=\varepsilon(-V_\lambda')p+V_\lambda q+
\varepsilon W_\lambda n$, then
\[
 |v_j-v_{j,0}|\le Ce_*\Theta^{26}V_\lambda
 \qquad(1\le\tau\le\Theta).
\]

\textit{Interval and later use.}
If $\Theta\ge x_j/\sqrt\beta$, these conclusions include the target
$t_j$. They include any later time in $I_j$ for which the assumed inputs
hold. For the choice $\Theta=\Theta_j$ in
Subsection~\ref{subsec:wave-parent-inputs}, the interval containments in
Lemmas~\ref{intro:lem:clock-child}--\ref{intro:lem:clock-parent}
place the retained interval, and hence the accumulation time, inside $I_j$.
The conclusions there remain conditional on the parent and activation
inputs throughout $I_j$. This proposition does not prove those induction inputs,
the pressure sign, or the next-frame conditions. It supplies a nonvanishing
normal, a positive growing velocity component, control of the full wave
relative to its ideal reference, and a propagator bound for subsequent
estimates; the positivity also makes target amplitude normalization
well defined when $\tau(t_j)\ge1$.

The following subsections prove the proposition: first the normal
comparison, then the velocity matrix bound, and finally the ideal and
actual propagator estimates.

\subsection{Comparing \texorpdfstring{$Z$ and $Z_0$}{Z and Z0}}

We first compute the propagator of \eqref{eq:ideal-ray-system}
and use it to transport the coefficient and activation errors.

To compute the ideal propagator, start from arbitrary data at scaled
time $s$. Successive integration,
with $d=\tau-s$, gives
\[
 N_0(\tau)=N_0(s),\qquad Q_0(\tau)=Q_0(s)-2\beta dN_0(s),\qquad
 P_0(\tau)=P_0(s)-dQ_0(s)+\beta d^2N_0(s).
\]
Equivalently, the matrix that maps the data at $s$ to the solution at
$\tau$ is
\begin{equation}
 \Phi_0(\tau,s)=
 \begin{pmatrix}
 1&-d&\beta d^2\\
 0&1&-2\beta d\\
 0&0&1
 \end{pmatrix},\qquad d=\tau-s.
 \label{eq:explicit-ray-propagator}
\end{equation}
This is the ideal propagator. For $0\le s\le\tau\le\Theta$,
using $\beta\lesssim1$ and $\Theta\ge1$, we obtain
\[
 \|\Phi_0(\tau,s)\|
 \le C(1+d+\beta d^2)\le C\Theta^2.
\]
In particular, $Z_0(0)=(0,0,1)^T$ gives
$Z_0(\tau)=(\beta\tau^2,-2\beta\tau,1)^T$ and
$\sup_{0\le\tau\le\Theta}|Z_0(\tau)|\le C\Theta^2$.

\emph{Proof of Proposition~\ref{prop:conditional-wave-control}(i).}
To derive the error equation and its integral estimate, fix a spatial
label $y$ and write the actual and ideal equations as
\[
 Z'=(A_0+R)Z,\qquad
 Z_0'=A_0Z_0,\qquad
 A_0=
 \begin{pmatrix}0&-1&0\\0&0&-2\beta\\0&0&0\end{pmatrix}.
\]
The input estimates give $\|R\|\le Ce_*$ uniformly in scaled time
and over the new layer's support,
while the first bound in \eqref{eq:budget-activation-data} gives
$|Z(0)-Z_0(0)|\le Ce_*$.
Here $R$ is the difference of the coefficient matrices, not of their
propagators. Set $F=Z-Z_0$. Subtraction gives the differential equation
\[
 F'=A_0F+RZ_0+RF.
\]
Variation of constants with the explicit ideal propagator yields
\[
 F(\tau)=\Phi_0(\tau,0)F(0)
 +\int_0^\tau\Phi_0(\tau,s)
                  R(s)\bigl(Z_0(s)+F(s)\bigr)\,ds.
\]
Define $D(\tau)=\sup_{0\le s\le\tau}|F(s)|$; the same estimates
hold after also taking the supremum over the new layer's support. The bounds above
imply the integral inequality
\[
 D(\tau)\le Ce_*\Theta^2
       +Ce_*\Theta^2\int_0^\tau\bigl(\Theta^2+D(s)\bigr)\,ds.
\]
The first term comes from the initial-data error. The integral transports
the forcing $RZ_0$ and the error-dependent term $RF$. Since
$\tau\le\Theta$ and $D(s)\le D(\Theta)$, it follows that
\[
 D(\Theta)\le Ce_*\Theta^2+Ce_*\Theta^5
                    +Ce_*\Theta^3D(\Theta)
 \le Ce_*\Theta^5+Ce_*\Theta^3D(\Theta).
\]
If $Ce_*\Theta^3<1/2$, absorption proves
\begin{equation}
 |P-P_0|+|Q-Q_0|+|N-1|\le Ce_*\Theta^5,
 \qquad N\ge1/2.
 \label{eq:actual-ray-comparison}
\end{equation}
The last assertion uses a correspondingly small $e_*\Theta^5$.
It justifies eliminating $W$ by \eqref{eq:scaled-transversality}.
Since $|Z|\le C\Theta^2$, we obtain
\begin{equation}
 \begin{aligned}
 |W|&\le C\Theta^2(|U|+|V|),\\
 |D_a-(1+P_0^2)|
 &\le (|P|+|P_0|)|P-P_0|+(|N|+1)|N-1|+\varepsilon^2Q^2
 \le Ce_*\Theta^7.
 \end{aligned}
 \label{eq:actual-rational-errors}
\end{equation}
For the last term, $|Q|\le C\Theta^2$ and
$\varepsilon^2\lesssim e_*$ by \eqref{eq:actual-error-budget}, so
$\varepsilon^2Q^2\le Ce_*\Theta^4$.
Also $D_a\ge N^2\ge1/4$ directly. The pressure-numerator comparison
is proved together with the full velocity matrix estimate below.

\subsection{The velocity coefficient matrix}

We now turn to the velocity system, using $A_0$ and $\Phi_0$ below
for its ideal coefficient matrix and propagator.
Eliminating $W$ by \eqref{eq:scaled-transversality} from the exact scaled
velocity equations \eqref{eq:actual-scaled-velocity} below, and using
$D_a\ge1/4$, gives the two-component system
\begin{equation}
 Y'=(A_0+\mathcal E)Y,\qquad Y=(U,V)^T,
 \qquad \|\mathcal E\|\le Ce_*\Theta^9,
 \label{eq:actual-matrix-comparison}
\end{equation}
where $A_0$ is the matrix of \eqref{eq:ideal-amplification-system}.

We verify \eqref{eq:actual-matrix-comparison} in three steps: the
nonprojected terms, the pressure numerator, and the pressure projection.
All bounds below are uniform on the scaled interval and the new layer's
support. We assume the smallness conditions used above, in particular
$e_*\Theta^5\ll1$, and use $\Theta\ge1$ and $\varepsilon\le1$.

\textit{1. The full nonprojected terms.}
Write $b_{ik}=B_{ik}/a$, $r_h=h/h_*-1$, and
$\mathfrak e_{ik}=\mathfrak E_{ik}/a$. The input estimates of
Subsection~\ref{subsec:three-errors}, with the budget
\eqref{eq:actual-error-budget}, imply
\[
 |b_{pq}-1|+|b_{nq}-\beta|+|r_h|\le Ce_*,\qquad
 \varepsilon|b_{ik}|\le Ce_*,\qquad
 |\mathfrak e_{ik}|\le Ce_*.
\]
The last bound includes both the preceding correction and spatial
variation. Indeed, $|\mathfrak E|\lesssim E_*+\ell K_h^{\kappa}$, while
$e_*$ includes the larger localization budget $\ell K_h^{19\kappa}/\varepsilon$.
The frame is the moving frame at the origin, so its derivatives are already
included in the component equations.

Restoring $B_{pq}/a$ and $h/h_*$ in place of the frozen coefficients
in \eqref{eq:short-velocity-full}, and adding the contributions of
$\mathfrak E$, gives the following exact scaled equations. The terms
$f_U$ and $f_V$ collect the resulting nonprojected errors:
\begin{equation}\label{eq:actual-scaled-velocity}
 U'=-2V+\frac{2PJ}{D_a}+f_U,\qquad
 V'=-U+\frac{2\varepsilon^2QJ}{D_a}+f_V,
\end{equation}
where
\begin{align*}
 f_U={}&-2(b_{pq}-1)V-\varepsilon b_{pp}U
       -2\varepsilon b_{pn}W
       -\varepsilon\mathfrak e_{pp}U-\mathfrak e_{pq}V
       -\varepsilon\mathfrak e_{pn}W,\\
 f_V={}&-r_hU-\varepsilon^2(b_{qp}-b_{pq})U
       -\varepsilon b_{qq}V-\varepsilon^2(b_{qn}+b_{nq})W\\
       &-\varepsilon^2\mathfrak e_{qp}U
       -\varepsilon\mathfrak e_{qq}V-\varepsilon^2\mathfrak e_{qn}W.
\end{align*}
Transversality and the ray estimates give
\[
 W=-\frac{PU+QV}{N},\qquad |W|\le C\Theta^2|Y|,
 \qquad Y=(U,V)^T.
\]
Every coefficient displayed in $f_U,f_V$ is bounded by $Ce_*$.
Consequently, after eliminating $W$,
\begin{equation}
 |f_U|+|f_V|\le Ce_*\Theta^2|Y|.
 \label{eq:all-nonprojected-errors}
\end{equation}
In particular, the physical shear error may have size $h_*|r_h|$,
but its direct scaled contribution is only $-r_hU$, because
$\varepsilon^2h_*/a=1$.

\textit{2. All errors in the pressure numerator.}
Expansion by rows, including the correction, gives
\[
 J=b_{pq}PV+(1+r_h)QU+b_{nq}NV+\rho_B+\rho_E,
\]
where the remaining background and correction terms are
\begin{align*}
 \rho_B={}&\varepsilon(Pb_{pp}U+Pb_{pn}W+Qb_{qq}V
                  +Nb_{np}U+Nb_{nn}W)
             +\varepsilon^2Q(b_{qp}U+b_{qn}W),\\
 \rho_E={}&P(\varepsilon\mathfrak e_{pp}U+\mathfrak e_{pq}V
                         +\varepsilon\mathfrak e_{pn}W)\\
 &+\varepsilon Q(\varepsilon\mathfrak e_{qp}U+\mathfrak e_{qq}V
                         +\varepsilon\mathfrak e_{qn}W)\\
 &+N(\varepsilon\mathfrak e_{np}U+\mathfrak e_{nq}V
                         +\varepsilon\mathfrak e_{nn}W).
\end{align*}
Using $|P|+|Q|+|N|\le C\Theta^2$ and the bound for $W$, every
term is at most $Ce_*\Theta^4|Y|$. Thus, with
$J_0=(P_0+\beta)V+Q_0U$,
\begin{align*}
 J-J_0={}&(b_{pq}-1)PV+r_hQU+(b_{nq}-\beta)NV
           +\rho_B+\rho_E\\
         &+(P-P_0)V+(Q-Q_0)U+\beta(N-1)V.
\end{align*}
The first line is bounded by $Ce_*\Theta^4|Y|$ and the second by
$Ce_*\Theta^5|Y|$, by \eqref{eq:actual-ray-comparison}. Hence
\begin{equation}
 |J-J_0|\le Ce_*\Theta^5|Y|,\qquad |J|+|J_0|\le C\Theta^2|Y|.
 \label{eq:all-numerator-errors}
\end{equation}
This includes the pressure contribution of the shear error, $r_hQU$,
as well as both parts of $\mathfrak E$. Initial-data errors enter here
through the ray comparison; they need no additional matrix term.

\textit{3. The denominator and the assembled matrix.}
Set $D_0=1+P_0^2$. By \eqref{eq:actual-rational-errors},
$|D_a-D_0|\le Ce_*\Theta^7$. Also $D_a\ge1/4$,
$|P|/D_a\le1/(2|N|)\le1$, and $|P_0|/D_0\le1/2$.
Therefore
\[
 \left|\frac P{D_a}-\frac{P_0}{D_0}\right|
 \le\frac{|P-P_0|}{D_a}
  +\frac{|P_0|}{D_0}\frac{|D_a-D_0|}{D_a}
 \le Ce_*\Theta^7.
\]
Subtracting the ideal projected term gives
\[
 \frac{2PJ}{D_a}-\frac{2P_0J_0}{D_0}
 =\frac{2P}{D_a}(J-J_0)
  +2\left(\frac P{D_a}-\frac{P_0}{D_0}\right)J_0,
\]
whose absolute value is at most $Ce_*\Theta^9|Y|$.
The projection omitted from the ideal $V'$ equation satisfies
\[
 \left|\frac{2\varepsilon^2QJ}{D_a}\right|
 \le C\varepsilon^2\Theta^4|Y|\le Ce_*\Theta^4|Y|.
\]
Here $\varepsilon^2\lesssim e_*$ follows from
$e_*\gtrsim\varepsilon\Theta G_*^2$, $G_*\ge1$, and
$\varepsilon\le1$.

Combining these estimates with \eqref{eq:all-nonprojected-errors},
for every vector $Y$ we obtain
\[
 |\mathcal E Y|\le
 C e_*(\Theta^2+\Theta^9+\Theta^4)|Y|.
\]
All expressions are linear in $Y$ after eliminating $W$, so this is
an operator-norm estimate, not merely a bound along one solution.
In particular,
\begin{equation}
 \|\mathcal E\|\le Ce_*\Theta^9,
 \label{eq:velocity-error-operator-bound}
\end{equation}
which verifies the stated bound with every coefficient
error included. The velocity activation-data discrepancy is separate:
it affects the subsequent solution comparison, not $\mathcal E$ itself.

\subsection{The ideal scalar solution and its propagator}

The velocity comparison needs a different estimate from the triangular
ray comparison. We first isolate the scalar properties needed for it.

\refstepcounter{equation}
\label{lem:ideal-positive}
\textbf{Lemma \theequation\ (Positivity and bounds for the ideal solution).}
Let $0<\beta\le1$, and let $V_0$ solve the ideal scalar amplification
equation \eqref{eq:scalar-amplification} with
$V_0(0)=1$ and $V_0'(0)=0$.
Then $V_0(\tau)>0$ for all $\tau\ge0$. It is nondecreasing on
$[0,1]$, and $\tau V_0(\tau)$ is strictly increasing for $\tau\ge1$.
Moreover, for every $\Theta\ge1$,
\[
 \frac{V_0(s)}{V_0(r)}\le\Theta
 \quad(0\le s\le r\le\Theta),\qquad
 \left|\frac{V_0'(\tau)}{V_0(\tau)}\right|\le\sqrt2
 \quad(0\le\tau\le\Theta).
\]

\textit{Proof.}
Write $D_0=1+\beta^2\tau^4$. The scalar equation is a linear ODE
with smooth coefficients and $D_0>0$, so its solution exists on every
finite interval. It remains to prove the stated bounds.

\textit{Positivity.}
On $0\le\tau\le1$, the coefficient $2(1-\beta^2\tau^2)$ is
nonnegative. As long as $V_0>0$, integration of the scalar equation gives
\[
 D_0(\tau)V_0'(\tau)
 =\int_0^\tau2(1-\beta^2r^2)V_0(r)\,dr\ge0.
\]
Starting from $V_0(0)=1$, this prevents a first zero and proves
$V_0\ge1$ and $V_0'\ge0$ on this interval.
For $\tau\ge1$, put $w(\tau)=\tau V_0(\tau)$. Substituting
$V_0=w/\tau$ into the scalar equation gives
\[
 \left(\frac{D_0}{\tau^2}w'\right)'
 =\frac{2(1-\tau^{-2})}{\tau^2}w.
\]
At $\tau=1$, both $w$ and $w'$ are positive. While $w>0$,
the right side is nonnegative, so $w'$ stays positive. A first zero
is therefore impossible: $w>0$ and $w'>0$ for every later time.
In particular, $V_0$ remains positive throughout the interval.

\textit{The growth ratio.}
For $1\le s\le r$, monotonicity of $w$ gives
$V_0(s)/V_0(r)\le r/s\le\Theta$.
For $s\le r\le1$, monotonicity of $V_0$ gives a ratio at most $1$.
If $s\le1\le r$, combine $V_0(s)\le V_0(1)$ with
$V_0(1)/V_0(r)\le r$. Hence, in all cases,
\[
 \frac{V_0(s)}{V_0(r)}\le\Theta,\qquad 0\le s\le r\le\Theta.
\]

\textit{The logarithmic derivative.}
Since $V_0>0$, $l_0=V_0'/V_0$ is well defined and satisfies
\[
 l_0'=q-p l_0-l_0^2,\qquad
 p=\frac{4\beta^2\tau^3}{1+\beta^2\tau^4}\ge0,\qquad
 q=\frac{2(1-\beta^2\tau^2)}{1+\beta^2\tau^4}\le2.
\]
The constant $\sqrt2$ is an upper barrier: at $l_0=\sqrt2$,
the right side is nonpositive. Since $l_0(0)=0$, scalar comparison
gives $l_0\le\sqrt2$. On $[0,1]$, $V_0'\ge0$ gives $l_0\ge0$.
For $\tau\ge1$, $w'=V_0+\tau V_0'>0$ gives
$l_0>-1/\tau\ge-1$. Thus
\[
 |l_0(\tau)|\le\sqrt2\qquad(0\le\tau\le\Theta).
\]
This proves the lemma.\hfill$\square$

\textbf{Deriving the weighted ideal propagator bound.}
Let $z$ be any solution of the ideal scalar equation, with arbitrary
values $z(s),z'(s)$ at scaled time $s$. We use a lowercase letter to
distinguish this scalar solution from the vector $Y=(U,V)^T$.
Both $z$ and $V_0$ solve $(D_0z')'=2(1-\beta^2\tau^2)z$, where
$D_0=1+\beta^2\tau^4$. Consequently their weighted Wronskian is constant:
\[
 \bigl[D_0(V_0z'-V_0'z)\bigr]'=0,
 \qquad
 D_0V_0^2(z/V_0)'=D_0(V_0z'-V_0'z).
\]
Set $l_0=V_0'/V_0$ and $c_s=z'(s)-l_0(s)z(s)$.
Evaluating that constant at $s$ and integrating gives
\begin{equation}
 z(\tau)=\frac{V_0(\tau)}{V_0(s)}\mathcal B(\tau,s),\qquad
 \mathcal B(\tau,s)=z(s)+D_0(s)c_s I_s(\tau),
 \label{eq:scalar-propagator-formula}
\end{equation}
where $\mathcal B$ is the scalar bracket and
\[
 I_s(\tau)=\int_s^\tau
       \frac{(V_0(s)/V_0(r))^2}{D_0(r)}\,dr.
\]
Apply Lemma~\ref{lem:ideal-positive} to bound the formula above.
Since $D_0(r)\ge1$, $D_0(s)\le C\Theta^4$, and $\tau-s\le\Theta$,
\[
 |I_s(\tau)|\le C\Theta^3,\qquad
 |c_s|\le C(|z(s)|+|z'(s)|).
\]
Therefore the bracket in \eqref{eq:scalar-propagator-formula} obeys
\[
 |\mathcal B(\tau,s)|\le C\Theta^7(|z(s)|+|z'(s)|).
\]
The propagator must control the derivative as well as the scalar value.
Differentiating the same formula gives
\[
 z'(\tau)=\frac{V_0(\tau)}{V_0(s)}
 \left[l_0(\tau)\mathcal B(\tau,s)
 +\frac{D_0(s)}{D_0(\tau)}c_s
                   \left(\frac{V_0(s)}{V_0(\tau)}\right)^2\right].
\]
The second term in brackets is bounded by
$C\Theta^6(|z(s)|+|z'(s)|)$; the first is bounded by
$C\Theta^7(|z(s)|+|z'(s)|)$. Thus both $z$ and $z'$ are controlled,
and since the ideal velocity vector is $(U,V)=(-z',z)$, we obtain
\[
 |(U,V)(\tau)|\le C\Theta^7\frac{V_0(\tau)}{V_0(s)}|(U,V)(s)|.
\]
This holds for arbitrary initial vectors and therefore gives the
operator-norm bound:
\begin{equation}
 \|\Phi_0(\tau,s)\|\le C\Theta^7\frac{V_0(\tau)}{V_0(s)}.
 \label{eq:actual-weighted-propagator}
\end{equation}
The exponential growth is contained in the displayed ratio.

\refstepcounter{equation}
\label{lem:ideal-nonnegative-slope}
\textbf{Lemma \theequation\ (Ideal solutions with nonnegative initial slope).}
Let $0<\beta\le1$, $\Theta\ge1$, and $\lambda\ge0$. Let $V_\lambda$
solve \eqref{eq:scalar-amplification} with
$V_\lambda(0)=1$, $V_\lambda'(0)=\lambda$, and set
$Y_\lambda=(-V_\lambda',V_\lambda)^T$ and $D_0=1+\beta^2\tau^4$.
Write $V_0$ for the solution with zero initial slope. Then
\begin{equation}
 V_\lambda(\tau)=V_0(\tau)(1+\lambda I(\tau)),\qquad
 I(\tau)=\int_0^\tau\frac{dr}{D_0(r)V_0(r)^2}.
 \label{eq:slope-solution-formula}
\end{equation}
In particular, $V_\lambda>0$ on $[0,\Theta]$, and
\begin{equation}
 V_\lambda(\tau)\ge c(1+\lambda)V_0(\tau),\qquad 1\le\tau\le\Theta.
 \label{eq:slope-lower-bound}
\end{equation}
Both components satisfy
\begin{equation}
 |Y_\lambda(\tau)|\le C\Theta^7(1+\lambda)V_0(\tau),
 \qquad 0\le\tau\le\Theta.
 \label{eq:ideal-Y-lambda-bound}
\end{equation}
Moreover,
\begin{equation}
 \frac{V_0(\tau)}{V_0(s)}
 \le\frac{V_\lambda(\tau)}{V_\lambda(s)},
 \qquad 0\le s\le\tau\le\Theta.
 \label{eq:ideal-slope-growth-ratio}
\end{equation}
The constants $c,C>0$ are independent of $\beta,\Theta,\lambda$.

\emph{Proof.} Apply \eqref{eq:scalar-propagator-formula} with
$z=V_\lambda$ and $s=0$. Since $D_0(0)=V_0(0)=1$ and
$l_0(0)=0$, the coefficient $c_s$ at $s=0$ equals $\lambda$, giving
\eqref{eq:slope-solution-formula}. The integral $I$ is nonnegative
and increasing, proving positivity and \eqref{eq:ideal-slope-growth-ratio}.

On $[0,1]$, Lemma~\ref{lem:ideal-positive} gives
$V_0(r)\le e^{\sqrt2 r}$, while $D_0(r)\le2$. Hence
\[
 I(1)\ge\frac12\int_0^1e^{-2\sqrt2 r}\,dr
 =\frac{1-e^{-2\sqrt2}}{4\sqrt2}=:c_0>0.
\]
For $\tau\ge1$, $1+\lambda I(\tau)\ge\min(1,c_0)(1+\lambda)$,
which proves \eqref{eq:slope-lower-bound}.
Finally, $Y_\lambda(\tau)=\Phi_0(\tau,0)(-\lambda,1)^T$.
Applying \eqref{eq:actual-weighted-propagator}, using $V_0(0)=1$
and $|(-\lambda,1)|\le1+\lambda$, proves
\eqref{eq:ideal-Y-lambda-bound}.\hfill$\square$

\subsection{The actual propagator and the slope-dependent comparison}

\emph{Proof of Proposition~\ref{prop:conditional-wave-control}(ii).}
We now combine the preceding estimates to prove the remaining conclusions.
Let $\Phi(\tau,s)$ be the actual velocity propagator. Since its
coefficient matrix is $A_0+\mathcal E$, while the ideal one is $A_0$,
variation of constants gives
\[
 \Phi(\tau,s)=\Phi_0(\tau,s)
 +\int_s^\tau\Phi_0(\tau,r)\mathcal E(r)\Phi(r,s)\,dr.
\]
In the integrand, $\Phi(r,s)$ evolves the data to time $r$,
$\mathcal E(r)$ applies the coefficient error, and $\Phi_0(\tau,r)$
transports its effect to time $\tau$. Thus $\mathcal E$ is the source
of the difference between the two evolutions.

Normalize the actual and ideal propagators by
\begin{equation}
\begin{aligned}
 \widetilde\Phi(\tau,s)&=\frac{V_0(s)}{V_0(\tau)}\Phi(\tau,s),\\
 \widetilde\Phi_0(\tau,s)&=\frac{V_0(s)}{V_0(\tau)}\Phi_0(\tau,s),
 \qquad \|\widetilde\Phi_0(\tau,s)\|\le C\Theta^7.
\end{aligned}
\label{eq:normalized-propagators}
\end{equation}
The bound follows from \eqref{eq:actual-weighted-propagator}.
Multiplying the variation-of-constants identity by
\[
 \frac{V_0(s)}{V_0(\tau)}
 =\frac{V_0(r)}{V_0(\tau)}\frac{V_0(s)}{V_0(r)}
\]
gives
\[
 \widetilde\Phi(\tau,s)=\widetilde\Phi_0(\tau,s)
 +\int_s^\tau\widetilde\Phi_0(\tau,r)\mathcal E(r)
                         \widetilde\Phi(r,s)\,dr.
\]
Its integral operator has supremum norm at most
$Ce_*\Theta^{7+9+1}=Ce_*\Theta^{17}$.
If this is less than $1/2$, the normalized equation is solved by a
convergent series. Using \eqref{eq:normalized-propagators}, we obtain
\begin{equation}
 \|\widetilde\Phi(\tau,s)\|\le2C\Theta^7,
 \qquad 0\le s\le\tau\le\Theta.
 \label{eq:actual-normalized-bound}
\end{equation}
With the definition \eqref{eq:normalized-propagators},
this gives \eqref{eq:conditional-velocity-propagator}.

\textbf{The comparison with a possibly large activation slope.}
The activation data at the origin are $V(0)=1$, $U(0)=-\lambda$, with
$\lambda\ge0$. Use the ideal reference $Y_\lambda$ from
Lemma~\ref{lem:ideal-nonnegative-slope}, which supplies its positivity,
amplitude lower bound, and two-component upper bound.

To compare with the actual solution at a spatial label $y$, let
$F_\lambda=Y-Y_\lambda$. Then the second bound in
\eqref{eq:budget-activation-data} gives $|F_\lambda(0)|\le Ce_*$,
uniformly over the new layer's support.
Subtracting the equations gives
\[
 F_\lambda'=(A_0+\mathcal E)F_\lambda+\mathcal E Y_\lambda.
\]
Using the actual propagator therefore gives the exact formula
\[
 F_\lambda(\tau)=\Phi(\tau,0)F_\lambda(0)
 +\int_0^\tau\Phi(\tau,r)\mathcal E(r)Y_\lambda(r)\,dr.
\]
Combining the bound \eqref{eq:conditional-velocity-propagator} for $\|\Phi(\tau,s)\|$
with $\|\mathcal E\|\le Ce_*\Theta^9$ from \eqref{eq:velocity-error-operator-bound} and the bound \eqref{eq:ideal-Y-lambda-bound} for $|Y_\lambda(r)|$ gives
\begin{align*}
 |F_\lambda(\tau)|
 &\le Ce_*\Theta^7V_0(\tau)\\
 &\quad+C\int_0^\tau
 \Theta^7\frac{V_0(\tau)}{V_0(r)}\,
 e_*\Theta^9\,\Theta^7(1+\lambda)V_0(r)\,dr\\
 &\le Ce_*\Theta^7V_0(\tau)
       +Ce_*\Theta^{7+9+7+1}(1+\lambda)V_0(\tau)\\
 &\le Ce_*\Theta^{24}(1+\lambda)V_0(\tau).
\end{align*}
The factors $V_0(r)$ cancel inside the integral, whose length supplies
the final power of $\Theta$, giving \eqref{eq:conditional-Y-absolute}.
Finally, dividing by \eqref{eq:slope-lower-bound} yields the first bound in
\eqref{eq:conditional-Y-relative}. The second bound follows from the first one,
the component bound $|V-V_\lambda|\le|Y-Y_\lambda|$ and the smallness of $e_*\Theta^{24}$.

To recover the full wave, the decomposition
\eqref{eq:scaled-components} and \eqref{eq:actual-ray-comparison}
give the stated bounds for $m_j$. For the velocity, transversality gives
\[
 W-W_\lambda=-\frac PN(U+V_\lambda')-\frac QN(V-V_\lambda)
 -\left(\frac PN-P_0\right)(-V_\lambda')
 -\left(\frac QN-Q_0\right)V_\lambda.
\]
Using $N\ge1/2$, $|Z|\le C\Theta^2$, and the normal comparison,
\[
 |W-W_\lambda|\le C\Theta^2|Y-Y_\lambda|
                     +Ce_*\Theta^7|Y_\lambda|.
\]
For $\tau\ge1$, \eqref{eq:ideal-Y-lambda-bound} and
\eqref{eq:slope-lower-bound} give $|Y_\lambda|\le C\Theta^7V_\lambda$.
Together with \eqref{eq:conditional-Y-relative} and $\varepsilon\le1$,
this proves $|v_j-v_{j,0}|\le Ce_*\Theta^{26}V_\lambda$.
The nonvanishing normal and the bounded linear velocity coefficients
also give continuation throughout $I_j$, completing the proposition.
\hfill$\square$

\textbf{Remark.} The relative bound \eqref{eq:conditional-Y-relative}
is independent of the possibly large slope $\lambda$: both the absolute error
\eqref{eq:conditional-Y-absolute} and the reference amplitude
\eqref{eq:slope-lower-bound} carry $1+\lambda$. On $0\le\tau<1$,
the absolute bound \eqref{eq:conditional-Y-absolute} still holds;
the uniform relative conclusion uses $I(\tau)\ge I(1)$ and is asserted
only for $\tau\ge1$.

For the stage iteration, Section~\ref{sec:stage-parameters} uses the
stronger common condition
\begin{equation}
 {e_*\Theta^{60}\le x_{j-1}^{-10},}
 \label{eq:actual-final-smallness}
\end{equation}
with a sufficiently small fixed constant understood. It provides margins
for the ray comparison, velocity comparison, pressure sign, and subsequent
frame transfer. Since $\Theta\gg1$, it requires extremely small errors.
The three choices are now visible: a large parent shear
shortens physical time, a large preceding frequency reduces $E_*$, and
a small new support width reduces the scaled neighboring-data error.

\section{Quantitative amplification and the next frame}
\label{sec:quantitative-amplification}
\fancyhead[R]{\small Amplification and the next frame}

Proposition~\ref{prop:conditional-wave-control} compares the actual wave
with its ideal reference. We now determine what that reference achieves.
The wave must gain enough amplitude to produce a large target shear
from a small prefactor, and its target directions must be suitable for
amplifying the next wave. We now prove the ideal growth and direction
estimates, then transfer them using Section~\ref{sec:comparison-proofs}.

\subsection{Exponential amplification before the turn}
\label{subsec:initial-exponential-gain}

Throughout this section $0<\beta\le\beta_0$, where $\beta_0>0$ is a
sufficiently small absolute constant, and $\lambda\ge0$.
Use the turning clock $x=\sqrt\beta\,\tau$ from
\eqref{intro:eq:stage-turning-clock}, and set $\tau_1=\beta^{-1/2}$.
The ideal solution $V_\lambda$ is defined in
Lemma~\ref{lem:ideal-nonnegative-slope}; it solves
\eqref{eq:scalar-amplification} and is positive. Write
$D_0=1+\beta^2\tau^4$.
The point here is the size of the gain: the interval $0\le x\le1$
lasts for a scaled time of order $\beta^{-1/2}$.

\refstepcounter{equation}
\label{lem:exponential-gain}
\textbf{Lemma \theequation\ (Initial exponential gain).}
On $0\le\tau\le\tau_1$, $V_\lambda$ is nondecreasing and
\begin{equation}
 V_\lambda(\tau)\ge
 \cosh(\tau/\sqrt2)+\frac{\lambda}{\sqrt2}\sinh(\tau/\sqrt2).
 \label{eq:initial-gain}
\end{equation}
In particular, for an absolute $c>0$,
\begin{equation}
 V_\lambda(\tau_1)\ge
 c(1+\lambda)\exp\!\left(\frac1{\sqrt{2\beta}}\right).
 \label{eq:exponential-gain}
\end{equation}
The logarithmic derivative $l=V_\lambda'/V_\lambda$ satisfies
\begin{equation}
 0\le l(\tau)\le\sqrt2\coth(\sqrt2\tau),
 \qquad 0<\tau\le\tau_1.
 \label{eq:initial-slope-bound}
\end{equation}

\emph{Proof.} On this interval $D_0\le2$ and
$2(1-\beta^2\tau^2)\ge1$, after taking $\beta_0\le1/2$.
Integrating the scalar equation once gives
\[
 D_0(\tau)V_\lambda'(\tau)
 =\lambda+\int_0^\tau2(1-\beta^2s^2)V_\lambda(s)\,ds\ge0.
\]
Dividing by $D_0\le2$ and integrating again yields
\[
 V_\lambda(\tau)\ge1+\tfrac\lambda2\tau
 +\tfrac12\int_0^\tau(\tau-s)V_\lambda(s)\,ds.
\]
Iterating this positive integral operator bounds $V_\lambda$ below by
the solution of $w''=w/2$, $w(0)=1$, $w'(0)=\lambda/2$.
This is \eqref{eq:initial-gain}, and evaluation at $\tau_1\ge1$
gives \eqref{eq:exponential-gain}.
Finally, $D_0\ge1$, $2(1-\beta^2\tau^2)\le2$,
$D_0'=4\beta^2\tau^3\ge0$, and $l=V_\lambda'/V_\lambda\ge0$, so
\[
 l'=\frac{2(1-\beta^2\tau^2)}{D_0}
       -\frac{D_0'}{D_0}l-l^2\le2-l^2.
\]
The function $L(\tau)=\sqrt2\coth(\sqrt2\tau)$ satisfies
$L'=2-L^2$ and tends to infinity at zero. For each finite $\lambda$,
$l<L$ sufficiently near zero; scalar comparison then proves
\eqref{eq:initial-slope-bound}, with a bound independent of $\lambda$.
\hfill$\square$

\subsection{Direction and amplitude beyond the turn}

The target prescribed in \eqref{eq:actual-target-time} has
$x=x_j\ge2$, later than the initial growth interval.
We therefore need to follow the wave after $x=1$. To distinguish the
two clocks, set $\mathcal V(x)=V_\lambda(x/\sqrt\beta)$.

\refstepcounter{equation}
\label{lem:target-direction}
\textbf{Lemma \theequation\ (Size and direction beyond the turn).}
For $x\ge1$, $x\mathcal V(x)$ is increasing. In particular,
\begin{equation}
 \mathcal V(x)\ge\frac{V_\lambda(\tau_1)}x.
 \label{eq:gain-retention}
\end{equation}
There is a bounded positive function $z(\rho)$ such that
\begin{equation}
 r_\lambda:=\frac{-V_\lambda'}{V_\lambda}
 =\frac{\sqrt\beta}{x}-\frac{z(1/x)}{x^2},\qquad x\ge1,
 \label{eq:target-direction-ratio}
\end{equation}
and, uniformly for $0\le\rho\le1/2$,
\begin{equation}
 z(\rho)^2=\frac2{1+\rho^4}+O(\sqrt\beta).
 \label{eq:target-root-estimate}
\end{equation}
In particular, $|r_\lambda|\le C$ for $\tau\ge1$, uniformly in $\lambda$.

\emph{Proof.} The slope formula \eqref{eq:slope-solution-formula}
writes $\tau V_\lambda=(\tau V_0)(1+\lambda I)$.
For $\tau\ge\tau_1\ge1$, the first factor is strictly increasing by
Lemma~\ref{lem:ideal-positive}; the second is positive and nondecreasing.
Thus $x\mathcal V(x)$ increases for $x\ge1$, proving
\eqref{eq:gain-retention}.

To compute the direction ratio, put $\rho=1/x$ and define
$f(\rho)=\mathcal V(1/\rho)/\rho=x\mathcal V(x)$. Then
\[
 \mathcal V(x)=\rho f(\rho),\qquad
 \mathcal V_x=-\rho^2f-\rho^3f_\rho.
\]
Since $d/d\tau=\sqrt\beta\,d/dx$, the left side of
\eqref{eq:target-direction-ratio} is
\[
 -\frac{V_\lambda'}{V_\lambda}
 =-\sqrt\beta\,\frac{\mathcal V_x}{\mathcal V}
 =\sqrt\beta\,\rho+\sqrt\beta\,\rho^2\frac{f_\rho}{f}
 =\frac{\sqrt\beta}{x}-\frac{z(1/x)}{x^2},
 \qquad z(\rho):=-\sqrt\beta\,\frac{f_\rho(\rho)}{f(\rho)}.
\]
This proves the representation and identifies the quantity to estimate.
The preceding monotonicity gives $f>0$ and $f_\rho<0$ on $(0,1]$,
so $z>0$ there.

To bound $z$, rewrite \eqref{eq:scalar-amplification} in the $x$ clock:
\[
 \bigl[(1+x^4)\mathcal V_x\bigr]_x
       =(2/\beta-2x^2)\mathcal V.
\]
The right-hand coefficient changes sign for large $x$. The reciprocal
variable converts this into an equation with a positive coefficient
on $0\le\rho\le1$. Indeed,
\[
 (1+x^4)\mathcal V_x
 =-(\rho^2+\rho^{-2})f-(\rho^3+\rho^{-1})f_\rho.
\]
Differentiating with $d/dx=-\rho^2d/d\rho$ and substituting the scalar
equation gives
\begin{equation}
 \bigl[(1+\rho^4)f_\rho\bigr]_\rho=(2/\beta-2\rho^2)f.
 \label{eq:reciprocal-scalar}
\end{equation}
This linear equation has smooth coefficients through $\rho=0$, so $f$
extends there; monotonicity gives $f(0)\ge f(1)>0$.
Set
\[
 \mu(\rho)=\sqrt{\frac{2-2\beta\rho^2}{1+\rho^4}},
 \qquad \sigma=1-\rho.
\]
Logarithmic differentiation of \eqref{eq:reciprocal-scalar} yields
\[
 \sqrt\beta\,z_\sigma
 =\mu^2-z^2+\sqrt\beta\frac{4\rho^3}{1+\rho^4}z.
\]
At $\sigma=0$, $z=\sqrt\beta+l(\tau_1)$ is uniformly bounded
by \eqref{eq:initial-slope-bound}. The right side is positive at $z=0$
and negative at a sufficiently large fixed value of $z$. Hence
$0<z\le C$. Also $\mu\ge c>0$ and $|\mu_\sigma|\le C$.
For $w=z-\mu$ the equation reads
\[
 \sqrt\beta\,w_\sigma
 =-(z+\mu)w+\sqrt\beta
       \left(\frac{4\rho^3}{1+\rho^4}z-\mu_\sigma\right).
\]
Since $0\le\rho\le1$, we have $4\rho^3/(1+\rho^4)\le2$, and hence
\[
 \left|\frac{4\rho^3}{1+\rho^4}z-\mu_\sigma\right|
 \le\frac{4\rho^3}{1+\rho^4}z+|\mu_\sigma|
 \le2z+O(1)=O(1).
\]
The integrating-factor formula, using this bound and $z+\mu\ge c$, gives
\[
 w(\sigma)=O\!\left(e^{-c\sigma/\sqrt\beta}\right)+O(\sqrt\beta).
\]
For $\sigma\ge1/2$, the exponential is $O(\sqrt\beta)$, so
$w=O(\sqrt\beta)$ and $z=\mu+O(\sqrt\beta)$;
squaring this last relation proves
\eqref{eq:target-root-estimate}. The representation
\eqref{eq:target-direction-ratio} and $z=O(1)$ give $r_\lambda=O(1)$
when $x\ge1$; \eqref{eq:initial-slope-bound} handles
$1\le\tau\le\tau_1$.\hfill$\square$

\subsection{Actual direction ratios and the pressure sign}

For $\tau\ge1$, Proposition~\ref{prop:conditional-wave-control} gives
$V>0$. Set $r=U/V$ and $w=W/V$. The decomposition
\eqref{eq:scaled-components} then gives
\begin{equation}
 \begin{aligned}
 \zeta:=m_j/s_0&=Pp+\varepsilon Qq+Nn,\qquad |\zeta|^2=D_a,\\
 \nu:=v_j/V&=\varepsilon rp+q+\varepsilon wn,\qquad
 E_v:=|\nu|^2=1+\varepsilon^2(r^2+w^2).
 \end{aligned}
 \label{eq:normalized-wave-vectors}
\end{equation}
Thus $|m_j|=s_0\sqrt{D_a}$, as in \eqref{eq:scaled-pressure-quantities},
and $|v_j|=V\sqrt{E_v}$. Below we identify these vectors with their
component columns $(P,\varepsilon Q,N)^T$ and
$(\varepsilon r,1,\varepsilon w)^T$ in the orthonormal frame at the origin.

\refstepcounter{equation}
\label{lem:actual-direction-pressure}
\textbf{Lemma \theequation\ (Actual direction ratios and pressure sign).}
Assume the hypotheses of Proposition~\ref{prop:conditional-wave-control},
with $0<\beta\le\beta_0$ as above, and strengthen its error condition to
\begin{equation}
 e_*\Theta^{28}\le c\beta,
 \label{eq:geometry-smallness}
\end{equation}
with $c$ sufficiently small. Then the following estimates hold uniformly
for $y\in\Omega$ and $1\le\tau\le\Theta$:
\begin{equation}
 r:=U/V=r_\lambda+O(e_*\Theta^{24}),\qquad
 w:=W/V=-P_0r_\lambda-Q_0+O(e_*\Theta^{26}).
 \label{eq:actual-direction-ratios}
\end{equation}
In particular, $|r|\le C$, $|w|\le C\Theta^2$, and
\begin{equation}
 \frac JV=P_0+\beta+Q_0r_\lambda+O(e_*\Theta^{26}).
 \label{eq:positive-numerator}
\end{equation}
Consequently, recalling $J=(m_j\cdot Mv_j)/(s_0a)$ from
\eqref{eq:scaled-pressure-quantities},
\begin{equation}
 m_j\cdot Mv_j=s_0aJ>0.
 \label{eq:forward-pressure-sign}
\end{equation}

\textbf{Remark.} Condition \eqref{eq:geometry-smallness} implies the
smallness condition of Proposition~\ref{prop:conditional-wave-control}.
It follows, for example, from \eqref{eq:actual-final-smallness} when
$\beta x_{j-1}^2\in[1/2,2]$, as in
\eqref{intro:eq:activation-ranges}, and $x_{j-1}\ge1$.

\emph{Proof.} Proposition~\ref{prop:conditional-wave-control} gives
$V\ge V_\lambda/2>0$. Its relative comparison bound
\eqref{eq:conditional-Y-relative} and
Lemma~\ref{lem:target-direction} give
\[
 r-r_\lambda=
 \frac{(U-U_\lambda)-r_\lambda(V-V_\lambda)}{V}
 =O(e_*\Theta^{24}).
\]
Using $w=-(Pr+Q)/N$, $|Z|\le C\Theta^2$ and $|Z-Z_0|\le Ce_*\Theta^5$ proves
the second estimate in \eqref{eq:actual-direction-ratios} and the
stated size bounds. Divide \eqref{eq:all-numerator-errors} by $V$;
since $|Y|/V\le C$ and $|Q_0|\le C\Theta^2$, replacing $r$ by
$r_\lambda$ then gives \eqref{eq:positive-numerator}.

For $x\le1$, $P_0=x^2\ge0$, $Q_0=-2\sqrt\beta x\le0$, and
$r_\lambda=-l\le0$ by \eqref{eq:initial-slope-bound}. Thus
$Q_0r_\lambda\ge0$, so the leading expression in
\eqref{eq:positive-numerator} is at least $\beta$.
For $x\ge1$, write $z=z(1/x)$ here and in the calculations below.
Then \eqref{eq:target-direction-ratio} gives
\[
 P_0+\beta+Q_0r_\lambda
 =x^2-\beta+\frac{2\sqrt\beta z}{x}\ge\tfrac12x^2.
\]
Here $z>0$ by Lemma~\ref{lem:target-direction}, and
$\beta\le\beta_0\le1/2$ as chosen in
Subsection~\ref{subsec:initial-exponential-gain}; since $x\ge1$,
$x^2-\beta\ge x^2/2$.
Condition \eqref{eq:geometry-smallness} makes the error smaller than
half these lower bounds. Thus $J/V>0$; since $V,s_0,a>0$,
\eqref{eq:forward-pressure-sign} follows.\hfill$\square$

In the leading pressure Hessian \eqref{eq:leading-pressure-hessian},
this makes the coefficient
$-2\alpha_j\chi_1(m_j\cdot Mv_j)m_jm_j^T/|m_j|^2$
negative semidefinite. The sign of the full leading Hessian still
includes $f_{\delta_j}'$: its large positive part is favorable, while
its bounded negative part must be included in the later pressure budget.

\subsection{The frame supplied at the target}

Let $\tau_j=x_j/\sqrt\beta\le\Theta$, $x_j\ge2$, and evaluate the
following quantities at $(t_j,0)$. Recall the new frame
\eqref{eq:next-layer-frame} and the next-stage coefficient definitions
of Subsection~\ref{intro:subsec:stage-parameters}:
\[
 p_j=m_j/|m_j|,\qquad q_j=v_j/|v_j|,\qquad n_j=p_j\times q_j,
 \qquad a_{\rm new}=p_j\cdot Mq_j,\quad
 \beta_{\rm new}=\frac{n_j\cdot Mq_j}{a_{\rm new}}.
\]
Here $M=\nabla U^{j-1}(t_j,0)$ will be the older background for
the next stage. Transversality ensures that the new frame is orthonormal.

\refstepcounter{equation}
\label{prop:target-frame}
\textbf{Proposition \theequation\ (Conditional target-frame estimates).}
Under the hypotheses of Lemma~\ref{lem:actual-direction-pressure}
and the target conditions above,
\begin{align}
 \frac{a_{\rm new}}a
 &=1+O\!\left(x_j^{-4}+\frac\beta{x_j^2}
       +\frac{\sqrt\beta}{x_j^3}+e_*\Theta^{28}\right),
 \label{eq:target-a-transfer}\\
 \beta_{\rm new}x_j^2
 &=1+O\!\left(\sqrt\beta+x_j^{-4}+e_*\Theta^{28}\right).
 \label{eq:target-beta-transfer}
\end{align}
After decreasing $\beta_0,c$ if necessary, both new coefficients are
positive. The new normal also satisfies the compression estimate
\begin{equation}
 p_j\cdot Mp_j\le
 -c\frac{\sqrt{a h_*}\sqrt\beta}{x_j}+C(G_*+E_*).
 \label{eq:target-compression}
\end{equation}

\textbf{Remark.} The precise asymptotic in \eqref{eq:target-beta-transfer}
is stronger than the later construction needs: in
\eqref{eq:induction-parameter-transfer} it is used only to restore
$1/2\le\beta_{j+1}x_j^2\le2$. Subsequent estimates use this fixed
two-sided range, not convergence to $1$ or its rate. In contrast,
the relative errors in $a_{j+1}/a_j$ must be summable to keep $a_j$
in its prescribed range. The cancellation in the proof remains
essential to obtain the positive scale $\beta_{j+1}\asymp x_j^{-2}$;
estimating the large terms separately would lose that scale.

\emph{Proof.} Write $x=x_j$, use $\zeta,\nu,E_v$ from
\eqref{eq:normalized-wave-vectors}, and set $T=M\nu/a$.
Direct normalization and the definition of $J$ in
\eqref{eq:scaled-pressure-quantities} give
\begin{equation}
 \begin{aligned}
 \frac{a_{\rm new}}a
 &=\frac{m_j\cdot Mv_j}{a|m_j||v_j|}
 =\frac{s_0aJ}{a\,s_0\sqrt{D_a}\,V\sqrt{E_v}}
 =\frac{J/V}{\sqrt{D_a}\sqrt{E_v}},\\
 \beta_{\rm new}
 &=\frac{n_j\cdot Mq_j}{a_{\rm new}}
 =\frac{(\zeta\times\nu)\cdot(aT)}
       {a_{\rm new}\sqrt{D_a}\sqrt{E_v}\sqrt{E_v}}
 =\frac{(\zeta\times\nu)\cdot T}{(J/V)\sqrt{E_v}}.
 \end{aligned}
 \label{eq:exact-frame-transfer}
\end{equation}
We first prove \eqref{eq:target-a-transfer}. For $J/V$, substituting
$P_0=x^2$, $Q_0=-2\sqrt\beta x$ and
\eqref{eq:target-direction-ratio} into \eqref{eq:positive-numerator}
gives the first estimate below. The estimate for $D_a$ is
\eqref{eq:actual-rational-errors} with $P_0=x^2$; that for $E_v$
follows from \eqref{eq:normalized-wave-vectors} and the bounds
$r=O(1)$, $w=O(\Theta^2)$ in \eqref{eq:actual-direction-ratios}:
\begin{equation}\label{eq:target-denominator-estimates}
 J/V=x^2-\beta+2\sqrt\beta z/x+O(e_*\Theta^{26}),\quad
 D_a=1+x^4+O(e_*\Theta^7),\quad E_v=1+O(\varepsilon^2\Theta^4).
\end{equation}
In particular,
\begin{equation}\label{eq:target-normalized-denominators}
 \frac{J/V}{x^2}
 =1+O\!\left(\frac\beta{x^2}+\frac{\sqrt\beta}{x^3}
                         +e_*\Theta^{26}\right),\qquad
 \frac{D_a}{1+x^4}=1+O(e_*\Theta^7).
\end{equation}
These normalized quantities and $E_v$ are bounded away from zero by
\eqref{eq:geometry-smallness} and $x\ge2$. Thus
\eqref{eq:exact-frame-transfer} gives
\[
 \frac{a_{\rm new}}a
 =\frac{(J/V)/x^2}{\sqrt{D_a/(1+x^4)}\sqrt{E_v}}
       \frac{x^2}{\sqrt{1+x^4}}
 =1+O\!\left(x^{-4}+\frac\beta{x^2}
       +\frac{\sqrt\beta}{x^3}+e_*\Theta^{26}\right).
\]
Here $x^2/\sqrt{1+x^4}=1+O(x^{-4})$, and
$\varepsilon^2\Theta^4\le Ce_*\Theta^{26}$ since $\varepsilon\le Ce_*$
and $e_*\le1$. This proves \eqref{eq:target-a-transfer}, with its
common exponent $28$; also $a_{\rm new}>0$ by
\eqref{eq:forward-pressure-sign}.

We next prove \eqref{eq:target-beta-transfer}.
Write $b_{ik}=B_{ik}/a$. The coefficient bounds
\eqref{eq:budget-principal-coefficients}--\eqref{eq:budget-other-coefficients},
with the budget \eqref{eq:actual-error-budget}, give
\[
 b_{pq}=1+O(e_*),\qquad b_{nq}=\beta+O(e_*),\qquad
 h/h_*=1+O(e_*),\qquad
 \varepsilon|b_{ik}|+|E|/a=O(e_*).
\]
Since $|r|\le C$ and $|w|\le C\Theta^2$ by
\eqref{eq:actual-direction-ratios}, the decomposition $M=B+hqp^T+E$
and $\nu=\varepsilon rp+q+\varepsilon wn$ yield
\[
 \begin{aligned}
 T_p
 &=b_{pq}+\varepsilon b_{pp}r+\varepsilon b_{pn}w
       +\frac{p\cdot E\nu}{a}
 =1+O(e_*\Theta^2),\\
 T_n
 &=b_{nq}+\varepsilon b_{np}r+\varepsilon b_{nn}w
       +\frac{n\cdot E\nu}{a}
 =\beta+O(e_*\Theta^2),\\
 \varepsilon T_q
 &=\frac h{h_*}r+\varepsilon^2b_{qp}r+\varepsilon b_{qq}
       +\varepsilon^2b_{qn}w+\frac{\varepsilon q\cdot E\nu}{a}\\
 &=\frac h{h_*}r+O(e_*\Theta^2)
 =r_\lambda+O(e_*\Theta^{24}).
 \end{aligned}
\]
Here the shear contributes $(h/h_*)r$ because $h_*\varepsilon^2/a=1$;
the last equality uses $r-r_\lambda=O(e_*\Theta^{24})$ from
\eqref{eq:actual-direction-ratios}.
The cross product is exactly
\[
 \zeta\times\nu=
 (-N+\varepsilon^2Qw,\ \varepsilon(Nr-Pw),\ P-\varepsilon^2Qr).
\]
Set $\widetilde T=T-T_q\nu$. Since $\nu_q=1$ and
$(\zeta\times\nu)\cdot\nu=0$,
\[
 \widetilde T_q=0,\qquad
 (\zeta\times\nu)\cdot T=(\zeta\times\nu)\cdot\widetilde T.
\]
Set $w_0=-P_0r_\lambda-Q_0$. The component estimates above and
\eqref{eq:actual-direction-ratios} give
\[
 \begin{aligned}
 \widetilde T_p
 &=T_p-(\varepsilon T_q)r
 =1-r_\lambda^2+O(e_*\Theta^{24}),\\
 \widetilde T_n
 &=T_n-(\varepsilon T_q)w
 =\beta-r_\lambda w_0+O(e_*\Theta^{26}).
 \end{aligned}
\]
Only $\varepsilon T_q$ occurs, so no inverse power of $\varepsilon$
is lost. At the target, \eqref{eq:target-direction-ratio} and
$P_0=x^2$, $Q_0=-2\sqrt\beta x$ give the exact leading identities
\[
 r_\lambda=\frac{\sqrt\beta x-z}{x^2},\qquad
 w_0=\sqrt\beta x+z,\qquad
 \beta-r_\lambda w_0
 =\beta-\frac{(\sqrt\beta x-z)(\sqrt\beta x+z)}{x^2}
 =\frac{z^2}{x^2}.
\]
Thus $\widetilde T_n=z^2/x^2+O(e_*\Theta^{26})$:
subtracting the velocity-parallel part exposes the cancellation.
By \eqref{eq:actual-ray-comparison}, the first and third cross-product
components satisfy
\[
 (\zeta\times\nu)_p=-1+O(e_*\Theta^5),\qquad
 (\zeta\times\nu)_n=x^2+O(e_*\Theta^5),
\]
where the terms with $\varepsilon^2$ are absorbed using
$\varepsilon\le Ce_*$, $|P|,|Q|,|w|\le C\Theta^2$ and $|r|\le C$.
Since $r_\lambda,z=O(1)$ and $x\le\Theta$, we obtain
\[
 \begin{aligned}
 (\zeta\times\nu)\cdot T
 &=-1+r_\lambda^2+z^2
   +O(e_*\Theta^{24}+x^2e_*\Theta^{26}+e_*\Theta^5)\\
 &=-1+r_\lambda^2+z^2+O(e_*\Theta^{28}).
 \end{aligned}
\]
For the remaining terms, \eqref{eq:target-direction-ratio} and
\eqref{eq:target-root-estimate} give
\[
 r_\lambda^2=O(\beta/x^2+x^{-4}),\qquad
 z^2=2+O(\sqrt\beta+x^{-4}).
\]
Since $\beta/x^2<\beta<\sqrt\beta$, the error bound for $z^2$
also absorbs the $r_\lambda^2$ contribution, giving
\[
 (\zeta\times\nu)\cdot T
 =1+O(\sqrt\beta+x^{-4}+e_*\Theta^{28}).
\]
Substituting \eqref{eq:target-normalized-denominators} and the estimate for
$E_v$ from \eqref{eq:target-denominator-estimates} into
\eqref{eq:exact-frame-transfer}, we obtain
\[
 \beta_{\rm new}x^2
 =\frac{(\zeta\times\nu)\cdot T}{[(J/V)/x^2]\sqrt{E_v}}
 =1+O(\sqrt\beta+x^{-4}+e_*\Theta^{28}),
\]
which proves \eqref{eq:target-beta-transfer}; division introduces
no additional power of $\Theta$. Positivity follows from $J/V>0$
and the earlier expression
$-1+r_\lambda^2+z^2+O(e_*\Theta^{28})$: indeed $r_\lambda^2\ge0$
and \eqref{eq:target-root-estimate} gives
$z^2\ge32/17-O(\sqrt\beta)$ for $x\ge2$.

Finally, the decomposition \eqref{eq:parent-shear} yields
\[
 p_j\cdot Mp_j=\frac{h\varepsilon PQ}{D_a}+O(G_*+E_*).
\]
At the target, the phase comparison and
\eqref{eq:geometry-smallness} give $P\ge cx^2$,
$Q\le-c\sqrt\beta x$, and $D_a\le Cx^4$.
Since $h\asymp h_*$ and $h_*\varepsilon=\sqrt{a h_*}$,
\eqref{eq:target-compression} follows.\hfill$\square$

These estimates explain the next stage's scales:
$a_{\rm new}$ stays close to $a$ (the relative errors are summable,
keeping $a_j$ in its prescribed interval), while
$\beta_{\rm new}\asymp x_j^{-2}$. In the next layer's scaled clock,
velocity amplification occurs on an order-one time scale, whereas
substantial normal turning occurs only at time
$\beta_{\rm new}^{-1/2}\asymp x_j\gg1$. This restores the long
amplification interval before the next turn.

\subsection{Amplitude ratios and normalization}

\refstepcounter{equation}
\label{lem:amplitude-ratios}
\textbf{Lemma \theequation\ (Target gain and amplitude ratios).}
Assume the hypotheses of Lemma~\ref{lem:actual-direction-pressure},
and let $x_j\ge2$ with $\tau_j=x_j/\sqrt\beta\le\Theta$.
Set $A(\tau,y)=|m_j|\,|v_j|$ and $A_j=A(\tau_j,0)$.
Then, uniformly for $y\in\Omega$ and $1\le\tau\le\Theta$,
\begin{equation}
 A(\tau,y)\asymp s_0\sqrt{1+x^4}\,V_\lambda(\tau).
 \label{eq:actual-amplitude-size}
\end{equation}
At the target $\tau=\tau_j$,
\begin{equation}
 A_j\ge c s_0 x_j(1+\lambda)
                  \exp\!\left(\frac1{\sqrt{2\beta}}\right).
 \label{eq:target-amplitude-gain}
\end{equation}
The amplitude ratios satisfy
\begin{equation}
 \sup_{1\le\tau\le\min(\Theta,\tau_j+1),\ y\in\Omega}
 \frac{A(\tau,y)}{A_j}\le C,
 \label{eq:forward-amplitude-ratio}
\end{equation}
\begin{equation}
 \sup_{0\le\tau\le1,\ y\in\Omega}\frac{A(\tau,y)}{A_j}
 \le C\Theta^{11}x_j^{-1}
                  \exp\!\left(-\frac1{\sqrt{2\beta}}\right).
 \label{eq:early-amplitude-ratio}
\end{equation}
After the budget normalization in \eqref{eq:actual-error-budget},
the constants are universal: in particular, they are independent of
the stage, $\lambda$, $y$, and the parent size constants $C_M,C_H$.
The latter affect the error budget required for the hypotheses,
not the amplitude-ratio constant.

\emph{Proof.} By \eqref{eq:normalized-wave-vectors},
$|m_j|=s_0\sqrt{D_a}$ and $|v_j|=V\sqrt{E_v}$.
The denominator estimate \eqref{eq:actual-rational-errors} gives
$D_a\asymp1+x^4$. The direction bounds
\eqref{eq:actual-direction-ratios}, together with
\eqref{eq:normalized-wave-vectors} and \eqref{eq:geometry-smallness},
give $E_v=1+O(\varepsilon^2\Theta^4)\asymp1$.
Finally, \eqref{eq:conditional-Y-relative} gives $V\asymp V_\lambda$.
These prove \eqref{eq:actual-amplitude-size}. More explicitly,
the relative comparisons may be made to lie between $1/2$ and $2$
by a universal smallness threshold on $e_*\Theta^{28}/\beta$.
The resulting comparison constants depend only on the ideal scalar
estimates, which are uniform in $\lambda$ and $0<\beta\le\beta_0$;
no parent norm enters them. Combining this at the target
with \eqref{eq:exponential-gain} and \eqref{eq:gain-retention} proves
\eqref{eq:target-amplitude-gain}.

For $x\le1$, $V_\lambda$ increases by Lemma~\ref{lem:exponential-gain};
for $x\ge1$, $x^2\mathcal V(x)=x\cdot x\mathcal V(x)$ increases by
Lemma~\ref{lem:target-direction} and is comparable to
$\sqrt{1+x^4}\mathcal V(x)$. Thus the amplitude up to the target
is at most $CA_j$. Beyond the target, use $x'=\sqrt\beta$ and
$V_\lambda'/V_\lambda=-r_\lambda$ from \eqref{eq:target-direction-ratio}:
\[
 \frac{d}{d\tau}\log(x^2\mathcal V(x))
 =\frac{2\sqrt\beta}{x}-r_\lambda
 =\frac{\sqrt\beta}{x}+\frac{z}{x^2}\le C.
\]
One extra unit of scaled time costs only a constant factor, giving
\eqref{eq:forward-amplitude-ratio} on the stated interval.

For $0\le\tau\le1$, Lemma~\ref{lem:ideal-positive} gives
$V_0(\tau)\le V_0(1)\le e^{\sqrt2}$. The ideal bound \eqref{eq:ideal-Y-lambda-bound}
and the absolute comparison \eqref{eq:conditional-Y-absolute} give
$|Y|\le C\Theta^7(1+\lambda)$.
Transversality gives $|W|\le C\Theta^2|Y|$, while
$|m_j|\le Cs_0\Theta^2$. Hence $A\le Cs_0\Theta^{11}(1+\lambda)$.
Dividing by \eqref{eq:target-amplitude-gain} proves
\eqref{eq:early-amplitude-ratio}; the factor $1+\lambda$ cancels.
\hfill$\square$

The normalization \eqref{intro:eq:amplitude-normalization}
and the gain \eqref{eq:target-amplitude-gain} give
\begin{equation}
 \alpha_j=\frac{\delta_jh_j}{A_j}\le
 \frac{C\delta_jh_j}{s_0x_j(1+\lambda)}
 \exp\!\left(-\frac1{\sqrt{2\beta}}\right).
 \label{eq:quantitative-normalization}
\end{equation}
Moreover, the leading gradient \eqref{eq:gradient} has wave-amplitude
factor
\[
 \frac{\alpha_j}{\delta_j}|m_j|\,|v_j|
 =h_j\frac{A(\tau,y)}{A_j}.
\]
The cutoff $\chi_1$ and normalized profile derivative $\delta_jf_{\delta_j}'$
multiply this factor. Thus \eqref{eq:forward-amplitude-ratio} controls
its size on the stated interval relative to the prescribed target, while \eqref{eq:early-amplitude-ratio}
gives its small early-to-target ratio. Extending that estimate back to
physical time zero requires the history construction in
Section~\ref{sec:history-activation}.

\section{The wave history and activation data}
\label{sec:history-activation}
\fancyhead[R]{\small History and activation}

Sections~\ref{sec:comparison-proofs} and~\ref{sec:quantitative-amplification}
start from activation data at $t_a=t_{j-1}$. In an unforced Euler
construction, the new wave must already have a history from physical
time zero. We now construct that history and derive the activation
and size estimates used above. The hypotheses concern the existing
parent flow; their verification at every stage remains part of the
pressure and parameter induction. The construction supplies the
activation inputs \eqref{eq:budget-activation-data} and the history
part of the amplitude estimate \eqref{eq:history-target-ratio}.

\subsection{The parent inputs and the transported normal}
\label{sec:history-parent-inputs}

Let $0<t_a\le1$ and use the parent particle map $X(t,b)$ in physical
labels, with $X(0,b)=b$. Write $F=\partial_bX$,
$M=(\nabla U^{j-1})\circ X$, and $H=(\nabla^2p^{j-1})\circ X$.
The parent identities \eqref{intro:eq:parent} hold in these physical
label coordinates as well.
All spatial estimates below concern $|b|\le\ell/2$.
Assume the following inputs throughout $0\le t\le t_a$.

\emph{Regularity and coercivity.} The matrices $F,F^{-1},F_t,M,H$
and their first physical-label derivatives are bounded by $CK_h^{\kappa}$,
and $t_a^{-1}\le CK_h^{\kappa}$, with $K_h\ge2$. Assume also the Hessian bound
\begin{equation}
 H(t,b)\le t_a^{-2}I.
 \label{eq:history-coercivity-input}
\end{equation}
The matrix inequality is an upper quadratic-form bound. Large negative
eigenvalues do not weaken it. The particle-map bounds described in
\eqref{eq:parent-exponent-budget} supply the stated regularity; when
higher label derivatives are used, assume the corresponding bounds
from that same hierarchy.

\emph{Activation geometry at the origin and history sizes.} At $(t_a,0)$,
use the inherited orthonormal frame $(p,q,n)$ and write
\begin{equation}
 M=h_*qp^T+\mathcal B,\qquad \mathcal B=B+E,
 \qquad \mathcal B_{pp}<0,\qquad \|\mathcal B\|\le\zeta h_*.
 \label{eq:history-activation-geometry}
\end{equation}
Here $h_*\ge1$ and $\zeta>0$ will be sufficiently small. Assume also
\begin{equation}
 \sup|M|\le C_Mh_*,\qquad \sup|H|\le C_Hh_*h_{\rm old},\qquad
 1\le h_{\rm old}\le h_*,\qquad t_a^{-1}\le h_*.
 \label{eq:history-size-input}
\end{equation}
The suprema cover the history interval and label ball. Constants below
may depend on $C_M,C_H$, but not on the stage.

Choose the fixed unit normal and its transported normal by
\begin{equation}
 m_0=\frac{F(t_a,0)^Tn}{|F(t_a,0)^Tn|},\qquad
 m(t,b)=F(t,b)^{-T}m_0.
 \label{eq:history-normal-choice}
\end{equation}
Then $m(t_a,0)=s_0n$, with
$s_0=|F(t_a,0)^Tn|^{-1}$ and $s_0+s_0^{-1}\le CK_h^{\kappa}$.

\subsection{The history problem as a variational problem}

We want to choose, at activation time, the wave's velocity whose
history also solves the wave equations from time zero. This suggests
a boundary-value problem amenable to variational methods. Indeed, from
\eqref{eq:leading-velocity} the leading increment is
$w_{\lead}=(\ell/k)\alpha\chi_1v f_\delta(\theta)$, with $v$ solving the linear
velocity-amplitude equation \eqref{intro:eq:wave}. Hence
$(1+s)w_{\lead}=(\ell/k)\alpha\chi_1(1+s)v f_\delta(\theta)$
could also serve as the leading increment. This motivates viewing $v$
as the amplitude of a velocity variation, with the scalar oscillatory
factor separated out. Thus we consider a smooth family of velocities $u_s$ and their
particle maps $X_s$, with $u_0=U^{j-1}$ and $X_0=X$:
\[
 \partial_tX_s(t,b)=u_s(t,X_s(t,b)),\qquad X_s(0,b)=b.
\]
Write $\eta=\partial_sX_s|_{s=0}$ and
$\widehat v=(\partial_su_s|_{s=0})(t,X(t,b))$.
Differentiating the flow equation in $s$ gives
\begin{equation}
 \eta_t=\widehat v+M\eta,\qquad
 \widehat v=\eta_t-M\eta,\qquad \eta(0,b)=0.
 \label{eq:history-kinematic-variation}
\end{equation}
This is the standard relation between particle displacement and Eulerian
velocity variation; see \cite[Theorem~6.1, equation~(6.8)]{holm-marsden-ratiu1998}.

Motivated by \eqref{eq:history-kinematic-variation}, we construct the
wave coefficient as $v=\eta_t-M\eta$, where $\eta$ now denotes the
displacement (distinct from the small coefficient parameter in
\eqref{eq:short-eta}),
with zero initial value $\eta(0)=0$ and a terminal value $\eta(t_a)$ to be chosen.
We impose $m\cdot\eta=0$ to obtain the wave's
transversality. Indeed, differentiating this
constraint using $m_t=-M^Tm$ gives $m\cdot v=0$.
Moreover, \eqref{intro:eq:parent} and invertibility of $F$
give $M_t+M^2=-H$. Thus $v_t+Mv=\eta_{tt}+H\eta$, and the desired
velocity equation requires its transverse component to vanish:
\begin{equation}
 v_t+Mv=\eta_{tt}+H\eta,\qquad
 \operatorname{proj}_{m^\perp}(\eta_{tt}+H\eta)=0.
 \label{eq:history-displacement-equation}
\end{equation}
The second equation is precisely the transverse part of the desired
velocity equation: its remaining component is parallel to $m$ and
will be fixed by $m\cdot v=0$.

This transverse displacement equation is the Euler--Lagrange equation
of the quadratic action below, with transverse variations vanishing
at both endpoints. To impose transversality, choose a fixed $3\times2$
matrix $R_\perp$ whose columns form an orthonormal basis of $m_0^\perp$:
\begin{equation}
 R_\perp^TR_\perp=I_2,\qquad
 \operatorname{ran}R_\perp=m_0^\perp.
 \label{eq:history-transverse-basis}
\end{equation}
Then write
$\eta=FR_\perp\xi$: indeed,
$m\cdot FR_\perp\xi=m_0\cdot R_\perp\xi=0$, and invertibility of $F$
shows that every transverse displacement has this form.

For a coordinate endpoint $\xi_{\mathrm{end}}\in\mathbb R^2$, the history problem
is therefore to minimize
\begin{equation}
 \xi(0,b)=0,\qquad \xi(t_a,b)=\xi_{\mathrm{end}},
 \qquad
 \mathcal Q_b(\eta)=\int_0^{t_a}(|\eta_t|^2-\eta\cdot H\eta)\,dt.
 \label{eq:history-action}
\end{equation}
This quadratic action is the second variation, with the parent pressure
held fixed, of the classical particle action
$\int_0^{t_a}(\tfrac12|\dot\gamma|^2-P^{j-1}(t,\gamma))\,dt$
at $\gamma=X(t,b)$, restricted here to $\eta\perp m$.
For the pressure-augmented least-action principle and the use of an
upper pressure-Hessian bound to obtain short-time minimization, see
\cite[\S II]{brenier2008}.

The same coordinate endpoint will be used at every label. Its physical
endpoint $F(t_a,b)R_\perp\xi_{\mathrm{end}}$ varies with $b$.
Zero initial displacement does not mean zero initial velocity.
The endpoint will be chosen in the next subsection to obtain the
required activation signs. Minimization first supplies a unique
history for each endpoint; its positive quadratic action will also
control the endpoint-to-velocity map used in that choice.

\refstepcounter{equation}
\label{prop:history-existence}
\textbf{Proposition \theequation\ (Existence, uniqueness, and wave dynamics).}
Assume the smooth parent identities
\eqref{intro:eq:parent} and the coercivity condition
\eqref{eq:history-coercivity-input}, together with the quantitative
regularity bounds stated under \emph{Regularity and coercivity} in
Section~\ref{sec:history-parent-inputs}.
Let $m$ and $R_\perp$ be defined by
\eqref{eq:history-normal-choice} and \eqref{eq:history-transverse-basis}.
For each fixed coordinate endpoint $\xi_{\mathrm{end}}\in\mathbb R^2$ and label $b$,
the action in \eqref{eq:history-action} has a unique stationary path
among $\xi\in H^1((0,t_a);\mathbb R^2)$ with the stated endpoints.
This path is the unique minimizer. It is smooth in time and in the
labels to the regularity of the parent, and its velocity
\begin{equation}
 v=\eta_t-M\eta=FR_\perp\xi_t
 \label{eq:history-velocity}
\end{equation}
satisfies $m\cdot v=0$ and the wave equations
\eqref{intro:eq:normal} and \eqref{intro:eq:wave}.

\textbf{Remark.} Existence and uniqueness require only smooth coefficients,
invertibility of $F$, and the stated coercivity condition. The polynomial
bounds above are retained to obtain the uniform history and label-derivative
estimates in the later subsections.

\emph{Proof.} For any $H^1$ path with $\eta(0)=0$, the identity
$\eta(t)=\int_0^t\eta_t(s)\,ds$ and Cauchy--Schwarz give
\begin{equation}
 |\eta(t)|^2\le t\int_0^t|\eta_t(s)|^2\,ds,
 \qquad
 \int_0^{t_a}|\eta|^2\,dt\le\frac{t_a^2}{2}
                              \int_0^{t_a}|\eta_t|^2\,dt.
 \label{eq:history-endpoint-poincare}
\end{equation}
For the second bound, integrate the first and interchange the order
of integration: the weight on $|\eta_t(s)|^2$ is
$(t_a^2-s^2)/2\le t_a^2/2$. These estimates also hold for
Hilbert-space-valued paths, with norms in place of absolute values.
They apply on any interval $[0,T]$; for a path vanishing at $T$,
reverse time to obtain
$\|\eta\|_{L_t^2}^2\le(T^2/2)\|\eta_t\|_{L_t^2}^2$ and
$\|\eta(0)\|^2\le T\|\eta_t\|_{L_t^2}^2$.
Consequently, the Hessian bound \eqref{eq:history-coercivity-input} implies
\begin{equation}
 \mathcal Q_b(\eta)
 \ge\|\eta_t\|_{L^2_t}^2-t_a^{-2}\|\eta\|_{L^2_t}^2
 \ge\tfrac12\|\eta_t\|_{L^2_t}^2.
 \label{eq:history-action-coercivity}
\end{equation}
This also controls $\xi_t$: differentiating
at a fixed label, with coefficient suprema over the history interval
and label ball as in Subsection~\ref{sec:history-parent-inputs},
$\xi=R_\perp^TF^{-1}\eta$ gives
\[
 \xi_t=R_\perp^TF^{-1}(\eta_t-M\eta),\qquad
 \|\xi_t\|_{L^2_t}
 \le \sup|F^{-1}|
       (1+t_a\sup|M|/\sqrt2)\|\eta_t\|_{L^2_t}
 \le CK_h^{2\kappa}\|\eta_t\|_{L^2_t}.
\]
Put $L(t)=t\xi_{\mathrm{end}}/t_a$. Write $\xi=L(t)+\widetilde\xi$, where
$\widetilde\xi\in\mathcal H:=H_0^1((0,t_a);\mathbb R^2)$, equipped with the norm
$\|z\|_{\mathcal H}=\|z_t\|_{L^2_t}$. This is a Hilbert norm by
\eqref{eq:history-endpoint-poincare}. The symmetric bilinear form obtained by
polarization is
\begin{equation}
 \mathfrak a_b(z,w)=\int_0^{t_a}
 \left((FR_\perp z)_t\cdot(FR_\perp w)_t
       -(FR_\perp z)\cdot H(FR_\perp w)\right)\,dt.
 \label{eq:history-bilinear-form}
\end{equation}
Thus $\mathcal Q_b(FR_\perp z)=\mathfrak a_b(z,z)$.
The coefficient bounds and $\|z\|_{L_t^2}\le t_a\|z_t\|_{L_t^2}/\sqrt2$ from
\eqref{eq:history-endpoint-poincare}
give boundedness, while \eqref{eq:history-action-coercivity} and the
preceding bound on $\xi_t$ give coercivity:
\begin{equation}
 |\mathfrak a_b(z,w)|\le CK_h^{3\kappa}
       \|z\|_{\mathcal H}\|w\|_{\mathcal H},\qquad
 \mathfrak a_b(z,z)\ge C^{-1}K_h^{-4\kappa}\|z\|_{\mathcal H}^2.
 \label{eq:history-form-bounds}
\end{equation}
The boundedness estimate also applies when one argument is $L$, since
$L(0)=0$. Consequently $w\mapsto-\mathfrak a_b(L,w)$ is a bounded
linear functional on $\mathcal H$. The Lax--Milgram theorem gives a
unique $\widetilde\xi\in\mathcal H$ satisfying
\[
 \mathfrak a_b(\widetilde\xi,w)=-\mathfrak a_b(L,w)
 \quad\text{for every }w\in\mathcal H.
\]
Hence $\xi=L+\widetilde\xi$ has the prescribed endpoints and
$\mathfrak a_b(\xi,w)=0$ for every zero-endpoint variation $w$.
This is exactly stationarity, since the first variation of the action
in direction $w$ is $2\mathfrak a_b(\xi,w)$.
Every other admissible path is $\xi+w$ for some $w\in\mathcal H$,
and symmetry gives
\[
 \mathcal Q_b(FR_\perp(\xi+w))-\mathcal Q_b(FR_\perp\xi)
 =2\mathfrak a_b(\xi,w)+\mathfrak a_b(w,w)
 \ge C^{-1}K_h^{-4\kappa}\|w\|_{\mathcal H}^2.
\]
The right side is strictly positive unless $w=0$. Thus $\xi$ is
the unique minimizer as well as the unique stationary path.

For a zero-endpoint variation $\psi=FR_\perp b_0$, stationarity means
\[
 \int_0^{t_a}(\eta_t\cdot\psi_t-H\eta\cdot\psi)\,dt=0.
\]
Integrating by parts gives $(FR_\perp)^T(\eta_{tt}+H\eta)=0$,
initially in distributions. Since the columns of $FR_\perp$ span
$m^\perp$, this is the second equation in
\eqref{eq:history-displacement-equation}. The first identity there
then implies $v_t+Mv=\alpha m$ for a scalar $\alpha$.
Moreover, $F_t=MF$ gives $v=\eta_t-M\eta=FR_\perp\xi_t$, so
$m\cdot v=0$. The constraint fixes the scalar component exactly as in the derivation
of \eqref{intro:eq:wave}, giving
$v_t=-Mv+2m(m\cdot Mv)/|m|^2$.
Smoothness in time follows from this
linear ODE and $\eta_t=M\eta+v$; label regularity is established below.
\hfill$\square$

\textbf{Remark.} For the later quantitative estimates, we record
the coordinate form. Substituting $\eta=FR_\perp\xi$ and
$F_{tt}=-HF$ into the stationarity equation $(FR_\perp)^T(\eta_{tt}+H\eta)=0$,
which follows from \eqref{eq:history-displacement-equation}, gives
\begin{equation}
 R_\perp^TF^T(\eta_{tt}+H\eta)=K_R\xi_{tt}+2R_\perp^TF^TF_tR_\perp\xi_t=0,
 \qquad K_R=R_\perp^TF^TFR_\perp.
 \label{eq:history-coordinate-ode}
\end{equation}

\subsection{Choosing the endpoint to obtain the activation signs}

The endpoint must produce $v_q=1$ and $v_p\le0$ at activation.
We first work only at the origin, with unnormalized velocities.

For a physical displacement endpoint $d\in n^\perp$, let $\eta^d$
be the minimizing history with $\eta^d(0,0)=0$ and
$\eta^d(t_a,0)=d$. Define
\[
 \Lambda d=\operatorname{proj}_{n^\perp}\eta_t^d(t_a).
\]
\refstepcounter{equation}
\label{lem:history-endpoint-map}
\textbf{Lemma \theequation\ (The displacement endpoint map).}
Assume the hypotheses of Proposition~\ref{prop:history-existence}
and the history size bounds \eqref{eq:history-size-input}.
Then $\Lambda:n^\perp\to n^\perp$ is linear and symmetric, and
\begin{equation}
 \frac{1}{2t_a}I\le\Lambda\le C_0h_*I.
 \label{eq:history-endpoint-map}
\end{equation}
Here $C_0$ depends only on the constants in
\eqref{eq:history-size-input}, not on the stage.

\emph{Proof.} Linearity follows from the linear variational equation
and uniqueness. Integration by parts and stationarity give
\[
 d\cdot\Lambda e
 =\int_0^{t_a}(\eta_t^d\cdot\eta_t^e-H\eta^d\cdot\eta^e)\,dt
 =e\cdot\Lambda d,
 \qquad d\cdot\Lambda d=\mathcal Q_0(\eta^d)\ge0.
\]
Thus $\Lambda$ is symmetric. Moreover, the endpoint conditions give
$d=\eta^d(t_a)$. By \eqref{eq:history-endpoint-poincare} and
\eqref{eq:history-action-coercivity},
\[
 d\cdot\Lambda d=\mathcal Q_0(\eta^d)
 \ge\frac12\int_0^{t_a}|\eta_t^d|^2\,dt
 \ge\frac{|d|^2}{2t_a}.
\]
This proves the lower bound in \eqref{eq:history-endpoint-map};
in particular, $\Lambda$ is positive definite on $n^\perp$.

It also has a uniform upper bound in units of $h_*$. Set
$L=t_a^{-1}+\sup|M|+\sup|H|^{1/2}\le Ch_*$.
As a comparison path, take zero displacement until $t_a-L^{-1}$,
then use
\[
 \eta^{\rm cmp}(t)=L(t-t_a+L^{-1})\,
                      \operatorname{proj}_{m(t)^\perp}d.
\]
This path is admissible and has endpoint $d$. The normal equation
$m_t=-M^Tm$ in \eqref{intro:eq:normal} gives
$|\partial_t\operatorname{proj}_{m^\perp}|\le C\sup|M|$.
To estimate its action explicitly, put $t_*=t_a-L^{-1}$ and
$\Pi(t)=\operatorname{proj}_{m(t)^\perp}$. On $[t_*,t_a]$,
\[
 \eta^{\rm cmp}=L(t-t_*)\Pi d,\qquad
 \eta_t^{\rm cmp}=L\Pi d+L(t-t_*)\Pi_t d.
\]
Since $0\le L(t-t_*)\le1$, $\|\Pi\|\le1$, and
$\|\Pi_t\|\le C\sup|M|$, we have
\[
 |\eta^{\rm cmp}|\le |d|,\qquad
 |\eta_t^{\rm cmp}|^2
 \le C(L^2+\sup|M|^2)|d|^2.
\]
Both the path and its weak derivative vanish before $t_*$.
The path is continuous at $t_*$, so its derivative has no singular
term there. Integrating over an interval of length $L^{-1}$ gives
\[
 \int_0^{t_a}|\eta_t^{\rm cmp}|^2\,dt
 \le C\left(L+\frac{\sup|M|^2}{L}\right)|d|^2,
 \qquad
 \left|\int_0^{t_a}\eta^{\rm cmp}\cdot H\eta^{\rm cmp}\,dt\right|
 \le\frac{\sup|H|}{L}|d|^2.
\]
Consequently
\[
 \mathcal Q_0(\eta^{\rm cmp})\le
 C\left(L+\frac{\sup|M|^2}{L}
            +\frac{\sup|H|}{L}\right)|d|^2
 \le C_0h_*|d|^2.
\]
Here $\sup|M|\le L$ and $\sup|H|\le L^2$, so each
term in parentheses is at most $L\le Ch_*$. Finally, minimality gives
$d\cdot\Lambda d=\mathcal Q_0(\eta^d)
\le\mathcal Q_0(\eta^{\rm cmp})\le C_0h_*|d|^2$ for every
$d\in n^\perp$, proving the upper bound in
\eqref{eq:history-endpoint-map}.
\hfill$\square$

\refstepcounter{equation}
\label{lem:history-endpoint-choice}
\textbf{Lemma \theequation\ (Activation endpoint).}
Assume the hypotheses of Lemma~\ref{lem:history-endpoint-map}
and \eqref{eq:history-activation-geometry}, with $\zeta$ sufficiently small.
Then there is a physical displacement endpoint $d$ with $|d|\le C/h_*$ whose history
velocity satisfies
\begin{equation}
 v_q(t_a,0)=1,\qquad -C\le v_p(t_a,0)\le0,\qquad v_n(t_a,0)=0.
 \label{eq:history-activation-signs}
\end{equation}

\emph{Proof.} By Lemma~\ref{lem:history-endpoint-map}, all entries of
$\Lambda$ in the basis $(p,q)$ are bounded in absolute value by $C_0h_*$,
and $\Lambda$ is positive definite. Put $b_*=\Lambda_{pq}$,
$d_*=\Lambda_{qq}$, and $u_* =\Lambda_{pp}-\mathcal B_{pp}$.
The lower bound in \eqref{eq:history-endpoint-map} and the assumed
sign $\mathcal B_{pp}<0$ in \eqref{eq:history-activation-geometry}
give $d_*\ge(2t_a)^{-1}>0$ and
$u_*>(2t_a)^{-1}>0$, so the divisions below are well defined.
For an unnormalized endpoint $d=d_pp+d_qq$, the terminal velocity
$w=\eta_t-M\eta$ has components
\begin{equation}\label{eq:history-endpoint-velocity-matrix}
 \begin{pmatrix}w_p\\w_q\end{pmatrix}
 =\begin{pmatrix}
 u_*&b_*-\mathcal B_{pq}\\
 b_*-h_*-\mathcal B_{qp}&d_*-\mathcal B_{qq}
 \end{pmatrix}
 \begin{pmatrix}d_p\\d_q\end{pmatrix}.
\end{equation}
If $b_*\le h_*/2$, choose $d=-p$. Then
$w_p=-u_*\le0$ and
$w_q=h_*+\mathcal B_{qp}-b_*\ge h_*/3$.
Also $|w_p|\le(C_0+\zeta)h_*$.

If $b_*>h_*/2$, positivity and the upper bound in
\eqref{eq:history-endpoint-map} give
$u_*>\Lambda_{pp}\ge b_*^2/d_*\ge h_*/(4C_0)$.
Choose $d=q-(b_*-\mathcal B_{pq})p/u_*$, whose norm is bounded
by a stage-independent constant. This makes $w_p=0$, while
\[
 w_q=\left(d_*-\frac{b_*^2}{u_*}\right)-\mathcal B_{qq}
 +\frac{b_*h_*+b_*(\mathcal B_{pq}+\mathcal B_{qp})
                 -h_*\mathcal B_{pq}-\mathcal B_{qp}\mathcal B_{pq}}{u_*}.
\]
The bracket is nonnegative because $u_*>\Lambda_{pp}$ and
$\Lambda_{pp}d_*\ge b_*^2$ by \eqref{eq:history-endpoint-map}.
The numerator is at least
$h_*^2/2-C(C_0+1)\zeta h_*^2$, whereas
$u_*\le(C_0+\zeta)h_*$. Thus $w_q\ge ch_*$ after fixing $\zeta$.

In either case divide the endpoint, displacement, and velocity by
the positive terminal value $w_q$. This proves the stated size and
component bounds; $v_n=0$ follows from transversality $m\cdot v=0$ and
$m(t_a,0)=s_0n$.\hfill$\square$

\textbf{Remark.} The signs supplied by positivity of $\Lambda$,
$\mathcal B_{pp}<0$, and $h_*>0$ ensure that the endpoint-to-velocity
map represented by the matrix in \eqref{eq:history-endpoint-velocity-matrix}
attains a velocity satisfying the sign and magnitude constraints in
\eqref{eq:history-activation-signs}, with the displacement endpoint bound
$|d|\le C/h_*$, even if the matrix is singular.

Use this normalized terminal displacement to fix
\begin{equation}
 \xi_{\mathrm{end}}=R_\perp^TF(t_a,0)^{-1}d
 \label{eq:history-common-endpoint}
\end{equation}
at every label. Since $m_0\cdot F(t_a,0)^{-1}d=m(t_a,0)\cdot d=0$,
and $R_\perp R_\perp^T$ is the projection on $m_0^\perp$, this choice
recovers the terminal displacement
$\eta(t_a,0)=F(t_a,0)R_\perp \xi_{\mathrm{end}}=d$
at the origin. It has $|\xi_{\mathrm{end}}|\le CK_h^{\kappa}$, using $h_*\ge1$.
For the scaled variables of Section~\ref{sec:comparison-proofs}, put
$\lambda=-v_p(t_a,0)/\varepsilon$. Then
\[
 Z(0,0)=(0,0,1)^T,\qquad Y(0,0)=(-\lambda,1)^T,
 \qquad 0\le\lambda\le C/\varepsilon.
\]
The forward (which means from activation onward, here and below)
comparison in Proposition~\ref{prop:conditional-wave-control}
only needs $\lambda\ge0$; the upper bound is
an additional output of the history construction.

\subsection{Uniform label estimates and the history size}

All constants in this subsection are independent of the stage.

\refstepcounter{equation}
\label{lem:history-label-bounds}
\textbf{Lemma \theequation\ (Spatial activation accuracy and history bounds).}
Assume the parent identities, coercivity, and quantitative regularity
hypotheses of Proposition~\ref{prop:history-existence} on
$0\le t\le t_a$, $|b|\le\ell/2$.
Use the activation endpoint supplied by
Lemma~\ref{lem:history-endpoint-choice}, with the same coordinate endpoint
\eqref{eq:history-common-endpoint} at every label.
Then
\begin{equation}
 \begin{aligned}
 \sup_{0\le t\le t_a,\,b}(|m|+|\nabla_bm|)&\le CK_h^\kappa,\\
 \sup_{0\le t\le t_a,\,b}|v|&\le CK_h^{12\kappa},\qquad
 \sup_{0\le t\le t_a,\,b}|\nabla_bv|\le CK_h^{19\kappa}.
 \end{aligned}
 \label{eq:history-uniform-bounds}
\end{equation}
If, in addition, the scaling parameter
$\varepsilon=\sqrt{a/h_*}$ defined in \eqref{eq:amplification-scales}
satisfies $0<\varepsilon\le1$, then the scaled activation data obey
\begin{equation}
 \sup_{y\in\Omega}|Z(0,y)-(0,0,1)^T|
 +\sup_{y\in\Omega}|Y(0,y)-(-\lambda,1)^T|
 \le C\ell K_h^{19\kappa}/\varepsilon,
 \label{eq:history-scaled-data}
\end{equation}
which supplies \eqref{eq:budget-activation-data}. The history amplitude
satisfies
\begin{equation}
 \sup_{0\le t\le t_a,\ y\in\Omega}|m(t,\ell y)|\,|v(t,\ell y)|
 \le CK_h^{13\kappa}.
 \label{eq:history-amplitude-bound}
\end{equation}

\emph{Proof.} Put $A=K_h^\kappa\ge1$ for this proof.
All $L_t^2$ and $L_t^\infty$ norms below are over $[0,t_a]$ at
fixed particle labels, with bounds uniform for $|b|\le\ell/2$.
For label differences, both labels lie in this ball. Unsubscripted
matrix norms are operator norms; the coefficient bounds hold uniformly
in both time and labels.

Since $m=F^{-T}m_0$ with $m_0$ a fixed unit vector,
$|m|+|\nabla_bm|\le\|F^{-1}\|+|\nabla_bF^{-1}|\le CA$
by the parent regularity bounds, proving the first line of
\eqref{eq:history-uniform-bounds}.

For paths vanishing at
zero, $\|\xi\|_{L_t^2}\le t_a\|\xi_t\|_{L_t^2}/\sqrt2$ and $t_a\le1$.
Let $\mathfrak a_b$ be the polarization of the action defined in
\eqref{eq:history-bilinear-form}. The boundedness and coercivity
estimates \eqref{eq:history-form-bounds}, proved in
Proposition~\ref{prop:history-existence}, give
\[
 |\mathfrak a_b(\xi,\psi)|\le CA^3\|\xi_t\|_{L_t^2}\|\psi_t\|_{L_t^2},
 \qquad
 \mathfrak a_b(\psi,\psi)\ge C^{-1}A^{-4}\|\psi_t\|_{L_t^2}^2.
\]
Since $|\xi_{\mathrm{end}}|\le CA$ and
$t_a^{-1}\le CA$, the linear interpolation $L(t)=t\xi_{\mathrm{end}}/t_a$ satisfies
$\|L_t\|_{L_t^2}\le CAt_a^{-1}t_a^{1/2}\le CA^{3/2}$.

Put $z=\xi-L$, which has zero endpoints.
Stationarity of $\xi$ gives $\mathfrak a_b(\xi,z)=0$, so
$\mathfrak a_b(z,z)=-\mathfrak a_b(L,z)$, hence
\[
 cA^{-4}\|z_t\|_{L_t^2}^2\le|\mathfrak a_b(z,z)|
 =|\mathfrak a_b(L,z)|
 \le CA^3\|L_t\|_{L_t^2}\|z_t\|_{L_t^2}.
\]
Here we used coercivity on the left and boundedness on the right.
Dividing by $\|z_t\|_{L_t^2}$ when it is nonzero gives
$\|z_t\|_{L_t^2}\le CA^{4+3}\|L_t\|_{L_t^2}=CA^7\|L_t\|_{L_t^2}$;
the same bound is immediate if $z_t=0$.
Finally, $\xi_t=L_t+z_t$ and the triangle inequality give
\[
 \|\xi_t\|_{L_t^2}\le C(1+A^7)\|L_t\|_{L_t^2}
 \le CA^{17/2}\le CA^9.
\]
For two labels, $\Delta\xi=\xi(\cdot,b_1)-\xi(\cdot,b_2)$
has zero endpoints. Testing the difference equation with $\Delta\xi$ gives
\[
 \mathfrak a_{b_1}(\Delta\xi,\Delta\xi)
 =-(\mathfrak a_{b_1}-\mathfrak a_{b_2})(\xi(\cdot,b_2),\Delta\xi).
\]
Differentiating the products in the form $\mathfrak a_b$ costs no extra power because
each first coefficient derivative is also bounded by $CA$. Thus
\[
 |(\mathfrak a_{b_1}-\mathfrak a_{b_2})(\xi(\cdot,b_2),\Delta\xi)|
 \le CA^3|b_1-b_2|\|\xi_t(\cdot,b_2)\|_{L_t^2}\|\Delta\xi_t\|_{L_t^2}.
\]
Coercivity and this coefficient-difference bound give
\[
 \begin{aligned}
 cA^{-4}\|\Delta\xi_t\|_{L_t^2}^2
 &\le\mathfrak a_{b_1}(\Delta\xi,\Delta\xi)\\
 &=\left|(\mathfrak a_{b_1}-\mathfrak a_{b_2})
                (\xi(\cdot,b_2),\Delta\xi)\right|\\
 &\le CA^3|b_1-b_2|\,\|\xi_t(\cdot,b_2)\|_{L_t^2}
                         \|\Delta\xi_t\|_{L_t^2}\\
 &\le CA^{12}|b_1-b_2|\,\|\Delta\xi_t\|_{L_t^2}.
 \end{aligned}
\]
In the last step we used $\|\xi_t(\cdot,b_2)\|_{L_t^2}\le CA^9$.
Dividing by $\|\Delta\xi_t\|_{L_t^2}$ when it is nonzero yields
\[
 \|\Delta\xi_t\|_{L_t^2}\le CA^{4+3+9}|b_1-b_2|
 =CA^{16}|b_1-b_2|.
\]
The same estimate is immediate if $\Delta\xi_t=0$.
In \eqref{eq:history-coordinate-ode}, write
$\xi_{tt}=-C_b\xi_t$, where
$C_b=2K_R^{-1}R_\perp^TF^TF_tR_\perp$.
The smallest singular value of $FR_\perp$ is at least
$\|F^{-1}\|^{-1}$, so
\[
 \|K_R^{-1}\|\le CA^2,\quad
 |\nabla_bK_R|\le CA^2,\quad
 |\nabla_bK_R^{-1}|\le CA^6,\quad
 |C_b|\le CA^4,\quad |\nabla_bC_b|\le CA^8.
\]
Consequently
\[
 \|\xi_{tt}\|_{L_t^2}\le CA^{13},\qquad
 \|\Delta\xi_{tt}\|_{L_t^2}
 \le C(A^4A^{16}+A^8A^9)|b_1-b_2|
 \le CA^{20}|b_1-b_2|.
\]
For $g\in H^1(0,t_a)$, the one-dimensional interpolation estimate is
\[
 \|g\|_{L_t^\infty}\le C\left(t_a^{-1/2}\|g\|_{L_t^2}
                  +\|g\|_{L_t^2}^{1/2}\|g_t\|_{L_t^2}^{1/2}\right).
\]
Indeed, choose $t_0$ with $|g(t_0)|^2\le t_a^{-1}\|g\|_{L_t^2}^2$
and integrate $(|g|^2)'=2g\cdot g_t$ to get
$\|g\|_{L_t^\infty}^2\le t_a^{-1}\|g\|_{L_t^2}^2+2\|g\|_{L_t^2}\|g_t\|_{L_t^2}$.
Applying this estimate and $t_a^{-1}\le CA$ gives
\[
 \begin{aligned}
 \|\xi_t\|_{L_t^\infty}
 &\le C\left(A^{1/2+9}+A^{(9+13)/2}\right)\le CA^{11},\\
 \|\Delta\xi_t\|_{L_t^\infty}
 &\le C\left(A^{1/2+16}+A^{(16+20)/2}\right)|b_1-b_2|
 \le CA^{18}|b_1-b_2|.
 \end{aligned}
\]
Recalling $v=FR_\perp\xi_t$ from \eqref{eq:history-velocity},
multiplication by $FR_\perp$, including its label derivative, gives
$|v|\le CA^{12}$ and $|\nabla_bv|\le CA^{19}$.
Difference quotients converge to derivatives by uniqueness of the
differentiated variational problem.

For completeness, at fixed label order $r$ choose a nondecreasing
input exponent $\kappa_r=\max\{7,r+6\}$, as justified by
\eqref{eq:parent-exponent-budget}, bounding all parent coefficient
derivatives through order $r$.
Repeated differentiation of the form gives the conservative bounds
\[
 \begin{gathered}
 \|\partial_b^r\xi_t\|_{L_t^2}\le C_rK_h^{(7r+9)\kappa_r},\qquad
 |\partial_b^r C_b|\le C_rK_h^{(4r+4)\kappa_r},\\
 \|\partial_b^rv\|_{L_t^\infty}\le C_rK_h^{(7r+12)\kappa_r}.
 \end{gathered}
\]
Indeed, each differentiated variational equation costs at most
$K_h^{4\kappa_r}K_h^{3\kappa_r}$ times the lower derivative bounds;
the inverse-matrix differentiation formula gives the middle bound.
The time equation gives an $L^2_t$ bound with exponent $(7r+13)\kappa_r$.
Interpolation with the $(7r+9)\kappa_r$ bound, followed by multiplication
by $F$, then
gives the last bound. Also
$|\partial_b^rm|\le C_rK_h^{\kappa_r}$.
These establish smoothness at every supplied order without an
exponential-in-$\sup|M| t_a$ estimate.

For the activation data, use the conversion in
Subsection~\ref{subsec:wave-activation-budget}: the frame at the origin,
$s_0$, and $\varepsilon$ are independent of the label, so only $m,v$
are differentiated. At activation, the mean-value theorem therefore gives
\[
 |Z(0,y)-Z(0,0)|\le C\ell A^2/\varepsilon,\qquad
 |Y(0,y)-Y(0,0)|\le C\ell A^{19}/\varepsilon,
\]
where the first bound includes $s_0^{-1}\le CA$.
This proves \eqref{eq:history-scaled-data}. Finally
$|m||v|\le CA^{13}$ proves \eqref{eq:history-amplitude-bound}.
\hfill$\square$

\textbf{Remark.} With the fixed input choice $\kappa=7$, the resulting
budget is explicit:
\[
 \begin{array}{c|c}
 \text{quantity} & \text{bound, up to a stage-independent constant}\\ \hline
 |m|+|\nabla_bm| & K_h^7\\
 |v| & K_h^{84}\\
 |\nabla_bv| & K_h^{133}\\
 |Z(0,y)-Z(0,0)| & \ell K_h^{14}/\varepsilon\\
 |Y(0,y)-Y(0,0)| & \ell K_h^{133}/\varepsilon\\
 |m||v| & K_h^{91}
 \end{array}
\]
Thus the common localization term in \eqref{eq:actual-error-budget}
is $\ell K_h^{133}/\varepsilon$. To make its contribution to
$e_*\Theta^N$ at most a prescribed $\delta>0$, it suffices to choose
$\ell\le c\delta\varepsilon K_h^{-133}\Theta^{-N}$.
These are conservative upper bounds, not optimal exponents; no further
enlargement of $\kappa$ is implicit in this budget.

Thus the pre-activation amplitude is bounded by
$H_{\rm hist}=CK_h^{13\kappa}$ on $0\le t\le t_a=t_{j-1}$.
With the forward hypotheses of Lemma~\ref{lem:amplitude-ratios}
as well, \eqref{eq:target-amplitude-gain} gives
\begin{equation}
 \sup_{0\le t\le t_a,\ y\in\Omega}
 \frac{|m(t,\ell y)|\,|v(t,\ell y)|}{A_j}
 \le\frac{CK_h^{13\kappa}}{s_0x_j(1+\lambda)}
                     \exp\!\left(-\frac1{\sqrt{2\beta}}\right).
 \label{eq:history-target-ratio}
\end{equation}
Estimate \eqref{eq:history-target-ratio} controls the pre-activation
interval $0\le t\le t_{j-1}$, while \eqref{eq:early-amplitude-ratio}
controls the early post-activation interval. Together they bound the
early-to-target amplitude ratio (i.e., the quotient on the left side
of \eqref{eq:history-target-ratio}) from physical time zero.

\subsection{Localization and the first stage}

Use the target normalization \eqref{eq:quantitative-normalization}
with the prescribed $h_j$ from
Subsection~\ref{intro:subsec:stage-parameters}. Since
$f_{\delta_j}'(0)=\delta_j^{-1}$ and $\chi_1(0)=1$,
\eqref{eq:gradient} makes the gradient of the leading increment at $(t_j,0)$
exactly $h_jq_jp_j^T$.
During the history interval, its amplitude is bounded by
\[
 \frac{\alpha_j}{\delta_j}|m_j||v_j|
 \le h_j\frac{H_{\rm hist}}{A_j},
\]
whose ratio to $h_j$ is controlled by \eqref{eq:history-target-ratio}.
Since $\beta^{-1/2}\asymp x_{j-1}$, the gain is exponential in the
preceding stage parameter. The later parameter choices must show that
it dominates the frequency and localization costs in the initial data.

The history equations are ordinary differential equations at each label;
they contain no spatial or angular derivatives. Multiplying the
displacement by the time-independent factor
$\alpha_j\chi_1(y)f_{\delta_j}(\theta)$ therefore multiplies its wave
velocity by the same factor. This produces the localized leading
coefficient $A_1$ used in the increment construction, supported in
$\operatorname{supp}\chi_1\subset\Omega$.
The physical leading velocity includes the additional prefactor
$\ell_j/k_j$, as in \eqref{eq:leading-velocity}.
The full coefficient recursion and initial-support bounds are treated in
Sections~\ref{sec:packet-expansion} and~\ref{sec:initial-packet}.

If $t_a=0$, there is no history interval or endpoint problem. Since
$F(0,b)=I$, choose $m_0=n(0)$ and the constant initial velocity
$v(0,b)=q(0)$, then solve the forward wave ODEs. This gives
$s_0=1$, $Z(0,y)=(0,0,1)^T$, and $Y(0,y)=(0,1)^T$,
so $\lambda=0$ and the neighboring activation errors vanish.

The history and forward arguments now supply a wave defined from the
common physical initial time, with the activation data, spatial
accuracy, and history size used in Sections~\ref{sec:comparison-proofs}
and~\ref{sec:quantitative-amplification}. What remains conditional is
the parent input: the iteration must preserve the pressure upper bound,
activation geometry at the origin, and regularity hypotheses. The construction
below then uses this wave to produce an exact new layer.

\section{The forced transverse equation}
\label{sec:forced-transverse}
\fancyhead[R]{\small The forced transverse equation}

Section~\ref{sec:history-activation} constructs the leading increment
from an unforced wave with a chosen terminal displacement. Here we use the same history
operator to construct forced corrections, with zero displacement at
both history endpoints, and then continue them forward after activation.

Use the independent-angle coordinates and chain rule
\eqref{intro:eq:chain}. For general periodic coefficients $A,\pi$,
write the physical pair in the form of \eqref{eq:leading-velocity}:
\[
 w_{\lead}(t,X(t,\ell y))=\mu A(t,y,km_0\cdot y),\qquad
 \pi_{\lead}(t,X(t,\ell y))=\frac{\mu\ell}{k}
                         \pi(t,y,km_0\cdot y),
\]
where $\mu$ is a constant prefactor. For the leading increment,
$\mu=\ell/k$ and $A=\alpha\chi_1v f_\delta$.
The chain rule gives, after restriction to the physical phase,
\[
 D_tw_{\lead}+Mw_{\lead}+\nabla_x\pi_{\lead}
 =\mu\left(A_t+MA+m\partial_\theta\pi+k^{-1}d\pi\right),
 \qquad d=F^{-T}\nabla_y.
\]
Thus cancellation of an oscillatory residual $\mu r$ requires
$A_t+MA+m\partial_\theta\pi=f$, with $f=-r$.
This is the forced transverse equation derived in
\eqref{intro:eq:forced-transverse}; the spatial pressure term $k^{-1}d\pi$
belongs to the next order. The coefficients are functions on the
lifted $(y,\theta)$-space, with $F,M,H,m$ evaluated at physical label $\ell y$.

We impose $m\cdot A=0$ and zero angle mean. The remaining divergence
is removed by the curl correction of
Subsection~\ref{intro:subsec:divergence}, whose quantitative estimates
are proved below. The angle-independent residual is handled separately
in Section~\ref{sec:mean-inverse}.

Fix the unit vector $m_0$
and the constant matrix $R_\perp$ from
\eqref{eq:history-normal-choice} and \eqref{eq:history-transverse-basis}.
Put
\begin{equation}
 G=FR_\perp,\qquad K_R=G^TG,\qquad
 C_\perp=2K_R^{-1}G^TF_tR_\perp,\qquad D_\perp=K_R^{-1}G^T.
 \label{eq:forced-coordinate-coefficients}
\end{equation}
Here $G$ maps $\mathbb R^2$ isomorphically onto $m^\perp$;
$K_R$ is positive definite. All constants below are independent of the
stage. Fixed powers denoted by $c$ in this section summarize the inverse
estimates; they do not change the fixed exponent $\kappa=7$ of Section~\ref{sec:history-activation}.

\subsection{The forced history problem}

Let $0<t_a\le1$. We seek a velocity $A$ and a zero-angle-mean
pressure coefficient $\pi$ such that
\begin{equation}
 A_t+MA+m\partial_\theta\pi=f,\qquad m\cdot A=0.
 \label{eq:forced-transverse-system}
\end{equation}
The letter $A$ here denotes a vector correction, not the scalar
amplitude product in Section~\ref{sec:quantitative-amplification}.
On $[0,t_a]$, prescribe zero displacement at both endpoints:
\begin{equation}
 \eta=G\xi,\qquad \xi(0)=\xi(t_a)=0,\qquad
 A=\eta_t-M\eta=G\xi_t.
 \label{eq:forced-history-endpoints}
\end{equation}
Neither endpoint velocity is prescribed. In particular, zero initial
displacement does not impose zero initial correction velocity.

\refstepcounter{equation}
\label{lem:forced-history}
\textbf{Lemma \theequation\ (Forced history and its pressure).}
Assume the smooth parent identities \eqref{intro:eq:parent}
and $H\le t_a^{-2}I$ on $[0,t_a]\times\Omega$, where
$\Omega=\{|y|\le1/2\}$. Suppose
\[
 \sup(\|F\|+\|F^{-1}\|+\|F_t\|+|M|+|H|)+t_a^{-1}
 \le\mathcal K,\qquad \mathcal K\ge2.
\]
For every smooth force $f$ with zero angle mean, there is a unique
history \eqref{eq:forced-history-endpoints} solving
\eqref{eq:forced-transverse-system}. At each fixed label, with norms in
$H^s(\mathbb T)$ for any fixed integer $s\ge0$,
\begin{equation}
 \|A\|_{L_t^\infty H_\theta^s}
 +\|A_t\|_{L_t^\infty H_\theta^s}
 +\|\pi\|_{L_t^\infty H_\theta^s}
 \le C\mathcal K^{12}\|f\|_{L_t^\infty H_\theta^s}.
 \label{eq:forced-history-bound}
\end{equation}
The norms cover $[0,t_a]$, uniformly in $y$. Both $A$ and $\pi$
retain zero angle mean. If a fixed closed spatial set contains the
support of $f$ at every time, it also contains their supports.

\emph{Proof.} On $H_0^1((0,t_a);\mathbb R^2)$ use the form
\eqref{eq:history-bilinear-form}. Solve
\begin{equation}
 \mathfrak a_y(\xi,b)
 =-\int_0^{t_a} f\cdot Gb\,dt
 \quad\hbox{for every }b\in H_0^1((0,t_a);\mathbb R^2).
 \label{eq:forced-weak-equation}
\end{equation}
By \eqref{eq:history-form-bounds}, with $K_h^\kappa$ replaced by the
present coefficient bound $\mathcal K$, the form is bounded and satisfies
\[
 \mathfrak a_y(\xi,\xi)\ge c\mathcal K^{-4}\|\xi_t\|_{L_t^2}^2.
\]
By \eqref{eq:history-endpoint-poincare} and $t_a\le1$, the right side
of \eqref{eq:forced-weak-equation} is at most
$\mathcal K\|f\|_{L_t^2}\|b_t\|_{L_t^2}$. The Lax--Milgram theorem
therefore gives a unique solution $\xi$ and
\[
 \|\xi_t\|_{L_t^2}\le C\mathcal K^5\|f\|_{L_t^2}.
\]
These arguments also hold for $H^s(\mathbb T)$-valued functions:
angle derivatives commute with every coefficient.

Integration by parts, using the zero endpoint values of the test function and $F_{tt}=-HF$,
gives the forced version of \eqref{eq:history-coordinate-ode}, namely,
\begin{equation}
G^T(\eta_{tt}+H\eta)=K_R\xi_{tt}+2G^TF_tR_\perp\xi_t=G^Tf.
 \label{eq:forced-projected-displacement}
\end{equation}
Inverting $K_R$ gives
\begin{equation}
 \xi_{tt}+C_\perp\xi_t=D_\perp f,
 \qquad A=G\xi_t.
 \label{eq:forced-history-strong}
\end{equation}
The coefficient bounds give $\|K_R^{-1}\|\le\mathcal K^2$,
$|C_\perp|\le2\mathcal K^4$ and $|D_\perp|\le\mathcal K^3$.
Consequently $\|\xi_{tt}\|_{L_t^2}\le
C\mathcal K^9\|f\|_{L_t^2}$.
The interpolation inequality used in Lemma~\ref{lem:history-label-bounds}
now yields
\begin{equation}
 \|\xi_t\|_{L_t^\infty}
 \le C\bigl(t_a^{-1/2}\|\xi_t\|_{L_t^2}
       +\|\xi_t\|_{L_t^2}^{1/2}\|\xi_{tt}\|_{L_t^2}^{1/2}\bigr)
 \le C\mathcal K^7\|f\|_{L_t^2}.
 \label{eq:forced-history-trace}
\end{equation}
Multiplication by $G$ bounds $A$ by $C\mathcal K^8\|f\|_{L_t^2}$.

Equation \eqref{eq:forced-projected-displacement}, together with
$A_t+MA=\eta_{tt}+H\eta$ from
\eqref{eq:history-displacement-equation}, says that
$f-A_t-MA$ is perpendicular to the columns of $G$, hence parallel
to $m$. The constraint calculation \eqref{intro:eq:forced-transverse}
therefore gives
\begin{equation}
 \begin{aligned}
 A_t&=-MA+2m\frac{m\cdot MA}{|m|^2}
                +\left(I-\frac{mm^T}{|m|^2}\right)f,\\
 \pi&=\partial_\theta^{-1}
       \left(\frac{m\cdot f-2m\cdot MA}{|m|^2}\right).
 \end{aligned}
 \label{eq:forced-pressure-formula}
\end{equation}
The inverse angle derivative has Fourier multiplier $(in)^{-1}$
for $n\ne0$ and zero multiplier at $n=0$, so its $H^s$ norm is at
most one on zero-mean functions. Also $|m|^{-1}\le\|F\|\le\mathcal K$.
These formulas bound $A_t$ and $\pi$ by fixed further powers of
$\mathcal K$; the common exponent $12$ in
\eqref{eq:forced-history-bound} safely covers all terms.

Uniqueness at each label implies $\xi=0$ wherever the force vanishes
for the entire history. Averaging \eqref{eq:forced-weak-equation} in
$\theta$ gives the homogeneous zero-endpoint problem, so
$\langle\xi\rangle_\theta=0$. This proves support and mean preservation
and makes the pressure formula well defined. Smooth dependence on
labels follows by difference quotients in the coercive equation;
the derivative estimates are made quantitative below.
\hfill$\square$

\textbf{Remark.} The activation geometry of Lemma~\ref{lem:history-endpoint-choice}
is not needed here: the displacement endpoints are zero. Coercivity
controls the forced boundary-value problem directly. It would not
justify prescribing an arbitrary terminal velocity and integrating it
backward.

\subsection{The forward propagator in transverse coordinates}

Work on $I_j=[t_a,t_a+\varepsilon\Theta/a]$, where $t_a=t_{j-1}$
and $\tau(t)=a(t-t_a)/\varepsilon$. Retain the fixed $m_0,R_\perp$
and the matrices $K_R$ and $C_\perp=C_\perp(t,y)$ from
\eqref{eq:forced-coordinate-coefficients}.
For the forward construction, write an arbitrary transverse velocity
as $A=FR_\perp a_\perp$, with $a_\perp\in\mathbb R^2$.
After activation, the forced equation in these coordinates is
$(a_\perp)_t+C_\perp a_\perp=D_\perp f$, with
$a_\perp(t_a)=\xi_t(t_a)$ supplied by the history problem.
For $t_a\le s\le t$ in $I_j$, define the homogeneous propagator
$\mathcal P(t,s;y)$ by
\begin{equation}
 \partial_t a_\perp+C_\perp(t,y)a_\perp=0,
 \qquad a_\perp(t)=\mathcal P(t,s;y)a_\perp(s).
 \label{eq:packet-propagator-definition}
\end{equation}
Equivalently, $\partial_t\mathcal P=-C_\perp\mathcal P$ and
$\mathcal P(s,s;y)=I_2$. The forced solution is given by the variation-of-constants formula
\eqref{eq:forced-duhamel} below.
Thus a bound on $\mathcal P$ controls both the history data and the
accumulated forcing. The following lemma obtains this bound from
Proposition~\ref{prop:conditional-wave-control}, avoiding an exponential
Gr\"onwall estimate in the coefficient size. Proposition~\ref{prop:forced-forward}
then applies it to the forced continuation.
For comparison with the growing wave, put $g(t)=V_\lambda(\tau(t))$,
with the same $\lambda\ge0$ as in
Proposition~\ref{prop:conditional-wave-control}.

\refstepcounter{equation}
\label{lem:packet-coordinate-propagator}
\textbf{Lemma \theequation\ (Propagation in $a_\perp$ coordinates).}
Assume all hypotheses of Proposition~\ref{prop:conditional-wave-control}
on $I_j$. In addition, suppose that, at each
label $y\in\Omega$, the invertible parent deformation matrix $F(t,y)$
is continuously differentiable in time and satisfies
\begin{equation}
 F_t=MF,\qquad m_j=F^{-T}m_0,\qquad
 \sup_{t\in I_j,y\in\Omega}(\|F\|+\|F^{-1}\|)\le C_FK_h^{c_F},
 \label{eq:packet-coordinate-inputs}
\end{equation}
where $K_h\ge2$ and $C_F,c_F$ are independent of the stage.
Then $g>0$, $g(t_a)=1$, and, uniformly in $y$,
\begin{equation}
 \|\mathcal P(t,s;y)\|
 \le CK_h^{2c_F}\varepsilon^{-1}\Theta^9
       \frac{V_0(\tau(t))}{V_0(\tau(s))}
 \le CK_h^{2c_F}\varepsilon^{-1}\Theta^9\frac{g(t)}{g(s)}.
 \label{eq:packet-coordinate-propagator}
\end{equation}
The constants depend only on the constants in the stated hypotheses,
not on $\lambda$, the stage, or the label.

\emph{Proof.} Fix $y$ and suppress it in the formulas. The change of variables $A=FR_\perp a_\perp$ identifies
$\mathbb R^2$ with $m_j^\perp$. The projected equation
\eqref{eq:forced-projected-displacement}, with $f=0$ and
$a_\perp=\xi_t$, gives \eqref{eq:packet-propagator-definition}.
Conversely, the pressure-recovery identity \eqref{eq:forced-pressure-formula}
shows that this coordinate equation is the homogeneous wave equation.
The matrix $K_R$ is positive definite and its coefficients are
continuous, so the linear initial-value problem defining $\mathcal P$
has a unique solution throughout $I_j$.

The propagator $\mathcal P$ is conjugate to $\Phi$ as follows:
write the scaled components of an arbitrary transverse $A$ as $\mathbf y=(u,v)^T$ and define
\[
 \mathcal L_t\mathbf y
 =R_\perp^TF^{-1}
 \left(\varepsilon up+vq-\varepsilon\frac{Pu+Qv}{N}n\right).
\]
Here $P,Q,N$ are the time-dependent phase components at the fixed label $y$;
$N\ge1/2$ by Proposition~\ref{prop:conditional-wave-control}.
The vector in parentheses is transverse because its dot product
with $m_j=s_0(Pp+\varepsilon Qq+Nn)$ is zero.
The inverse map is
\[
 \mathcal L_t^{-1}a_\perp
 =\begin{pmatrix}
 \varepsilon^{-1}p\cdot FR_\perp a_\perp\\
 q\cdot FR_\perp a_\perp
 \end{pmatrix}.
\]
Indeed, the first two components recover $u,v$, and transversality
uniquely fixes the third component. These formulas also show directly
that the maps are inverses. Orthonormality of the frame, $N\ge1/2$,
$|P|+|Q|\le C\Theta^2$, and $\varepsilon\le1$ give
\[
 \|\mathcal L_t\|\le CK_h^{c_F}(1+\varepsilon\Theta^2)
                         \le CK_h^{c_F}\Theta^2,
 \qquad
 \|\mathcal L_t^{-1}\|\le CK_h^{c_F}\varepsilon^{-1}.
\]
The scaled two-component system of Section~\ref{sec:comparison-proofs}
is exactly the wave equation in these components, for arbitrary
transverse data. Uniqueness of the two linear evolutions therefore gives
\begin{equation}
 \mathcal P(t,s;y)
 =\mathcal L_t\Phi(\tau(t),\tau(s);y)\mathcal L_s^{-1}.
 \label{eq:packet-propagator-conjugacy}
\end{equation}
Apply \eqref{eq:conditional-velocity-propagator} and the two endpoint
map bounds. Their powers of $\Theta$ add to $2+7=9$, proving the
first inequality of \eqref{eq:packet-coordinate-propagator}, with an
explicit exponent $2c_F$ on $K_h$.
Lemma~\ref{lem:ideal-nonnegative-slope} gives $g>0$, $g(t_a)=1$, and
\eqref{eq:ideal-slope-growth-ratio}; this proves the second inequality.
\hfill$\square$

\textbf{Amplitude normalization.} Assume additionally the hypotheses
of Lemma~\ref{lem:amplitude-ratios}, choose $\alpha_j$ by
\eqref{eq:quantitative-normalization}, and restrict to
$0\le\tau(t)\le\min(\Theta,\tau_j+1)$.
For $\tau\ge1$, $s_0g(t)\le CA(\tau,y)\le CA_j$ by
\eqref{eq:actual-amplitude-size} and
\eqref{eq:forward-amplitude-ratio}.
For $0\le\tau\le1$, monotonicity of $V_\lambda$ from
Lemma~\ref{lem:exponential-gain} gives $g(t)\le V_\lambda(1)$,
so the same bound follows from the case $\tau=1$. In both cases,
\begin{equation}\label{eq:packet-normalized-amplitude}
 \sup\alpha_jg(t)\le C\delta_jh_j/s_0.
\end{equation}
Together with \eqref{eq:packet-coordinate-propagator}, this supplies
the forward propagation and amplitude inputs assembled in
Proposition~\ref{prop:wave-packet-inputs}.
The history is supplied by Section~\ref{sec:history-activation}; the iteration
must still verify the coefficient regularity and simultaneous parameter bounds.

\subsection{Continue the forced correction beyond activation}

\refstepcounter{equation}
\label{prop:forced-forward}
\textbf{Proposition \theequation\ (The inverse on the full interval).}
Let $0\le t_a\le S\le1$, where $S$ is the physical endpoint time.
Assume the parent identities \eqref{intro:eq:parent} and
the coefficient supremum bounds of Lemma~\ref{lem:forced-history}
throughout $[0,S]\times\Omega$, with a common bound $\mathcal K$.
When $t_a>0$, retain its history coercivity assumption on $[0,t_a]$
and its bound on $t_a^{-1}$. Assume also a positive function
$g$, with $g(t_a)=1$, such that
\begin{equation}
 \|\mathcal P(t,s;y)\|\le L\frac{g(t)}{g(s)},
 \qquad t_a\le s\le t\le S\le1.
 \label{eq:forced-propagator-input}
\end{equation}
Lemma~\ref{lem:packet-coordinate-propagator} supplies this input for the
growing wave when $[t_a,S]\subseteq I_j$, that is,
$S\le t_a+\varepsilon\Theta/a$.
Enlarge $\mathcal K$ if necessary so that $L\le\mathcal K$.
Set $\rho(t)=1$ on the history $[0,t_a]$ and $\rho(t)=g(t)$ on the forward interval
$[t_a,S]$. Then the forced history has a unique continuation satisfying
\eqref{eq:forced-transverse-system}, with
\begin{equation}
 \sup_{t,y}\frac{\|A(t,y)\|_{H_\theta^s}
          +\|A_t(t,y)\|_{H_\theta^s}+\|\pi(t,y)\|_{H_\theta^s}}
                  {\rho(t)}
 \le C\mathcal K^{12}
       \sup_{t,y}\frac{\|f(t,y)\|_{H_\theta^s}}{\rho(t)}.
 \label{eq:forced-full-bound}
\end{equation}
It preserves the spatial support and zero angle mean of the force.
When $t_a=0$, omit the history assumptions and prescribe $A(0)=0$;
the stated conclusions then hold using only the coefficient bounds
of Lemma~\ref{lem:forced-history}, assumed throughout $[0,S]\times\Omega$,
and the forward propagator bound
\eqref{eq:forced-propagator-input}.

\emph{Proof.} Write $A=Ga_\perp$ and solve
\begin{equation}
 (a_\perp)_t+C_\perp a_\perp=D_\perp f,
 \qquad a_\perp(t_a)=\xi_t(t_a)\quad(t_a>0).
 \label{eq:forced-forward-equation}
\end{equation}
If $t_a=0$, take $a_\perp(0)=0$. Variation of constants gives
\begin{equation}
 a_\perp(t)=\mathcal P(t,t_a)a_\perp(t_a)
       +\int_{t_a}^t\mathcal P(t,s)D_\perp(s)f(s)\,ds.
 \label{eq:forced-duhamel}
\end{equation}
Dividing by $g(t)$ and using \eqref{eq:forced-propagator-input} and $g(t_a)=1$ gives
\[
 \sup_{t\ge t_a}\frac{\|a_\perp(t)\|_{H_\theta^s}}{g(t)}
 \le L\|a_\perp(t_a)\|_{H_\theta^s}
 +L\sup|D_\perp|\int_{t_a}^S
                 \frac{\|f(s)\|_{H_\theta^s}}{g(s)}\,ds.
\]
Put
\[
 \mathcal F=\sup_{0\le t\le S,\,y\in\Omega}
                 \frac{\|f(t,y)\|_{H_\theta^s}}{\rho(t)}.
\]
For $t_a>0$, the history trace estimate
\eqref{eq:forced-history-trace}, applied to
$a_\perp(t_a)=\xi_t(t_a)$, gives
\[
 \|a_\perp(t_a)\|_{H_\theta^s}
 \le C\mathcal K^7\|f\|_{L^2((0,t_a);H_\theta^s)}
 \le C\mathcal K^7t_a^{1/2}\mathcal F
 \le C\mathcal K^7\mathcal F.
\]
Here $\rho=1$ on the history interval and $t_a\le1$.
Since $L\le\mathcal K$, the first term in the preceding Duhamel
estimate is at most $C\mathcal K^8\mathcal F$.
For $t_a=0$, this term is zero.

For the second term, the definitions
\eqref{eq:forced-coordinate-coefficients} and the coefficient estimates
following \eqref{eq:forced-history-strong} give
$\sup|D_\perp|\le\mathcal K^3$, now on the full interval by assumption.
Since $\rho=g$ on $[t_a,S]$,
\[
 \int_{t_a}^S\frac{\|f(s)\|_{H_\theta^s}}{g(s)}\,ds
 \le(S-t_a)\mathcal F\le\mathcal F.
\]
Thus the second term is at most $\mathcal K^4\mathcal F$, and
\[
 \sup_{t_a\le t\le S}
       \frac{\|a_\perp(t)\|_{H_\theta^s}}{g(t)}
 \le C(\mathcal K^8+\mathcal K^4)\mathcal F
 \le C\mathcal K^8\mathcal F.
\]
All estimates are uniform in $y$. Multiplication by $G=FR_\perp$,
with $\|G\|\le\mathcal K$, yields
\[
 \sup_{t_a\le t\le S,\,y\in\Omega}
       \frac{\|A(t,y)\|_{H_\theta^s}}{g(t)}
 \le C\mathcal K^9\mathcal F.
\]
Finally, \eqref{eq:forced-pressure-formula}, together with
$|M|\le\mathcal K$, $|m|^{-1}\le\mathcal K$ and the
$H_\theta^s$ bound for $\partial_\theta^{-1}$ used in
Lemma~\ref{lem:forced-history}, gives
\[
 \begin{aligned}
 \|A_t\|_{H_\theta^s}
 &\le3\mathcal K\|A\|_{H_\theta^s}+\|f\|_{H_\theta^s},\\
 \|\pi\|_{H_\theta^s}
 &\le\mathcal K\|f\|_{H_\theta^s}
       +2\mathcal K^2\|A\|_{H_\theta^s}.
 \end{aligned}
\]
After division by $g$, these are bounded by
$C\mathcal K^{10}\mathcal F$ and $C\mathcal K^{11}\mathcal F$,
respectively. Combining them with the history estimate
\eqref{eq:forced-history-bound} proves \eqref{eq:forced-full-bound}
with the common exponent $12$.

The coordinate equation agrees with \eqref{eq:forced-history-strong}
at activation, so the velocities join smoothly for smooth data.
Uniqueness of the labelwise initial-value problem and the pressure
formula preserve support and zero angle mean. No estimate differentiates
$g$, and no exponential in $\sup|C_\perp|$ is introduced.
\hfill$\square$

\subsection{Gevrey-2 bounds with a fixed derivative scale}

The finite expansion will apply this inverse repeatedly. Estimates
with arbitrary constants depending on the derivative order would not
suffice. We therefore record a Gevrey-2 bound with a single derivative
scale that can be kept fixed at every application.
These are Gevrey class $2$ estimates in the spatial labels $y$ and,
for oscillatory fields, the angle $\theta$; they make no assertion
of Gevrey regularity in time.

Use ordered derivative words in $(y,\theta)$. For a fixed integer $s\ge0$, write
\[
 |u(t)|_n=\sum_{|I|=n}\|\partial^Iu(t)\|_{H^s(\mathbb R^3\times\mathbb T)},
 \qquad
 |B(t)|_{n,\infty}=\sum_{|I|=n}
       \|\partial^IB(t)\|_{W^{s,\infty}(\Omega\times\mathbb T)}.
\]
Coefficients have zero angle derivatives. On the full interval $[0,S]$
define the time norms, with each supremum inside the derivative sum, by
\begin{equation}\label{eq:transverse-wordwise-time-norms}
 \begin{aligned}
 \mathcal N_{n,\rho}(u)
 &=\sum_{|I|=n}\sup_{0\le t\le S}
       \frac{\|\partial^Iu(t)\|_{H^s}}{\rho(t)},\\
 \mathcal C_n(B)
 &=\sum_{|I|=n}\sup_{0\le t\le S}
       \|\partial^IB(t)\|_{W^{s,\infty}}.
 \end{aligned}
\end{equation}
For a tuple, sum the norms of its entries. A restriction to the history
or forward interval restricts each supremum to that interval. These
norms are distinct from the supremum of the instantaneous derivative
sum. Assume coefficients are smooth on a
neighborhood of $\Omega$; forces are smooth and supported in a fixed
closed subset of its interior, and are extended by zero outside.

\refstepcounter{equation}
\label{prop:forced-factorial}
\textbf{Proposition \theequation\ (Gevrey-2 estimates for the transverse inverse).}
Assume Proposition~\ref{prop:forced-forward}, including its coercivity
and propagator inputs. Suppose, for some $R_c\ge1$, that the parent
coefficients satisfy, for each $B\in\{F,F^{-1},F_t,F_{tt},M,H\}$,
\begin{equation}
 \mathcal C_n(B)
 \le\mathcal K R_c^n(n!)^2\qquad(n\ge0).
 \label{eq:forced-factorial-input}
\end{equation}
Put $\gamma_s=4s^2+24s+22$. There is a constant $C=C(s)$, depending
only on the fixed Sobolev order $s$ and independent of $n$, the force,
and the number of inverse applications, such that, whenever
$R\ge C\mathcal K^{\gamma_s}R_c$, the bounds
\begin{equation}
 \mathcal N_{n,\rho}(f)\le hR^n(n!)^2
 \quad(n\ge0)
 \label{eq:forced-force-factorial}
\end{equation}
imply
\begin{equation}
 \mathcal N_{n,\rho}(A,A_t,\pi)
 \le C\mathcal K^{\gamma_s} hR^n(n!)^2.
 \label{eq:forced-output-factorial}
\end{equation}
The same derivative scale $R$ works for every force satisfying
\eqref{eq:forced-force-factorial}.

The proof follows the history--trace--forward argument of
Proposition~\ref{prop:forced-forward}, applied to differentiated equations;
the additional step is a Gevrey-2 induction controlling the coefficient
commutators with one fixed derivative scale. An explicit admissible
choice of the auxiliary constants and scale is
\[
 R_0=C_s'\mathcal K^{4s+4}R_c,\qquad
 B_0=C_s\mathcal K^{4s^2+20s+16},\qquad
 R\ge4(B_0+1)R_0,
\]
where $C_s,C_s'$ are sufficiently large constants depending only on
the fixed Sobolev order $s$. The resulting output prefactor in
\eqref{eq:forced-output-factorial} is
$C_s\mathcal K^{\gamma_s}$, with the same exponent $\gamma_s$ as
in the condition on $R$.
The details are given in the
\hyperref[app:forced-factorial-proof]{appendix to this section}.

\subsection{The divergence correction needed in the expansion}

Apply the curl construction of Subsection~\ref{intro:subsec:divergence}
to the transverse amplitude supplied by
Proposition~\ref{prop:forced-factorial}. For
$\widehat A(t,X(t,\ell y))=A(t,y,km_0\cdot y)$, the chain rule
\eqref{intro:eq:chain} gives
\begin{equation}
 (\nabla_x\cdot\widehat A)(t,X(t,\ell y))
 =\ell^{-1}\left[d\cdot A+k\partial_\theta(m\cdot A)\right]
                      _{\theta=km_0\cdot y}.
 \label{eq:forced-physical-divergence}
\end{equation}
Transversality removes the angle-derivative term; the spatial residual
$d\cdot A$ remains. Recall the potential and correction
\eqref{intro:eq:curl-correction}, now for a general transverse $A$:
Define
\begin{equation}
 Q_A=-\partial_\theta^{-1}\frac{m\times A}{|m|^2},
 \qquad C_A=d\times Q_A.
 \label{eq:forced-curl-correction}
\end{equation}
The identities proved in Subsection~\ref{intro:subsec:divergence} give
\begin{equation}
 k^{-1}\tilde d\times Q_A
 =k^{-1}d\times Q_A+m\times\partial_\theta Q_A=A+k^{-1}C_A,
 \qquad
 \tilde d\cdot(A+k^{-1}C_A)=0.
 \label{eq:forced-exact-divergence}
\end{equation}
Here $\tilde d=d+km\partial_\theta$ as in
\eqref{intro:eq:lifted-derivative}. The second identity follows from
$[\tilde d_i,\tilde d_j]=0$, proved in
Subsection~\ref{intro:subsec:divergence}; it holds before and after
restriction to the physical phase. The new issue is the derivative
cost of this correction.

We also bound $C_A$, $(C_A)_t$, and $d\pi$, which enter the momentum
residual at the next order in $k^{-1}$. The bounds display the cost of
one spatial derivative while preserving the Gevrey-2 derivative scale $R$.
These bounds supply the correction estimates used in
Section~\ref{sec:packet-expansion}.

\refstepcounter{equation}
\label{prop:forced-divergence-estimates}
\textbf{Proposition \theequation\ (Bounds for the divergence correction).}
Assume the hypotheses and parameter choices of
Proposition~\ref{prop:forced-factorial}, and let $(A,\pi)$ be its
forced solution. Define $Q_A,C_A$ by
\eqref{eq:forced-curl-correction}, and use $\gamma_s$ from that proposition.
For every integer $n\ge0$, the wordwise time norms satisfy
\begin{equation}
 \begin{aligned}
 \mathcal N_{n,\rho}(Q_A)
 &\le C_s\mathcal K^{\gamma_s+4s+3}hR^n(n!)^2,\\
 \mathcal N_{n,\rho}((Q_A)_t)
 &\le C_s\mathcal K^{\gamma_s+8s+8}hR^n(n!)^2.
 \end{aligned}
 \label{eq:forced-potential-estimates}
\end{equation}
Moreover,
\begin{equation}
 \mathcal N_{n,\rho}(C_A,(C_A)_t,d\pi)
 \le C_s\mathcal K^{4s^2+32s+31}hR^{n+1}((n+1)!)^2.
 \label{eq:forced-one-derivative-cost}
\end{equation}
The constants $C_s$ depend only on $s$, and are independent of the
stage, $n$, and the force. These fields have zero angle mean and
remain supported in the fixed spatial support specified for the
force. All estimates use the same scale $R$.

\emph{Proof.} Set $b_m=m/|m|^2$, so
$Q_A=-\partial_\theta^{-1}(b_m\times A)$.
The identity $m_t=-M^Tm$ gives
\[
 (b_m)_t=-\frac{M^Tm}{|m|^2}
          +\frac{2m(m\cdot M^Tm)}{|m|^4}.
\]
The input bounds for $m=F^{-T}m_0$ and $M$ have prefactor
$C_s\mathcal K$. The reciprocal bound
\eqref{eq:forced-inverse-coefficient-bounds} in the appendix has
prefactor $C_s\mathcal K^{4s+2}$ and scale $R_0$.
The product estimate \eqref{eq:forced-gevrey-product-sum},
whose normalized Leibniz sum is at most $3$,
therefore gives
\[
 \begin{aligned}
 \mathcal C_n(b_m)
 &\le C_s\mathcal K^{4s+3}R_0^n(n!)^2,\\
 \mathcal C_n((b_m)_t)
 &\le C_s\left(\mathcal K^{4s+4}+\mathcal K^{8s+8}\right)
                    R_0^n(n!)^2
 \le C_s\mathcal K^{8s+8}R_0^n(n!)^2.
 \end{aligned}
\]
In the second term of $(b_m)_t$, the four factors $m,m,M,m$
cost $\mathcal K^4$ and the two reciprocal factors cost
$\mathcal K^{8s+4}$. This explains the exponent $8s+8$.

For clarity, if a coefficient $B$ and an amplitude $u$ satisfy
\[
 \mathcal C_n(B)\le bR_0^n(n!)^2,\qquad
 \mathcal N_{n,\rho}(u)\le u_*R^n(n!)^2,
\]
the product rule and $R_0/R\le1/4$ give
\[
 \begin{aligned}
 \mathcal N_{n,\rho}(Bu)
 &\le C_sbu_*\sum_{l=0}^n\binom nl
           R_0^l(l!)^2R^{n-l}((n-l)!)^2\\
 &\le \frac{C_sbu_*}{1-R_0/R}R^n(n!)^2
 \le C_sbu_*R^n(n!)^2.
 \end{aligned}
\]
The same estimate holds for cross products. Since
$\partial_\theta^{-1}$ commutes with time and spatial derivatives
and is bounded on zero-mean $H^s$, the identities
\[
 Q_A=-\partial_\theta^{-1}(b_m\times A),\qquad
 (Q_A)_t=-\partial_\theta^{-1}
       \bigl((b_m)_t\times A+b_m\times A_t\bigr)
\]
and \eqref{eq:forced-output-factorial} prove \eqref{eq:forced-potential-estimates}.
Here the angle means vanish because $b_m$ and $(b_m)_t$ are
independent of $\theta$ and $A,A_t$ have zero angle mean.
No time derivative of $\rho$ is taken.

Next consider a product $B\partial_{y_r}u$. The same assumptions
give a convolution with $R^{n+1-l}((n+1-l)!)^2$ in place of
$R^{n-l}((n-l)!)^2$. The normalized factorial factor is
\[
 \binom nl\frac{(l!)^2((n+1-l)!)^2}{((n+1)!)^2}
 =\binom nl^{-1}\left(\frac{n+1-l}{n+1}\right)^2\le1.
\]
Thus, also after summing over the three spatial directions,
\[
 \mathcal N_{n,\rho}(B\nabla_yu)
 \le C_sbu_*R^{n+1}((n+1)!)^2.
\]
This is the precise one-derivative cost, with the same $R$.

Finally, $F_t=MF$ implies
\[
 (F^{-T})_t=-M^TF^{-T},\qquad
 d_t=-M^TF^{-T}\nabla_y,\qquad
 (C_A)_t=d_t\times Q_A+d\times(Q_A)_t.
\]
Here $d_t$ differentiates only the coefficients of $d$.
The coefficients of $d$ and $d_t$ have Gevrey-2 bounds with prefactors
$C_s\mathcal K$ and $C_s\mathcal K^2$, respectively, and scale
at most $R_0$, by \eqref{eq:forced-factorial-input} and the product
estimate. Applying the preceding one-derivative estimate to
$C_A=d\times Q_A$, to both terms of $(C_A)_t$, and to $d\pi$ gives
\[
 \begin{aligned}
 \mathcal N_{n,\rho}(C_A)
 &\le C_s\mathcal K^{\gamma_s+4s+4}hR^{n+1}((n+1)!)^2,\\
 \mathcal N_{n,\rho}((C_A)_t)
 &\le C_s\left(\mathcal K^{\gamma_s+4s+5}
               +\mathcal K^{\gamma_s+8s+9}\right)
                 hR^{n+1}((n+1)!)^2,\\
 \mathcal N_{n,\rho}(d\pi)
 &\le C_s\mathcal K^{\gamma_s+1}hR^{n+1}((n+1)!)^2.
 \end{aligned}
\]
Combining these bounds proves \eqref{eq:forced-one-derivative-cost},
since $\gamma_s+8s+9=4s^2+32s+31$. Multiplication by smooth
coefficients and spatial or time differentiation preserve the fixed
spatial support; $\partial_\theta^{-1}$ acts only in the angle variable.
The same operations preserve zero angle mean because all coefficients
are independent of $\theta$. This proves the remaining assertions.
\hfill$\square$

\begingroup
\renewcommand{\theequation}{\thesection.A.\arabic{equation}}
\renewcommand{\theHequation}{forced-appendix.\arabic{equation}}
\setcounter{equation}{0}
\section*{Appendix to Section \ref{sec:forced-transverse}: Gevrey-2 estimates}
\phantomsection
\label{app:forced-factorial-proof}
\fancyhead[R]{\small Gevrey-2 estimates: proof}

\emph{Proof of Proposition~\ref{prop:forced-factorial}.}
We extend the history energy estimate and forward propagator estimate
of Proposition~\ref{prop:forced-forward} to label and angle derivatives.

The coefficient bounds come first, followed by the fixed Sobolev-order
estimate and the Gevrey-2 induction. We use $G$, $K_R$, $C_\perp$ and
$D_\perp$ from \eqref{eq:forced-coordinate-coefficients} throughout.
On the history interval let $\mathcal T=\partial_t$.
Expanding the bilinear form \eqref{eq:history-bilinear-form}, used on
the left side of \eqref{eq:forced-weak-equation}, gives
\[
 \int (\mathcal T\xi)^T B_{11}\mathcal T b
      +(\mathcal T\xi)^TB_{10}b
      +\xi^TB_{01}\mathcal T b+\xi^TB_{00}b\,dt,
\]
where $B_{11}=G^TG=K_R$, $B_{10}=G^TG_t$,
$B_{01}=G_t^TG$, and $B_{00}=G_t^TG_t-G^THG$.
The product rule and \eqref{eq:forced-factorial-input} give, for $n\ge0$,
\begin{equation}
 \begin{aligned}
 \mathcal C_n(B_{11},B_{10},B_{01})
 &\le C_s\mathcal K^2 R_c^n(n!)^2,\\
 \mathcal C_n(B_{00})
 &\le C_s\mathcal K^3 R_c^n(n!)^2.
 \end{aligned}
 \label{eq:forced-form-coefficient-bounds}
\end{equation}
Indeed, $G=FR_\perp$ and $G_t=F_tR_\perp$ inherit the input bounds.
For a product of two factors and $n\ge2$, the normalized Leibniz sum is
\begin{equation}
 \sum_{l=0}^n {n\choose l}
       \frac{(l!)^2((n-l)!)^2}{(n!)^2}
 =\sum_{l=0}^n{n\choose l}^{-1}\le2+\frac{n-1}{n}<3.
 \label{eq:forced-gevrey-product-sum}
\end{equation}
Here the interior binomial coefficients are at least $n$.
For $n=0,1$, the sum is $1,2$, respectively.
The $W^{s,\infty}$ product constant depends only on the fixed $s$.
Applying this estimate twice controls the three factors in $G^THG$;
the quadratic term $G_t^TG_t$ is absorbed since $\mathcal K\ge2$.

For the inverse coefficients, take $J=K_R$ or the scalar $J=|m|^2$.
Both satisfy
\[
 \mathcal C_n(J)\le C_s\mathcal K^2R_c^n(n!)^2,
 \qquad \sup_{t,y}|J^{-1}(t,y)|\le\mathcal K^2.
\]
For $K_R$, the first bound is in \eqref{eq:forced-form-coefficient-bounds},
and the second bound follows from
$|Gz|\ge\|F^{-1}\|^{-1}|z|$, so $\mathcal K^2|z||K_Rz|\ge\mathcal K^2|Gz|^2\ge|z|^2$;
for $|m|^2$, the first bound follows from \eqref{intro:eq:normal}, and the second bound
follows from $1=|m_0|=|F^Tm|\le\|F\||m|$.
In each case, differentiating $JJ^{-1}=I$ gives,
for a nonempty derivative word $I$ of length $n$,
\begin{equation}
 \partial^I J^{-1}
 =-J^{-1}\sum_{\varnothing\ne E\subset\{1,\ldots,n\}}
       (\partial^{I_E}J)(\partial^{I_{E^c}}J^{-1}).
 \label{eq:forced-inverse-identity}
\end{equation}
First differentiate only through the fixed order $s$. Each resulting
term has at most $s$ derivatives of $J$ and $s+1$ inverse factors.
Their supremum bounds therefore give
$\sup_t\|J^{-1}(t)\|_{W^{s,\infty}}
\le C_s\mathcal K^{4s+2}$.

For the additional derivative words set
$U_n=\mathcal C_n(J^{-1})$, with the norms \eqref{eq:transverse-wordwise-time-norms}.
Applying the $W^{s,\infty}$ product estimate to
\eqref{eq:forced-inverse-identity} gives
\[
 U_0\le C_s\mathcal K^{4s+2},\qquad
 U_n\le C_s\mathcal K^{4s+4}
       \sum_{l=1}^n{n\choose l}R_c^l(l!)^2U_{n-l}.
\]
Choose $R_0=C_s'\mathcal K^{4s+4}R_c$ with $C_s'$ sufficiently large.
After division by $R_0^n(n!)^2$, the induction sum is bounded by
\[
 C_s\mathcal K^{4s+4}\sum_{l=1}^n
          \frac{(R_c/R_0)^l}{\binom nl}
 \le C_s\mathcal K^{4s+4}\frac{R_c/R_0}{1-R_c/R_0}
 \le\tfrac12.
\]
Starting from $U_0$, induction proves the explicit common bound
\begin{equation}
 \mathcal C_n(K_R^{-1},|m|^{-2})
 \le C_s\mathcal K^{4s+2}R_0^n(n!)^2,
 \qquad R_0=C_s'\mathcal K^{4s+4}R_c.
 \label{eq:forced-inverse-coefficient-bounds}
\end{equation}
The constants and exponents are independent of $n$.

To apply this bound to the forced system, write
\[
 \Pi_m=I-\frac{mm^T}{|m|^2},\qquad
 T_m=-M+\frac{2mm^TM}{|m|^2},\qquad
 p_m=\frac{m^T}{|m|^2},\qquad q_m=\frac{2m^TM}{|m|^2}.
\]
Then \eqref{eq:forced-pressure-formula} reads
$A_t=T_mA+\Pi_m f$ and
$\pi=\partial_\theta^{-1}(p_m f-q_m A)$.
Since $G,G_t,m,M$ have prefactor at most $C_s\mathcal K$ in the
input Gevrey-2 bounds, the product estimate
\eqref{eq:forced-gevrey-product-sum} gives
\begin{equation}
 \begin{aligned}
 \mathcal C_n(D_\perp)
       +\mathcal C_n(p_m)
 &\le C_s\mathcal K^{4s+3}R_0^n(n!)^2,\\
 \mathcal C_n(C_\perp)
       +\mathcal C_n(\Pi_m)
       +\mathcal C_n(q_m)
 &\le C_s\mathcal K^{4s+4}R_0^n(n!)^2,\\
 \mathcal C_n(T_m)
 &\le C_s\mathcal K^{4s+5}R_0^n(n!)^2.
 \end{aligned}
 \label{eq:forced-derived-coefficient-bounds}
\end{equation}
Thus all weak-form and forced-system coefficients are bounded by
$C_s\mathcal K^{4s+5}R_0^n(n!)^2$. The recurrence prefactor $B_0$
will also include the costs of the history and forward inverses below.

We first prove the history estimate at the fixed Sobolev order $s$,
when $t_a>0$. In this paragraph $H^j$ means
$H^j(\mathbb R^3_y\times\mathbb T)$ and all time norms cover $[0,t_a]$.
Write $\mathcal T^0=I$, $\mathcal T^1=\partial_t$, so the expanded
form is $\sum_{i,k=0}^1\int(\mathcal T^i\xi)^T
B_{ik}\mathcal T^k b\,dt$. For a multi-index $\alpha$ in $(y,\theta)$,
differentiating \eqref{eq:forced-weak-equation} gives
\begin{equation}
 \begin{aligned}
 \mathfrak a_y(\partial^\alpha\xi,b)
 ={}&-\int_0^{t_a}\partial^\alpha(G^Tf)\cdot b\,dt\\
 &-\sum_{i,k=0}^1\sum_{0<\gamma\le\alpha}
       {\alpha\choose\gamma}\int_0^{t_a}
       (\mathcal T^i\partial^{\alpha-\gamma}\xi)^T
       (\partial^\gamma B_{ik})\mathcal T^k b\,dt.
 \end{aligned}
 \label{eq:forced-history-differentiated-form}
\end{equation}
Coefficient angle derivatives vanish. Every remaining commutator has
fewer than $|\alpha|$ derivatives on $\xi$.

Set $E_j=\|\xi_t\|_{L_t^2H^j}$ and
$\mathcal F_j=\|f\|_{L_t^2H^j}$. All label and angle derivatives
of $\xi$ still vanish at the time endpoints. Thus \eqref{eq:history-endpoint-poincare},
with $t_a\le1$, controls an undifferentiated-in-time path by its time
derivative. The $n=0$ case of \eqref{eq:forced-form-coefficient-bounds}
bounds every derivative of $B_{ik}$ through order $s$ by
$C_s\mathcal K^3$, while $\|G\|_{W^{s,\infty}}\le C_s\mathcal K$.
In \eqref{eq:forced-history-differentiated-form}, test with
$b=\partial^\alpha\xi$, integrate also over $(y,\theta)$, and sum
over $|\alpha|\le j$. Coercivity from
\eqref{eq:history-form-bounds} and Cauchy--Schwarz give, for $1\le j\le s$,
\[
 c\mathcal K^{-4}E_j^2
 \le C_s\bigl(\mathcal K\mathcal F_j
                    +\mathcal K^3E_{j-1}\bigr)E_j.
\]
The same argument at $j=0$ has no commutator. Dividing by $E_j$ when
it is nonzero gives
\begin{equation}
 E_0\le C\mathcal K^5\mathcal F_0,\qquad
 E_j\le C_s\bigl(\mathcal K^5\mathcal F_j
                     +\mathcal K^7E_{j-1}\bigr),\qquad
 E_s\le C_s\mathcal K^{7s+5}\mathcal F_s.
 \label{eq:forced-fixed-order-energy}
\end{equation}
The last bound follows by induction, since $\mathcal F_j\le\mathcal F_s$
and $\mathcal K\ge2$. If $E_j=0$, its bound is immediate.
To justify the differentiation, first use spatial difference quotients
in the weak equation. The same estimates are uniform in the difference
step; weak compactness and uniqueness identify the derivative limit.
Angle derivatives commute exactly. Repeating this through order $s$
establishes the regularity used in the test above.

The strong equation \eqref{eq:forced-history-strong} and the $n=0$
coefficient estimates \eqref{eq:forced-derived-coefficient-bounds} now give
\[
 \|\xi_{tt}\|_{L_t^2H^s}
 \le C_s\bigl(\mathcal K^{4s+4}E_s
                    +\mathcal K^{4s+3}\mathcal F_s\bigr)
 \le C_s\mathcal K^{11s+9}\mathcal F_s.
\]
Apply the same Hilbert-space interpolation inequality as in
\eqref{eq:forced-history-trace}, using $t_a^{-1}\le\mathcal K$:
\begin{equation}
 \begin{aligned}
 \|\xi_t\|_{L_t^\infty H^s}
 &\le C\left(t_a^{-1/2}E_s
        +E_s^{1/2}\|\xi_{tt}\|_{L_t^2H^s}^{1/2}\right)\\
 &\le C_s\left(\mathcal K^{7s+11/2}
                   +\mathcal K^{9s+7}\right)\mathcal F_s
 \le C_s\mathcal K^{9s+7}\mathcal F_s.
 \end{aligned}
 \label{eq:forced-fixed-order-trace}
\end{equation}
This controls the activation trace $\xi_t(t_a)$ as well as the entire
history. Each exponent depends only on the fixed order $s$.

On the forward interval, let $\alpha$ be a multi-index in $(y,\theta)$.
Differentiating \eqref{eq:forced-forward-equation} gives
\begin{equation}
 (\partial^\alpha a_\perp)_t+C_\perp\partial^\alpha a_\perp
 =\partial^\alpha(D_\perp f)
 -\sum_{0<\gamma\le\alpha}{\alpha\choose\gamma}
      (\partial^\gamma C_\perp)
      \partial^{\alpha-\gamma}a_\perp.
 \label{eq:forced-differentiated-forward}
\end{equation}
Only spatial derivatives of $C_\perp$ occur, since it is independent
of $\theta$. For $0\le j\le s$, set
\[
 Q_j=\sup_{t_a\le t\le S}\frac{\|a_\perp(t)\|_{H^j}}{g(t)},
 \qquad \mathcal I_j=\|a_\perp(t_a)\|_{H^j},\qquad
 \mathcal N_s=\sup_{0\le t\le S}
                          \frac{\|f(t)\|_{H^s}}{\rho(t)}.
\]
Here $H^j=H^j(\mathbb R^3_y\times\mathbb T)$. For $t_a>0$,
\eqref{eq:forced-fixed-order-trace} and $t_a\le1$ give
$\mathcal I_j\le \mathcal I_s\le C_s\mathcal K^{9s+7}\mathcal N_s$;
for $t_a=0$, every $\mathcal I_j$ is zero.

Write $\mathcal R_\alpha$ for the right side of
\eqref{eq:forced-differentiated-forward}. Apply the Duhamel formula
\eqref{eq:forced-duhamel} to $\partial^\alpha a_\perp$, with forcing
$\mathcal R_\alpha$ and initial value $\partial^\alpha a_\perp(t_a)$.
This uses the original propagator without differentiating it.
The $n=0$ case of
\eqref{eq:forced-derived-coefficient-bounds} yields, for $1\le j\le s$,
\[
 \left(\sum_{|\alpha|\le j}
           \|\mathcal R_\alpha(t)\|_{L^2}^2\right)^{1/2}
 \le C_s\left(\mathcal K^{4s+3}\|f(t)\|_{H^j}
        +\mathcal K^{4s+4}\|a_\perp(t)\|_{H^{j-1}}\right).
\]
The second term is absent at $j=0$. Apply
\eqref{eq:forced-propagator-input}, divide by $g(t)$, and use Minkowski's
inequality in $L^2$ and in the finite derivative sum. Since the
propagator coefficients are independent of $\theta$, the same bound
acts on $L^2(\Omega\times\mathbb T)$. We obtain
\begin{equation}
 \begin{aligned}
 Q_0&\le C\mathcal K \mathcal I_0+C_s\mathcal K^{4s+4}\mathcal N_s,\\
 Q_j&\le C_s\left(\mathcal K \mathcal I_j
          +\mathcal K^{4s+4}\mathcal N_s
          +\mathcal K^{4s+5}Q_{j-1}\right)
       \qquad(1\le j\le s).
 \end{aligned}
 \label{eq:forced-forward-fixed-recurrence}
\end{equation}
Indeed, each forcing integral costs at most $S-t_a\le1$, and the
factor $L$ is at most $\mathcal K$. Inserting the history bound for
$\mathcal I_j$ gives $Q_0\le C_s\mathcal K^{9s+8}\mathcal N_s$ and, by induction,
\begin{equation}
 Q_j\le C_s\mathcal K^{(4s+5)j+9s+8}\mathcal N_s
 \quad(0\le j\le s),\qquad
 Q_s\le C_s\mathcal K^{4s^2+14s+8}\mathcal N_s.
 \label{eq:forced-forward-fixed-bound}
\end{equation}
There are only $s$ induction steps, so their constants are absorbed
into $C_s$. These bounds also cover $t_a=0$, when the history term is absent.

Multiplication by $G$, whose $W^{s,\infty}$ norm is at most
$C_s\mathcal K$, bounds $A$ by one further power of $\mathcal K$.
Using $A_t=T_mA+\Pi_m f$ and
$\pi=\partial_\theta^{-1}(p_m f-q_mA)$, with the coefficient bounds
\eqref{eq:forced-derived-coefficient-bounds}, costs at most another
$4s+5$ powers. Thus a common forward estimate is
\begin{equation}
 \sup_{t_a\le t\le S}
 \frac{\|A(t)\|_{H^s}+\|A_t(t)\|_{H^s}+\|\pi(t)\|_{H^s}}{g(t)}
 \le C_s\mathcal K^{4s^2+18s+14}\mathcal N_s.
 \label{eq:forced-forward-fixed-outputs}
\end{equation}
The same exponent covers the history outputs by
\eqref{eq:forced-fixed-order-trace} and the same multiplication estimates.
This supplies the fixed-$s$ estimate needed to start the Gevrey-2
induction, without an exponential in $\sup|C_\perp|$.

For the additional $n$ derivative words, denote by $Z_n$ the sum of:
the history $L_t^2H^s$ norms of $\partial^I\xi_t$ and
$\partial^I\xi_{tt}$; the history time suprema of
$\partial^I\xi_t$; and the forward suprema of
$\partial^Ia_\perp/g$, summed over $|I|=n$.
Omit the history terms if $t_a=0$. Put
$F_n=\mathcal N_{n,\rho}(f)$, with the norms \eqref{eq:transverse-wordwise-time-norms}.
Write $Z_n=\mathcal H_n+\mathcal Q_n$, where $\mathcal H_n$ is the sum of the three history
quantities just listed and $\mathcal Q_n$ is the forward quantity. Define the two convolutions
\[
 U_n=\sum_{l=0}^n{n\choose l}R_0^l(l!)^2F_{n-l},
 \qquad
 V_n=\sum_{l=1}^n{n\choose l}R_0^l(l!)^2Z_{n-l},
 \qquad V_0=0.
\]
For a word $I$, differentiating the weak equation gives
\eqref{eq:forced-history-differentiated-form}, with the subsets of
the positions of $I$ replacing the multi-index splittings. Its force
term is a sum of products $(\partial^E G^T)
\partial^{I\setminus E}f$. Every other term has the form
\[
 -\int_0^{t_a}
 (\mathcal T^i\partial^{I\setminus E}\xi)^T
 (\partial^E B_{ik})\mathcal T^k b\,dt,
 \qquad E\ne\varnothing,
 \qquad i,k\in\{0,1\}.
\]
Estimate \eqref{eq:history-endpoint-poincare} controls $\mathcal T^0\partial^{I\setminus E}\xi$
by its time derivative. Thus the product bounds
\eqref{eq:forced-form-coefficient-bounds} and
\eqref{eq:forced-derived-coefficient-bounds} bound the sum of these
weak forcing norms by $C_s\mathcal K^{4s+5}(U_n+V_n)$.
The time $L^2$ norm of each force is bounded by its time supremum
because $t_a\le1$; all sums use the wordwise time-norm convention.

The energy induction \eqref{eq:forced-fixed-order-energy} applies
also to these weak forcing terms: testing against the differentiated
path bounds each term by its $L_t^2H^j$ coefficient norm times the
energy norm. Starting with the coercivity factor $\mathcal K^4$,
each of the $s$ base Sobolev steps costs at most $\mathcal K^7$.
Consequently, with $E_n$ denoting the word sum of
$\|\partial^I\xi_t\|_{L_t^2H^s}$,
\[
 E_n\le C_s\mathcal K^{11s+9}(U_n+V_n).
\]
In the differentiated strong equation
\eqref{eq:forced-history-strong}, the undifferentiated $C_\perp$
costs at most $\mathcal K^{4s+4}$, while every differentiated
coefficient contributes to $U_n$ or $V_n$. Hence
\[
 \sum_{|I|=n}\|\partial^I\xi_{tt}\|_{L_t^2H^s}
 \le C_s\mathcal K^{15s+13}(U_n+V_n).
\]
Apply the trace inequality used in
\eqref{eq:forced-fixed-order-trace} to each word and then
Cauchy--Schwarz to the sum of the square-root terms. Since
$t_a^{-1/2}\le\mathcal K^{1/2}$, this proves the convenient common bound
\begin{equation}
 \mathcal H_n\le C_s\mathcal K^{15s+14}(U_n+V_n).
 \label{eq:forced-factorial-history-step}
\end{equation}
The powers here are deliberately rounded upward; $\mathcal H_n=0$ when $t_a=0$.

On the forward interval, the equation for $\partial^I a_\perp$
is \eqref{eq:forced-differentiated-forward}. The force products are
bounded by $C_s\mathcal K^{4s+3}U_n$, and its nonempty coefficient
splittings by $C_s\mathcal K^{4s+4}V_n$.
Its initial value is bounded by $\mathcal H_n$.
Applying the base-order induction
\eqref{eq:forced-forward-fixed-recurrence} to this equation gives
\begin{equation}
 \mathcal Q_n\le C_s\mathcal K^{(4s+5)s}
 \left(\mathcal K \mathcal H_n+
       \mathcal K^{4s+5}(U_n+V_n)\right).
 \label{eq:forced-factorial-forward-step}
\end{equation}
The $s$ steps here concern only the fixed $H^s$ norm, so their
constants do not depend on the additional order $n$. Combining
\eqref{eq:forced-factorial-history-step} and
\eqref{eq:forced-factorial-forward-step}, we may take
$B_0=C_s\mathcal K^{4s^2+20s+16}$, increasing $C_s$ if necessary,
in the recurrence
\begin{equation}
 Z_n\le B_0\sum_{l=0}^n{n\choose l}R_0^l(l!)^2F_{n-l}
       +B_0\sum_{l=1}^n{n\choose l}R_0^l(l!)^2Z_{n-l},
 \qquad B_0,R_0/R_c\le C_s\mathcal K^{c_s}.
 \label{eq:forced-factorial-recurrence}
\end{equation}
In the second sum at least one derivative lands on a coefficient;
this is why no $Z_n$ occurs on its right. The history and forward
estimates are applied successively, with the history trace inserted
as forward data. Their finitely many prefactors are included in $B_0$.
The binomial coefficient counts which positions of an ordered word
land on the coefficient.

Use the choice $R\ge4(B_0+1)R_0$ stated after
Proposition~\ref{prop:forced-factorial}, and set $q=R_0/R$.
The normalized Leibniz identity in \eqref{eq:forced-gevrey-product-sum}
gives, under the induction hypothesis
$Z_j\le4B_0hR^j(j!)^2$ for $j<n$,
\[
 \frac{Z_n}{hR^n(n!)^2}
 \le B_0\sum_{l=0}^n\frac{q^l}{\binom nl}
       +4B_0^2\sum_{l=1}^n\frac{q^l}{\binom nl}
 \le\frac{B_0}{1-q}
       +4B_0^2\frac{q}{1-q}
 \le4B_0.
\]
Indeed, $q\le[4(B_0+1)]^{-1}\le1/4$ and $4B_0q\le1$,
so the penultimate expression is at most $8B_0/3$.
The case $n=0$ follows directly from
\eqref{eq:forced-factorial-recurrence}. This closes the induction
with one $R$, rather than enlarging $R$ at every derivative order.

Finally, put $\mathcal A_n=\mathcal N_{n,\rho}(A)$, and define
$\mathcal A_{t,n}$ and $\mathcal P_n$ analogously for $A_t$ and $\pi$.
The products $A=Ga_\perp$ (or $G\xi_t$ on the history) and
$Z_n\le4B_0hR^n(n!)^2$ give
\[
 \mathcal A_n\le C_s\mathcal K
 \sum_{l=0}^n{n\choose l}R_0^l(l!)^2Z_{n-l}
 \le 4C_s\mathcal KB_0hR^n(n!)^2\sum_{l=0}^nq^l
 \le \frac{4C_s\mathcal K B_0}{1-q}\,hR^n(n!)^2.
\]
Using $A_t=T_mA+\Pi_m f$ and
$\pi=\partial_\theta^{-1}(p_m f-q_mA)$ now gives
\[
 \mathcal A_{t,n}+\mathcal P_n
 \le C_s\mathcal K^{4s+5}
 \sum_{l=0}^n{n\choose l}R_0^l(l!)^2
             (\mathcal A_{n-l}+F_{n-l})
 \le C_s\mathcal K^{4s+6}B_0 hR^n(n!)^2.
\]
Recalling $B_0=C_s\mathcal K^{4s^2+20s+16}$, a common output prefactor is therefore
$C_s\mathcal K^{4s^2+24s+22}$.
In both lines the same geometric sum is bounded by $(1-q)^{-1}$;
there is no further increase of $R$. The operator
$\partial_\theta^{-1}$ commutes with derivatives and is bounded on
zero-mean $H^s$. This proves \eqref{eq:forced-output-factorial}.
\hfill$\square$

Together with Proposition~\ref{prop:forced-divergence-estimates} and the
mean inverse, this gives the common shift bounds of
Corollary~\ref{cor:mean-shared-inverses}. The fixed derivative scale,
polynomial prefactor and one-derivative loss then enter the coefficient
accounting of Section~\ref{sec:packet-expansion} and the exact correction
of Section~\ref{sec:exact-correction}.

\endgroup

\section{The mean equation and initial support}
\label{sec:mean-inverse}
\fancyhead[R]{\small The mean equation}

The mean step \eqref{intro:eq:mean-step} must be solved before the
oscillatory step treated in Section~\ref{sec:forced-transverse}.
Here we construct its inverse for a prescribed angle-independent force:
\begin{equation}
 B_t+MB+d\bar q=f,\qquad d\cdot B=0,
 \qquad d=F^{-T}\nabla_y.
 \label{eq:mean-averaged-equation}
\end{equation}
The mean pressure $\bar q$ enforces $d\cdot B=0$. It differs from
the oscillatory pressure coefficient $\pi$ in
Section~\ref{sec:forced-transverse}: the term
$m\partial_\theta\pi$ has zero angle average and cannot supply this
mean pressure gradient.

At order $p$, take $(B,\bar q)=(B_p,\bar q_p)$ and
$f=\langle f_p^{\rm old}\rangle_\theta$, with the known force
\eqref{intro:eq:old-force}. Section~\ref{sec:packet-expansion} uses
this inverse in the quantitative coefficient recursion.

The pressure in \eqref{eq:mean-averaged-equation} is spatial and can
be nonlocal. Accordingly, we require compact support only of $B(0)$,
not of $B(t)$ at positive times. We select the initial velocity through
a displacement boundary condition. The common estimates for this
inverse and the transverse inverse are collected in
Corollary~\ref{cor:mean-shared-inverses}.

\subsection{Global parent inputs and the localized boundary operator}

Fix $0<S\le1$ and $0<\ell\le1$. All parent assumptions below hold on
$[0,S]\times\mathbb R^3_y$ unless stated otherwise. Assume the parent identities
\eqref{intro:eq:parent}, which also imply
$F_t(0)=M(0)$. Fix $K_+,B_e,B_c\ge0$ and $0<r<1/4$.
Let $\mathcal K\ge2$ satisfy
\begin{equation}
 \begin{gathered}
 C_1B_c+1\le\mathcal K,\qquad S^{-1}\le\mathcal K,\\
 \sup_{t,y}\bigl(|F|+|F^{-1}|+|F_t|+|M|+|H|\bigr)
 \le\mathcal K.
 \end{gathered}
 \label{eq:mean-basic-inputs}
\end{equation}
Here $C_1$ is the constant of Lemma~\ref{lem:mean-boundary-operator}
below; enlarging $\mathcal K$ to meet its first constraint is included
in the polynomial input budget. Assume $H$ is symmetric and
\begin{equation}
 \begin{aligned}
 H(t,y)&\le K_+I,\\
 \operatorname{sym}M(0,y)&\ge-B_eI &&(|\ell y|\ge r),\\
 \operatorname{sym}M(0,y)&\ge-B_cI &&(|\ell y|<r).
 \end{aligned}
 \label{eq:mean-quadratic-inputs}
\end{equation}
The large negative part of the initial strain is thus confined to a
small ball in the physical label $\ell y$.

Fix a smooth cutoff $0\le\chi_o\le1$, equal to one on the unit ball
and supported in the ball of radius two, with Gevrey-2 derivative
bounds $C^{n+1}(n!)^2$. Put
\begin{equation}
 \chi(y)=\chi_o(\ell y),\qquad
 N=(-\Delta_y)^{-1},\qquad
 \mathcal A_\chi=\operatorname{curl}_y\,\chi N\chi\,\operatorname{curl}_y.
 \label{eq:mean-boundary-operator}
\end{equation}
Multipliers are composed from right to left. The outer curl makes
the output divergence free; the outer cutoff confines its support
to $\{|\ell y|\le2\}$. The Newton operator is understood through
its Fourier multiplier or the homogeneous Sobolev pairing below.

\refstepcounter{equation}
\label{lem:mean-boundary-operator}
\textbf{Lemma \theequation\ (The boundary operator and the interior mass).}
The operator $\mathcal A_\chi$ is bounded, self-adjoint, and
nonnegative on $L^2(\mathbb R^3;\mathbb R^3)$, with norm bounded
independently of $\ell$. For every divergence-free $v\in L^2$,
\begin{equation}
 \|v\|_{L^2(|\ell y|<r)}^2
 \le C_1\langle\mathcal A_\chi v,v\rangle
        +C_2r^3\|v\|_{L^2}^2,
 \label{eq:mean-interior-mass}
\end{equation}
where $C_1,C_2$ are fixed constants depending only on the chosen cutoff
and the dimension.

\emph{Proof.} If $\|\psi\|_\infty+\|\nabla\psi\|_{L^3}<\infty$,
integration by parts, Holder's inequality, and the Sobolev inequality
$\|\phi\|_{L^6}\le C\|\nabla\phi\|_{L^2}$ give
\[
 \begin{aligned}
 |\langle\psi\operatorname{curl}v,\phi\rangle|
 &=|\langle v,\psi\operatorname{curl}\phi+
                       \nabla\psi\times\phi\rangle|\\
 &\le C(\|\psi\|_\infty+\|\nabla\psi\|_{L^3})
                          \|v\|_{L^2}\|\nabla\phi\|_{L^2}.
 \end{aligned}
\]
Thus $\|\psi\operatorname{curl}v\|_{\dot H^{-1}}
\le C(\|\psi\|_\infty+\|\nabla\psi\|_{L^3})\|v\|_{L^2}$.
For $\psi=\chi$, both multiplier norms are independent of $\ell$:
the factor $\ell$ in $\nabla\chi$ cancels the $L^3$ scaling factor
$\ell^{-1}$. Consequently
\[
 \langle\mathcal A_\chi v,w\rangle
 =\langle N^{1/2}\chi\operatorname{curl}v,
          N^{1/2}\chi\operatorname{curl}w\rangle
\]
defines a bounded symmetric form, and proves nonnegativity and the
claimed operator bound. The support and divergence assertions follow
from the outer cutoff and curl, also distributionally for $L^2$ data.

Set $Z=\|N^{1/2}\chi\operatorname{curl}v\|_{L^2}$ and
$w=N\operatorname{curl}(\chi\operatorname{curl}v)$.
The Fourier multiplier $N^{1/2}\operatorname{curl}$ has norm at most
one, so $\|w\|_{L^2}\le Z$. In $|y|<\ell^{-1}$, the cutoff is one
and $\operatorname{div}v=0$, giving
\[
 -\Delta(v-w)=-\Delta v-\operatorname{curl}\operatorname{curl}v=0.
\]
The mean-value estimate for this harmonic field gives
\[
 \sup_{|y|<(2\ell)^{-1}}|v-w|^2
 \le C\ell^3\|v-w\|_{L^2(|y|<\ell^{-1})}^2.
\]
Integrating over $|y|<r/\ell$, and using
$|v|^2\le2|w|^2+2|v-w|^2$, yields
\[
 \|v\|_{L^2(|\ell y|<r)}^2
 \le2Z^2+Cr^3(\|v\|_{L^2}^2+Z^2).
\]
Since $r<1/4$, this implies \eqref{eq:mean-interior-mass}.
\hfill$\square$

Choose the boundary coefficient $\lambda_b$ such that
\begin{equation}\label{eq:mean-boundary-choice}
 0\le C_1B_c\le\lambda_b\le\mathcal K.
\end{equation}
The interval is nonempty by \eqref{eq:mean-basic-inputs}.
Separately, assume
\begin{equation}\label{eq:mean-coercivity-inputs}
 \frac{K_+S^2}{2}+B_eS+C_2B_cr^3S\le\frac12.
\end{equation}
For the stage choices, Lemma~\ref{lem:stage-coercivity-margin}
verifies this inequality with right side $1/4$: the short time controls
the first two terms, and the small volume factor $r^3$ compensates
for the large interior strain $B_c$.
If the exterior lower bound holds everywhere, one may take $B_c=0$
and $\lambda_b=0$.

\subsection{The variational inverse and pressure recovery}

Write $L^2_\sigma$ for divergence-free fields in $L^2(\mathbb R^3)$.
For a displacement $\eta=Fz$, the relation used in
\eqref{eq:history-velocity} becomes $B=\eta_t-M\eta=Fz_t$.
The divergence identity \eqref{intro:eq:piola}, proved geometrically
in Subsection~\ref{intro:subsec:order-expansion}, then gives
$d\cdot B=\operatorname{div}_y z_t=0$ whenever
$\operatorname{div}_y z=0$. Define the Hilbert space
\[
 \mathcal V=\{z\in H^1((0,S);L^2_\sigma):z(S)=0\},\qquad
 \|z\|_{\mathcal V}=\|z_t\|_{L_t^2L_y^2}.
\]
The terminal condition selects a displacement representative. Its
initial value is free, allowing a natural boundary condition to select
the initial velocity. With $\langle\cdot,\cdot\rangle$ denoting the
$L_y^2$ pairing, set
\begin{equation}
 \begin{aligned}
 \mathfrak a_m(z,b)
 ={}&\int_0^S\bigl(\langle(Fz)_t,(Fb)_t\rangle
                         -\langle HFz,Fb\rangle\bigr)\,dt\\
 &+\langle M(0)z(0),b(0)\rangle
       +\lambda_b\langle\mathcal A_\chi z(0),b(0)\rangle.
 \end{aligned}
 \label{eq:mean-bilinear-form}
\end{equation}

\refstepcounter{equation}
\label{prop:mean-variational-inverse}
\textbf{Proposition \theequation\ (Mean inverse with localized initial velocity).}
Assume \eqref{eq:mean-basic-inputs}, \eqref{eq:mean-quadratic-inputs},
and \eqref{eq:mean-coercivity-inputs}, with $\lambda_b$ chosen by
\eqref{eq:mean-boundary-choice}
and the parent identities
\eqref{intro:eq:parent}. For every angle-independent
$f\in L_t^2L_y^2$, there is a unique $z\in\mathcal V$ satisfying
\begin{equation}
 \mathfrak a_m(z,b)=-\int_0^S\langle f,Fb\rangle\,dt
 \qquad(b\in\mathcal V).
 \label{eq:mean-weak-equation}
\end{equation}
It has $z_t\in H^1((0,S);L^2_\sigma)$. The fields
$B=Fz_t$ and a pressure $\bar q$, determined up to a function of time,
satisfy
\begin{equation}
 \begin{gathered}
 B_t+MB+d\bar q=f,\qquad d\cdot B=0,\\
 B(0)=z_t(0)=\lambda_b\mathcal A_\chi z(0),\qquad
 \operatorname{supp}B(0)\subset\{|\ell y|\le2\}.
 \end{gathered}
 \label{eq:mean-inverse-conclusions}
\end{equation}
In particular, $B(0)=0$ if $\lambda_b=0$.
The weak problem \eqref{eq:mean-weak-equation} obeys
\begin{equation}
 \mathfrak a_m(z,z)\ge c\mathcal K^{-4}\|z\|_{\mathcal V}^2,
 \qquad
 \|z_t\|_{L_t^2L_y^2}\le C\mathcal K^5\|f\|_{L_t^2L_y^2}.
 \label{eq:mean-energy-bound}
\end{equation}

\textbf{Remark.} The form need not be symmetric, because $M(0)$
need not be symmetric. We use a coercive bilinear problem, rather
than claim a minimization principle. The support conclusion concerns
only time zero; the pressure and positive-time mean velocity can
extend beyond the support of $f$.

\emph{Proof.} For $\eta=Fz$, put
$E=\int_0^S\|\eta_t\|_{L^2}^2\,dt$ and
$Z^2=\langle\mathcal A_\chi z(0),z(0)\rangle$.
Since $\eta(S)=0$, the time-reversed Hilbert-space version of
\eqref{eq:history-endpoint-poincare}, with $T=S$, gives
\[
 \int_0^S\|\eta\|_{L^2}^2\,dt\le\frac{S^2}{2}E,
 \qquad \|z(0)\|_{L^2}^2=\|\eta(0)\|_{L^2}^2\le SE.
\]
The coefficient bounds \eqref{eq:mean-quadratic-inputs} and
\eqref{eq:mean-interior-mass}, together with
\eqref{eq:mean-coercivity-inputs} and the choice
\eqref{eq:mean-boundary-choice}, imply
\[
 \begin{aligned}
 \mathfrak a_m(z,z)
 &\ge\left(1-\frac{K_+S^2}{2}-B_eS-C_2B_cr^3S\right)E
                     +(\lambda_b-C_1B_c)Z^2
 \ge\tfrac12E.
 \end{aligned}
\]
Also $z_t=F^{-1}(\eta_t-M\eta)$, so
$\|z_t\|_{L_t^2L_y^2}\le C\mathcal K^2E^{1/2}$.
This proves coercivity. The coefficient bounds, the trace estimate from
\eqref{eq:history-endpoint-poincare},
$\|b(0)\|_{L^2}\le\sqrt S\|b\|_{\mathcal V}$, and
Lemma~\ref{lem:mean-boundary-operator} show
$|\mathfrak a_m(z,b)|\le C\mathcal K^3
\|z\|_{\mathcal V}\|b\|_{\mathcal V}$.
The right side of \eqref{eq:mean-weak-equation} has norm at most
$\mathcal K\|f\|_{L_t^2L_y^2}$ on $\mathcal V$.
Lax--Milgram therefore gives the unique solution and
\eqref{eq:mean-energy-bound}.

To recover the strong equation, let $\mathbb P_\sigma$ be the
orthogonal projection onto $L^2_\sigma$, and define, on that space,
\[
 K=F^TF,\qquad T(t)v=\mathbb P_\sigma(K(t)v).
\]
The operator $T$ is positive and invertible:
$\langle Tv,v\rangle=\|Fv\|_{L^2}^2
\ge\mathcal K^{-2}\|v\|_{L^2}^2$.
It and its inverse are continuously differentiable in time as bounded
operators on $L^2_\sigma$; indeed
$(T^{-1})_t=-T^{-1}T_tT^{-1}$.
Testing \eqref{eq:mean-weak-equation} away from the time endpoints,
and using $F_{tt}=-HF$ to cancel the terms containing $z$, gives
\[
 (Tz_t)_t=\mathbb P_\sigma\bigl(F^Tf+(K_t-2F^TF_t)z_t\bigr).
\]
The right side lies in $L_t^2L_y^2$. Thus $Tz_t$ and then
$z_t=T^{-1}(Tz_t)$ belong to $H_t^1L_y^2$, and
\begin{equation}
 Tz_{tt}=\mathbb P_\sigma(F^Tf-2F^TF_tz_t).
 \label{eq:mean-strong-equation}
\end{equation}

Now integrate the weak equation by parts in time. The interior term
vanishes by \eqref{eq:mean-strong-equation}. At time zero, the remaining
pairing is
\[
 \langle-(Fz)_t(0)+M(0)z(0)+
          \lambda_b\mathcal A_\chi z(0),b(0)\rangle
 =\langle-z_t(0)+\lambda_b\mathcal A_\chi z(0),b(0)\rangle=0.
\]
Here $F(0)=I$ and $F_t(0)=M(0)$ cancel the initial strain terms.
Every element of $L^2_\sigma$ occurs as $b(0)$, by taking
$b(t)=(1-t/S)b(0)$. Both vectors in the last pairing are
divergence free. Hence the claimed initial boundary condition holds.

Finally, $B_t+MB=Fz_{tt}+2F_tz_t$, so
\eqref{eq:mean-strong-equation} says that
$F^T(f-B_t-MB)$ is orthogonal to $L^2_\sigma$.
The Helmholtz decomposition supplies $\bar q$ with
\[
 \nabla_y\bar q=F^T(f-B_t-MB),\qquad
 d\bar q=f-B_t-MB.
\]
This determines its gradient in $L_t^2L_y^2$. Applying
\eqref{intro:eq:piola} to $B=Fz_t$ proves incompressibility,
and the initial support follows from the boundary operator.
\hfill$\square$

\subsection{Gevrey-2 bounds with one fixed derivative scale}

All spatial norms in this subsection are on $\mathbb R^3_y$.
Use the ordered-word convention of Section~\ref{sec:forced-transverse},
with no angle derivatives. For a fixed integer $s\ge0$, assume
\begin{equation}
 \sum_{|I|=n}\sup_{0\le t\le S}
       \|\partial_y^I D(t)\|_{W^{s,\infty}}
 \le\mathcal K R_c^n(n!)^2\quad(n\ge0),
 \quad D\in\{F,F^{-1},F_t,M,H\}.
 \label{eq:mean-factorial-inputs}
\end{equation}
The parent identities still supply $F_{tt}=-HF$.
Take $R_c\ge1$ large enough to include the fixed cutoff constants
below. This requirement depends only on $s$ and the chosen cutoff.

\refstepcounter{equation}
\label{prop:mean-factorial-inverse}
\textbf{Proposition \theequation\ (Gevrey-2 estimates for the mean inverse).}
Assume Proposition~\ref{prop:mean-variational-inverse} and
\eqref{eq:mean-factorial-inputs}. There is $C_s$, independent of the
stage and derivative order, such that, if
$R\ge C_s\mathcal K^{15s+19}R_c$ and
\begin{equation}
 F_n:=\sum_{|I|=n}\sup_{0\le t\le S}
               \|\partial_y^If(t)\|_{H^s}
 \le h_0R^n(n!)^2\qquad(n\ge0),
 \label{eq:mean-force-factorial}
\end{equation}
then the mean inverse satisfies
\begin{equation}
 \sum_{|I|=n}\sup_{0\le t\le S}
  \bigl(\|\partial_y^IB(t)\|_{H^s}
       +\|\partial_y^IB_t(t)\|_{H^s}
       +\|\partial_y^I(d\bar q)(t)\|_{H^s}\bigr)
 \le C_s\mathcal K^{15s+19}h_0R^n(n!)^2.
 \label{eq:mean-output-factorial}
\end{equation}
For smooth data the constructed fields are smooth, and their initial
velocity has the support stated in \eqref{eq:mean-inverse-conclusions}.
The scale $R$ is independent of the number of inverse applications.

\emph{Proof.} We give the boundary commutator estimates explicitly,
since they are the new ingredient relative to the transverse inverse.
For an ordered word $J$ of length $j$, scaling the cutoff gives
\[
 \|\partial^J\chi\|_\infty+
       \|\nabla\partial^J\chi\|_{L^3}
 \le C^{j+1}(j!)^2.
\]
In the second norm the scale is $\ell^{j+1}\ell^{-1}=\ell^j\le1$.
The pairing argument in Lemma~\ref{lem:mean-boundary-operator}
applied to two differentiated cutoffs bounds
\[
 \mathcal A_{J,K}
 =\operatorname{curl}(\partial^J\chi)N
                  (\partial^K\chi)\operatorname{curl}
\]
on $L^2$. Commuting at most $s$ further derivatives through it gives
\begin{equation}
 \|\mathcal A_{J,K}\|_{H^s\to H^s}
 \le C_s C_s^{j+k}((j+k)!)^2,
 \qquad j=|J|,\quad k=|K|.
 \label{eq:mean-cutoff-commutators}
\end{equation}
Indeed, each extra derivative either falls on the input or one of
the cutoffs; for fixed $s$, the factors
$(j+k+s)!/(j+k)!$ are polynomial in $j+k$ and are absorbed by
$C_s^{j+k}$, with the case $j+k=0$ included in $C_s$.
Since curl and $N$ commute with translations and derivatives,
$[\partial^I,\mathcal A_\chi]z$ is the sum over nonempty sets of
positions assigned to the two cutoffs, with the remaining derivatives
on $z$. Grouping by the total number $l\ge1$ assigned to the cutoffs
and absorbing the at most $2^l$ assignments in $R_c$ yields
\begin{equation}
 \sum_{|I|=n}\|[\partial^I,\mathcal A_\chi]v\|_{H^s}
 \le C_s\sum_{l=1}^n\binom nl R_c^l(l!)^2
               \sum_{|J|=n-l}\|\partial^Jv\|_{H^s}.
 \label{eq:mean-boundary-factorial-commutator}
\end{equation}

First obtain the estimate at the fixed Sobolev order. Allow general
weak data in the form
\[
 \mathfrak a_m(z,b)=\int_0^S
       (\langle g_1,b_t\rangle+\langle g_0,b\rangle)\,dt
                              +\langle g_b,b(0)\rangle.
\]
Coercivity and the time-reversed estimates
\eqref{eq:history-endpoint-poincare} give
$\|z_t\|_{L_t^2L_y^2}\le C\mathcal K^4
(\|g_1\|_{L_t^2L_y^2}+\|g_0\|_{L_t^2L_y^2}+\|g_b\|_{L_y^2})$.
Spatial difference quotients preserve $L^2_\sigma$ and $z(S)=0$.
In the translated weak equation, every commutator has a coefficient
difference multiplying a lower derivative of $z$ or $z_t$.
The four interior coefficients are those in
\eqref{eq:forced-form-coefficient-bounds}, with $G=F$ in place of
$FR_\perp$: $F^TF,F^TF_t,F_t^TF$, and $F_t^TF_t-F^THF$.
The same product estimate \eqref{eq:forced-gevrey-product-sum}
bounds their $W^{s,\infty}$ norms by $C_s\mathcal K^3$.
The boundary terms use derivatives of $M(0)$
and \eqref{eq:mean-cutoff-commutators}, multiplied by
$\lambda_b\le\mathcal K$. The time-reversed estimates
\eqref{eq:history-endpoint-poincare} control both
$\partial^Iz$ in $L_t^2H^j$ and $\partial^Iz(0)$ by the corresponding
norm of $\partial^Iz_t$.

Writing $E_j=\|z_t\|_{L_t^2H^j}$ and
$D_j=\|g_1\|_{L_t^2H^j}+\|g_0\|_{L_t^2H^j}+\|g_b\|_{H^j}$,
testing the differentiated equation gives
\[
 E_0\le C\mathcal K^4D_0,\qquad
 E_j\le C_s(\mathcal K^4D_j+\mathcal K^7E_{j-1}),\qquad
 E_s\le C_s\mathcal K^{7s+4}D_s.
\]
Uniform difference-quotient bounds and weak compactness justify
each derivative before proceeding to the next one. For the original
force, $g_0=-F^Tf$ and $g_1=g_b=0$, giving
$E_s\le C_s\mathcal K^{7s+5}\|f\|_{L_t^2H^s}$.

We also need the spatial inverse of $T$. For $Tv=h$ in
$L^2_\sigma$, the projection commutes with derivatives, and
\begin{equation}
 T\partial^Iv=\partial^Ih-
  \mathbb P_\sigma\sum_{\varnothing\ne E\subset\{1,\ldots,n\}}
       (\partial^{I_E}K)\partial^{I_{E^c}}v.
 \label{eq:mean-projected-inverse-derivatives}
\end{equation}
The $L^2$ bound $\|T^{-1}\|\le\mathcal K^2$ starts a spatial
difference-quotient induction. Each fixed Sobolev step costs at most
$\mathcal K^4$, since $K=F^TF$ has coefficient bound
$C_s\mathcal K^2$. Hence
\[
 \|T^{-1}h\|_{H^s}\le C_s\mathcal K^{4s+2}\|h\|_{H^s}.
\]
Equation \eqref{eq:mean-strong-equation} now gives $z_{tt}$ in
$L_t^2H^s$. In particular, the traces used below are well defined.

For the additional derivative words define
\[
 \begin{aligned}
 E_n^*&=\sum_{|I|=n}\|\partial^Iz_t\|_{L_t^2H^s},&
 J_n^*&=\sum_{|I|=n}\|\partial^Iz_{tt}\|_{L_t^2H^s},\\
 U_n^*&=\sum_{|I|=n}\sup_t\|\partial^Iz_t(t)\|_{H^s},&
 V_n^*&=\sum_{|I|=n}\sup_t\|\partial^Iz_{tt}(t)\|_{H^s},
 \end{aligned}
\]
and $Z_n=E_n^*+J_n^*+U_n^*+V_n^*$. Put
\[
 \mathcal F_n=\sum_{l=0}^n\binom nl R_c^l(l!)^2F_{n-l},\qquad
 \mathcal Z_n=\sum_{l=1}^n\binom nl R_c^l(l!)^2Z_{n-l}.
\]
The differentiated weak problem has force products and nonempty
coefficient splittings, including the boundary terms in
\eqref{eq:mean-boundary-factorial-commutator}. Its weak data norms
are at most $C_s\mathcal K^3(\mathcal F_n+\mathcal Z_n)$.
Apply the fixed-$s$ estimate to each word to obtain
\[
 E_n^*\le C_s\mathcal K^{7s+7}(\mathcal F_n+\mathcal Z_n).
\]
Next differentiate \eqref{eq:mean-strong-equation}. The new terms
are derivatives of $F^T$, $2F^TF_t$, or $K$ multiplying lower
derivatives of $f,z_t,z_{tt}$; the undifferentiated $z_t$ term
is bounded by $E_n^*$. Applying the fixed-$s$ bound for $T^{-1}$ gives
\[
 J_n^*\le C_s\mathcal K^{4s+4}
              (E_n^*+\mathcal F_n+\mathcal Z_n)
 \le C_s\mathcal K^{11s+11}(\mathcal F_n+\mathcal Z_n).
\]
For $w\in H^1((0,S);H^s)$, averaging its integral identity gives
\[
 \sup_t\|w(t)\|_{H^s}
 \le S^{-1/2}\|w\|_{L_t^2H^s}
                  +S^{1/2}\|w_t\|_{L_t^2H^s}.
\]
Since $S^{-1}\le\mathcal K$ and $S\le1$, this gives
$U_n^*\le C_s\mathcal K^{11s+12}(\mathcal F_n+\mathcal Z_n)$.
Applying the differentiated strong equation once more, now in the
time supremum, gives
\[
 V_n^*\le C_s\mathcal K^{4s+4}
              (U_n^*+\mathcal F_n+\mathcal Z_n)
 \le C_s\mathcal K^{15s+16}(\mathcal F_n+\mathcal Z_n).
\]
All these arguments first apply to finite smooth approximations or
difference quotients; the resulting bounds justify their limits.
There is no time-growth exponential in these estimates.

Consequently $Z_n\le D_s(\mathcal F_n+\mathcal Z_n)$ with
$D_s=C_s\mathcal K^{15s+17}$. This is the recurrence
\eqref{eq:forced-factorial-recurrence}, with
$(B_0,R_0,h)=(D_s,R_c,h_0)$. Its induction therefore gives
\[
 Z_n\le4D_sh_0R^n(n!)^2\qquad(n\ge0),
 \qquad R\ge4(D_s+1)R_c.
\]
The same $R$ works at every derivative order. Finally use
\[
 B=Fz_t,\qquad B_t=F_tz_t+Fz_{tt},\qquad
 d\bar q=f-B_t-MB.
\]
Their product convolutions cost at most two further powers of
$\mathcal K$, giving \eqref{eq:mean-output-factorial}.
The stated threshold $C_s\mathcal K^{15s+19}R_c$ includes both
this prefactor and the induction threshold. The boundary identity
from Proposition~\ref{prop:mean-variational-inverse} supplies the
initial support. Spatial regularity of every order follows from
the derivative estimates, and time regularity for smooth data follows
successively from \eqref{eq:mean-strong-equation}.
\hfill$\square$

\subsection{A common estimate for the coefficient expansion}

The mean inverse uses a constant time profile, while the transverse
inverse uses the history/forward profile $\rho(t)$ from
Section~\ref{sec:forced-transverse}. Both keep a single Gevrey-2 derivative
scale. For the coefficient accounting in
Section~\ref{sec:packet-expansion}, say that a field has integer shift
$\nu\ge0$ and profile $h(t)$ when
\begin{equation}
 \sum_{|I|=n}\sup_t
       \frac{\|\partial^Iu(t)\|_{H^s}}{h(t)}
 \le R^{n+\nu}((n+\nu)!)^2\qquad(n\ge0).
 \label{eq:mean-shared-shift}
\end{equation}
The derivatives are spatial for mean fields and include the angle
for oscillatory fields.

The preceding inductions are uniform in $\nu$: the factorial ratio
used whenever $l$ derivatives fall on a coefficient is
\[
 \binom nl\frac{(l!)^2((n-l+\nu)!)^2}{((n+\nu)!)^2}
 \le1.
\]
Indeed, $\binom{n+\nu}{l}\ge\binom nl$ and the left side equals
$\binom nl/\binom{n+\nu}{l}^2$.
Thus the same proofs give a fixed polynomial prefactor at the same
shift for the mean outputs and transverse potentials; the transverse
curl and pressure gradient cost one additional shift, as in
Proposition~\ref{prop:forced-divergence-estimates}.
Enlarging the common $R$ once to dominate the finitely many polynomial
prefactors absorbs each such prefactor into one further shift.

\refstepcounter{equation}
\label{cor:mean-shared-inverses}
\textbf{Corollary \theequation\ (The two inverses in a common convention).}
Assume the inputs of the mean inverse above and of
Propositions~\ref{prop:forced-factorial} and
\ref{prop:forced-divergence-estimates}, using a common bound
$\mathcal K$ and a sufficiently large common polynomial choice of $R$.
For every integer $\nu\ge0$, a mean force of shift $\nu$ and
constant profile $h_0$ gives $B,B_t,d\bar q$ of shift at most $\nu+1$
with profile $h_0$. A zero-mean transverse force of shift $\nu$
and profile $h_0\rho(t)$ gives $A,A_t,\pi,Q_A,C_A,(C_A)_t,d\pi$
of shift at most $\nu+2$ with that profile. In particular, a common
two-shift allowance suffices for both inverses. The support conclusions
are those of \eqref{eq:mean-inverse-conclusions} for the mean initial
velocity and Proposition~\ref{prop:forced-divergence-estimates} for
the transverse fields. The choice of $R$ is independent of $\nu$ and
the number of coefficients in the expansion.

\textbf{Remark.} For a fixed parent, choose $R$ once, large enough for
both inverses and the subsequent product estimates. This same $R$ must be kept throughout the coefficient recursion in
Section~\ref{sec:packet-expansion}; it cannot be enlarged at successive
orders or depend on the truncation length. The accumulating derivative
cost is recorded in the shifts, rather than in a changing Gevrey scale.

\emph{Proof.} The shifted mean estimate is bounded by
$C_s\mathcal K^a h_0R^{n+\nu}((n+\nu)!)^2$, for a fixed exponent $a$.
Choose $R\ge C_s\mathcal K^a$ to absorb this prefactor into one shift.

For the transverse outputs with a final spatial derivative, the bound
has the form
\[
 C_s\mathcal K^b h_0\rho(t)
 R^{n+\nu+1}((n+\nu+1)!)^2
 \le h_0\rho(t)R^{n+\nu+2}((n+\nu+2)!)^2,
\]
provided also $R\ge C_s\mathcal K^b$. The other transverse outputs
require only one extra shift and satisfy the stated weaker bound.
The exponents $a,b$ and the required lower bound on $R$ are independent
of $\nu$. Each inverse supplies its respective support properties.
\hfill$\square$

Section~\ref{sec:packet-expansion} now uses these inverses in the ordered
mean and oscillatory recursion and proves the residual estimates required
by Section~\ref{sec:exact-correction}.

\section{The finite order approximation}
\label{sec:packet-expansion}
\fancyhead[R]{\small The finite order approximation}

Sections~\ref{sec:forced-transverse} and \ref{sec:mean-inverse} provide
the two linear inverses. We now use them to cancel the nonlinear Euler
equation through the coefficient recursion of
Subsection~\ref{intro:subsec:order-expansion}. The point of stopping after finitely
many terms is quantitative: the coefficients lose spatial derivatives,
but a sufficiently large frequency makes the retained corrections small
and the uncancelled residual much smaller still.

\subsection{Inputs and the divergence-free ansatz}
\label{subsec:packet-inputs}

Fix $0\le t_a\le S\le1$ and the parent data of
Corollary~\ref{cor:mean-shared-inverses}, globally for the mean inverse
and on the wave support for the transverse inverse. All coefficient,
coercivity, and propagator hypotheses of that corollary are assumed,
with the Sobolev order fixed in Section~\ref{subsec:packet-coefficient-bounds}.
Let $P\ge2$ bound their constants and their initial Gevrey radii by
$P^{c_0}$, with a fixed exponent $c_0$.
These are conditional parent inputs; their compatibility at every stage
will be checked in the iteration.

Use $d,m,\tilde d$ from \eqref{intro:eq:normal} and
\eqref{intro:eq:lifted-derivative}, and the normalized physical
increments \eqref{intro:eq:normalized-increment}. The equation to
approximate is \eqref{intro:eq:lifted-W}.
Choose a fixed Gevrey-2 cutoff $\chi_1$ compactly supported in
$\{|y|<1/2\}$ and the mean-zero profile $f_\delta$.
The leading coefficient $A_1$ and its pressure are
\eqref{intro:eq:leading-coefficients} and \eqref{intro:eq:wave}, with
$B_1=\bar q_1=0$ as in Subsection~\ref{intro:subsec:order-expansion}.
Here $v$ is the history-selected wave velocity of
Section~\ref{sec:history-activation}, continued
by its homogeneous wave equation. At $t_a=0$ it is selected by prescribed
transverse initial data instead. The cutoff $\chi_1$ and the profile
$f_\delta$ are assumed to have Gevrey-2 constants bounded by $P^{c_0}$.
The selected coordinate endpoint $\xi_{\mathrm{end}}$ satisfies
$|\xi_{\mathrm{end}}|\le P^{c_0}$; when $t_a=0$, the prescribed
initial velocity amplitude $v(0,y)$ has Gevrey-2 bounds with constants
at most $P^{c_0}$.
Let $g$ be the positive amplitude profile in the propagator bound
\eqref{eq:forced-propagator-input}, normalized by $g(t_a)=1$.
Take $\alpha>0$ and set
\begin{equation}\label{eq:packet-amplitude-profile}
 \gamma(t)=\alpha\rho(t)=
 \begin{cases}
  \alpha,&0\le t\le t_a,\\
  \alpha g(t),&t_a\le t\le S.
 \end{cases}
\end{equation}
Here $\rho$ is the history/forward profile used by the transverse inverse.
Assume
\[
 H_0:=\max\{1,\sup_{[0,S]}\gamma\}\le P^{c_0}.
\]

For every order, use the transverse $A_p$, angle-independent
$B_p$, and curl corrections $(Q_p,C_p)$ of
\eqref{intro:eq:curl-correction}. The finite sums are
\eqref{intro:eq:finite-packet}. Their divergence vanishes exactly,
as proved in Subsection~\ref{intro:subsec:divergence}; only the
momentum residual requires a subsequent exact correction.

\subsection{Constructing the coefficients with the two inverses}\label{subsec:packet-construction}

At step $p\ge2$, the lower-order coefficients and $C_{p-1}$ are
known. Recall that $V_p=A_p+B_p+C_{p-1}$ in
\eqref{intro:eq:total-velocity-coefficient} collects the entire velocity
coefficient at order $k^{-p}$. The calculation
\eqref{intro:eq:pth-equation}--\eqref{intro:eq:old-force} separates the
known force $f_p^{\rm old}$ from the only current-order term,
$-(B_p\cdot m)\partial_\theta A_1$, whose angle average is zero.
This gives the following construction.

First solve the mean equation \eqref{intro:eq:mean-step} by
Proposition~\ref{prop:mean-variational-inverse}, which selects the
initial velocity as well as $B_p$ and $d\bar q_p$.
Then solve the oscillatory equation \eqref{intro:eq:wave-step}, whose
force now includes the known $B_p$, by
Proposition~\ref{prop:forced-forward}. For $t_a>0$, its history
problem is Lemma~\ref{lem:forced-history} with zero displacement at
both endpoints; for $t_a=0$, it uses $A_p(0)=0$.
Finally form $Q_p,C_p$ by \eqref{intro:eq:curl-correction}.
Thus the formal matching in Section~\ref{sec:wave-derivation} becomes
a construction to every finite order, with no simultaneous nonlinear
solve at a given order.

The oscillatory coefficients $A_p,\pi_p,Q_p,C_p$ remain supported in
$\operatorname{supp}\chi_1$. Indeed, terms containing only angle-independent
coefficients disappear on subtracting the average, and every other term
contains a previously constructed oscillatory coefficient or its derivative.
The transverse inverse preserves this support by
Proposition~\ref{prop:forced-forward}, and so do the curl corrections
by Proposition~\ref{prop:forced-divergence-estimates}. The mean velocities
$B_p$ need not be supported there at positive times. Their initial support
is supplied by \eqref{eq:mean-inverse-conclusions}, namely $|\ell y|\le2$.

\subsection{Uniform control of the coefficient recursion}
\label{subsec:packet-coefficient-bounds}

Fix an integer $s\ge3$. This ensures that
$H^s(\mathbb R^3\times\mathbb T)$ is an algebra for the nonlinear
product estimates below. Constants may depend on $s,c_0$ and the fixed
cutoffs, but not on the stage or truncation order.

Use the shifted Gevrey convention \eqref{eq:mean-shared-shift},
with time interval $[0,S]$. Derivative words use $(y,\theta)$;
angle-independent fields have the corresponding spatial norm.
For a family of fields, sum their norms in that definition.

\refstepcounter{equation}
\label{prop:packet-coefficient-bounds}
\textbf{Proposition \theequation\ (Bounds for the finite coefficient sequence).}
Under the preceding inputs, a common $R=P^c$ can be chosen independently
of $N$ so that the recursion is defined to every finite order, with
\begin{align}
 &(A_p,\partial_t A_{p},\pi_p,Q_p,C_p,\partial_t C_{p},d\pi_p):
 &&\text{shift }a_p=9p-8,
 &&\text{profile }\gamma H_0^{2p-2},\label{eq:packet-wave-shifts}\\
 &(B_p,\partial_t B_{p},d\bar q_p):
 &&\text{shift }b_p=9p-13,
 &&\text{profile }H_0^{2p-2}\quad(p\ge2).
 \label{eq:packet-mean-shifts}
\end{align}
The residual of \eqref{intro:eq:finite-packet} has no powers
$k^{-p}$ for $p\le N$.

\textbf{Proof.}
First recall the product accounting. If two fields have shifts $a,b\ge0$,
their product has shift $a+b$, up to a constant depending only on $s$.
The relevant convolution is bounded independently of the shifts because
\[
 \sum_{l=0}^n\binom nl
 \frac{((l+a)!)^2((n-l+b)!)^2}{((n+a+b)!)^2}
 =\sum_{l=0}^n\frac{\binom nl}{\binom{n+a+b}{l+a}^2}
 \le\sum_{l=0}^n\binom nl^{-1}\le2+\frac{n-1}{n}<3.
\]
The last inequality is \eqref{eq:forced-gevrey-product-sum};
for $n=0,1$, the corresponding upper bounds are $1,2$.

Multiplication by a parent coefficient preserves the shift, with
prefactor $3C_sK_0\le3C_sP^{c_0}$, where $C_s$ is the Sobolev
product constant. Indeed, if its derivative bounds are
$K_0R_c^l(l!)^2$ and $R\ge R_c$, the same convolution, with one shift
set to zero, bounds multiplication on any shifted class by $3C_sK_0$,
uniformly in the shift. The parent bounds
\eqref{eq:forced-factorial-input} and \eqref{eq:mean-factorial-inputs},
imported in Section~\ref{subsec:packet-inputs}, give
$K_0=\mathcal K\le P^{c_0}$ for $F^{-T},M$.
Transposition preserves the Euclidean matrix norm, and contraction with
the fixed unit vector $m_0$ gives the same bound for $m=F^{-T}m_0$.
The factor $3C_s$ comes from the product estimate, not the parent inputs.
Each forcing term in the table below contains at most one such parent
multiplier. We therefore retain a prefactor $C_sP^{c_0}$ until the
summation step, where one spare shift absorbs it together with the
number of terms. Sobolev multiplication is valid since $s\ge3$
on the four-dimensional lifted $(y,\theta)$-space.

For $p=1$, use the amplitude profile $\gamma=\alpha\rho$ from
\eqref{eq:packet-amplitude-profile}, which includes the prefactor $\alpha$
in $A_1$. The homogeneous history problem has
polynomial Gevrey bounds by the argument for
Proposition~\ref{prop:forced-factorial}, using linear interpolation between the endpoints.
Explicitly, if its coordinate endpoint is $h(y,\theta)$, write
$\xi=t h/t_a+\xi_0$. The first term contributes velocity
$FR_\perp h/t_a$; its transverse equation has forcing
$2F_tR_\perp h/t_a$. Thus $\xi_0$, which vanishes at both ends,
is estimated by the same zero-endpoint inverse with the negative of
this forcing. All constants in this linear interpolation argument are polynomial in the assumed data,
including $t_a^{-1}$ when $t_a>0$. Differentiating the forward equation
and using \eqref{eq:forced-propagator-input} gives the same profile $g$
and the Gevrey convolution of that proposition. For $t_a=0$, the
prescribed Gevrey initial value replaces the history trace in this argument.
Proposition~\ref{prop:forced-divergence-estimates} then bounds $Q_1,C_1$
and the indicated derivatives. Before fixing the common $R$, these
estimates give polynomial constants $C_0,R_0\le P^c$ such that, for
the whole family $\mathcal U_1=(A_1,\partial_t A_{1},\pi_1,Q_1,C_1,\partial_t C_{1},d\pi_1)$,
\[
 \sum_{u\in\mathcal U_1}\sum_{|I|=n}\sup_t
 \frac{\|\partial^Iu(t)\|_{H^s}}{\gamma(t)}
 \le C_0R_0^n(n!)^2.
\]
Here the one additional spatial derivative is absorbed into $C_0,R_0$:
$R_*^{n+1}((n+1)!)^2
\le C R_*(2R_*)^n(n!)^2$. This enlargement is made only for the
data at $p=1$, not at successive expansion orders. Taking
$R\ge\max\{C_0,R_0\}$ gives
\[
 C_0R_0^n(n!)^2\le R^{n+1}((n+1)!)^2,
\]
which proves the bound at $p=1$ with $a_1=1$.

Suppose the bounds hold below order $p$. The following table records
upper bounds on the shifts before applying either inverse. A term
involving $B_p$ is used only after the mean equation has been solved.
\[
\begin{array}{c|c|c}
 \text{term}&\text{shift bound}&\substack{\text{nonzero}\\\text{mean?}}\\\hline
 V_i\cdot dV_j,\ i+j=p
       &a_i+a_j+1=9p-15&\text{yes}\\
 (B_i\cdot m)\partial_\theta(A_j+C_{j-1}),\ i+j=p+1
       &b_i+a_j+1=9p-11&\text{no}\\
 (C_{i-1}\cdot m)\partial_\theta(A_j+C_{j-1}),\ i+j=p+1
       &a_{i-1}+a_j+1=9p-15&\text{yes}\\
 \partial_t C_{p-1},\ MC_{p-1},\ d\pi_{p-1}
       &a_{p-1}=9p-17&\text{no}
\end{array}
\]
Here $A_i\cdot m=0$ and $\partial_\theta B_j=0$ eliminate the other
phase terms. A spatial or angle derivative costs one shift, accounting
for the $+1$ in the first three rows; the derivatives in the last row are already
included in the coefficient-family bounds. Replacing $A_j$ by $B_j$ or
$C_{j-1}$ lowers the shift whenever that term is present. The second row
has zero mean because $B_i,m$ are angle independent.

The mean forcing has profile $H_0^{2p-2}$. The oscillatory forcing has
profile $\gamma H_0^{2p-2}$: each of its terms contains an oscillatory
factor, and any additional $\gamma$ can be bounded by $H_0$.
For example, a label product of two oscillatory factors contributes
$\gamma^2H_0^{2p-4}\le\gamma H_0^{2p-2}$ and also
$\le H_0^{2p-2}$; a phase product $B_iA_j$ contributes exactly
$\gamma H_0^{2p-2}$ when $i+j=p+1$.

The largest mean-force shift in the table is $9p-15=b_p-2$;
the largest oscillatory-force shift is $9p-11=a_p-3$.
There are at most $Cp$ terms, each with prefactor at most
$C_sP^{c_0}$. To absorb both costs, increase the forcing shift by one:
\[
 \nu_m=b_p-1=9p-14,\qquad \nu_o=a_p-2=9p-10.
\]
Indeed, for either $\nu=\nu_m$ or $\nu_o$,
\[
 C_spP^{c_0}R^{n+\nu-1}((n+\nu-1)!)^2
 \le R^{n+\nu}((n+\nu)!)^2
 \quad\text{if}\quad C_spP^{c_0}\le R(n+\nu)^2.
\]
For $p\ge2$ both target shifts are at least $p$, so
$R\ge C_sP^{c_0}$ suffices for this absorption, uniformly in $n,p$.
Choose the common $R$ to satisfy this bound as well as the inverse
and $p=1$ requirements already stated. In the oscillatory force, use the mean bound at order $p$ only
after solving the mean equation. Thus no unknown prefactor enters this
choice of $R$.
Corollary~\ref{cor:mean-shared-inverses} now gives
\eqref{eq:packet-mean-shifts} using its one-shift mean allowance in
\eqref{intro:eq:mean-step}, then \eqref{eq:packet-wave-shifts}
using its two-shift transverse allowance in
\eqref{intro:eq:wave-step}. Finally,
\eqref{intro:eq:pth-equation} cancels every coefficient through
order $N$; the order-one equation is the homogeneous equation for $A_1$.
\hfill$\square$

\subsection{Truncation and the input for the exact correction}\label{subsec:packet-truncation}

For a radius $\varrho>0$ write
\[
 \|u\|_{\varrho,s}
 =\sum_{n\ge0}\frac{\varrho^n}{(n!)^2}
       \sum_{|I|=n}\|\partial^Iu\|_{H^s}.
\]
Time suprema below are understood term by term in this sum.

\refstepcounter{equation}
\label{prop:packet-truncation-bounds}
\textbf{Proposition \theequation\ (Small transport and residual).}
There are fixed constants $c,k_0$ depending only on $s$, the exponent
$c_0$ and the fixed cutoff choices in Section~\ref{subsec:packet-inputs}, independently of
$P$, the stage, $k$ and $N$, such that, if
\begin{equation}\label{eq:packet-frequency-choice}
 \vartheta=10^{-6},\qquad N=\lfloor k^\vartheta\rfloor,
 \qquad P^c\le k^\vartheta,\qquad k\ge k_0,
\end{equation}
the approximation \eqref{intro:eq:finite-packet} has, at a common radius
$\varrho=P^{-c}$, the bounds
\[
 z^a=kF^{-1}W^a,\qquad
 b^a=(k^{-1} z^a,m_0\cdot z^a),\qquad
 r^a=kF^{-1}\mathcal R^a,
\]
where $\mathcal R^a=W_t^a+MW^a+W^a\cdot\tilde dW^a+\tilde dq^a$,
with
\begin{align}
 \|z^a\|_{\varrho,s}+\|z_t^a\|_{\varrho,s}
       +\|\nabla_{y,\theta}z^a\|_{\varrho,s}&\le P^c,
       \label{eq:packet-scaled-velocity}\\
 \|b^a\|_{\varrho,s}&\le k^{-1}P^c+O(k^{-1.9})\le k^{-1/2},
       \label{eq:packet-small-transport}\\
 \|r^a\|_{\varrho,s}&\le\exp(-0.7k^\vartheta\log k).
       \label{eq:packet-small-residual}
\end{align}
Moreover,
\begin{equation}\label{eq:packet-leading-remainders}
 \|W^a-k^{-1} A_1\|_{\varrho,s}\le k^{-2}P^c,\qquad
 \left\|\sum_{p=2}^Nk^{-(p+1)}\pi_p\right\|_{\varrho,s}
       \le k^{-3}P^c,\qquad
 \left\|\sum_{p=2}^Nk^{-p} d\bar q_p\right\|_{\varrho,s}
       \le k^{-2}P^c.
\end{equation}
These conclusions also hold with one additional fixed label or angle
derivative after decreasing the radius by a fixed factor.

\textbf{Proof.}
The shifts in Proposition~\ref{prop:packet-coefficient-bounds}
must be tracked uniformly up to $N$. We use their slope $9$ below,
while retaining the stated frequency condition and residual exponent.
For $n,d\ge0$,
\[
 R^{n+d}((n+d)!)^2=R^{n+d}\binom{n+d}d^2(n!)^2(d!)^2
 \le (4R)^n(n!)^2(4R)^d(d!)^2.
\]
The quantities $V_p,\partial_t V_p$ entering the estimates for $z^a$
and $z_t^a$ have shifts at most $9p-8$, and $\nabla_{y,\theta}V_p$
entering the estimate for $\nabla_{y,\theta}z^a$ has shift at most $9p-7$. These bounds also cover $V_{N+1}=C_N$.
The residual identity \eqref{intro:eq:residual} gives
\[
 \mathcal R^a=\sum_{p=N+1}^{2N+2}k^{-p}\mathcal R_p,
 \qquad
 r^a=kF^{-1}\mathcal R^a
     =\sum_{p=N+1}^{2N+2}k^{1-p}F^{-1}\mathcal R_p.
\]
For a coefficient $\mathcal R_p$ of the unnormalized momentum residual
$\mathcal R^a$, label products have shift at most $9p-15$.
The phase products come from the $km\partial_\theta$ part of
$W^a\cdot\tilde dW^a$. At order $k^{-p}$ they are
\[
 \sum_{\substack{i+j=p+1\\1\le i,j\le N+1}}
 (V_i\cdot m)\partial_\theta V_j
 =\sum_{\substack{i+j=p+1\\1\le i,j\le N+1}}
 \bigl((B_i+C_{i-1})\cdot m\bigr)
 \partial_\theta(A_j+C_{j-1}),
\]
where we set $A_{N+1}=B_{N+1}=0$ for this formula, consistent with
$V_{N+1}=C_N$. The equality uses $A_i\cdot m=0$ and
$\partial_\theta B_j=0$. The largest remaining shift is
\[
 b_i+a_j+1=9p-11,\qquad i+j=p+1,
\]
coming from $(B_i\cdot m)\partial_\theta A_j$, because $C_{i-1}$ has a smaller shift than both $B_i$ and $A_i$.

The contribution left by the last retained divergence correction and
oscillatory pressure is a linear tail:
\[
 k^{-(N+1)}\bigl(\partial_t C_N+MC_N+d\pi_N\bigr).
\]
All lower orders have been cancelled by the recursion. This tail occurs
only at $p=N+1$; its coefficient has shift $a_N=9N-8=9p-17$,
because $\partial_t C_N$ and $d\pi_N$ are already included in
\eqref{eq:packet-wave-shifts}.

Thus $\mathcal R_p$ has shift at most $9p-11$.
Multiplication by $F^{-1}$ adds a polynomial prefactor without changing
this shift; the normalization contributes the factor $k$, so that the
corresponding term in $r^a$ has power $k^{1-p}$.

Thus $d_p=9p-7$ bounds the shifts of all the quantities being estimated.
Since $p\le2N+2$, we have $d_p\le CN$ and
\[
 (d_p!)^2\le(CN)^{2d_p}=(CN)^{18p-14}.
\]
For $\mathcal R_p$, the sharper exponent is $18p-22$.
The profiles in \eqref{eq:packet-wave-shifts} and
\eqref{eq:packet-mean-shifts} are $\gamma H_0^{2p-2}$ and
$H_0^{2p-2}$. Since $\gamma\le H_0\le P^{c_0}$, they contribute
at most $P^{c_0(2p-1)}$ and $P^{c_0(2p-2)}$, respectively.
Their products at residual order $p$ also contribute only a power of
$P$ whose exponent grows linearly with $p$. Multiplication by the
parent coefficients $M,F^{-T},m$ in the residual, and by $F^{-1}$
in its normalization, contributes fixed polynomial prefactors as
estimated in Section~\ref{subsec:packet-coefficient-bounds}.
The factor $(4R)^{d_p}$ from the factorial estimate must also be included:
writing the already fixed scale as $R=P^{c_R}$, we have
\[
 (4R)^{d_p}\le P^{(c_R+2)(9p-7)},\qquad P\ge2.
\]
Together these factors are bounded by $P^{c(p+1)}$ for a fixed $c$.
There are at most $Cp\le CN$ products at each order. Absorbing all
these factors into a common allowance, at a radius
$\varrho=P^{-c}$ every such coefficient is bounded by
\begin{equation}\label{eq:packet-coefficient-majorant}
 \mathcal B_k^{p+1},\qquad
 \mathcal B_k=P^c(CN)^{18}\le k^{0.01}.
\end{equation}
The spare powers between $18p-14$ (or $18p-22$) and $18(p+1)$
absorb the number of terms. Choose the fixed $c$ to dominate all
polynomial exponents in these estimates and the radius requirement.
Then \eqref{eq:packet-frequency-choice} gives
\[
 \mathcal B_k\le C^{18}k^{19\vartheta}\le k^{0.01}
\]
for sufficiently large $k_0$, since $19\vartheta<0.01$.
The radius can be fixed independently of $N$, since all $n$ dependence
above is confined to $(4R)^n(n!)^2$. An additional fixed derivative
is absorbed by shrinking that radius. Products with $F^{-1}$ and its
time derivative are controlled at the same smaller radius.

The first two coefficients have fixed shifts, and hence bounds $P^c$
independent of $N$. For the remaining velocity terms, the normalized
tail is bounded by
\[
 \sum_{p=3}^{N+1}k^{1-p}\mathcal B_k^{p+1}
 \le\sum_{p=3}^{N+1}k^{1-p+0.01(p+1)}
 =O(k^{-1.9}).
\]
This proves \eqref{eq:packet-scaled-velocity} and the velocity part of
\eqref{eq:packet-leading-remainders}; the same argument proves the
two pressure bounds. Only the gradient of the mean pressure is asserted,
as required by the mean inverse.

For \eqref{eq:packet-small-transport}, the first component has the
explicit factor $k^{-1}$. In the second component the leading term
vanishes: $m_0\cdot F^{-1}A_1=m\cdot A_1=0$.
The first possible contribution is thus the order-two coefficient,
which is $O(k^{-1} P^c)$ after normalization; the remaining tail is
the one just estimated. Since $P^c\le k^\vartheta$, this gives
$k^{-1}P^c+O(k^{-1.9})\le k^{-1/2}$ for sufficiently large $k_0$.
This explains why the lifted transport is small
despite the factor $k$ in $\tilde d$.

Finally, all residual orders through $N$ vanish. Equations
\eqref{eq:packet-coefficient-majorant} and
\eqref{eq:packet-frequency-choice} give
\[
 \|r^a\|_{\varrho,s}
 \le P^c\sum_{p=N+1}^{2N+2}k^{1-p+0.01(p+1)}
 \le C P^c k^{-0.99N+0.03}
 \le \exp(-0.7k^\vartheta\log k).
\]
In the last step $N=\lfloor k^\vartheta\rfloor$ and the fixed
lower threshold $k_0$ absorbs the floor and the polynomial factor,
using $P^c\le k^\vartheta$.
\hfill$\square$

Taking $s=6$, Proposition~\ref{prop:packet-truncation-bounds} supplies
the finite order estimates used in Section~\ref{sec:exact-correction},
where the residual is removed under the additional frequency condition
\eqref{eq:corr-frequency}. The verification of the parent inputs at
each stage, and the full child-flow estimates needed to repeat the
construction, remain separate steps.

\section{From finite order approximation to an exact Euler solution}
\label{sec:exact-correction}
\fancyhead[R]{\small From approximation to exact solution}

Proposition~\ref{prop:packet-truncation-bounds} supplies the finite order
estimates, including a small momentum residual. We now remove that
residual on the whole interval $[0,S]$, with zero initial velocity
correction. We solve the correction equation \eqref{intro:eq:correction}
derived in Subsection~\ref{intro:subsec:exact-identities}.
The pressure inverse and weighted energy estimates below lead to the
exact-correction proposition; the last corollary recovers the leading
velocity gradient and pressure Hessian.

\subsection{The input estimates and the conclusion}
\label{subsec:corr-inputs}

Use the smooth parent and pulled-back matrices from
\eqref{intro:eq:parent} on $[0,S]$, with $0<S\le1$.
The data bound $P\ge2$ is the one used in Section~\ref{sec:packet-expansion}.
All constants and fixed exponents below depend only on the fixed input
parameters specified there, independently of the stage, frequency,
derivative cutoff and auxiliary viscosity. As before, $P^c$ denotes
a fixed polynomial bound, enlarged finitely many times when necessary.
Fix $s=6$ and work on the lifted $(y,\theta)$-space, with
$\theta\in\mathbb T$ independent of $y$. The operators and normalized
physical increments are those of \eqref{intro:eq:lifted-derivative}
and \eqref{intro:eq:normalized-increment}; their equation is
\eqref{intro:eq:lifted-W}.

For an ordered word $I$ of derivatives in $(y,\theta)$, define
\[
 |f|_{n,r}=\sum_{|I|=n}\|\partial^If\|_{H^r(\R^3\times\mathbb T)},
 \qquad
 \|f\|_\varrho:=\|f\|_{\varrho,s}=\sum_{n\ge0}\frac{\varrho^n}{(n!)^2}|f|_{n,s}.
\]
For coefficients independent of $\theta$, use $W^{r,\infty}$ in place of
$H^r$; such coefficients need not be square integrable. Put
$G=F^{-1}F^{-T}$ and $K=G^{-1}=F^TF$. Here we assume the parent bounds
\begin{equation}\label{eq:corr-parent-bounds}
 \begin{gathered}
 \sum_{|I|=n}\sup_{0\le t\le S}
 \|\partial_y^I(F,F^{-1},M,H,m,G,K)(t,\cdot)\|_{W^{s,\infty}(\R^3)}
 \le P^c(P^c)^n(n!)^2\qquad(n\ge0),\\
 P^{-c}I\le G(t,y),K(t,y)\le P^c I
 \qquad((t,y)\in[0,S]\times\R^3).
 \end{gathered}
\end{equation}
The tuple norm is the sum of its component norms. Any fixed additional
label derivatives used below obey the same form of bound after a fixed
enlargement of $c$. The bounds for $m,G,K$ follow from those for $F,F^{-1}$
and $m=F^{-T}m_0$, with the same convention on $c$.

Let $(W^a,q^a)$ be the finite order approximation. It is exactly divergence free,
and we recall its lifted velocity, transport field, and residual from
Proposition~\ref{prop:packet-truncation-bounds}:
\begin{equation}
 \begin{aligned}
 z^a&=kF^{-1}W^a,
 &b^a&=(k^{-1} z^a,m_0\cdot z^a),\\
 r^a&=kF^{-1}
  (W^a_t+MW^a+W^a\cdot\tilde dW^a+\tilde dq^a).
 \end{aligned}
 \label{eq:corr-input-defs}
\end{equation}
Here $b^a$ has four components: three for transport in $y$ and one for
transport in $\theta$. At a radius $\varrho_a=P^{-c}$, the inputs are
\begin{equation}
 \begin{aligned}
 \|z^a\|_{\varrho_a}+\|z^a_t\|_{\varrho_a}
       +\|\nabla_{y,\theta}z^a\|_{\varrho_a}&\le P^c,\\
 \|b^a\|_{\varrho_a}&\le B_0\le k^{-1/2},\\
 \|r^a\|_{\varrho_a}&\le R_k:=\exp(-0.7k^\vartheta\log k),
 \qquad \vartheta=10^{-6}.
 \end{aligned}
 \label{eq:corr-input-bounds}
\end{equation}
These estimates hold uniformly on $[0,S]$ by
\eqref{eq:packet-scaled-velocity}--\eqref{eq:packet-small-residual}.
The sharper bound \eqref{eq:packet-small-transport} gives
$\|b^a\|_{\varrho_a}\le k^{-1}P^c+O(k^{-1.9})$. Its smallness comes from
transversality: the leading part of $z^a$ has $m_0\cdot z^a=0$,
so the leading angle transport vanishes. This is where the wave geometry
enters the nonlinear correction.

\refstepcounter{equation}
\label{prop:exact-packet-correction}
\textbf{Proposition \theequation\ (Exact correction of the finite order approximation).}
Assume the parent bounds \eqref{eq:corr-parent-bounds} and the finite order
estimates \eqref{eq:corr-input-bounds}, with their polynomial exponents fixed.
There are sufficiently large fixed constants $Q,k_0$, with the
dependencies stated above, such that the following conclusion holds whenever
\begin{equation}
 P^Q\le k^\vartheta,\qquad k\ge k_0.
 \label{eq:corr-frequency}
\end{equation}
Then there is a smooth correction $e$, with $e(0)=0$, and a vector pressure
correction $\mathfrak p^e$, such that, at a fixed radius $\varrho_*=P^{-c}$,
with $c$ enlarged as needed under the convention above,
\begin{equation}
 \sup_{0\le t\le S}
 \bigl(\|e(t)\|_{\varrho_*}+\|e_t(t)\|_{\varrho_*}
                     +\|\mathfrak p^e(t)\|_{\varrho_*}\bigr)\le P^c\Delta_k,
 \qquad \Delta_k:=\exp(-k^\vartheta/2).
 \label{eq:corr-conclusion}
\end{equation}
The pair $(e,\mathfrak p^e)$ solves \eqref{eq:corr-e-equation} with $\nu=0$,
where $\mathfrak p^e$ lies in the gradient space $\mathcal G$ defined below.
The fields \eqref{intro:eq:physical-child} are a smooth exact Euler
solution on $[0,S]$, and their velocity increment at time zero is
\eqref{eq:corr-initial-preserved}. If the parent and approximate velocity
have the prescribed odd parity, the velocity correction has that parity.

\textbf{Remark.} The variable $\mathfrak p^e$ is a three-component pressure
variable; the scalar physical pressure is recovered in
Section~\ref{subsec:corr-physical}. The condition
\eqref{eq:corr-frequency} is stronger than the frequency condition of
Section~\ref{sec:packet-expansion}: it also pays for the energy growth
in the correction argument. The next four lemmas give the pressure inverse, product and pressure
estimates, energy inequality, and bootstrap used in the proof of the proposition.

\subsection{The lifted correction equation}

Use $D$, $z=kF^{-1}W$ and $\mathfrak p=k^2Dq$ from
\eqref{intro:eq:fixed-div}. The identities
\eqref{intro:eq:fixed-div-identity} turn incompressibility into the
fixed constraint $\operatorname{div}_Dz=0$. The momentum equation and
its approximate counterpart are \eqref{intro:eq:fixed-momentum} and
\eqref{intro:eq:approximate-fixed-momentum}.

For the existence proof, add an auxiliary viscosity to their difference
\eqref{intro:eq:correction}. With $z=z^a+e$ and
$\mathfrak p=\mathfrak p^a+\mathfrak p^e$, solve
\begin{equation}
 e_t+b\cdot\nabla_{y,\theta}e+Z(e)+G\mathfrak p^e
       =-r^a+\nu\Delta_{y,\theta}e,
 \quad \operatorname{div}_De=0,\quad e(0)=0,
 \qquad 0\le\nu\le1,
 \label{eq:corr-e-equation}
\end{equation}
where $b,Z(e)$ are defined in \eqref{intro:eq:b-Z}.
We construct solutions for $\nu>0$ and pass to $\nu=0$.
All derivatives of the unknown $e$ lie in transport or viscosity:
$Z(e)$ contains no derivative of $e$. Moreover,
$\operatorname{div}_{y,\theta}b=\operatorname{div}_D(z^a+e)=0$.
These two facts are the starting point of the energy estimate.

\subsection{The pressure inverse}
\label{subsec:corr-pressure}

Let $\mathcal G$ be the $L^2(\R^3\times\mathbb T)$ closure of fields
$D\phi$ with smooth periodic $\phi$ compactly supported in $y$.
Integration by parts shows
\[
 \mathcal G^\perp=\{v:\operatorname{div}_Dv=0\}.
\]
The $L^2$-orthogonal projection $\Pi$ onto $\mathcal G$ has Fourier symbol
\[
 \widehat\Pi(\xi,j)
  =\frac{(k^{-1}\xi+jm_0)\otimes(k^{-1}\xi+jm_0)}
         {|k^{-1}\xi+jm_0|^2},\qquad (\xi,j)\in\R^3\times\mathbb Z.
\]
The value on the measure-zero singular set is immaterial. This is a bounded
matrix projection: its norm is at most one, uniformly in $k$, and it
commutes with every ordinary derivative. No lower bound on
$|k^{-1}\xi+jm_0|$ is needed.

\refstepcounter{equation}
\label{lem:corr-pressure-estimate}
\textbf{Lemma \theequation\ (Pressure inversion without frequency loss).}
Under the coefficient bounds of Section~\ref{subsec:corr-inputs}, for
each $f\in L^2$ there is a unique $p\in\mathcal G$ with
$f-Gp\in\mathcal G^\perp$. Write $p=\mathcal T f$. There is a fixed
exponent $c$, large enough for both polynomial factors in
\eqref{eq:corr-pressure-recursion}, such that for $r=s-1,s$ and every
fixed radius $0<\varrho\le\tfrac14P^{-2c}$,
\begin{equation}\label{eq:corr-pressure-weighted-bound}
 \sum_{n=0}^J\frac{\varrho^n}{(n!)^2}|p|_{n,r}
 \le 2P^c\sum_{n=0}^J\frac{\varrho^n}{(n!)^2}|f|_{n,r}
 \qquad(J\ge0).
\end{equation}
The constants are uniform in time, $k$ and $J$.

\emph{Proof.} Given $f$, seek $p\in\mathcal G$ such that $f-Gp\in\mathcal G^\perp$.
Equivalently, $\Pi Gp=\Pi f$. On $\mathcal G$,
\[
 \langle\Pi Gp,p\rangle=\langle Gp,p\rangle=\|F^{-T}p\|_2^2
       \ge P^{-c}\|p\|_2^2.
\]
The operator is bounded and self-adjoint, so coercivity gives a unique
solution and the bound $\|p\|_2\le P^c\|f\|_2$. Write
\begin{equation}
 p=\mathcal T f=(\Pi G|_{\mathcal G})^{-1}\Pi f.
 \label{eq:corr-pressure-inverse}
\end{equation}

Here is the derivative estimate needed later. Differentiate
$\Pi Gp=\Pi f$ by a word $I$, move the terms with a derivative on $G$
to the right, and use the same inverse. Through the fixed base orders
$r=s-1,s$, this costs a fixed power of $P$. For the additional derivatives,
Leibniz's rule gives
\begin{equation}
 |p|_{n,r}\le P^c\left(
 |f|_{n,r}+\sum_{l=1}^n\binom nl P^{cl}(l!)^2|p|_{n-l,r}\right).
 \label{eq:corr-pressure-recursion}
\end{equation}
To justify differentiating a solution initially known only in $L^2$, first
use spatial difference quotients. Translations preserve $\mathcal G$, and
the translated equation has the same coercivity. Passing to weak limits
establishes the next derivative before it is used in the induction.
Angle derivatives commute with $G$.

For $w_n=\varrho^n/(n!)^2$, the normalized convolution coefficient is
\[
 \frac{w_n\binom nl P^{cl}(l!)^2}{w_{n-l}}
       =\frac{(\varrho P^c)^l}{\binom nl}
       \le(\varrho P^c)^l.
\]
For the stated radius bound, $\varrho P^c\le\tfrac14P^{-c}\le1/4$, and
\[
 P^c\sum_{l\ge1}(\varrho P^c)^l
 =\frac{\varrho P^{2c}}{1-\varrho P^c}\le\frac13<\frac12.
\]
Multiplying
\eqref{eq:corr-pressure-recursion} by $w_n$ and summing then absorbs
the commutator terms, with a factor at most $2$ on the right.
This works with any finite derivative cutoff and
with constants independent of that cutoff. This proves
\eqref{eq:corr-pressure-weighted-bound}.
\hfill$\square$

Apply $\Pi$ to \eqref{eq:corr-e-equation}. Since $D$ commutes with the
Laplacian, viscosity has zero projection. Thus
\begin{equation}
 \mathfrak p^e=\mathfrak p^0+\mathfrak p^1,\qquad
 \mathfrak p^0=-\mathcal T(Z(e)+r^a),\qquad
 \mathfrak p^1=-\mathcal T(b\cdot\nabla e).
 \label{eq:corr-pressure-split}
\end{equation}
The split distinguishes the part with no derivative of $e$ from the
transport pressure. The latter needs one special weighted estimate.

\subsection{The weighted energy and its derivative accounting}
\label{subsec:corr-energy}

The correct energy uses $K=G^{-1}$ so that its highest pressure term cancels.
For a word $I$, put
\[
 E_I^2=\sum_{|\gamma|\le s}
     \int_{\R^3\times\mathbb T}
       (K\partial^\gamma\partial^Ie)\cdot
          \partial^\gamma\partial^Ie,
 \qquad E_n=\sum_{|I|=n}E_I.
\]
Then $P^{-c}|e|_{n,s}\le E_n\le P^c|e|_{n,s}$.
Here $J$ denotes a derivative cutoff, independently of the
truncation order $N$ in Section~\ref{sec:packet-expansion}. Define
\begin{equation}
 X_J=\sum_{n=0}^Jw_nE_n,\qquad
 Y_J=\sum_{n=0}^Jnw_nE_n,
 \qquad w_n=\frac{\varrho(t)^n}{(n!)^2}.
 \label{eq:corr-energies}
\end{equation}
$X_J$ measures the correction; $Y_J$ records the derivative cost that a
decreasing radius can absorb. The letter $X_J$ here is an energy, not a
particle map.

\refstepcounter{equation}
\label{lem:corr-product-pressure}
\textbf{Lemma \theequation\ (Preliminary product and pressure estimates).}
Assume the parent bounds and finite order estimates of
Section~\ref{subsec:corr-inputs}. Fix a time, a finite cutoff $J\ge0$,
and $0<\varrho\le\min\{\varrho_a,\tfrac14P^{-2c}\}$, with $c$
at least the exponent of Lemma~\ref{lem:corr-pressure-estimate}.
For $r=s-1,s$, weighted products satisfy
\[
 \sum_{n=0}^Jw_n|fg|_{n,r}
 \le C\left(\sum_{n=0}^Jw_n|f|_{n,r}\right)
       \left(\sum_{n=0}^Jw_n|g|_{n,r}\right).
\]
The same estimate holds with the first factor measured in
$W^{r,\infty}$ instead of $H^r$.
For a smooth correction $e$, define $b,Z(e)$ by \eqref{intro:eq:b-Z},
$\mathfrak p^0,\mathfrak p^1$ by \eqref{eq:corr-pressure-split}, and $X_J,Y_J$ by
\eqref{eq:corr-energies}. Then
\begin{align}
 \sum_{n=0}^Jw_n|\mathfrak p^0|_{n,s}
   &\le P^c(X_J+X_J^2+R_k),\label{eq:corr-p0-est}\\
 \sum_{n=0}^Jw_n|\mathfrak p^1|_{n,s-1}
   &\le P^c(B_0+X_J)X_J.\label{eq:corr-p1-low}
\end{align}
\begin{equation}
 \sum_{j=0}^{J-1}(j+1)w_{j+1}|\mathfrak p^1|_{j,s}
       \le P^c(B_0+X_J)Y_J.
 \label{eq:corr-p1-shift}
\end{equation}
The sum in \eqref{eq:corr-p1-shift} is zero when $J=0$.
All constants are independent of time, $k$, and $J$.

\emph{Proof.} Both $H^s$ and $H^{s-1}$ are algebras on the
four-dimensional product space. The weighted Leibniz coefficient is
\[
 \frac{w_n\binom nl}{w_lw_{n-l}}=\binom nl^{-1}\le1.
\]
Summing the Sobolev product estimate proves the asserted product bound,
including its coefficient version. In particular,
\[
 \sum_{n=0}^Jw_n|b|_{n,s}\le P^c(B_0+X_J),\qquad
 \sum_{n=0}^Jw_n|Z(e)|_{n,s}\le P^c(X_J+X_J^2).
\]
Applying Lemma~\ref{lem:corr-pressure-estimate} to
\eqref{eq:corr-pressure-split} proves
\eqref{eq:corr-p0-est}--\eqref{eq:corr-p1-low}.
For the latter, the derivative in $\nabla e$ is paid for by the base
$H^s$ norm of $e$, since pressure is measured in $H^{s-1}$.

To prove \eqref{eq:corr-p1-shift}, consider
the term with $l$ derivatives on $b$ and $j-l$ on $\nabla e$.
Its weight relative to $w_l|b|_{l,s}$ and
$(j-l+1)w_{j-l+1}|e|_{j-l+1,s}$ is
\[
 \frac{(j+1)w_{j+1}\binom jl}
      {w_l(j-l+1)w_{j-l+1}}
        =\binom{j+1}{l}^{-1}\le1.
\]
The largest derivative order on $e$ is $j+1\le J$. Applying
\eqref{eq:corr-pressure-recursion} with these shifted weights introduces
only the same absorbable geometric series in $\varrho P^c$.
This proves \eqref{eq:corr-p1-shift} without using derivatives beyond $J$.

\hfill$\square$

\refstepcounter{equation}
\label{lem:corr-energy-estimate}
\textbf{Lemma \theequation\ (Weighted energy inequality).}
Assume the parent bounds and finite order estimates of
Section~\ref{subsec:corr-inputs}. Let $e$ be a smooth solution of
\eqref{eq:corr-e-equation} with $0\le\nu\le1$, and let $\varrho(t)>0$
be differentiable, with
$\varrho(t)\le\min\{\varrho_a,\tfrac14P^{-2c}\}$ throughout the interval.
Here $c$ is a fixed exponent sufficiently large for the estimate below
and at least the exponent of Lemma~\ref{lem:corr-pressure-estimate}.
For every finite derivative cutoff $J$, the energies
\eqref{eq:corr-energies} satisfy
\begin{equation}
 {\begin{aligned}
 X_J'\le{}&P^c(X_J+X_J^2+R_k)\\
 &+\left[\frac{\varrho'}\varrho
       +P^c(\varrho^{-1}+P^c)(B_0+X_J)\right]Y_J.
 \end{aligned}}
 \label{eq:corr-main-energy}
\end{equation}
with constants independent of $J$ and $\nu$.

\emph{Proof.} We use the product and pressure estimates of
Lemma~\ref{lem:corr-product-pressure} to estimate the contributions to
$X_J'$ from the correction equation.

\textbf{Pressure cancellation.}
Apply $\partial^\gamma\partial^I$ to the correction equation and test
against $K\partial^\gamma\partial^Ie$. If
$h=\partial^\gamma\partial^Ie$, the highest pressure term is
\[
 \langle Kh,G\partial^\gamma\partial^I\mathfrak p^e\rangle
   =\langle h,\partial^\gamma\partial^I\mathfrak p^e\rangle=0.
\]
Indeed $K^T=K$, $KG=I$, $\operatorname{div}_Dh=0$, and the pressure
derivative belongs to $\mathcal G$. Only terms with a derivative on $G$ remain.

If at least one base derivative falls on $G$ and all $n$ additional
derivatives fall on pressure, at most $s-1$ base derivatives remain on
pressure. Use \eqref{eq:corr-p1-low}. If $l\ge1$ additional derivatives
fall on $G$, put $j=n-l$. The weight relative to the shifted pressure
sum is
\begin{equation}
 \frac{w_{j+l}\binom{j+l}{l}}{w_l(j+1)w_{j+1}}
       =\varrho^{-1}\frac{j+1}{\binom{j+l}{l}}\le\varrho^{-1}.
 \label{eq:corr-shift-ratio}
\end{equation}
The weighted derivatives of $G$ with $l\ge1$ sum to at most
$\varrho P^{2c}$. Combined with \eqref{eq:corr-p1-shift}, their contribution
is bounded by $P^{2c}(B_0+X_J)Y_J$. The $\mathfrak p^0$ contribution uses
\eqref{eq:corr-p0-est} and has no derivative loss.

\textbf{Transport.}
The term with no derivative on $b$ satisfies
\[
 \langle Kh,b\cdot\nabla h\rangle
       =-\frac12\int ((b\cdot\nabla K)h)\cdot h,
\]
because $\operatorname{div}b=0$. Base commutators are bounded by
$C\|\nabla b\|_{H^{s-1}}\|\nabla\partial^Ie\|_{H^{s-1}}$;
they use no extra derivative beyond $|e|_{n,s}$.
For $l\ge1$ additional derivatives on $b$, the remaining factor is
$|e|_{j+1,s}$, where $j=n-l$. The weight ratio
\eqref{eq:corr-shift-ratio} now bounds the total by
$P^c\varrho^{-1}(B_0+X_J)Y_J$.

\textbf{Lower-order terms, viscosity, and radius.}
The coefficient bounds give a contribution $P^cX_J$ from $K_t$.
Lemma~\ref{lem:corr-product-pressure} bounds $Z(e)$ by
$P^c(X_J+X_J^2)$ and the residual
by $P^cR_k$. For each tested field $h$, integration by parts and Young's
inequality give
\[
 \nu\langle Kh,\Delta h\rangle
 \le-\tfrac12\nu P^{-c}\|\nabla h\|_2^2+\nu P^c\|h\|_2^2.
\]
Thus viscosity contributes at most $P^cX_J$, uniformly for $\nu\le1$.
Finally $w_n'=(n\varrho'/\varrho)w_n$, so differentiation of the weights
contributes $(\varrho'/\varrho)Y_J$.

Combining these estimates proves \eqref{eq:corr-main-energy}.

At a zero of $E_I$, use $(E_I^2+\eta^2)^{1/2}$ first and pass
to $\eta\downarrow0$ at fixed cutoff. Equivalently the inequality can
be read in the integrated or upper-derivative sense. All its constants
are independent of $J$ and $\nu$.
\hfill$\square$

\subsection{The radius and bootstrap estimates}
\label{subsec:corr-bootstrap}

\refstepcounter{equation}
\label{lem:corr-bootstrap-estimate}
\textbf{Lemma \theequation\ (Positive radius and small correction).}
Let $c$ be the fixed exponent in \eqref{eq:corr-main-energy}. Fix
$A\ge2c$ and $0<c_1\le1/4$ such that
$\varrho_i:=c_1P^{-A}\le\varrho_a$.
Assume \eqref{eq:corr-input-bounds} and \eqref{eq:corr-frequency}, with
$Q\ge A+c$ and $k_0$ sufficiently large for these fixed choices.
Then there is a decreasing radius $\varrho(t)$, independent of the derivative
cutoff and viscosity, with $\varrho(t)\ge\varrho_i/2>0$ on $[0,S]$.
Every solution covered by Lemma~\ref{lem:corr-energy-estimate} with
$e(0)=0$, hence $X_J(0)=0$, satisfies
\[
 X_J(t)\le P^cS e^{2P^cS}R_k<\tfrac12\Delta_k
\]
on its existence interval in $[0,S]$, uniformly in $J$ and $\nu$.
Recall $\Delta_k=\exp(-k^\vartheta/2)$ from \eqref{eq:corr-frequency}.

\textbf{Remark.} There \emph{is} a Gronwall exponential in this proof, but it is
$e^{P^cS}$, paid for by the residual
$e^{-0.7k^\vartheta\log k}$. We do not require $e^{P^c}\ll k$.
The chosen frequency condition handles both this residual comparison and
the much slower loss of radius caused by the small transport field.

\emph{Proof.} The stated choices give
$\varrho_i\le\min\{\varrho_a,\tfrac14P^{-2c}\}$, so the pressure absorption condition of
Lemma~\ref{lem:corr-pressure-estimate} holds for every
$0<\varrho\le\varrho_i$, as required by
Lemmas~\ref{lem:corr-product-pressure} and~\ref{lem:corr-energy-estimate}.
Take
\begin{equation}
 \varrho(t)=\varrho_i-Lt,\qquad
 L=2P^c(B_0+\Delta_k),
 \label{eq:corr-radius-choice}
\end{equation}
Since $\varrho P^c\le\varrho_iP^c\le1$, we have $P^c\le\varrho^{-1}$, so
under the bootstrap $X_J\le\Delta_k$ the coefficient of $Y_J$ in
\eqref{eq:corr-main-energy} satisfies
\[
 \frac{\varrho'}{\varrho}
 +P^c(\varrho^{-1}+P^c)(B_0+X_J)
 \le\frac{-L+2P^c(B_0+\Delta_k)}{\varrho}\le0.
\]

We must also check that this radius survives until $S$. The condition
\eqref{eq:corr-frequency} implies $P^a\le k^{a\vartheta/Q}$
for every fixed exponent $a$. Since $Q\ge A+c$, we have
$P^{A+c}\le k^\vartheta$.
Using $S\le1$ and $B_0\le k^{-1/2}$ gives
\[
 \frac{LS}{\varrho_i}
 \le C P^{A+c}(k^{-1/2}+e^{-k^\vartheta/2})\longrightarrow0.
\]
Thus, after increasing $k_0$, $\varrho(t)\ge\varrho_i/2$ on $[0,S]$.
This check is essential: an estimate at a radius that reaches zero would
not give the claimed Gevrey control.

Now bootstrap $X_J\le\Delta_k<1$. With the same fixed exponent $c$,
\eqref{eq:corr-main-energy} reduces to
\[
 X_J'\le P^c(2X_J+R_k),\qquad X_J(0)=0.
\]
Integration yields
\begin{equation}
 X_J(t)\le P^cS e^{2P^cS}R_k<\tfrac12\Delta_k.
 \label{eq:corr-bootstrap-bound}
\end{equation}
For clarity, the last inequality follows by taking logarithms. Its left
bound has logarithm at most
\[
 c\log P+2P^c-0.7k^\vartheta\log k.
\]
Since $Q\ge A+c\ge3c$, we have $P^c\le k^{\vartheta/3}
\le k^{\vartheta/2}$.
This expression is then smaller than $-k^\vartheta/2-\log2$.
Continuity excludes a first time at which $X_J=\Delta_k$.
\hfill$\square$

\subsection{Proof of the exact-correction proposition}

\emph{Proof of Proposition~\ref{prop:exact-packet-correction}.}
Lemmas~\ref{lem:corr-pressure-estimate}--\ref{lem:corr-bootstrap-estimate}
supply estimates uniform in viscosity. To construct a solution,
define the bounded projection
\[
 \mathcal P(t)=I-G\mathcal T.
\]
Its range is the fixed space $\mathcal G^\perp$; it fixes every field in
that space and annihilates $Gp$ for $p\in\mathcal G$. Apply it to
the non-heat terms of \eqref{eq:corr-e-equation} to get
\begin{equation}
 e_t-\nu\Delta e
   =-\mathcal P(t)(b\cdot\nabla e+Z(e)+r^a),\qquad e(0)=0.
 \label{eq:corr-parabolic}
\end{equation}
There is no $\mathcal P_t$ term: we are projecting the equation, and
$e_t$ already lies in its fixed range.

For each fixed $q=s+J$, the right side is locally Lipschitz from $H^q$
to $H^{q-1}$. The heat estimate
\[
 \|e^{\nu t\Delta}f\|_{H^q}
    \le C\bigl(1+(\nu t)^{-1/2}\bigr)\|f\|_{H^{q-1}}
\]
has an integrable time singularity. In the Duhamel formula it gives a
contraction on a small ball in $C([0,t_1];H^q)$. This constructs a local
solution, extendible while its $H^q$ norm remains bounded. The fixed
divergence constraint is preserved by both the heat semigroup and the
right side of \eqref{eq:corr-parabolic}.

The energy computation is valid for these solutions. One justification
is heat maximal regularity: on a compact existence interval it gives
$e\in L^2_tH^{q+1}$ and $e_t\in L^2_tH^{q-1}$. After differentiating
through order $q$, the equation can therefore be tested by the
$H^{-1}$--$H^1$ pairing. Multiplication by $K$ preserves $H^1$.
This justifies the same integrations by parts as for smooth functions;
the resulting bound is independent of $\nu$.

For every finite $J$, \eqref{eq:corr-bootstrap-bound} bounds $H^{s+J}$
on $[0,S]$ and hence prevents loss of existence there. Uniqueness at
the base index identifies the solutions obtained for different $J$.
We obtain smooth viscous solutions $e^\nu$ with the uniform estimates
above. For every fixed integer $j\ge0$, their equations and the pressure
inverse imply
\[
 \sup_{0<\nu\le1}
 \left(\|e^\nu\|_{L^\infty_tH^{j+2}}
       +\|e_t^\nu\|_{L^\infty_tH^j}
       +\|\mathfrak p^{e,\nu}\|_{L^\infty_tH^j}\right)<\infty.
\]
These bounds may depend on $j$ and the fixed data of
Subsection~\ref{subsec:corr-inputs}, including the parent bounds
\eqref{eq:corr-parent-bounds} and the finite order estimates
\eqref{eq:corr-input-bounds}, but are uniform in $\nu$, which suffices
for compactness. The two extra spatial derivatives bound
$\nu\Delta e^\nu$.

Local Sobolev compactness and time equicontinuity now give, along a
diagonal sequence $\nu\downarrow0$,
\[
 e^\nu\to e\quad\hbox{in }C([0,S];H^j_{\rm loc})
 \quad\hbox{for every finite }j.
\]
Pressure has a weakly convergent subsequence in each corresponding
global Sobolev space. Also $\nu\Delta e^\nu\to0$ in $L^\infty_tH^j$,
using the uniform bound two orders higher. Local strong convergence
passes the nonlinear terms to the limit; the weak pressure limit remains
in $\mathcal G$. Thus \eqref{eq:corr-e-equation} holds with $\nu=0$,
and $e(0)=0$.

At every fixed time, weak lower semicontinuity preserves the finite sums
$X_J$. Monotone convergence in $J$ then proves the full weighted bound
at $\varrho_{\min}=\varrho_i/2$. To estimate $e_t$ and $\mathfrak p^e$ we need one additional
derivative of $e$. Shrink once more to $\varrho_*=\varrho_{\min}/2$; the cost is
\begin{equation}
 \sup_{n\ge0}
 \frac{\varrho_*^n((n+1)!)^2}{\varrho_{\min}^{n+1}(n!)^2}
   =\varrho_{\min}^{-1}\sup_{n\ge0}(n+1)^2 2^{-n}
   \le P^c.
 \label{eq:corr-radius-derivative}
\end{equation}
The pressure formula and the inviscid correction equation now prove
\eqref{eq:corr-conclusion}. They also give continuity in each spatial
Sobolev norm. Differentiating these equations in time, including
$\Pi Gp=\Pi f$ for pressure, gives smoothness in time by induction,
since the parent and approximate fields are smooth.

If the parent and approximate velocity have the prescribed odd parity
under $(y,\theta)\mapsto(-y,-\theta)$, the viscous equation has the same
symmetry. Uniqueness gives odd $e^\nu$, and oddness passes to $e$.
This is the symmetry that keeps the origin stationary
for the following layer; see Subsection~\ref{subsec:parity-origin}.

\subsection{The physical solution and initial data}
\label{subsec:corr-physical}

The smooth vector pressure $\mathfrak p^e\in\mathcal G$ supplies a
scalar $q^e$ through \eqref{intro:eq:scalar-pressure-recovery}:
\[
 \nabla_yq^e=k^{-1}(\mathfrak p^e)^\diag.
\]
The fields \eqref{intro:eq:physical-child} therefore solve Euler:
adding the approximate equation to the inviscid correction equation
cancels the residual, and \eqref{intro:eq:fixed-div-identity} gives
incompressibility. This holds at every physical point.

Passing from Sobolev estimates on the product space to global Sobolev
estimates on the graph is elementary:
one-dimensional Sobolev embedding in $\theta$ gives
\[
 \int_{\R^3}|f(y,km_0\cdot y)|^2\,dy
   \le C\int_{\R^3}\|f(y,\cdot)\|_{H^1(\mathbb T)}^2\,dy.
\]
For $r$ spatial derivatives, the chain rule gives a further finite factor
$C_rk^r$ and requires derivatives of $f$ through order $r+1$ before
restriction. The Gevrey bounds provide all these derivatives. The
bounded derivatives of $F^{\pm1}$ and the volume-preserving physical
coordinate change transfer the estimates to physical space.

Finally $e(0)=0$, so
\begin{equation}
 u^{\rm new}(0,\ell y)-u(0,\ell y)
                  =\ell(W^a)^{\rm diag}(0,y).
 \label{eq:corr-initial-preserved}
\end{equation}
The exact velocity increment $u^{\rm new}-u$ therefore inherits the initial
support and initial Sobolev estimates of the approximate increment
$\ell(W^a)^{\rm diag}$. The pressure
and the positive-time velocity correction need not be spatially localized.
This completes the proof of Proposition~\ref{prop:exact-packet-correction}.
\hfill$\square$

\subsection{The leading gradient and pressure Hessian}

\refstepcounter{equation}
\label{cor:corr-physical-expansions}
\textbf{Corollary \theequation\ (Leading physical derivatives).}
Assume Proposition~\ref{prop:exact-packet-correction}, with the finite order
approximation of Section~\ref{sec:packet-expansion}, with the fixed constants
$Q,k_0$ enlarged if necessary. The physical fields
\eqref{intro:eq:physical-child} satisfy, uniformly in space and on $[0,S]$,
\begin{equation}
 \begin{aligned}
 \nabla u^{\rm new}-\nabla u
   &=\alpha\chi_1v\otimes m\,f_\delta'(\theta)+O(k^{-1/2}),\\
 \nabla^2p^{\rm new}-\nabla^2p
   &=-2\alpha\chi_1(m\cdot Mv)
       \frac{mm^T}{|m|^2}f_\delta'(\theta)+O(k^{-1/2}).
 \end{aligned}
 \label{eq:corr-final-expansions}
\end{equation}
where $x=X(t,\ell y)$ and $\theta=km_0\cdot y$.

\emph{Proof.} We can now quantify the contribution of the exact correction to the two
main physical errors. At a smaller fixed radius, every fixed number of
ordinary derivatives costs only $P^c$, by
\eqref{eq:corr-radius-derivative}. One physical derivative of the velocity
correction $\ell k^{-1} Fe$ gives
\[
 \nabla_x(\ell k^{-1} Fe)^{\rm diag}
       =\bigl(\tilde d(k^{-1} Fe)\bigr)^{\rm diag}.
\]
The largest phase factor is $k$, canceled by $k^{-1}$; derivatives of
$F$ are controlled coefficients. Hence this gradient is $O(P^c\Delta_k)$
globally and throughout $[0,S]$.

For pressure, write $p^e(t,X(t,\ell y))=\ell^2q^e(t,y)$. Then
\eqref{intro:eq:scalar-pressure-recovery} gives
\[
 (\nabla_xp^e)(t,X(t,\ell y))
 =\ell F^{-T}\nabla_yq^e
 =\ell k^{-1}F^{-T}(\mathfrak p^e)^{\rm diag}.
\]
Applying $\nabla_x$ once more
again cancels the factor $\ell$ and costs at most $kP^c$. Thus the
pressure Hessian correction is also $O(P^c\Delta_k)$.
Under \eqref{eq:corr-frequency}, both are smaller than $k^{-1/2}$
for sufficiently large $k$.

To recover the \emph{complete} leading expansions, also use
\eqref{eq:packet-leading-remainders}: its higher velocity terms begin at $k^{-2}$,
and its higher oscillatory pressure terms at $k^{-3}$. After one and
two graph derivatives, respectively, their errors are $O(k^{-1}P^c)$.
The leading phase derivatives are already computed in
\eqref{eq:gradient} and \eqref{eq:leading-pressure-hessian}; they give
the main terms on the right side of \eqref{eq:corr-final-expansions}.
The remainder from the leading velocity has no phase derivative, so its
contribution to the velocity gradient is $O(k^{-1}P^c)$.
In the remaining terms from the leading oscillatory pressure, at most one
derivative falls on the phase, so their contribution to the pressure
Hessian is also $O(k^{-1}P^c)$.
The mean pressure terms are estimated through
their spatial gradients; they have no phase derivative and begin at
order $k^{-2}$, so their contribution to the pressure Hessian is
$O(k^{-2}P^c)$.

Combining all these contributions, the total remainder in each of the
velocity-gradient and pressure-Hessian expansions is
$O(k^{-1}P^c+P^c\Delta_k)$, for a fixed exponent $c$.
Choose $Q\ge2c\vartheta$ in \eqref{eq:corr-frequency}. Then
\[
 k^{-1}P^c\le k^{-1+c\vartheta/Q}\le k^{-1/2},\qquad
 P^c\Delta_k\le k^{c\vartheta/Q}e^{-k^\vartheta/2}
 \le k^{-1/2}
\]
for sufficiently large $k_0$. This proves the stated $O(k^{-1/2})$
remainders. These are the global error estimates needed to preserve
the next shear geometry and the upper pressure-Hessian bound.
\hfill$\square$

\textbf{Remark.}
This proves the exact-correction step under the stated inputs. The
quantitative child particle-map bound used by the next stage is an
additional result, proved in Section~\ref{sec:child-map} by the
lifted-flow composition argument; mere smoothness
of $u^{\rm new}$ is not a substitute for it. Section~\ref{sec:packet-expansion} supplies the small lifted transport,
polynomial coefficient bounds and exponentially small residual under its
parent inputs. Those inputs, the amplification and frame-transfer hypotheses,
and the common choice of scales must still be verified at every stage.

\section{The child particle map}
\label{sec:child-map}
\fancyhead[R]{\small The child particle map}

Section~\ref{sec:exact-correction} produces an exact child velocity
$u^{\rm new}$. To use it as the parent of another layer, we also need
quantitative bounds for its particle map. Directly exponentiating a
physical Lipschitz bound would obscure the required dependence on $k$.
Instead, the leading wave has small label transport and zero angle
transport by transversality. We estimate its flow on the lifted
$(y,\theta)$-space $\R^3\times\mathbb T$, then restrict to the physical phase graph and
compose with the parent flow, as derived in
Subsection~\ref{intro:subsec:physical-child}. The new task is to quantify
that construction in a Gevrey-2 norm.

\subsection{The additional input and the child-map bound}

Retain $s=6$, $0<S\le1$, and all hypotheses and conclusions of
Proposition~\ref{prop:exact-packet-correction}. In addition assume
$0<\ell\le1$, $\ell^{-1}\le P^{c_0}$, and the following bound in the
physical initial label $a$:
\begin{equation}\label{eq:child-parent-map}
 \sum_{|I|=n}\sup_{0\le t\le S}
 \|\partial_a^I(X-\mathrm{id},X_t,X_{tt})(t)\|_{H^s(\R^3)}
 \le K_h^{n+1}(n!)^2,\qquad 1\le K_h\le P^{c_0},\quad n\ge0.
\end{equation}
The norm of a tuple means the sum of its component norms. This is an
additional induction input, not a consequence of the correction theorem.
All constants below are independent of the stage and $k$. The fixed
exponents may depend on the input budget, including $c_0$.

Use the child particle map $X^{\rm new}=X\circ\Psi$ from
\eqref{intro:eq:relative-flow}; its flow equation and initial condition
were verified in Subsection~\ref{intro:subsec:physical-child}.

\refstepcounter{equation}\label{prop:child-map}
\textbf{Proposition \theequation\ (Gevrey bounds for the child particle map).}
Under these assumptions, after a fixed enlargement of $Q$ and $k_0$ in
\eqref{eq:corr-frequency}, the particle map $X^{\rm new}$ of
$u^{\rm new}$ satisfies
\begin{equation}\label{eq:child-map-output}
 \sum_{|I|=n}\sup_{0\le t\le S}
 \|\partial_a^I(X^{\rm new}-\mathrm{id},X_t^{\rm new},X_{tt}^{\rm new})(t)\|_{H^s}
 \le (k^{C_*})^{n+1}(n!)^2,\qquad C_*=10(s+2),\quad n\ge0.
\end{equation}
The map is a smooth volume-preserving diffeomorphism for every
$t\in[0,S]$. In particular, the bound covers the child's whole retained interval $[0,S]$,
including its history, rather than only the activation-to-target interval.

\textbf{Remark.} Spatial derivatives of $X^{\rm new}$ and the inverse
matrix $(\nabla_aX^{\rm new})^{-1}$ consequently have Gevrey-2 bounds
with size and derivative-growth constants bounded by fixed powers of $k$.
For the inverse matrix, use $\det\nabla_aX^{\rm new}=1$ and the cofactor
formula, so only products of first derivatives of the map are needed.
This is a statement about the inverse deformation matrix at label $a$.

\subsection{Composition and a small lifted flow}

For this subsection let
$|f|_{n,\infty}=\sum_{|I|=n}\|\partial^If\|_{L^\infty}$ and
$|f|_{n,2}=\sum_{|I|=n}\|\partial^If\|_{L^2}$, with component norms
summed. These are derivative word sums, not the $H^r$-based notation
of Section~\ref{subsec:corr-inputs}. On a product with $\mathbb T$,
map displacements are defined using the angular lift to $\R$ and are
periodic in that variable.

\refstepcounter{equation}\label{lem:child-composition}
\textbf{Lemma \theequation\ (Composition at a fixed Gevrey scale).}
Let $g=\mathrm{id}+h$ be a smooth map in fixed dimension, with
\[
 |h|_{n,\infty}\le A_iR_i^n(n!)^2,\qquad
 |f|_{n,\infty}\le A_oR_o^n(n!)^2\quad(n\ge0),
 \qquad R_i,R_o\ge1.
\]
There is a dimensional constant $d$ such that, if
\begin{equation}\label{eq:child-composition-radius}
 2R_i\zeta\le1,\qquad R_o(d+2A_iR_i)\zeta\le\tfrac12,
\end{equation}
then
\[
 |f\circ g|_{n,\infty}\le C A_o\zeta^{-n}(n!)^2.
\]
If $g$ is a volume-preserving diffeomorphism and the bounds on $f$ hold
also in $L^2$, the conclusion holds in $L^2$ as well.

\emph{Proof.} In the ordered chain rule, group $n$ derivatives into
$q$ nonempty blocks of sizes $j_1,\ldots,j_q$. Each block groups the
differentiations that act on one factor $g$: the resulting term has
$q$ derivatives of $f$, evaluated at $g$, applied to derivatives of $g$
of orders $j_1,\ldots,j_q$, with $j_1+\cdots+j_q=n$.
Summing over ordered block sizes uses the coefficient $n!/(q!\prod_i j_i!)$. After dividing
by $(n!)^2$ and inserting the outer Gevrey factor $(q!)^2$, the
coefficient for the normalized inner derivatives is
\begin{equation}\label{eq:child-composition-factor}
 \frac{q!\prod_{i=1}^q j_i!}{n!}\le1,
 \qquad j_i\ge1,\quad \sum_i j_i=n.
\end{equation}
It equals one when every $j_i=1$. Increasing any $j_i$ multiplies the
ratio by $(j_i+1)/(n+1)\le1$, which proves the inequality.
Component contractions are bounded by the component and word sums.
If $U_g(\zeta)=\sum_{j\ge1}|g|_{j,\infty}\zeta^j/(j!)^2$, the
chain rule is therefore bounded coefficientwise by
$A_o\sum_{q\ge1}(R_oU_g)^q$. The identity contributes at most $d\zeta$,
and
\[
 U_g(\zeta)\le d\zeta+
 A_i\frac{R_i\zeta}{1-R_i\zeta}
 \le(d+2A_iR_i)\zeta.
\]
Condition \eqref{eq:child-composition-radius} bounds the geometric
series. The order-zero bound is direct. For $L^2$, put the outer
factor $(\partial^q f)\circ g$ in $L^2$ and every derivative of $g$
in $L^\infty$. Volume preservation gives
$\|(\partial^q f)\circ g\|_2=\|\partial^q f\|_2$, so the same
calculation applies. All estimates can first be truncated at a finite
order, avoiding any assumption that a derivative series converges.
\hfill$\square$

\refstepcounter{equation}\label{lem:child-small-flow}
\textbf{Lemma \theequation\ (Flow of a small Gevrey transport field).}
Let $b$ be a smooth divergence-free field on $\R^3\times\mathbb T$,
with, uniformly on $[0,S]$,
\[
 |b|_{n,\infty}+|b|_{n,2}\le B R_b^n(n!)^2,
 \qquad R_b\ge1,\qquad B R_b S\le\tfrac14.
\]
Define its flow by
\[
 \partial_t\Phi(t,y_0,\theta_0)
   =b\bigl(t,\Phi(t,y_0,\theta_0)\bigr),\qquad
 \Phi(0,y_0,\theta_0)=(y_0,\theta_0).
\]
Then $\Phi(t)$ is a volume-preserving diffeomorphism, and for a dimensional $C$,
\begin{equation}\label{eq:child-lifted-flow-bound}
 |\Phi-\mathrm{id}|_{n,\infty}+|\Phi-\mathrm{id}|_{n,2}
 \le C BS(CR_b)^n(n!)^2\quad(n\ge0).
\end{equation}
If also $|b_t|_{n,\infty}+|b_t|_{n,2}\le A_tR_t^n(n!)^2$,
then $\Phi_t,\Phi_{tt}$ have bounds of the same Gevrey form, with
size and growth constants polynomial in $1+B+R_b+A_t+R_t$.

\emph{Proof.} Boundedness and the Lipschitz bound for $b$ give a global
flow on $[0,S]$; solving backward gives its inverse. The Jacobian
satisfies $\partial_t\det D\Phi=(\operatorname{div}b)\circ\Phi\det D\Phi=0$.
Fix a derivative cutoff $J\ge1$ and put
\[
 U_J(t,\zeta)=d\zeta+
 \sum_{n=1}^J\frac{|\Phi(t)-\mathrm{id}|_{n,\infty}}{(n!)^2}\zeta^n,
 \qquad \zeta=(4dR_b)^{-1}.
\]
Using \eqref{eq:child-composition-factor} in
$\Phi-\mathrm{id}=\int_0^t b(t',\Phi(t'))\,dt'$ gives
\begin{equation}\label{eq:child-flow-series}
 U_J(t,\zeta)\le d\zeta+
 \int_0^t B\frac{R_bU_J(t',\zeta)}{1-R_bU_J(t',\zeta)}\,dt'
\end{equation}
while $R_bU_J<1$. The inequality follows by retaining terms only
through degree $J$, then bounding by the nonnegative full geometric
series. While $R_bU_J\le1/2$, Gronwall applied to
\eqref{eq:child-flow-series} gives
$R_bU_J\le\tfrac14e^{2BR_bt}\le\tfrac14e^{1/2}<1/2$.
Continuity closes this bootstrap on $[0,S]$.
Returning to the integral equation bounds its displacement series by
$BS$. This proves \eqref{eq:child-lifted-flow-bound} in $L^\infty$.
For $L^2$, put the outer derivatives of $b$ in $L^2$ and use volume
preservation, exactly as in Lemma~\ref{lem:child-composition}. The same
series is bounded by $BS$. At order zero, direct integration gives
both bounds. Constants are independent of $J$, so every order is covered.

Finally,
\begin{equation}\label{eq:child-lifted-time}
 \Phi_t=b\circ\Phi,\qquad
 \Phi_{tt}=(b_t+(b\cdot\nabla)b)\circ\Phi.
\end{equation}
One derivative is absorbed by replacing a Gevrey growth constant $R$
by $2R$, since $(n+1)^2\le C2^n$. The product estimate of
Lemma~\ref{lem:corr-product-pressure}, with $L^\infty$ factors and one
$L^2$ factor when needed, then bounds $(b\cdot\nabla)b$.
Lemma~\ref{lem:child-composition} applies to both compositions in
\eqref{eq:child-lifted-time}. Its admissible $\zeta^{-1}$ is polynomial
in the stated constants, which proves the last assertion.
\hfill$\square$

\subsection{The transport field and the physical phase graph}

Use $z=z^a+e$ and the divergence-free transport field
$b=(k^{-1}z,m_0\cdot z)$ from \eqref{intro:eq:b-Z}.
The leading contribution to its angle component vanishes because
$m_0\cdot F^{-1}A_1=m\cdot A_1=0$. Equations
\eqref{eq:packet-small-transport} and \eqref{eq:corr-conclusion}, together
with Sobolev embedding, give at the common radius $\varrho_*$ of
\eqref{eq:corr-conclusion}, with no further radius loss,
\begin{equation}\label{eq:child-transport-bounds}
 \begin{aligned}
 |b|_{n,\infty}+|b|_{n,2}&\le B R_b^n(n!)^2,
 &B&\le Ck^{-1/2},& R_b&=\varrho_*^{-1}=P^c,\\
 |b_t|_{n,\infty}+|b_t|_{n,2}&\le P^c R_b^n(n!)^2.
 \end{aligned}
\end{equation}
The second line uses the bounds for $z_t^a$ and $e_t$ in
\eqref{eq:corr-input-bounds} and \eqref{eq:corr-conclusion}.
The exponentially small $P^c\Delta_k$ is absorbed into the first line.
Enlarging $Q,k_0$ in \eqref{eq:corr-frequency} makes $BR_bS\le1/4$.
Lemma~\ref{lem:child-small-flow} now gives bounds of size and growth
constant $P^c$ for $\Phi-\mathrm{id},\Phi_t,\Phi_{tt}$ in both norms,
uniformly on $[0,S]$.

The graph restriction and rescaling in \eqref{intro:eq:relative-flow}
turn this lifted flow into the relative physical flow $\Psi$.
The following lemma quantifies their derivative cost.

\refstepcounter{equation}\label{lem:child-graph-map}
\textbf{Lemma \theequation\ (Relative flow on the physical graph).}
Let $\Phi$ be the flow \eqref{intro:eq:lifted-flow} of the field
just estimated, and define $\Psi$ by \eqref{intro:eq:relative-flow}.
Then $\Psi$ is a volume-preserving diffeomorphism of $\R^3$,
$\Psi(0,a)=a$, and, for every $n\ge0$,
\begin{equation}\label{eq:child-relative-bounds}
 \sum_{|I|=n}\sup_{0\le t\le S}
 \|\partial_a^I(\Psi-\mathrm{id},\Psi_t,\Psi_{tt})\|_{L^2\cap L^\infty}
 \le P^c(kP^c)^n(n!)^2.
\end{equation}
Here the intersection norm is the sum of the two norms.
The child particle map is $X^{\rm new}=X\circ\Psi$, as in
\eqref{intro:eq:relative-flow}.

\emph{Proof.}
Subsection~\ref{intro:subsec:physical-child} proves that $\Phi$ and its
inverse preserve $\theta=km_0\cdot y$, that the restricted label
velocity is divergence free, and that $X\circ\Psi$ solves the child
flow equation. Thus $\Psi$ is a volume-preserving diffeomorphism
with $\Psi(0)=\mathrm{id}$.

To prove \eqref{eq:child-relative-bounds}, apply the graph trace bound
in Section~\ref{subsec:corr-physical} to the first three components of
$\Phi-\mathrm{id},\Phi_t,\Phi_{tt}$. Each graph derivative costs at
most $Ck$, and the $L^2$ trace uses one further angle derivative.
For any of these fields $f$, the rescaling $a=\ell y$ therefore gives
\[
 \begin{aligned}
 |\partial_a^n[\ell f(a/\ell,km_0\cdot a/\ell)]|_{\infty}
 &\le \ell^{1-n}(Ck)^n P^c(P^c)^n(n!)^2,\\
 |\partial_a^n[\ell f(a/\ell,km_0\cdot a/\ell)]|_{2}
 &\le \ell^{5/2-n}(Ck)^n P^c(P^c)^{n+1}((n+1)!)^2.
 \end{aligned}
\]
Here $\partial_a^n$ denotes the word sum. Use $\ell\le1$,
$\ell^{-1}\le P^{c_0}$ and absorb $(n+1)^2$ into a fixed enlargement
of the growth constant. To place the time supremum inside this sum,
note that there are $3^n$ derivative words in the physical labels $a$.
For any field $u$ and either spatial norm $\mathcal X=L^2,L^\infty$,
\begin{equation}\label{eq:child-wordwise-time}
 \sum_{|I|=n}\sup_{0\le t\le S}\|\partial_a^Iu(t)\|_{\mathcal X}
 \le 3^n\sup_{0\le t\le S}
          \sum_{|I|=n}\|\partial_a^Iu(t)\|_{\mathcal X}.
\end{equation}
Indeed, each summand on the left is bounded by the supremum on the
right. The factor $3^n$ is absorbed by multiplying the derivative-growth
constant by three. The same argument applies to tuples and to the sum
of the two norms, proving \eqref{eq:child-relative-bounds}.
\hfill$\square$

\subsection{Transfer of the parent bounds}

\emph{Proof of Proposition~\ref{prop:child-map}.}
The preceding lemmas construct the map on the full interval $[0,S]$.
For its displacement and time derivatives, use
\begin{equation}\label{eq:child-time-composition}
 \begin{aligned}
 X^{\rm new}-\mathrm{id}
   &=(\Psi-\mathrm{id})+(X-\mathrm{id})\circ\Psi,\\
 X_t^{\rm new}&=X_t\circ\Psi+(\nabla_aX\circ\Psi)\Psi_t,\\
 X_{tt}^{\rm new}&=X_{tt}\circ\Psi
   +2(\nabla_aX_t\circ\Psi)\Psi_t
   +(\nabla_a^2X\circ\Psi)[\Psi_t,\Psi_t]
   +(\nabla_aX\circ\Psi)\Psi_{tt}.
 \end{aligned}
\end{equation}
In the factors containing $\nabla_aX$, write
$\nabla_aX=I+\nabla_a(X-\mathrm{id})$; the constant matrix is not
placed in $L^2$. The remaining outer functions are derivatives of
$X-\mathrm{id},X_t,X_{tt}$ of order at most two. Hypothesis
\eqref{eq:child-parent-map} and Sobolev embedding bound them in
$L^2$ and $L^\infty$ with size and growth constants $P^c$.
Absorbing these two fixed derivative shifts only enlarges those constants.

Apply Lemma~\ref{lem:child-composition} with $g=\Psi$,
$A_i\le P^c$, $R_i\le kP^c$, and $A_o,R_o\le P^c$.
Both requirements in \eqref{eq:child-composition-radius} hold with
$\zeta^{-1}=kP^c$ after fixing a sufficiently large exponent.
The $L^2$ conclusion applies because $\Psi$ preserves volume.
Products in \eqref{eq:child-time-composition} use one $L^2$ factor
and the others in $L^\infty$; the weighted product calculation from
Lemma~\ref{lem:corr-product-pressure} preserves a common enlarged
Gevrey scale. These estimates first bound the derivative-word sum
uniformly in time. Applying \eqref{eq:child-wordwise-time} and absorbing
its factor $3^n$ into the growth constant gives
\[
 \sum_{|I|=n}\sup_t
 \|\partial_a^I(X^{\rm new}-\mathrm{id},X_t^{\rm new},X_{tt}^{\rm new})\|_{L^2\cap L^\infty}
 \le P^c(kP^c)^n(n!)^2.
\]
Passing to $H^s$ uses at most $s$ additional derivatives. Since
$((n+s)!/n!)^2\le C_s^{n+1}$ for fixed $s$, the result is bounded by
\[
 C_s^{n+1}P^c(kP^c)^{n+s}(n!)^2.
\]
Choose $Q$ large enough that every fixed $P$-power here is at most $k$,
and increase $k_0$ so that $C_s\le k$. This is possible under
$P^Q\le k^\vartheta$. The last bound is then at most
$k^{3n+2s+2}(n!)^2$, which is bounded by
$(k^{10(s+2)})^{n+1}(n!)^2$. This proves
\eqref{eq:child-map-output}. Finally, both $X$ and $\Psi$ preserve
volume and are diffeomorphisms, so their composition has the same
properties. No exponential in a positive power of $k$ enters the bound.
\hfill$\square$

The exact child solution now has the particle-map estimate needed by the next
stage, conditional on \eqref{eq:child-parent-map} and the inputs of
Subsection~\ref{subsec:corr-inputs}. The full initial support and Sobolev estimate, including all terms of
the finite approximation, is proved next in Section~\ref{sec:initial-packet}.

\section{Support and Sobolev bounds at time zero}
\label{sec:initial-packet}
\fancyhead[R]{\small Bounds at time zero}

The exact correction $e$ satisfies $e(0)=0$ by
\eqref{eq:corr-e-equation}, so its construction does
not alter the new layer's initial data. We now record its support and Sobolev
bounds, retaining the factor $\alpha$ in the oscillatory part.
Equation \eqref{eq:corr-initial-preserved} reduces this task to the
finite approximation \eqref{intro:eq:finite-packet}. The bounds will
be used to sum the initial increments in the stage construction;
that summability still requires the stage parameters.

\subsection{The initial-data proposition}

Assume the hypotheses of Proposition~\ref{prop:exact-packet-correction},
with the approximation of Section~\ref{sec:packet-expansion},
$0<\ell\le1$, and the frequency condition \eqref{eq:corr-frequency}.
All constants have the fixed-input dependencies of those sections.
No additional parent particle-map bound is needed here, because
$X(0,a)=a$ and $F(0,y)=I$.
Set $y_a=a/\ell$, $\theta_a=km_0\cdot a/\ell$, and define
\begin{equation}\label{eq:initial-split}
 \begin{aligned}
 u_{\mathrm{osc},0}(a)
 &=\ell\sum_{p=1}^N
       \bigl(k^{-p}A_p+k^{-(p+1)}C_p\bigr)(0,y_a,\theta_a),\\
 u_{\mathrm{mean},0}(a)
 &=\ell\sum_{p=2}^N k^{-p}B_p(0,y_a).
 \end{aligned}
\end{equation}
The symbols denote two parts of the physical initial velocity increment,
not two separate Euler solutions.

\refstepcounter{equation}\label{prop:initial-packet}
\textbf{Proposition \theequation\ (Initial support and derivative bounds).}
After a fixed enlargement of $Q,k_0$, independent of the Sobolev order
below, the exact child satisfies
\begin{equation}\label{eq:initial-packet-support}
 \begin{gathered}
 u^{\rm new}(0)-u(0)=u_{\mathrm{osc},0}+u_{\mathrm{mean},0},\\
 \operatorname{div}_a u_{\mathrm{osc},0}
 =\operatorname{div}_a u_{\mathrm{mean},0}=0,\\
 \operatorname{supp}u_{\mathrm{osc},0}\subseteq\{|a|\le\ell/2\},
 \qquad
 \operatorname{supp}u_{\mathrm{mean},0}\subseteq\{|a|\le2\}.
 \end{gathered}
\end{equation}
For every fixed integer $r\ge0$, there is an exponent $c_r$, independent
of the stage and frequency, such that
\begin{equation}\label{eq:initial-packet-sobolev}
 \|u_{\mathrm{osc},0}\|_{H^r(\R^3)}
       \le\ell^{-r}\alpha k^rP^{c_r},\qquad
 \|u_{\mathrm{mean},0}\|_{H^r(\R^3)}
       \le\ell^{-r}k^{-2}P^{c_r}.
\end{equation}
If $t_a=0$ and $\lambda_b=0$ in the mean boundary condition
\eqref{eq:mean-inverse-conclusions}, then $u_{\mathrm{mean},0}=0$ and
\begin{equation}\label{eq:initial-curl-formula}
 u^{\rm new}(0,a)-u(0,a)
 =-\frac{\ell^2\alpha}{k^2}\operatorname{curl}_a
 \left[\chi_1(a/\ell)\bigl(m_0\times v(0,a/\ell)\bigr)
                  \varphi_\delta(km_0\cdot a/\ell)\right],
 \qquad \varphi_\delta=\partial_\theta^{-1}f_\delta.
\end{equation}

\textbf{Remark.} The same $Q,k_0$ work for every fixed $r$; only $c_r$
changes. No smallness of $\alpha$ is assumed in this proposition.
The factor $\alpha$ must remain visible for the later initial-data
summability argument.

\subsection{Sum the initial coefficients before taking the physical trace}

Let $\|\cdot\|_{\varrho,s}$ be the product-space Gevrey norm of
Section~\ref{subsec:packet-truncation}. For an angle-independent field use the corresponding
$H^s(\R^3_y)$ norm. Define the unscaled initial profiles
\[
 \mathcal W_{\mathrm{osc},0}
 =\sum_{p=1}^N k^{-p}(A_p+k^{-1}C_p)(0),\qquad
 \mathcal W_{\mathrm{mean},0}
 =\sum_{p=2}^N k^{-p}B_p(0).
\]

\refstepcounter{equation}\label{lem:initial-profile-sum}
\textbf{Lemma \theequation\ (Initial profile sums with the amplitude retained).}
At a radius $\varrho=P^{-c}$ independent of $N$, the initial profiles obey
\begin{equation}\label{eq:initial-profile-sums}
 \|\mathcal W_{\mathrm{osc},0}\|_{\varrho,s}
       \le\alpha k^{-1}P^c,\qquad
 \|\mathcal W_{\mathrm{mean},0}\|_{\varrho,s}
       \le k^{-2}P^c.
\end{equation}
For every fixed integer $r\ge0$, consequently
\begin{equation}\label{eq:initial-profile-sobolev}
 \|\mathcal W_{\mathrm{osc},0}\|_{H^{r+1}(\R^3\times\mathbb T)}
       \le\alpha k^{-1}P^{c_r},\qquad
 \|\mathcal W_{\mathrm{mean},0}\|_{H^r(\R^3)}
       \le k^{-2}P^{c_r}.
\end{equation}

\emph{Proof.} By \eqref{eq:packet-amplitude-profile}, $\gamma(0)=\alpha$,
including when $t_a=0$. Thus the profiles in
\eqref{eq:packet-wave-shifts}--\eqref{eq:packet-mean-shifts} give
\[
 \begin{aligned}
 |(A_p,C_p)(0)|_{n,s}
 &\le\alpha H_0^{2p-2}R^{n+a_p}((n+a_p)!)^2,
 &&a_p=9p-8,\\
 |B_p(0)|_{n,s}
 &\le H_0^{2p-2}R^{n+b_p}((n+b_p)!)^2,
 &&b_p=9p-13\quad(p\ge2).
 \end{aligned}
\]
The tuple norm is the sum of its component norms; fixed prefactors can
be absorbed into the common scale. This retains $\alpha$ for all
oscillatory corrections, including those selected through a positive
history interval.

Apply the derivative-order separation and summation used to prove
\eqref{eq:packet-coefficient-majorant}, choosing $4R\varrho\le1/2$.
The derivative sum contributes at most $2$, uniformly in $p,N$;
the remaining factors are precisely the expansion-order factors
estimated there. Keeping the outside factor $\alpha$ in the oscillatory case, the same
fixed choices give
\[
 \alpha^{-1}\|(A_p,C_p)(0)\|_{\varrho,s}\le\mathcal B_k^{p+1},
 \qquad \|B_p(0)\|_{\varrho,s}\le\mathcal B_k^{p+1},
 \qquad \mathcal B_k\le k^{0.01}.
\]
At the fixed orders $p=1,2$ use instead the direct bounds $\alpha P^c$
for the oscillatory coefficients and $P^c$ for $B_2(0)$.
For the remaining orders,
\[
 \sum_{p=3}^N k^{-p}\mathcal B_k^{p+1}
 \le\sum_{p=3}^{\infty}k^{-0.99p+0.01}
 \le Ck^{-2.96}\le Ck^{-2}.
\]
The $C_p$ terms have another factor $k^{-1}$. These estimates prove
\eqref{eq:initial-profile-sums}, with the fixed numerical constants
absorbed into $P^c$.

For each fixed $r$, the finitely many derivative terms needed for
$H^{r+1}$ are bounded by the Gevrey norm times
$C_r\sum_{n=0}^{r+1}\varrho^{-n}(n!)^2\le P^{c_r}$, after enlarging
$c_r$ since $P\ge2$. This proves \eqref{eq:initial-profile-sobolev}.
Only a prefactor depends on $r$: the sum over expansion orders has
already been performed at a common radius. Thus $Q,k_0$ need not
depend on $r$.
\hfill$\square$

\subsection{Physical support, scaling, and the first-stage formula}

\emph{Proof of Proposition~\ref{prop:initial-packet}.}
Equation \eqref{eq:corr-initial-preserved}, the finite expansion
\eqref{intro:eq:finite-packet}, and $B_1=0$ give the exact splitting
\eqref{eq:initial-split}. Every $A_p,Q_p,C_p$ is supported in
$\operatorname{supp}\chi_1\subset\{|y|<1/2\}$, by the support argument
in Section~\ref{subsec:packet-construction}. Differentiating the curl potential
does not enlarge its support. The mean inverse gives
$\operatorname{supp}B_p(0)\subset\{|\ell y|\le2\}$ in
\eqref{eq:mean-inverse-conclusions}. Restriction to the graph and
$a=\ell y$ yield the two claimed physical supports.

Each oscillatory pair is a lifted curl by
\eqref{eq:forced-exact-divergence}, and each $B_p$ satisfies
$d\cdot B_p=0$. The physical chain rule therefore shows that both
initial increments are divergence free separately. The correction
$e(0)=0$ adds neither initial support nor initial derivatives.

For completeness, set $T_kf(y)=f(y,km_0\cdot y)$.
The chain rule and the one-dimensional Sobolev embedding in $\theta$
used in Section~\ref{subsec:corr-physical} give
\begin{equation}\label{eq:initial-trace-scaling}
 \|T_k f\|_{H^r(\R^3)}\le C_r k^r
                      \|f\|_{H^{r+1}(\R^3\times\mathbb T)},\qquad
 \|\ell h(\cdot/\ell)\|_{H^r(\R^3)}
       \le\ell^{5/2-r}\|h\|_{H^r(\R^3)}.
\end{equation}
The first inequality is the graph trace estimate from that subsection,
with its derivative cost made explicit. For a spatial multi-index
$\beta$, rescaling gives exactly
$\|\partial_a^\beta[\ell h(a/\ell)]\|_2
=\ell^{5/2-|\beta|}\|\partial_y^\beta h\|_2$.
Since $0<\ell\le1$, summing over $|\beta|\le r$ proves the second inequality.

Applying \eqref{eq:initial-profile-sobolev} now gives the intermediate bounds
\[
 \|u_{\mathrm{osc},0}\|_{H^r}
       \le\ell^{5/2-r}\alpha k^{r-1}P^{c_r},\qquad
 \|u_{\mathrm{mean},0}\|_{H^r}
       \le\ell^{5/2-r}k^{-2}P^{c_r}.
\]
The mean has no phase derivative and hence no factor $k^r$.
Using $\ell\le1$, $k\ge1$ yields the convenient bounds
\eqref{eq:initial-packet-sobolev}.

Finally suppose $t_a=0$ and $\lambda_b=0$. The prescribed initial data
for every forced transverse coefficient give $A_p(0)=0$ for $p\ge2$
by the construction in Section~\ref{subsec:packet-construction}; hence
$Q_p(0)=C_p(0)=0$. The mean boundary condition
\eqref{eq:mean-inverse-conclusions} gives $B_p(0)=0$ for every $p$.
Only the leading oscillation and its curl correction remain:
\[
 u^{\rm new}(0,a)-u(0,a)
 =\ell(k^{-1}A_1+k^{-2}C_1)(0,y_a,\theta_a)
 =\ell k^{-2}(\tilde d\times Q_1)(0,y_a,\theta_a).
\]
At time zero $F=I$, $m=m_0$, and
\[
 Q_1(0,y,\theta)
 =-\alpha\chi_1(y)(m_0\times v(0,y))\varphi_\delta(\theta).
\]
The physical graph chain rule is
$\operatorname{curl}_a[Q_1(0,y_a,\theta_a)]
=\ell^{-1}(\tilde d\times Q_1)(0,y_a,\theta_a)$.
Substitution proves \eqref{eq:initial-curl-formula}, also when the
prescribed initial amplitude depends on $y$.
\hfill$\square$

The three outputs of the new layer construction are now available as conditional results:
leading physical derivatives in Corollary~\ref{cor:corr-physical-expansions},
the child particle map in Proposition~\ref{prop:child-map}, and the
initial data in Proposition~\ref{prop:initial-packet}. Section~\ref{sec:complete-packet}
assembles them with a common input budget. The verification of that budget
for the seed flow and every subsequent stage remains part of the
parameter selection and induction in
Sections~\ref{sec:stage-parameters}--\ref{sec:stage-induction}.

\section{The complete conditional new-layer theorem}
\label{sec:complete-packet}
\fancyhead[R]{\small The conditional new-layer theorem}

To use the new layer in an iteration, its three outputs must hold for the
same exact solution and the same frequency choice. We now collect
Sections~\ref{sec:packet-expansion}--\ref{sec:initial-packet} into the
conditional theorem below, using Corollary~\ref{cor:corr-physical-expansions}
and Propositions~\ref{prop:child-map} and~\ref{prop:initial-packet}.
The purpose is to make their common inputs explicit; verifying those
inputs from the growing wave and the stage parameters is the task of
Sections~\ref{sec:wave-packet-interface}--\ref{sec:stage-induction}.

\subsection{The common inputs and outputs}

Fix $s=6$, $0\le t_a\le S\le1$ with $S>0$, and a smooth exact parent
Euler solution $(u,p)$ with particle map $X(0,a)=a$. Use the normalized
coefficients $F,M,H,m$ and unit vector $m_0$ from
\eqref{intro:eq:parent}--\eqref{intro:eq:normal}.
Fix $P\ge2$ and an input exponent $c_0$. Assume the following inputs
simultaneously, with their constants and reciprocal Gevrey radii bounded
by $P^{c_0}$:

\begin{enumerate}
\item The mean and transverse inverse hypotheses of
Corollary~\ref{cor:mean-shared-inverses}, on the respective domains
specified in Section~\ref{subsec:packet-inputs}. These include the
coercivity and forward propagator hypotheses, not merely coefficient
regularity.
\item The data for the leading increment in Section~\ref{subsec:packet-inputs}:
$\chi_1,f_\delta$, the selected endpoint (or prescribed initial
amplitude when $t_a=0$), and the amplitude profile
\eqref{eq:packet-amplitude-profile}, with $\alpha>0$ and
$H_0\le P^{c_0}$.
\item The parent coefficient and inverse bounds
\eqref{eq:corr-parent-bounds}, including the fixed additional
derivatives. The finite order estimates
\eqref{eq:corr-input-bounds} will be obtained from the construction,
rather than imposed as separate inputs here.
\item The physical parent-map bound \eqref{eq:child-parent-map}, with
$K_h\le P^{c_0}$, and $0<\ell\le1$, $\ell^{-1}\le P^{c_0}$.
\end{enumerate}
All hypotheses not expressed as upper bounds, such as positivity,
support, transversality and the parent identities, retain their stated
meaning. All constants below depend only on these fixed input exponents,
the Sobolev order and the fixed cutoff choices, independently of the
stage, $P$ and $k$.

\refstepcounter{equation}\label{prop:complete-packet}
\textbf{Proposition \theequation\ (The exact new layer and its three outputs).}
Under the common inputs above, there are fixed $Q,k_0$ such that whenever
\begin{equation}\label{eq:complete-packet-frequency}
 P^Q\le k^\vartheta,\qquad k\ge k_0,\qquad
 \vartheta=10^{-6},\qquad N=\lfloor k^\vartheta\rfloor,
\end{equation}
the coefficient construction through order $N$ and the correction
\eqref{intro:eq:physical-child} produce a smooth exact Euler solution
$(u^{\rm new},p^{\rm new})$ on $[0,S]\times\R^3$ with the following outputs.

\emph{Leading physical derivatives.} The velocity gradient and pressure
Hessian obey \eqref{eq:corr-final-expansions}, with uniform
$O(k^{-1/2})$ errors in space and throughout $[0,S]$.

\emph{Child particle map.} The volume-preserving particle map
$X^{\rm new}$ satisfies \eqref{eq:child-map-output} on the whole interval
$[0,S]$, with $C_*=10(s+2)=80$.

\emph{Initial support and size.} The initial increment has the
divergence-free decomposition \eqref{eq:initial-split}, with supports
\eqref{eq:initial-packet-support} and Sobolev bounds
\eqref{eq:initial-packet-sobolev} for every fixed integer $r\ge0$.
The same $Q,k_0$ work for all these fixed orders $r$; only $c_r$ depends
on $r$. If $t_a=0$ and $\lambda_b=0$, the initial increment is exactly
\eqref{eq:initial-curl-formula} and its mean part vanishes.

\textbf{Remark.} This proposition constructs one exact child under
specified parent inputs. It does not yet assert that the child satisfies
the next stage's coercivity, propagator and amplitude hypotheses.
In particular, the output map bound is polynomial in $k$; fitting it
into the next input budget is part of the stage parameter choice.

\subsection{Compatibility of the budgets and proof}

\emph{Proof of Proposition~\ref{prop:complete-packet}.}
First fix the input exponent $c_0$. Choose one derivative scale
$R=P^{c_R}$ large enough for Corollary~\ref{cor:mean-shared-inverses}
and the product bounds and bounds on the data for the leading increment in
Proposition~\ref{prop:packet-coefficient-bounds}, independently of $N$
and the coefficient order. That proposition
constructs all coefficients with the shifts
\eqref{eq:packet-wave-shifts}--\eqref{eq:packet-mean-shifts} at this same
scale. No change of $R$ during the recursion is needed.

Proposition~\ref{prop:packet-truncation-bounds} gives
\eqref{eq:corr-input-bounds} at a radius $\varrho_a=P^{-c_a}$, for fixed
$c_a$, once the frequency dominates its fixed polynomial budget.
Together with \eqref{eq:corr-parent-bounds}, these are precisely
the inputs to Proposition~\ref{prop:exact-packet-correction}.
It produces one correction $e$, with $e(0)=0$, and the exact physical
solution \eqref{intro:eq:physical-child}. Its estimates hold at
$\varrho_*=P^{-c_*}\le\varrho_a$, with $c_*$ fixed independently of $N$.

There is no competing radius requirement among the three outputs.
The physical derivative proof uses a fixed smaller radius to control
the finitely many derivatives in \eqref{eq:corr-final-expansions}.
The transport estimates \eqref{eq:child-transport-bounds} use
$\varrho_*$ itself, with $R_b=\varrho_*^{-1}$.
The initial coefficient sums in Lemma~\ref{lem:initial-profile-sum}
use another fixed radius of the form $P^{-c}$, before taking any
physical trace. If a common radius is desired for these intermediate
bounds, take their minimum. All the weighted norms decrease when the
radius decreases, and the reciprocal of this minimum is still a fixed
power of $P$. This choice is made once, not at successive coefficient
orders or derivative orders.

Choose $Q$ to dominate the finitely many frequency exponents required
by Proposition~\ref{prop:packet-truncation-bounds},
Proposition~\ref{prop:exact-packet-correction},
Corollary~\ref{cor:corr-physical-expansions}, and
Propositions~\ref{prop:child-map} and~\ref{prop:initial-packet}.
Choose $k_0$ to dominate their fixed lower thresholds. This is possible
because all input exponents were fixed first. Increasing $Q$ only
strengthens $P^Q\le k^\vartheta$; it does not change $R$, $N$'s defining
exponent, or any previously constructed coefficient. With this single
choice, Corollary~\ref{cor:corr-physical-expansions} proves
\eqref{eq:corr-final-expansions}, Proposition~\ref{prop:child-map}
proves the map assertion, and Proposition~\ref{prop:initial-packet}
proves the initial assertions for this same exact solution.
The latter proposition sums the coefficients at a fixed Gevrey radius
before estimating any fixed Sobolev order, so it imposes no
$r$-dependent enlargement of $Q$ or $k_0$.
\hfill$\square$

Section~\ref{sec:wave-packet-interface} identifies how the wave estimates
supply these inputs; Section~\ref{sec:stage-parameters} checks the proposed
stage scales. Applying this theorem at every stage still requires all
its common inputs to be verified for each successive parent;
Proposition~\ref{prop:stage-induction} performs this verification.

\subsection{Parity and the stationary trajectory at the origin}
\label{subsec:parity-origin}

\refstepcounter{equation}\label{lem:packet-parity}
\textbf{Lemma \theequation\ (Parity preservation).}
Under the hypotheses of Proposition~\ref{prop:complete-packet}, suppose
the parent velocity is odd in physical space, both cutoffs $\chi_1,\chi_o$
are even, and $f_\delta$ is odd. Choose the leading increment using a constant
coordinate endpoint when $t_a>0$, or an even initial amplitude when
$t_a=0$, and a label-independent $m_0$. Then the exact child velocity
is odd. Its particle map satisfies $X^{\rm new}(t,0)=0$ on $[0,S]$.

\emph{Proof.} Reflection invariance and uniqueness of the parent flow
give $X(t,-a)=-X(t,a)$. Hence $F,M,H$ are even in the label; the
pressure Hessian is even also from $F_{tt}=-HF$.
The transported normal is even. The homogeneous history problem has
even coefficients and the same coordinate endpoint at opposite labels,
so uniqueness makes its velocity amplitude even. The same conclusion
holds for the forward wave ODE with even initial data. Thus $v$ is even.

Let $\mathcal Sf(y,\theta)=f(-y,-\theta)$, acting componentwise.
Call a field even or odd according as $\mathcal Sf=f$ or $-f$.
The coefficients of $d$ and $m$ are even, so $d$, $\partial_\theta$,
and the zero-mean inverse $\partial_\theta^{-1}$ reverse parity.
Therefore $A_1=\alpha\chi_1vf_\delta$ is odd, while $\pi_1$ and
$Q_1$ are even and $C_1=d\times Q_1$ is odd.

Inductively assume every known velocity coefficient is odd and every
known pressure coefficient is even. In \eqref{intro:eq:old-force},
the time derivative and multiplication by $M$ preserve oddness,
a pressure gradient is odd, and each quadratic term is an odd velocity
times a derivative of an odd velocity, hence is odd. Angular averaging
commutes with reflection. Thus the mean force is odd in $y$.

The mean inverse preserves this parity. The Newton operator commutes
with reflection and each curl reverses it; the two curls and even
cutoffs in $\mathcal A_\chi$ make that operator commute with reflection.
The mean bilinear form \eqref{eq:mean-bilinear-form} is invariant under
simultaneous reflection of its arguments. For odd forcing, the path
$-z(t,-y)$ solves the same weak problem as $z$. Uniqueness gives odd
$z$, hence odd $B=Fz_t$ and odd $d\bar q=f-B_t-MB$.
Since $\nabla_y\bar q=F^T d\bar q$ is odd, its scalar potential
$\bar q$ is even; its additive time constant is immaterial.

After this mean solve, the remaining force in
\eqref{intro:eq:wave-step}, including $-(B_p\cdot m)\partial_\theta A_1$,
is odd. The transverse zero-endpoint history problem is invariant
under $A\mapsto-\mathcal SA$, as is its forward initial-value problem.
Uniqueness and the pressure formula give odd $A_p$ and even $\pi_p$;
the curl construction again gives even $Q_p$ and odd $C_p$.
This proves the induction, including oddness of $W^a,z^a,r^a$.

Finally the lifted correction equation is invariant under
$e\mapsto-\mathcal Se$ and
$\mathfrak p^e\mapsto-\mathcal S\mathfrak p^e$.
Indeed $D$ reverses parity, its gradient space is invariant, and its
orthogonal projection and the pressure inverse commute with reflection
because $G=F^{-1}F^{-T}$ is even. The transport field transforms as an
odd four-component vector field; the Laplacian also commutes with
reflection. Uniqueness of each viscous solution with zero initial data
therefore gives odd $e^\nu$, and this parity passes to the inviscid limit
in Proposition~\ref{prop:exact-packet-correction}.
The graph $\theta=km_0\cdot y$ is reflection invariant. Restriction to
it in \eqref{intro:eq:physical-child} gives an odd physical increment
over the odd parent. Thus $u^{\rm new}(t,0)=0$, and uniqueness of its
particle trajectories gives $X^{\rm new}(t,0)=0$.
\hfill$\square$

\section{From the growing wave to the new-layer inputs}
\label{sec:wave-packet-interface}
\fancyhead[R]{\small From the wave to new-layer inputs}

Proposition~\ref{prop:complete-packet} separates the construction of an
exact new layer from the verification of its inputs. We now check which
inputs are supplied by the wave analysis. In particular, the amplitude
bound must hold through the retained endpoint $S$, not just through the
target time $t_j$. We combine the wave-model comparison of
Proposition~\ref{prop:conditional-wave-control}, the amplification bounds
of Section~\ref{sec:quantitative-amplification}, and the history estimates
of Section~\ref{sec:history-activation} to supply the corresponding inputs
of Proposition~\ref{prop:complete-packet}. Global parent
regularity and coercivity remain inputs to be maintained in the stage
induction.

\subsection{The wave inputs on the retained interval}

Set $t_a=t_{j-1}$ and use the clock and target of
\eqref{eq:amplification-scales} and \eqref{eq:actual-target-time}:
\[
 \tau(t)=\frac a\varepsilon(t-t_a),\qquad
 \tau_j=\tau(t_j)=\frac{x_j}{\sqrt\beta}.
\]
Here $P$ denotes the input budget, not the component of the
normal in the moving frame. We allow a fixed exponent $d$ in the
intermediate input bounds $P^d$; the input exponent $c_0$ will be
chosen after these bounds have been assembled.

\refstepcounter{equation}\label{prop:wave-packet-inputs}
\textbf{Proposition \theequation\ (Wave inputs through the retained endpoint).}
Assume the hypotheses of Lemma~\ref{lem:packet-coordinate-propagator}
and Lemma~\ref{lem:amplitude-ratios}, on the same wave and support.
For $t_a>0$, use the history selected by
Lemma~\ref{lem:history-endpoint-choice} with the common endpoint
\eqref{eq:history-common-endpoint}, under the hypotheses of
Lemma~\ref{lem:history-label-bounds}. For $t_a=0$, use the prescribed
initial wave in Section~\ref{sec:history-activation}.
Suppose $t_j\le S\le1$ and
\begin{equation}\label{eq:interface-horizon}
 \tau(S)\le\Theta,\qquad \tau(S)-\tau_j\le1.
\end{equation}
Assume, for a fixed $d$, that
\begin{equation}\label{eq:interface-budget}
 K_h+C_F+\varepsilon^{-1}+\Theta+s_0^{-1}
       +h_j+\delta_j\le P^d,
\end{equation}
where $C_F$ is from \eqref{eq:packet-coordinate-inputs} and the
exponent $c_F$ there is fixed. Choose
$\alpha_j=\delta_jh_j/A_j$ as in
\eqref{eq:quantitative-normalization}. Then, with
$g(t)=V_\lambda(\tau(t))$, the forward propagator satisfies
\begin{equation}\label{eq:interface-propagator}
 \|\mathcal P(t,s;y)\|\le P^c\frac{g(t)}{g(s)},
 \qquad t_a\le s\le t\le S,\quad y\in\Omega,
 \qquad g(t_a)=1.
\end{equation}
For the history/forward profile \eqref{eq:packet-amplitude-profile},
\begin{equation}\label{eq:interface-amplitude}
 H_0=\max\{1,\sup_{[0,S]}\gamma(t)\}
 \le\max\{1,C\delta_jh_j/s_0\}\le P^c.
\end{equation}
At positive-history stages the activation signs and scaled accuracy are
\eqref{eq:history-activation-signs} and
\eqref{eq:history-scaled-data}; in particular their spatial error is
at most $C\ell K_h^{19\kappa}/\varepsilon$.
The selected coordinate endpoint is constant in the label and has
polynomial size in $P$.

\textbf{Remark.} The smallness of the activation error is not a
conclusion of a polynomial bound on $K_h$ alone. The localization width
$\ell$ must make this error small enough for
Proposition~\ref{prop:conditional-wave-control}. That comparison
hypothesis, and its other error bounds, are included in the hypotheses
above and must be verified in the stage construction.

\emph{Proof.} Lemma~\ref{lem:packet-coordinate-propagator} gives a
propagator prefactor $CK_h^{2c_F}\varepsilon^{-1}\Theta^9$.
It is polynomial in $P$ by \eqref{eq:interface-budget}, proving
\eqref{eq:interface-propagator}. The first inequality in
\eqref{eq:interface-horizon} ensures that this estimate covers all
of $[t_a,S]$.

The amplitude normalization estimate
\eqref{eq:packet-normalized-amplitude} applies because
\eqref{eq:interface-horizon} gives
$\tau(S)\le\min(\Theta,\tau_j+1)$. It bounds
$\alpha_jg(t)$ by $C\delta_jh_j/s_0$ on $[t_a,S]$.
Since $g(t_a)=1$, it also bounds $\alpha_j$, the constant history
value of $\gamma$. Together with \eqref{eq:interface-budget}, this
proves \eqref{eq:interface-amplitude} on all of $[0,S]$.

The activation conclusions are precisely those of
Lemmas~\ref{lem:history-endpoint-choice} and
\ref{lem:history-label-bounds}. The first gives a physical endpoint
of size $C/h_*$. Changing to its coordinate endpoint by
\eqref{eq:history-common-endpoint} costs the polynomial parent inverse
bound, and the same coordinate endpoint is used at every label.
Thus its higher label derivatives vanish. When $t_a=0$, the prescribed
constant initial velocity in Section~\ref{sec:history-activation}
replaces this endpoint construction and has zero activation variation.
\hfill$\square$

For the stage choices of Subsection~\ref{intro:subsec:stage-parameters},
take $S=S_j$. Under its stated geometry assumptions,
Lemma~\ref{intro:lem:clock-child} gives both inequalities in
\eqref{eq:interface-horizon}. Applying
Lemmas~\ref{intro:lem:clock-child}--\ref{intro:lem:clock-parent} at
stage $j+1$ also gives $S_{j+1}\le S_j$, so the child is available
through the next layer's retained interval. These timing conclusions
remain conditional on the corresponding stage geometry; the other
parent inputs are addressed below.

\subsection{The parent inputs still to be verified}

The proposition settles the forward propagator and amplitude parts of
the new-layer hypotheses. The remaining parts have distinct sources.

\emph{History and activation.} Section~\ref{sec:history-activation}
constructs the selected wave under the history upper Hessian bound
\eqref{eq:history-coercivity-input}, the geometry
\eqref{eq:history-activation-geometry}, and the size bounds
\eqref{eq:history-size-input}. These inequalities must hold for the
actual parent. The first-derivative estimates in
\eqref{eq:history-uniform-bounds} do not replace the all-order Gevrey
inputs needed for the expansion.

\emph{All-order coefficients and data for the leading increment.} The relevant hierarchies
are \eqref{eq:forced-factorial-input},
\eqref{eq:mean-factorial-inputs}, and \eqref{eq:corr-parent-bounds}.
Their time derivatives are tied to the parent identities
\eqref{intro:eq:parent}, in particular
$F_t=MF$ and $F_{tt}=-HF$; they are not new independent wave estimates.
The preceding child-map output \eqref{eq:child-map-output} provides a
candidate regularity budget for the next parent. Converting it to all
the required coefficient bounds, with the new label scale, and fitting
the resulting constants into the next $P$ are verified by
Lemma~\ref{lem:stage-map-coefficients} and
Proposition~\ref{prop:stage-frequency-time}.
The cutoff and profile Gevrey constants, $\delta_j^{-1}$ and
$\ell^{-1}$ must also fit that budget. Once these bounds hold, the
argument for the leading coefficient $A_1$ in Proposition~\ref{prop:packet-coefficient-bounds}
supplies its Gevrey estimates using the constant coordinate endpoint.

\emph{Global mean coercivity.} The mean inverse requires the global
one-sided bounds \eqref{eq:mean-quadratic-inputs} and the numerical
conditions \eqref{eq:mean-coercivity-inputs}.
These concern the pressure Hessian throughout space and the initial
strain inside and outside a physical ball. They do not follow from the
wave-model comparison at the origin. The stage induction must preserve
these bounds and satisfy \eqref{eq:mean-coercivity-inputs} by its choices of
time and localization. Likewise, $S^{-1}$ and, at positive-history
stages, $t_a^{-1}$ must fit the inverse budgets.

\emph{Common frequency and retained time.} After these inputs are
bounded by fixed powers of $P$, choose $c_0$ to cover those powers.
Proposition~\ref{prop:complete-packet} then fixes $Q$. Only afterwards
choose the stage frequency to satisfy \eqref{eq:complete-packet-frequency}.
The timing lemmas applied above supply \eqref{eq:interface-horizon}
once their geometry hypotheses hold for that stage.

\section{Choosing the stage parameters}
\label{sec:stage-parameters}
\fancyhead[R]{\small The stage parameters}

The preceding section lists the inequalities needed to apply the layer
theorem repeatedly. We now verify that the scales introduced in
Subsection~\ref{intro:subsec:stage-parameters} satisfy these numerical
requirements, with the common frequency threshold of
Proposition~\ref{prop:complete-packet}. This does not yet
construct the seed or prove that the parent has the required signs and
decomposition. The estimates below explain why, once those structural
properties are maintained, the sizes and time intervals are compatible.

\subsection{The scales and the order of choice}

Use the stage indexing, recurrence $x_j=j^2x_{j-1}$, and
scales $\ell_j,k_j,\delta_j,h_j$ of
Subsection~\ref{intro:subsec:stage-parameters}, including the seed
\eqref{intro:eq:seed-scales} with $D_0=1000$.
Write $x_0=x_{J-1}$ and set $r=\ell_{J-1}=x_0^{-D_0}$ for the
physical ball in the mean-coercivity hypothesis
\eqref{eq:mean-quadratic-inputs}. The wave amplitude is fixed by
\eqref{eq:quantitative-normalization}; the unit-amplitude seed has
$\alpha_{J-1}=\delta_{J-1}h_{J-1}=x_0^{-10}$ as in
Subsection~\ref{intro:subsec:stage-parameters}.

The order of choices is important. First fix the global size and history
geometry constants in the order justified by
Lemma~\ref{lem:induction-constant-order}; this does not use the stage
construction. Fix all exponents in the wave and
inverse estimates first, then the input exponent $c_0$ and the resulting
layer threshold $Q$ of Proposition~\ref{prop:complete-packet}.
With $\vartheta=10^{-6}$ and $C_*=80$, one convenient conservative choice is
\[
 D_k\ge1000\max\{1,(D_0+20)Q/\vartheta\}.
\]
Next choose $J$ sufficiently large for the estimates below, depending
only on these fixed constants. Finally choose $x_0$ sufficiently large.
Neither $Q$ nor $D_k$ is chosen again when $x_0$ is increased.

The candidate normal-stage input budget is
\begin{equation}\label{eq:stage-input-budget}
 P_j=C\bigl(2+k_{j-1}^{C_*}+\ell_j^{-1}+\delta_j^{-1}
       +h_j+h_{j-1}+h_{J-1}+r^{-1}+S_{\rm base}^{-1}
       +x_j+x_{j-1}\bigr),
\end{equation}
where $C$ is a sufficiently large fixed constant. For the seed we use
$P_{J-1}\le C(J)x_0^{D_0+20}$, conditional on a base flow with uniform
bounds; Lemma~\ref{lem:seed-inputs} verifies this budget. A fixed power of an
entry of \eqref{eq:stage-input-budget} is covered by choosing $c_0$
large enough. The frequency $k_j$ itself is not an input entry at stage
$j$; the inherited map bound is $K_{h,j}=k_{j-1}^{C_*}$.

\subsection{Regularity passed from the preceding particle map}

\refstepcounter{equation}\label{lem:stage-map-coefficients}
\textbf{Lemma \theequation\ (Parent coefficients from the preceding map).}
Suppose the parent is an exact incompressible flow and its particle map
has the bound \eqref{eq:child-map-output} with frequency $k_{j-1}$.
Then at the new scale $0<\ell_j\le1$ the coefficient hierarchies
\eqref{eq:forced-factorial-input}, \eqref{eq:mean-factorial-inputs}, and
\eqref{eq:corr-parent-bounds} have size and derivative-growth constants
bounded by fixed powers of $K_{h,j}$, after including the finitely many
additional derivatives required there.

\emph{Proof.} Write $\mathsf F=\nabla_aX$ in the physical label.
Volume preservation gives $\mathsf F^{-1}=\operatorname{cof}(\mathsf F)^T$.
The required physical-label identities are
\[
 \widehat M=(\nabla_aX_t)\mathsf F^{-1},\qquad
 \widehat H=-(\nabla_aX_{tt})\mathsf F^{-1},\qquad
 \mathsf F_t=\nabla_aX_t,\quad \mathsf F_{tt}=\nabla_aX_{tt}.
\]
The second identity follows by differentiating the material acceleration
$X_{tt}=-(\nabla p)\circ X$. Thus each coefficient is a finite product
of controlled derivatives; the inverse is a quadratic polynomial in
$\mathsf F$. Sobolev embedding and the Gevrey product estimate used in
Proposition~\ref{prop:packet-coefficient-bounds} give polynomial size
and growth constants. Fixed extra derivatives are absorbed by
$(n+q)!^2\le C_q^{n+1}(n!)^2$ for fixed $q$.
The identity part of $\mathsf F=I+\nabla_a(X-\mathrm{id})$ is bounded
in $L^\infty$ and has zero derivatives of positive order.
Finally $a=\ell_jy$ contributes $\ell_j^n\le1$ to $n$ label
derivatives. Products give the corresponding bounds for $m,G,K$ and
their required inverses, since $|F^{-T}m_0|\ge\|F^T\|^{-1}$.
\hfill$\square$

\subsection{Frequency and time compatibility}

Use the target times from \eqref{eq:actual-target-time} and
$W_j,S_j,\Theta_j$ from Subsection~\ref{intro:subsec:time-intervals}.
In particular, $t_{J-1}=0$ and the seed is retained through
\[
 S_{\rm base}:=S_{J-1}=2W_J=6J^2x_0^{2-D_0/2}.
\]
Unlike the scale sequences, the times depend on the activation geometry
of the actual parent; they are fixed successively as that parent is
constructed.

\refstepcounter{equation}\label{prop:stage-frequency-time}
\textbf{Proposition \theequation\ (Compatible frequencies and horizons).}
For sufficiently large $J$ and then $x_0$, the budgets
\eqref{eq:stage-input-budget} satisfy $P_j^Q\le k_j^\vartheta$,
uniformly in $j\ge J$, and the seed budget satisfies the same condition.
If the ranges \eqref{intro:eq:activation-ranges} hold at each constructed activation,
then
\begin{equation}\label{eq:stage-time-conclusions}
 \frac{W_j}{6}\le t_j-t_{j-1}\le\frac{2W_j}{3},\qquad
 S_j\le S_{j-1},\qquad
 \tau(S_j)\le\min(\Theta_j,\tau_j+1).
\end{equation}
The whole comparison interval $I_j$ of Proposition~\ref{prop:conditional-wave-control}
lies in $[t_{j-1},S_{j-1}]$. The increasing targets have a finite limit
$T\le S_j$ for every $j$. Moreover, at positive-history stages,
$t_{j-1}^{-1},S_j^{-1}\le12S_{\rm base}^{-1}$.

\emph{Proof.} Write $X_j=x_{j-1}$. For $j>J$ the exact identities
\begin{equation}\label{eq:stage-log-identities}
 \log k_{j-1}=\frac{X_j}{(j-1)^4},\qquad
 \log h_{j-1}=\frac{X_j}{(j-1)^7}
\end{equation}
follow from $x_{j-1}=(j-1)^2x_{j-2}$.
Taking logarithms of the finite sum in \eqref{eq:stage-input-budget}
and dividing by $\log k_j=X_j/j^2$ gives, for $j>J$,
\[
 \frac{\log P_j}{\log k_j}
 \le C\left(\frac1j+\frac{C_*}{j^2}
 +\frac{j^2[\log X_j+(D_0+1)\log x_0+\log J+1]}{X_j}\right).
\]
The first two terms become small by choosing $J$. The last tends to
zero uniformly by increasing $x_0$: indeed
$X_j=x_0\prod_{i=J}^{j-1}i^2$, which dominates every fixed power of
$j$ and its own logarithm. At $j=J$, the inherited frequency term is
instead $C_*D_k\log x_0$, whose ratio to $\log k_J=x_0/J^2$ also tends
to zero. Thus the ratios are at most $\vartheta/Q$.
For the seed, $(D_0+20)Q<D_k\vartheta$ leaves a strict margin to
absorb $C(J)$ by increasing $x_0$.

The activation geometry directly gives the first two-sided bound in
\eqref{eq:stage-time-conclusions}. Also
\[
 \frac{W_{j+1}}{W_j}=j^2(j+1)^2\sqrt{\frac{h_{j-1}}{h_j}}.
\]
For $j>J$ the logarithm of the square root is
$-X_j/(2j^5)+X_j/(2(j-1)^7)$; at $j=J$ it is
$-x_0/(2J^5)+(D_0/2)\log x_0$. The negative term dominates all
logarithmic factors. Hence, uniformly, $W_{j+1}\le W_j/4$ and
\[
 2\sqrt{a_jh_{j-1}}W_{j+1}
 =6\sqrt{a_j}\,x_{j+1}x_j\sqrt{h_{j-1}/h_j}
 \le\Theta_j^{-60}.
\]
In particular, the additional scaled time is at most $1$.
Lemma~\ref{intro:lem:clock-child} gives the full bound on $\tau(S_j)$,
and Lemma~\ref{intro:lem:clock-parent} places $I_j$ inside the parent
interval. Together they give $S_j\le S_{j-1}$, as recorded in
Subsection~\ref{intro:subsec:time-intervals}.
The geometric decay of $W_j$ implies convergence of the targets;
nesting gives $T\le S_j$. Finally $t_J\ge W_J/6=S_{\rm base}/12$,
which proves the inverse-time bounds. Increasing $x_0$ also makes
$S_{\rm base}\le1$ and places it inside any fixed base existence interval.
\hfill$\square$

\subsection{Wave errors and separation}

Set $h_{J-2}=1$, $G_{*,j}=1+h_{j-2}$ and
$E_{*,j}=k_{j-1}^{-1/2}$, consistent with the remainder in
\eqref{eq:corr-final-expansions}. Use $K_{h,j}=k_{j-1}^{C_*}$ and
$\kappa=7$. Constants multiplying these bounds can be absorbed below.

\refstepcounter{equation}\label{lem:stage-wave-smallness}
\textbf{Lemma \theequation\ (Small wave errors and separated shears).}
With the preceding choices, the numerical budget
\begin{equation}\label{eq:stage-error-choice}
 e_{*,j}=C\left(\varepsilon_j\Theta_jG_{*,j}^2
       +k_{j-1}^{-1/2}
       +\ell_jK_{h,j}^{19\kappa}/\varepsilon_j\right)
\end{equation}
satisfies $e_{*,j}\Theta_j^{60}\le x_{j-1}^{-10}$ uniformly, after
increasing $J,x_0$ as above. In addition the ratios
\begin{equation}\label{eq:stage-separations}
 \frac{h_j}{h_{j-1}^2},\qquad
 \frac{\sqrt{h_{j-1}}}{x_{j-1}x_j(G_{*,j}+1)},\qquad
 \frac r{\ell_j}
\end{equation}
can all be made uniformly larger than any prescribed fixed constant.

\emph{Proof.} For $j>J$, \eqref{eq:stage-log-identities} gives
\[
 \log\frac{\ell_jK_{h,j}^{19\kappa}}{\varepsilon_j}
 \le-\frac{X_j}{j^{7/2}}
 +\frac{19\kappa C_*X_j}{(j-1)^4}
 +\frac{X_j}{2(j-1)^7}+C
 \le-\frac{X_j}{2j^{7/2}}.
\]
Here the last inequality fixes how large $J$ must be; in particular
it can be chosen larger than a fixed multiple of $(19\kappa C_*)^2$.
The inherited error has logarithm $-X_j/(2(j-1)^4)$.
For $j\ge J+2$, the older-field term obeys
\[
 \log(\varepsilon_jG_{*,j}^2)
 \le-\frac{X_j}{2(j-1)^7}
 +\frac{2X_j}{(j-1)^2(j-2)^7}+C
 \le-\frac{2X_j}{5(j-1)^7}+C.
\]
At $j=J+1$ the positive term instead is $O(D_0\log x_0)$,
whereas the negative term is $-x_0/(2J^5)$.
These negative exponential terms dominate $\Theta_j^{61}X_j^{10}$,
uniformly by the same growth of $X_j$ used in the preceding proof.

At the first normal stage $\varepsilon_J\le Cx_0^{-500}$,
$G_{*,J}=2$, and $\Theta_J=2J^2x_0^2$. Thus
$\varepsilon_JG_{*,J}^2\Theta_J^{61}\le C(J)x_0^{-378}$.
The inherited error contributes at most
$C(J)x_0^{-D_k/2+120}$, and the spatial error contains
$\exp(-x_0/J^{7/2})$ times a fixed power of $x_0$.
This proves the claimed error bound also at $j=J$.

Finally, for $j>J$,
$\log(h_j/h_{j-1}^2)=X_j/j^5-2X_j/(j-1)^7$ tends uniformly to
infinity. The logarithm of the second ratio in
\eqref{eq:stage-separations} has positive leading term
$X_j/(2(j-1)^7)$, whereas $\log(G_{*,j}+1)$ has two additional inverse
powers of $j$ and $\log(x_{j-1}x_j)$ is logarithmic.
The stages $J,J+1$ follow directly from the seed formulas; at $J$ the
second ratio grows like $x_0^{498}/J^2$.
Lastly $\log(r/\ell_j)=X_j/j^{7/2}-D_0\log x_0$ tends uniformly to
infinity. This proves all three separations.
\hfill$\square$

This lemma verifies \eqref{eq:actual-final-smallness}, covering the
polynomial comparison losses of
Proposition~\ref{prop:conditional-wave-control}, including
$e_*\Theta^{26}$, and the activation variation in
\eqref{eq:history-scaled-data}. It estimates the budgets; it does not
assert the parent decomposition or activation signs without induction.

\subsection{Summable costs and the coercivity margin}

Let $b>0$ and $c$ be fixed constants in the early-to-target estimates
\eqref{eq:early-amplitude-ratio} and \eqref{eq:history-target-ratio}.
The numerical allowance for the early gradient and Hessian changes is
\[
 b_j=C h_jh_{j-1}K_{h,j}^c\Theta_j^c\varepsilon_j^{-c}
          e^{-b x_{j-1}}.
\]
For $j>J$, the logarithm of the prefactor is at most
$C(c,C_*)X_j/(j-1)^4+C\log(X_jx_0)$; at $j=J$ it is
$O(x_0/J^5)+O(\log x_0)$. Choose $J$ so that both are less than
$bX_j/2$, then increase $x_0$. Consequently $b_j\le e^{-bX_j/2}$.
Likewise,
\[
 \log(\delta_jh_jh_{j-1})
 =-\frac{X_j}{j^3}+\frac{X_j}{j^5}+\log h_{j-1}
 \le-\frac{X_j}{2j^3}.
\]
The seed case is $\delta_{J-1}h_{J-1}=x_0^{-10}$.
Since $X_j/j^A$ grows at least geometrically for every fixed $A$ once
$J$ is large, the sums of $b_j$, $\delta_jh_jh_{j-1}$,
$k_j^{-1/2}$, and $x_{j-1}^{-d}$ for any fixed $d>0$ tend to zero
as $x_0$ increases. These are the costs available for preserving the
one-sided pressure and initial-strain bounds; their application to the
actual solutions is part of the induction.

\refstepcounter{equation}\label{lem:stage-coercivity-margin}
\textbf{Lemma \theequation\ (Numerical coercivity margin).}
Fix base constants $K_B,B_e,C_M$. Put
$K_+=K_B+1$, $B_c=C_Mh_{J-1}+2$, and
$\lambda_b=C_1B_c+1$. Increasing
$x_0$ makes
\begin{equation}\label{eq:stage-coercivity-margin}
 \frac{K_+S_{\rm base}^2}{2}+B_eS_{\rm base}
       +C_2B_cr^3S_{\rm base}\le\frac14,
\end{equation}
which holds with $S_{\rm base}$ replaced by any $0<S\le S_{\rm base}$
by monotonicity.
Thus the numerical conditions \eqref{eq:mean-coercivity-inputs} have
a strict margin, once their actual parent quadratic-form bounds hold
and $\lambda_b$ is included in the polynomial input budget.
At every positive-history stage,
$t_a=t_{j-1}<S_{j-1}\le S_{\rm base}$ by
\eqref{eq:stage-time-conclusions}. Consequently, the history condition
\eqref{eq:history-coercivity-input} holds whenever the parent satisfies
$H\le K_+I$.

\emph{Proof.} The first two terms tend to zero because
$S_{\rm base}=6J^2x_0^{-498}$. The last is bounded by a constant
times $x_0^{D_0}x_0^{-3D_0}S_{\rm base}$, which also tends to zero.
Moreover $\lambda_b$ is polynomial in the budget entries
$h_{J-1},r^{-1}$. Finally, \eqref{eq:stage-coercivity-margin} gives
$K_+S_{\rm base}^2\le\tfrac12$, so $H\le K_+I\le t_a^{-2}I$
under the stated assumptions.
\hfill$\square$

The numerical scales are now fixed, with the target times determined
successively by the actual activation geometry.
Section~\ref{sec:base-seed} constructs the uniformly controlled base flow
and seed. The remaining induction must prove that every child preserves
the parent signs, decomposition and global quadratic-form bounds.
The estimates here provide its frequency, localization, time and
summability margins.

\section{The base flow and seed layer}
\label{sec:base-seed}
\fancyhead[R]{\small The base flow and seed}

We now supply an actual starting parent for the numerical choices of
Section~\ref{sec:stage-parameters}. First construct a uniformly controlled
base flow whose gradient at the origin fixes the first activation geometry.
Then apply Proposition~\ref{prop:complete-packet} once, without a history
interval, to insert the seed shear. All constants in this section are
independent of $x_0=x_{J-1}$ unless a dependence on the already fixed $J$
is displayed. Increasing $x_0$ does not change $c_0,Q,D_k$ or $J$.

\subsection{A compactly supported base datum}

Fix an oriented orthonormal basis $(p_0,q_0,n_0)$ and define
\begin{equation}\label{eq:base-linear-map}
 L_Bq_0=p_0+x_0^{-2}n_0,\qquad L_Bp_0=L_Bn_0=0.
\end{equation}
This map has trace zero and uniformly bounded norm for $x_0\ge1$.
Choose a fixed even compactly supported Gevrey-2 cutoff $\chi_B$,
equal to one on a ball $B(0,r_B)$, and set
\begin{equation}\label{eq:base-datum}
 u_{B,0}(a)=\operatorname{curl}_a
       \left[-\frac13\chi_B(a)\,a\times L_Ba\right].
\end{equation}
The cutoff can be constructed from the extension by zero of
$\exp(-1/t)$ for $t>0$: its derivative bound
$C^{n+1}(n!)^2$ follows from Cauchy's estimate on a circle of radius
proportional to $t$, followed by maximizing $t^{-n}e^{-c/t}$.
Products, integration and fixed rescalings give an even cutoff with
the stated support and plateau. We choose $\chi_1$ and $\chi_o$ by the
same construction, with $0\le\chi_1,\chi_o\le1$ and $\chi_1=1$ near zero.

\refstepcounter{equation}\label{lem:base-flow}
\textbf{Lemma \theequation\ (Uniform base flow).}
There are $T_B>0$ and $K_B,K_{h,B}\ge1$, independent of $x_0\ge1$,
such that \eqref{eq:base-datum} generates a smooth odd Euler solution
$U_B$ on $[0,T_B]$. It is divergence free, has compactly supported
initial datum, and $\nabla U_B(0,a)=L_B$ for $|a|<r_B$.
Writing $X_B$ for its particle map and
$H_B=\nabla_x^2p_B$, one has
\begin{equation}\label{eq:base-uniform-bounds}
 \begin{gathered}
 \sup_{[0,T_B]\times\R^3}
 (|\nabla U_B|+|H_B|)\le K_B,\\
 \sum_{|I|=n}\sup_{0\le t\le T_B}
 \|\partial_a^I(X_B-\mathrm{id},X_{B,t},X_{B,tt})(t)\|_{H^6}
 \le K_{h,B}^{n+1}(n!)^2\qquad(n\ge0).
 \end{gathered}
\end{equation}
In particular $H_B\le K_BI$ and $X_B(t,0)=0$.

\emph{Proof.} For a trace-free linear map,
$\operatorname{curl}(a\times L_Ba)=-3L_Ba$.
Thus \eqref{eq:base-datum} equals $L_Ba$ on the plateau.
Being a curl, it is divergence free; its vector potential is even,
so the velocity is odd. The compact support and the cutoff bounds give
\[
 \sum_{|I|=n}\|\partial^Iu_{B,0}\|_{H^6}
       \le C^{n+1}(n!)^2
\]
uniformly in $x_0$.

We recall the local existence estimate to make the uniform time explicit.
For integer $m\ge3$, the differentiated Euler equation gives
\begin{equation}\label{eq:base-sobolev-energy}
 \frac{d}{dt}\|u\|_{H^m}
       \le C_m\|\nabla u\|_{L^\infty}\|u\|_{H^m}.
\end{equation}
The undifferentiated transport term cancels by incompressibility,
and pressure cancels against divergence-free derivatives. The remaining
commutators are bounded by interpolation between $\nabla u$ in
$L^\infty$ and the top derivative in $L^2$, giving the right side.
Since $H^3(\R^3)\hookrightarrow W^{1,\infty}$, the uniformly bounded
initial $H^3$ norm supplies a common positive time. One construction is
to truncate Fourier frequencies in the projected equation. The same
energy estimate holds uniformly for the resulting ordinary differential
equations. Higher Sobolev bounds and the $L^2$ difference estimate give
convergence, first in $C_tL^2$ and then, by interpolation, in every fixed
Sobolev norm. Passing to the equation gives a smooth solution;
the same difference estimate gives uniqueness. Reflection invariance
and uniqueness preserve oddness.

To retain all derivative orders on a common interval, use the weighted
norms of Section~\ref{subsec:corr-energy}, now in three physical variables
and with the identity energy matrix. Put
\[
 \mathcal X_{\rho,J}=\sum_{n=0}^J\frac{\rho^n}{(n!)^2}
                    \sum_{|I|=n}\|\partial^I U_B\|_{H^6},\qquad
 \mathcal Y_{\rho,J}=\sum_{n=0}^J\frac{n\rho^n}{(n!)^2}
                    \sum_{|I|=n}\|\partial^I U_B\|_{H^6}.
\]
The product and transport weight calculations in
Lemmas~\ref{lem:corr-product-pressure}--\ref{lem:corr-energy-estimate}
give
\begin{equation}\label{eq:base-gevrey-energy}
 \mathcal X_{\rho,J}'\le C\mathcal X_{\rho,J}^2+
 \left(\frac{\rho'}\rho+
             \frac{C\mathcal X_{\rho,J}}\rho\right)\mathcal Y_{\rho,J}.
\end{equation}
Here pressure cancels at every order, so there are no variable-matrix
pressure commutators. Choose $\rho(0)>0$ below the uniform initial
radius, and let $A$ bound all initial sums. Taking $\rho'=-2CA$,
with a larger fixed $C$ if necessary, makes the last coefficient
nonpositive while $\mathcal X_{\rho,J}\le2A$.
The remaining scalar inequality preserves this bound on a fixed short
interval; shorten it also so that $\rho\ge\rho(0)/2$.
These choices are independent of $J,x_0$. The argument applies first
to the Fourier approximations; lower semicontinuity and increasing $J$
give the full Gevrey bound.

At a smaller fixed radius, the projected equation bounds $U_{B,t}$
and the material acceleration. The ordinary gradient projection is
bounded on every Sobolev space, so it introduces no derivative loss
beyond the one derivative in $U_B\cdot\nabla U_B$.
One more fixed derivative bounds $H_B$, and Sobolev embedding gives
the first line of \eqref{eq:base-uniform-bounds}.
The flow argument of Lemma~\ref{lem:child-small-flow}, in three
dimensions, applies after shortening the interval so that its fixed
size times growth constant times time is at most $1/4$.
It bounds $X_B-\mathrm{id},X_{B,t},X_{B,tt}$ in $L^2$ and $L^\infty$.
Absorbing six additional derivatives into the fixed growth constant
proves the second line. Oddness fixes the origin.
\hfill$\square$

\subsection{The seed inputs}

Use the seed scales \eqref{intro:eq:seed-scales}, set
$r=\ell_{J-1}$, and retain the seed through
$S_{\rm base}=6J^2x_0^{-498}$. For large $x_0$ this lies in
$[0,T_B]$ and $r/2<r_B$. Set $U^{J-2}=U_B$.
For this layer choose
\begin{equation}\label{eq:seed-wave-data}
 t_a=0,\qquad m(0,y)=p_0,\qquad v(0,y)=q_0,\qquad
 \alpha=\delta_{J-1}h_{J-1}=x_0^{-10},\qquad \lambda_b=0.
\end{equation}
These are the seed wave data; the first normal stage will use a different
normal, $n_0$.

We also make explicit the profile bounds needed for the pressure sign.
Choose an even nonnegative Gevrey-2 bump $g_B\le1$ supported inside
$(-\pi,\pi)$, with $g_B(0)=1$, and let
$c_B=\langle\delta^{-1}g_B(\theta/\delta)\rangle$, independent of small
$\delta$. Define $f_\delta$ as the odd periodic primitive of
\begin{equation}\label{eq:seed-profile-definition}
 f_\delta'(\theta)=
       \frac{\delta^{-1}g_B(\theta/\delta)-c_B}{1-c_B\delta}.
\end{equation}
The derivative has zero mean and its primitive has zero mean by oddness.
For sufficiently small $\delta$, this gives
\begin{equation}\label{eq:seed-profile-bounds}
 f_\delta'(0)=\delta^{-1},\qquad
 -C\le f_\delta'\le\delta^{-1},\qquad
 \|\partial_\theta^n f_\delta\|_{H^8(\mathbb T)}
       \le\delta^{-c}(C\delta^{-c})^n(n!)^2.
\end{equation}
Scaling the bump proves the last bound; the fixed eight derivatives
are absorbed into the polynomial prefactor and growth constant.
This realizes the profile required in Section~\ref{subsec:packet-inputs}.

\refstepcounter{equation}\label{lem:seed-inputs}
\textbf{Lemma \theequation\ (Verification of the seed inputs).}
The data \eqref{eq:seed-wave-data} over $U_B$ satisfy all inputs of
Proposition~\ref{prop:complete-packet} on $[0,S_{\rm base}]$, with
$g=1$ and a budget $P_{J-1}\le C(J)x_0^{D_0+20}$.
The frequency condition holds for the choices in
Section~\ref{sec:stage-parameters}. On the seed wave support,
\begin{equation}\label{eq:seed-pressure-sign}
 |m|+|v|+|m|^{-1}\le C,\qquad m\cdot M_Bv\ge\tfrac12,
 \qquad M_B=(\nabla U_B)\circ X_B.
\end{equation}

\emph{Proof.} Lemma~\ref{lem:base-flow} and the coefficient argument
of Lemma~\ref{lem:stage-map-coefficients} give uniform Gevrey bounds
for all parent coefficients, their inverses, and the parent map.
For the mean inverse use $K_+=K_B$, $B_e=K_B$, $B_c=0$ and
$\lambda_b=0$. The global initial lower-strain bound follows from
$|\nabla u_{B,0}|\le K_B$, and
$K_BS_{\rm base}^2/2+K_BS_{\rm base}\le1/2$ for large $x_0$.
Thus \eqref{eq:mean-coercivity-inputs} holds.
There is no history interval and hence no history coercivity or
$t_a^{-1}$ requirement.

The normal equation and transversality-preserving velocity equation
\eqref{intro:eq:normal}, \eqref{intro:eq:wave} have uniformly bounded
coefficients over this short interval. Gronwall gives uniform bounds
for $m,v$ and $|m|^{-1}$; the coordinate matrices of
\eqref{eq:forced-coordinate-coefficients} are consequently bounded.
Their propagator is bounded by $e^{CS_{\rm base}}\le C$, which verifies
\eqref{eq:forced-propagator-input} with $g=1$.
Proposition~\ref{prop:forced-forward} and its Gevrey estimates therefore
apply directly, without an amplification hypothesis for this seed.
The constant initial wave is transverse and has uniform derivative bounds.
Here $\gamma=\alpha$ and $H_0=1$.

The remaining sizes are $\ell_{J-1}^{-1}=x_0^{D_0}$,
$\delta_{J-1}^{-1}=x_0^{D_0+10}$ and
$S_{\rm base}^{-1}=(6J^2)^{-1}x_0^{498}$.
Together with \eqref{eq:seed-profile-bounds}, these fit fixed powers
of the stated budget. Choose the fixed input exponent $c_0$ to
include these powers before fixing $Q,D_k,J$, as prescribed in
Section~\ref{sec:stage-parameters}. The seed frequency verification
in Proposition~\ref{prop:stage-frequency-time} then applies.

Initially the whole seed support lies in the plateau of $\chi_B$,
so $m\cdot M_Bv=p_0\cdot L_Bq_0=1$ there.
Along a base trajectory,
$\partial_tM_B=-M_B^2-H_B\circ X_B$, by \eqref{intro:eq:parent}.
The uniform base and wave bounds therefore give
$|\partial_t(m\cdot M_Bv)|\le C$. Since $S_{\rm base}\to0$,
this proves \eqref{eq:seed-pressure-sign} on the whole retained interval.
\hfill$\square$

\subsection{The initialized parent and its first activation}

\refstepcounter{equation}\label{prop:seed-packet}
\textbf{Proposition \theequation\ (Seed for the first normal stage).}
For the fixed choices of Section~\ref{sec:stage-parameters} and
sufficiently large $x_0$, there is a smooth odd exact Euler solution
$U^{J-1}$ on $[0,S_{\rm base}]$, obtained from $U_B$ by constructing the seed layer.
Its initial increment is the compactly supported curl
\eqref{eq:initial-curl-formula}, with $m_0=p_0$ and $v(0)=q_0$,
and is supported in $|a|\le r/2$.
Write $M^{J-1}=\nabla U^{J-1}$ and
$H^{J-1}=\nabla^2p^{J-1}$. There are fixed constants $C_M,C_H$ such that
\begin{equation}\label{eq:seed-global-bounds}
 \begin{gathered}
 \|M^{J-1}(t)\|_\infty\le C_Mh_{J-1},\qquad
 \|H^{J-1}(t)\|_\infty\le C_Hh_{J-1},\\
 H^{J-1}(t,x)\le(K_B+1)I,\\
 \operatorname{sym}M^{J-1}(0,x)\ge
 \begin{cases}
  -B_eI,&|x|\ge r,\\
  -B_cI,&|x|<r,
 \end{cases}\\
 B_e=K_B+1,\quad B_c=C_Mh_{J-1}+2.
 \end{gathered}
\end{equation}
At the origin, for all $0\le t\le S_{\rm base}$,
\begin{equation}\label{eq:seed-central-shear}
 \begin{gathered}
 M^{J-1}(t,0)=M_B(t,0)+h(t)q(t)p(t)^T+E(t),\\
 p=m(t,0)/|m(t,0)|,\quad q=v(t,0)/|v(t,0)|,\quad
 h(t)=\frac{\alpha}{\delta_{J-1}}|m(t,0)|\,|v(t,0)|,\\
 h(0)=h_{J-1},\qquad |E(t)|\le Ck_{J-1}^{-1/2}
                                      \le k_{J-1}^{-1/4}.
 \end{gathered}
\end{equation}
The frame and shear satisfy \eqref{eq:inherited-frame-evolution} and
\eqref{eq:old-wave-norms} with older gradient $B=M_B(t,0)$.
At the first normal activation $t_{J-1}=0$,
\begin{equation}\label{eq:seed-first-geometry}
 a_J=1,\qquad \beta_J=x_0^{-2}.
\end{equation}
The particle map has \eqref{eq:child-map-output} with frequency
$k_{J-1}$, so the next parent budget is $K_{h,J}=k_{J-1}^{C_*}$.
For the stage-$J$ wave choose $m_J(0,y)=n_0$, $v_J(0,y)=q_0$.
Then the activation data \eqref{eq:budget-activation-data} are exact
with $s_0=1$, $\lambda=0$; the parent, frame and error hypotheses of
Proposition~\ref{prop:conditional-wave-control} hold on $I_J$.

\textbf{Remark.} This initializes the zero-history stage. The
positive-history endpoint geometry is not required at $j=J$;
it must be obtained when the next layer transfers its frame.
The global mean-coercivity bounds in \eqref{eq:seed-global-bounds}
are available for that next layer with
$\lambda_b=C_1B_c+1$, rather than the zero value used to build the seed.

\emph{Proof.} Lemma~\ref{lem:seed-inputs} permits an application of
Proposition~\ref{prop:complete-packet}. Choose even cutoffs and the odd
profile above. The constant initial wave and odd base satisfy
Lemma~\ref{lem:packet-parity}; hence $U^{J-1}$ is odd and its flow fixes
the origin.
Since $t_a=0$ and $\lambda_b=0$, Proposition~\ref{prop:initial-packet}
gives the asserted curl formula, zero initial mean increment and support.
The child-map assertion is Proposition~\ref{prop:child-map}.

Apply \eqref{eq:corr-final-expansions} to the seed. The bounds
\eqref{eq:seed-pressure-sign} and \eqref{eq:seed-profile-bounds},
together with $\alpha/\delta_{J-1}=h_{J-1}$, bound the full gradient
and Hessian increments by $Ch_{J-1}+Ck_{J-1}^{-1/2}$.
Absorbing the bounded base gives the first line of
\eqref{eq:seed-global-bounds}.
For the upper Hessian bound, its leading increment is
\[
 -2\alpha\chi_1(m\cdot M_Bv)
             \frac{mm^T}{|m|^2}f_{\delta_{J-1}}'.
\]
It is negative semidefinite wherever $f_{\delta_{J-1}}'\ge0$.
Where that derivative is negative, its magnitude is bounded by a
fixed constant, so its upper-eigenvalue cost is at most $C\alpha$.
Outside the wave support this leading term is zero. Consequently,
uniformly in physical space and time,
\[
 \lambda_{\max}(H^{J-1})
 \le K_B+C\alpha+Ck_{J-1}^{-1/2}
 =K_B+Cx_0^{-10}+Cx_0^{-D_k/2}<K_B+1.
\]
At time zero the increment vanishes outside $r/2$, so the exterior
strain bound is inherited unchanged from the base. Inside $r$ the
global gradient estimate gives the stated bound with $B_c$.
Lemma~\ref{lem:stage-coercivity-margin} now verifies the numerical
mean-coercivity conditions for the next layer.

At the stationary origin, $\chi_1(0)=1$ and
$f_{\delta_{J-1}}'(0)=\delta_{J-1}^{-1}$. The gradient expansion
therefore gives \eqref{eq:seed-central-shear}.
Transversality makes $p,q$ orthogonal, and their evolution and that
of $h$ are exactly the identities proved in Section~\ref{sec:amplification}.
At time zero $p=p_0$, $q=q_0$, $n=n_0$ and the older gradient is $L_B$.
Thus \eqref{eq:amplification-scales} and \eqref{eq:base-linear-map}
give $a_J=p_0\cdot L_Bq_0=1$ and
$\beta_J=(n_0\cdot L_Bq_0)/a_J=x_0^{-2}$.

Finally the base has $|B|+|\dot B|\le C$, so the older-field input
\eqref{eq:actual-Gstar} holds with $G_{*,J}=2$.
The remainder in \eqref{eq:seed-central-shear} holds uniformly in time,
and the child-map bound supplies the spatial coefficient estimates
of Lemma~\ref{lem:spatial-coefficients}.
With the constant first normal-stage data, $Z(0,y)=(0,0,1)$ and
$Y(0,y)=(0,1)$ at every label: there is no activation variation.
Lemma~\ref{lem:stage-wave-smallness} bounds the remaining error budget
and implies the smallness required by
Proposition~\ref{prop:conditional-wave-control}. Constants in the
$O(k_{J-1}^{-1/2})$ remainder are included in that budget.
Lemmas~\ref{intro:lem:clock-child}--\ref{intro:lem:clock-parent}
place the whole comparison interval inside $[0,S_{\rm base}]$.
This verifies the claimed first-stage hypotheses.
\hfill$\square$

The base and seed provide the starting pair $(U^{J-2},U^{J-1})$
for the induction in Section~\ref{sec:stage-induction}.

\section{Closing the stage induction}
\label{sec:stage-induction}
\fancyhead[R]{\small The stage induction}

At stage $j$, the parent $U^{j-1}$ supplies the large shear, while
$U^{j-2}$ supplies its older background. We construct $U^j$ and show
that $(U^{j-1},U^j)$ has the same properties at the next activation.
The numerical inequalities are already proved in
Section~\ref{sec:stage-parameters}; the point here is to apply them to
the actual fields, including the global pressure and initial-strain
bounds. Constants in this section are independent of the stage.

\subsection{The bounds maintained by the induction}

Write $M^i=\nabla U^i$ and $H^i=\nabla^2p^i$ in physical coordinates.
Use the scales, times and budgets of
Section~\ref{sec:stage-parameters}, with $h_{J-2}=1$.
We first specify the constants used in this state, to avoid a
circular choice of the global size bounds.

\refstepcounter{equation}\label{lem:induction-constant-order}
\textbf{Lemma \theequation\ (Order of the induction constants).}
The constants can be fixed in the order
\[
 C_{\rm amp},C_g,C_p,C_{\rm seed},K_B
 \ ;\ C_M,C_H\ ;\ \zeta
 \ ;\ c_0,Q,D_k,J,x_0,
\]
where the first group is independent of $C_M,C_H,\zeta$ and the
stage. All subsequent wave smallness requirements can be met by
choosing $J$ and then $x_0$ sufficiently large, without revising
the earlier constants.

\emph{Proof.} Normalize the raw coefficient and activation errors
as in \eqref{eq:actual-error-budget}. Proposition~\ref{prop:conditional-wave-control}
then has universal comparison constants and a universal smallness
threshold. The proof of Lemma~\ref{lem:amplitude-ratios} gives a
universal constant $C_{\rm amp}$ in \eqref{eq:forward-amplitude-ratio},
uniform also in the selected slope $\lambda$. Parent size constants
affect the prefactor in $e_*$, not $C_{\rm amp}$.

The fixed cutoff and profile bounds and
\eqref{eq:corr-final-expansions} consequently give leading increment
bounds $C_g h_j$ and $C_p C_M h_jh_{j-1}$, with $C_g,C_p$ universal
for the chosen cutoffs. The history and early-time contributions are
instead charged to $b_j$ as in \eqref{eq:induction-early-amplitude};
their possibly parent-dependent prefactors are fixed only after
$C_M,C_H$. They are made at most one by the later parameter choice.
The exact remainders are also at most one after increasing the fixed
frequency threshold. Thus neither contribution enlarges $C_g,C_p$
by a parent-dependent amount.

Fix the base cutoff first. The base lemma gives a uniform $K_B$;
on a sufficiently short fixed base interval the seed normal and
velocity have norms at most two. The proof of
Proposition~\ref{prop:seed-packet} therefore gives a fixed seed size
constant $C_{\rm seed}$ depending only on these base and profile data.
Choose, for example,
\[
 C_M\ge4(C_g+C_{\rm seed}+K_B+1),\qquad
 C_H\ge4(C_pC_M+C_{\rm seed}+K_B+1).
\]
The factor four leaves a fixed margin for the older fields and the
remainders in the preservation estimates below. Now the endpoint
map bound and Lemma~\ref{lem:history-endpoint-choice} fix a sufficiently
small $\zeta=\zeta(C_M,C_H)>0$.

Fix the exponents in all inverse and coefficient estimates and choose
$c_0$ to cover the resulting input powers; Proposition~\ref{prop:complete-packet}
then fixes $Q$. Choose $D_k$, then $J$, then $x_0$ as in
Section~\ref{sec:stage-parameters}. All parent-dependent prefactors
are now fixed. Lemma~\ref{lem:stage-wave-smallness} and the summable
cost estimates tend to zero uniformly in the stage under these choices,
so they meet the comparison thresholds, the chosen $\zeta$, and the
fixed size margins. No leading amplitude constant is chosen again.
\hfill$\square$

Use these constants in the maintained bounds below. The statements
about the seed and parameter choices in the preceding sections are
understood with this order of choice.

For bookkeeping, use $b_i$ from the summable-cost calculation in
Section~\ref{sec:stage-parameters}, with its fixed exponent enlarged
to cover both \eqref{eq:early-amplitude-ratio} and
\eqref{eq:history-target-ratio}. Define
\begin{equation}\label{eq:induction-cumulative-costs}
 \begin{aligned}
 \rho_{J-1}&=C\bigl(x_0^{-10}+k_{J-1}^{-1/2}\bigr),\\
 \rho_i&=\rho_{J-1}
       +C\sum_{\nu=J}^{i}
           \bigl(b_\nu+\delta_\nu h_\nu h_{\nu-1}+k_\nu^{-1/2}\bigr),\\
 \eta_j&=C_a\sum_{\nu=J}^{j-1}x_{\nu-1}^{-4},\qquad \eta_J=0.
 \end{aligned}
\end{equation}
Here $C,C_a$ are fixed to cover the estimates below. The summability
proved there permits the single choice
\begin{equation}\label{eq:induction-total-margins}
 \sup_i\rho_i<\tfrac12,\qquad \sup_j\eta_j<\tfrac12\log2.
\end{equation}
The cumulative bounds prevent successive stages from using the same
margin more than once.

At stage $j\ge J$ we maintain the following state.

\emph{Parent and older flow.} The pair $(U^{j-2},U^{j-1})$ consists of
smooth odd exact Euler solutions on $[0,S_{j-1}]$. At the origin,
for $t_{j-1}\le t\le S_{j-1}$,
\begin{equation}\label{eq:induction-parent-shear}
 \begin{gathered}
 M^{j-1}(t,0)=M^{j-2}(t,0)
       +h(t)q_{j-1}(t)p_{j-1}(t)^T+E_{j-1}(t),\\
 h(t_{j-1})=h_{j-1},\qquad
 |E_{j-1}(t)|\le C_E k_{j-1}^{-1/2}\le k_{j-1}^{-1/4}.
 \end{gathered}
\end{equation}
The orthonormal frame and $h$ satisfy
\eqref{eq:inherited-frame-evolution} and \eqref{eq:old-wave-norms}
with $B=M^{j-2}(t,0)$. The fixed constant $C_E$ is the uniform
remainder constant in \eqref{eq:corr-final-expansions}.

\emph{Activation geometry.} Define $a_j,\beta_j$ by
\eqref{eq:amplification-scales} at $t=t_{j-1}$. We require
\begin{equation}\label{eq:induction-activation}
 |\log a_j|\le\eta_j,\qquad
 \tfrac12\le\beta_jx_{j-1}^2\le2.
\end{equation}
In particular $a_j\in[1/2,2]$ by
\eqref{eq:induction-total-margins}. For $j>J$ we also require
\eqref{eq:history-activation-geometry}, with
$\mathcal B=M^{j-2}(t_{j-1},0)+E_{j-1}(t_{j-1})$ and $h_*=h_{j-1}$.
For $j=J$, use the prescribed zero-history data in
Proposition~\ref{prop:seed-packet} instead.

\emph{Global size and one-sided bounds.} On their retained intervals,
each constructed field $U^i$, $i\ge J-1$, satisfies
\begin{equation}\label{eq:induction-global-size}
 \|M^i(t)\|_{L_x^\infty}\le C_Mh_i,\qquad
 \|H^i(t)\|_{L_x^\infty}\le C_Hh_ih_{i-1}.
\end{equation}
In addition, throughout physical space,
\begin{equation}\label{eq:induction-quadratic-bounds}
 \begin{gathered}
 H^i(t,x)\le(K_B+\rho_i)I,\\
 \operatorname{sym}M^i(0,x)\ge
 \begin{cases}
  -(K_B+\rho_i)I,&|x|\ge r,\\
  -(C_Mh_{J-1}+\rho_i)I,&|x|<r.
 \end{cases}
 \end{gathered}
\end{equation}
These imply the bounds with $K_+=K_B+1$, $B_e=K_B+1$ and
$B_c=C_Mh_{J-1}+2$ used in
Lemma~\ref{lem:stage-coercivity-margin}. For $U^{J-2}=U_B$ use
\eqref{eq:base-uniform-bounds}.

\emph{Parent regularity.} Each constructed particle map has
\eqref{eq:child-map-output} at its own frequency $k_i$.
The current parent therefore has $K_{h,j}=k_{j-1}^{C_*}$,
$C_*=80$; the older base map uses \eqref{eq:base-uniform-bounds}.
All these estimates may be restricted to the common interval
$[0,S_{j-1}]$.

\subsection{The new layer and the next parent}

We first isolate the global estimates needed to preserve the state.
They explain why the large new shear need not consume a large part
of the upper pressure-Hessian bound.

\refstepcounter{equation}\label{lem:stage-increment-costs}
\textbf{Lemma \theequation\ (Global bounds for the amplified layer).}
Use the stage parameters with the activation ranges
\eqref{intro:eq:activation-ranges}.
Suppose the stage-$j$ wave has the conclusions of
Lemmas~\ref{lem:actual-direction-pressure} and~\ref{lem:amplitude-ratios},
the history estimate \eqref{eq:history-target-ratio} when $t_{j-1}>0$,
and the horizon bound \eqref{eq:interface-horizon} with $S=S_j$.
Normalize $\alpha_j$ by \eqref{eq:quantitative-normalization}.
Suppose the exact child has \eqref{eq:corr-final-expansions} and
the parent obeys \eqref{eq:induction-global-size}.
Then, uniformly on $[0,S_j]$,
\begin{equation}\label{eq:induction-absolute-increments}
 \begin{aligned}
 \|M^j-M^{j-1}\|_{L_x^\infty}&\le C_g h_j+Ck_j^{-1/2},\\
 \|H^j-H^{j-1}\|_{L_x^\infty}&\le C_p C_Mh_jh_{j-1}+Ck_j^{-1/2}.
 \end{aligned}
\end{equation}
The one-sided and initial estimates are smaller:
\begin{equation}\label{eq:induction-small-increments}
 \begin{aligned}
 H^j-H^{j-1}
   &\le C\bigl(b_j+\delta_jh_jh_{j-1}+k_j^{-1/2}\bigr)I,\\
 \|M^j(0)-M^{j-1}(0)\|_{L_x^\infty}
   &\le C\bigl(b_j+k_j^{-1/2}\bigr).
 \end{aligned}
\end{equation}

\emph{Proof.} Let $A(t,y)=|m_j(t,y)|\,|v_j(t,y)|$ and
$A_j=A(t_j,0)$. In the leading expressions below, the parent gradient
is evaluated at $X^{j-1}(t,\ell_jy)$ before taking the physical trace.
The normalization gives
$\alpha_j A/\delta_j=h_j A/A_j$.
For $\tau\ge1$, \eqref{eq:forward-amplitude-ratio} and the horizon
assumption bound this by $Ch_j$ on the support of $\chi_1$.
For the history and $0\le\tau\le1$, respectively,
\eqref{eq:history-target-ratio} and \eqref{eq:early-amplitude-ratio}
give, for a fixed exponent $c$,
\begin{equation}\label{eq:induction-early-amplitude}
 h_jh_{j-1}\frac{A(t,y)}{A_j}
 \le C h_jh_{j-1}K_{h,j}^c\Theta_j^c\varepsilon_j^{-c}
                     e^{-b x_{j-1}}\le b_j.
\end{equation}
Here $\beta_jx_{j-1}^2\in[1/2,2]$, by
\eqref{intro:eq:activation-ranges}, gives the fixed $b>0$,
and $s_0^{-1}\le CK_{h,j}^{\kappa}$ follows from
\eqref{eq:history-normal-choice}; for $j=J$, $s_0=1$ and there is
no history interval. The definition of $b_j$ includes the fixed
constant in this estimate.

In \eqref{eq:corr-final-expansions}, the leading gradient has norm
at most $C\alpha_j A/\delta_j$, and the leading Hessian has norm
at most $C\alpha_j A\|M^{j-1}\|/\delta_j$.
The profile bounds \eqref{eq:seed-profile-bounds} apply with
$\delta=\delta_j$. These observations, the parent size bound and
\eqref{eq:induction-early-amplitude} prove
\eqref{eq:induction-absolute-increments}. They also bound both early
increments by $Cb_j+Ck_j^{-1/2}$, including at physical time zero.

For the remaining times $\tau\ge1$, the leading Hessian is
\[
 -2\alpha_j\chi_1(m_j\cdot M^{j-1}v_j)
                   \frac{m_jm_j^T}{|m_j|^2}f_{\delta_j}'.
\]
By \eqref{eq:forward-pressure-sign} and $\chi_1\ge0$, it is negative
semidefinite where $f_{\delta_j}'\ge0$. Where $f_{\delta_j}'<0$,
the bound $|f_{\delta_j}'|\le C$ gives an upper-eigenvalue cost of
at most $C\alpha_j A\|M^{j-1}\|\le C\delta_jh_jh_{j-1}$.
Outside the wave support the leading expression vanishes.
The remainder is uniformly $Ck_j^{-1/2}$ throughout physical space.
Adding it proves \eqref{eq:induction-small-increments}.
In particular this is a global pressure estimate; it does not assume
that pressure is spatially localized.\hfill$\square$

The following ledger identifies the source of every group of inputs
used in the next proposition. Rows concerning the parent use the
maintained state; outputs of the new layer are used only to verify the
next state.
\begin{center}
\small
\begin{tabularx}{\linewidth}{@{}p{0.34\linewidth}X@{}}
\hline
Input or conclusion & Source \\
\hline
Shear, frame, and remainder at the origin & \eqref{eq:induction-parent-shear}; seed Proposition~\ref{prop:seed-packet}. \\
Wave-model comparison and pressure sign & Proposition~\ref{prop:conditional-wave-control}, Lemma~\ref{lem:actual-direction-pressure}, and smallness Lemma~\ref{lem:stage-wave-smallness}. \\
History coercivity and endpoint & \eqref{eq:induction-quadratic-bounds}, Lemmas~\ref{lem:stage-coercivity-margin} and~\ref{lem:history-endpoint-choice}. \\
Label-uniform activation data & Lemma~\ref{lem:history-label-bounds}; the localization term in \eqref{eq:stage-error-choice}. \\
All-order parent coefficients & Previous map output \eqref{eq:child-map-output} and Lemma~\ref{lem:stage-map-coefficients}. \\
Forward propagator and profile & Proposition~\ref{prop:wave-packet-inputs}. \\
Global mean coercivity & \eqref{eq:induction-quadratic-bounds}, Lemma~\ref{lem:stage-coercivity-margin}, and Section~\ref{sec:mean-inverse}. \\
Cutoff and profile data & The fixed cutoff construction and \eqref{eq:seed-profile-bounds}. \\
Frequency, horizon, inverse time & Proposition~\ref{prop:stage-frequency-time}. \\
Odd child and stationary origin & Lemma~\ref{lem:packet-parity}. \\
Initial support and all fixed Sobolev orders & Proposition~\ref{prop:initial-packet}, with one common frequency threshold; summation in Section~\ref{sec:limiting-initial-datum}. \\
\hline
\end{tabularx}
\end{center}

\refstepcounter{equation}\label{prop:stage-induction}
\textbf{Proposition \theequation\ (Closure of one stage).}
For the fixed choices above and sufficiently large $x_0$, suppose
the maintained state holds at stage $j\ge J$. Then the target time
$t_j$ and horizon $S_j$ are defined by
\eqref{eq:actual-target-time} and
Subsection~\ref{intro:subsec:time-intervals}.
There is a smooth odd exact solution $U^j$ on $[0,S_j]$ obtained by
Proposition~\ref{prop:complete-packet}, with
\begin{equation}\label{eq:induction-target-gradient}
 M^j(t_j,0)=M^{j-1}(t_j,0)+h_jq_jp_j^T+E_j(t_j),
 \qquad |E_j(t_j)|\le C_Ek_j^{-1/2}.
\end{equation}
The pair $(U^{j-1},U^j)$ satisfies the maintained state at stage $j+1$.
The initial support and Sobolev estimates
\eqref{eq:initial-packet-support}--\eqref{eq:initial-packet-sobolev}
hold for its increment, and its amplitude satisfies
\begin{equation}\label{eq:induction-amplitude-smallness}
 \alpha_j\le P_j^c e^{-b x_{j-1}}
\end{equation}
with fixed $c,b>0$.

\emph{Proof.}
\emph{Verification of the wave and new-layer inputs.}
Oddness fixes the origin. Differentiating Euler there gives
\[
 B=M^{j-2}(t,0),\qquad \dot B=-B^2-H^{j-2}(t,0).
\]
The size bounds \eqref{eq:induction-global-size}, with
$h_{j-3}\le h_{j-2}$ when that index occurs, give
$|B|\le CG_{*,j}$ and $|\dot B|\le CG_{*,j}^2$,
where $G_{*,j}=1+h_{j-2}$. At $j=J$ use the uniform base bounds.
Together with \eqref{eq:induction-parent-shear}, these are the
inputs of Lemma~\ref{lem:central-coefficients}.
The parent-map bound and Lemma~\ref{lem:stage-map-coefficients}
supply the spatial and all-order coefficient inputs.

For $j>J$, the inverse-time estimate in
Proposition~\ref{prop:stage-frequency-time} gives
$t_{j-1}^{-1}\le12S_{\rm base}^{-1}\le h_{j-1}$ and a fixed power
of $K_{h,j}$. The last inequalities follow from the seed scales and
the increase of $h_{j-1},K_{h,j}$.
The upper Hessian bound \eqref{eq:induction-quadratic-bounds} and
Lemma~\ref{lem:stage-coercivity-margin} give
\eqref{eq:history-coercivity-input}. The remaining history inputs
\eqref{eq:history-activation-geometry}--\eqref{eq:history-size-input}
are precisely the maintained geometry and size bounds.
Thus Lemmas~\ref{lem:history-endpoint-choice} and
\ref{lem:history-label-bounds} select a transverse wave with the
activation data \eqref{eq:history-scaled-data}. At $j=J$ use the
constant initial wave of Proposition~\ref{prop:seed-packet}.

Lemma~\ref{lem:stage-wave-smallness} bounds all errors by
\eqref{eq:stage-error-choice} and gives
$e_{*,j}\Theta_j^{60}\le x_{j-1}^{-10}$.
Consequently Proposition~\ref{prop:conditional-wave-control},
Lemma~\ref{lem:actual-direction-pressure}, and
Lemma~\ref{lem:amplitude-ratios} apply.
The wave may be continued on the parent interval to cover the whole
comparison interval $I_j$. Proposition~\ref{prop:stage-frequency-time}
then supplies the required horizon through $S_j$.
Normalize $\alpha_j$ by \eqref{eq:quantitative-normalization}.
The target gain \eqref{eq:target-amplitude-gain},
$s_0^{-1}\le CK_{h,j}^{\kappa}$ and the budget
\eqref{eq:stage-input-budget} give
\eqref{eq:induction-amplitude-smallness}.

Proposition~\ref{prop:wave-packet-inputs} supplies the forward
propagator, amplitude profile and selected endpoint inputs.
Its budget condition \eqref{eq:interface-budget} follows from
\eqref{eq:stage-input-budget}, the parent-map bounds and
\eqref{eq:packet-coordinate-inputs}.
The global mean inverse uses the fixed
$\lambda_b=C_1B_c+1$, the parent bounds
\eqref{eq:induction-quadratic-bounds}, and
Lemma~\ref{lem:stage-coercivity-margin}; its coefficient and
inverse-time bounds fit fixed powers of $P_j$.
Lemma~\ref{lem:stage-map-coefficients} supplies the remaining
coefficient hierarchies, while \eqref{eq:seed-profile-bounds}
and the fixed cutoffs supply the bounds on the data for the leading increment.
Finally Proposition~\ref{prop:stage-frequency-time} gives
$P_j^Q\le k_j^\vartheta$.
Thus all common inputs of Proposition~\ref{prop:complete-packet}
hold with the already fixed input exponent.
It constructs $U^j$ and gives its map and initial-data outputs.
Lemma~\ref{lem:packet-parity} proves oddness of this child: the
cutoffs are even, the profile is odd, and the selected coordinate
endpoint is constant (the first normal stage instead has constant
initial wave data). In particular, the origin is stationary.

\emph{Preservation of the global bounds.}
Lemma~\ref{lem:stage-increment-costs} applies to this exact child.
Adding \eqref{eq:induction-small-increments} to the parent bounds
and using the definition of $\rho_j$ proves
\eqref{eq:induction-quadratic-bounds} for $i=j$, both inside and
outside $|x|<r$. In particular, the total positive Hessian cost
and total initial-strain cost remain below their fixed margins.
For the absolute bounds, \eqref{eq:induction-absolute-increments}
gives
\[
 \begin{aligned}
 \|M^j\|_{L_x^\infty}
   &\le C_Mh_{j-1}+C_g h_j+Ck_j^{-1/2},\\
 \|H^j\|_{L_x^\infty}
   &\le C_Hh_{j-1}h_{j-2}+C_p C_Mh_jh_{j-1}+Ck_j^{-1/2}.
 \end{aligned}
\]
The choices in Lemma~\ref{lem:induction-constant-order} give
$C_g\le C_M/4$ and $C_pC_M\le C_H/4$, leaving a fixed margin
for the older fields and the remainders.
The separation \eqref{eq:stage-separations} absorbs the older
terms and remainders, proving \eqref{eq:induction-global-size}
with the same constants.

\emph{The shear at the origin and next activation.}
At the stationary origin the phase is zero. Using
$\chi_1(0)=1$, $f_{\delta_j}'(0)=\delta_j^{-1}$ and
\eqref{eq:corr-final-expansions}, define on $[t_j,S_j]$
\begin{equation}\label{eq:induction-child-frame}
 \begin{gathered}
 p_j(t)=\frac{m_j(t,0)}{|m_j(t,0)|},\qquad
 q_j(t)=\frac{v_j(t,0)}{|v_j(t,0)|},\qquad
 h(t)=\frac{\alpha_j}{\delta_j}|m_j(t,0)|\,|v_j(t,0)|,\\
 M^j(t,0)=M^{j-1}(t,0)+h(t)q_j(t)p_j(t)^T+E_j(t),\qquad
 |E_j(t)|\le C_Ek_j^{-1/2}\le k_j^{-1/4}.
 \end{gathered}
\end{equation}
The amplitude normalization makes $h(t_j)=h_j$, proving
\eqref{eq:induction-target-gradient}.
The wave equations and transversality imply the frame and norm
identities \eqref{eq:inherited-frame-evolution} and
\eqref{eq:old-wave-norms}, now with older gradient $M^{j-1}$.
Thus the full time-dependent decomposition, rather than only its
value at the target, is inherited by the next stage.

For the new coefficients, Proposition~\ref{prop:target-frame} gives
\begin{equation}\label{eq:induction-parameter-transfer}
 \begin{aligned}
 \left|\frac{a_{j+1}}{a_j}-1\right|
 &\le C\left(x_j^{-4}+\frac{\beta_j}{x_j^2}
       +\frac{\sqrt{\beta_j}}{x_j^3}+e_{*,j}\Theta_j^{28}\right)
 \le Cx_{j-1}^{-4},\\
 |\beta_{j+1}x_j^2-1|
 &\le C\left(\sqrt{\beta_j}+x_j^{-4}
                         +e_{*,j}\Theta_j^{28}\right)\le Cx_{j-1}^{-1}.
 \end{aligned}
\end{equation}
Here $x_j\ge x_{j-1}$ and the previously verified error bound
justify the last inequalities. Make the first relative error at
most $1/2$. Then $|\log(a_{j+1}/a_j)|\le C_a x_{j-1}^{-4}$,
so \eqref{eq:induction-activation} propagates by the definition
of $\eta_{j+1}$. The second line places $\beta_{j+1}x_j^2$ in
$[1/2,2]$ for sufficiently large $x_0$, uniformly in $j$; no product
of these second errors is required.

The positive-history geometry uses
$\mathcal B_{j+1}=M^{j-1}(t_j,0)+E_j(t_j)$, not just its older
part. By \eqref{eq:target-compression} and
\eqref{eq:stage-separations},
\begin{equation}\label{eq:induction-next-history-geometry}
 \begin{aligned}
 p_j\cdot M^{j-1}(t_j,0)p_j
   &\le-\frac{c\sqrt{h_{j-1}}}{x_{j-1}x_j}
             +C(G_{*,j}+E_{*,j})<-2,\\
 (\mathcal B_{j+1})_{p_jp_j}&<0,\qquad
 \frac{\|\mathcal B_{j+1}\|}{h_j}
       \le\frac{C_Mh_{j-1}+1}{h_j}<\zeta.
 \end{aligned}
\end{equation}
In the second line we used $|E_j(t_j)|<1$.
This is exactly \eqref{eq:history-activation-geometry} at the next
activation. The new global bounds hold on $[0,t_j]$, and
$t_j^{-1}\le12S_{\rm base}^{-1}\le h_j$, so they also provide
\eqref{eq:history-size-input} for the next history.
The child-map output supplies its regularity budget
$K_{h,j+1}=k_j^{C_*}$. Restricting the preceding parent to $[0,S_j]$
therefore gives every part of the maintained state at stage $j+1$.
\hfill$\square$

\subsection{The finite-stage sequence and accumulation time}

\refstepcounter{equation}\label{cor:finite-stage-sequence}
\textbf{Corollary \theequation\ (Exact finite stages with growing target shear).}
There are smooth odd exact Euler solutions $U^j$ on $[0,S_j]$ for
every $j\ge J-1$, with all the maintained bounds above. Their target
times increase to a finite $T>0$, and
\begin{equation}\label{eq:induction-accumulation}
 0<T-t_j\le\frac89 W_{j+1}<2W_{j+1}=S_j-t_j,
 \qquad \|M^j(t_j,0)\|\ge\tfrac12 h_j\longrightarrow\infty.
\end{equation}
Every normal-stage initial increment has the support and Sobolev
estimates of Proposition~\ref{prop:initial-packet}, with
\eqref{eq:induction-amplitude-smallness}.

\emph{Proof.} Proposition~\ref{prop:seed-packet} initializes the state:
its proof gives the stronger cumulative bounds
\eqref{eq:induction-quadratic-bounds} with $\rho_{J-1}$, and
\eqref{eq:seed-first-geometry} gives $\eta_J=0$.
Apply Proposition~\ref{prop:stage-induction} successively.
The choices and margins are uniform in the number of stages, so
no parameter needs to be chosen again at a later step.

The time estimates in Proposition~\ref{prop:stage-frequency-time}
give $W_{i+1}\le W_i/4$ and $t_i-t_{i-1}\le2W_i/3$.
Summing the resulting geometric series proves
\[
 T-t_j=\sum_{i=j+1}^{\infty}(t_i-t_{i-1})
 \le\frac23\sum_{i=j+1}^{\infty}W_i\le\frac89W_{j+1}.
\]
Each summand is positive, and $t_J>0$, so $0<T<\infty$ and
$t_j<T<S_j$. At the target, \eqref{eq:induction-target-gradient}
and $\|q_jp_j^T\|=1$ give
\[
 \|M^j(t_j,0)\|\ge h_j-C_Mh_{j-1}-C_Ek_j^{-1/2}
                         \ge\tfrac12 h_j.
\]
The last inequality follows from \eqref{eq:stage-separations}, and
the prescribed $h_j$ tends to infinity. The initial assertions are
outputs of the same stage proposition.\hfill$\square$

\textbf{Remark.} These are exact solutions with different initial
data. To obtain blowup for a single initial datum, one must still
sum the initial increments and use stability to transfer the growing
target gradients to a hypothetical smooth solution from that limit.
The finite-stage induction supplies the uniform inputs for that
argument in Sections~\ref{sec:limiting-initial-datum}--\ref{sec:limiting-solution-stability}.

\section{The limiting initial datum}
\label{sec:limiting-initial-datum}
\fancyhead[R]{\small The limiting initial datum}

Corollary~\ref{cor:finite-stage-sequence} constructs exact solutions
with growing target gradients, but their initial data differ from
stage to stage. We now show that these data converge to one smooth,
compactly supported velocity. The two parts of each initial increment
have different sources of smallness: the oscillatory part retains the
exponentially small amplitude $\alpha_j$, whereas the mean part starts
at order $k_j^{-2}$ and has no angle derivatives. Both features are
needed to sum derivatives of arbitrarily high fixed order.

\subsection{Summability of the two initial increments}

For each normal stage $j\ge J$, use the physical decomposition
\eqref{eq:initial-split} and write
\begin{equation}\label{eq:limit-initial-increment}
 w_{j,0}:=U^j(0)-U^{j-1}(0)
       =u_{j,\mathrm{osc},0}+u_{j,\mathrm{mean},0}.
\end{equation}
The seed contributes the single fixed smooth datum $U^{J-1}(0)$;
only the normal-stage increments need to be summed. All parameters
remain fixed as in Sections~\ref{sec:stage-parameters}--\ref{sec:stage-induction}.

\refstepcounter{equation}\label{lem:initial-increment-summability}
\textbf{Lemma \theequation\ (Summability in every fixed Sobolev norm).}
Let the increments be those of
Corollary~\ref{cor:finite-stage-sequence}, and let $b>0$ be as in
\eqref{eq:induction-amplitude-smallness}. For every integer $m\ge0$
there is an index $j_m\ge J$ such that, for all $j\ge j_m$,
\begin{equation}\label{eq:limit-increment-decay}
 \|u_{j,\mathrm{osc},0}\|_{H^m}
       \le e^{-b x_{j-1}/2},\qquad
 \|u_{j,\mathrm{mean},0}\|_{H^m}
       \le e^{-x_{j-1}/j^2}.
\end{equation}
Consequently $\sum_{j\ge J}\|w_{j,0}\|_{H^m}<\infty$.
Increasing $j_m$ if necessary, its tails satisfy, for $j\ge j_m$,
\begin{equation}\label{eq:limit-initial-tail}
 \sum_{i=j+1}^{\infty}\|w_{i,0}\|_{H^m}
       \le 2e^{-b x_j/2}+2e^{-x_j/(j+1)^2}.
\end{equation}

\textbf{Remark.} Only the starting index $j_m$ depends on the
Sobolev order. The parameters $J,x_0$, the input exponent and
the frequency choices are not changed when $m$ increases.

\emph{Proof.} Set $X_j=x_{j-1}$. The input budget
\eqref{eq:stage-input-budget} and the calculation in
Proposition~\ref{prop:stage-frequency-time} give, with all parameters
now fixed,
\begin{equation}\label{eq:limit-budget-asymptotic}
 \frac{\log P_j}{X_j/j^2}
 \le C\left(\frac1j+\frac{C_*}{j^2}
 +\frac{j^2[\log X_j+(D_0+1)\log x_0+\log J+1]}{X_j}\right)
 \longrightarrow0.
\end{equation}
Indeed $X_j=x_0\prod_{\nu=J}^{j-1}\nu^2$, so the last term also
tends to zero. This asymptotic statement is stronger than the fixed
frequency threshold $P_j^Q\le k_j^\vartheta$ and is what allows
arbitrarily large fixed $m$.

Combine \eqref{eq:initial-packet-sobolev} with
\eqref{eq:induction-amplitude-smallness} and the stage scales
of Subsection~\ref{intro:subsec:stage-parameters}. After absorbing
the amplitude's fixed power of $P_j$ into a constant $C_m$, we obtain
\begin{equation}\label{eq:limit-logarithmic-costs}
 \begin{aligned}
 \|u_{j,\mathrm{osc},0}\|_{H^m}
 &\le\exp\!\left[-bX_j+\frac{mX_j}{j^2}
                      +\frac{mX_j}{j^{7/2}}+C_m\log P_j\right],\\
 \|u_{j,\mathrm{mean},0}\|_{H^m}
 &\le\exp\!\left[-\frac{2X_j}{j^2}
                      +\frac{mX_j}{j^{7/2}}+C_m\log P_j\right].
 \end{aligned}
\end{equation}
The term $mX_j/j^2$ in the first line is the cost of $k_j^m$.
It is absent from the second line because the mean coefficient is
independent of the angle; only the localization cost
$\ell_j^{-m}$ remains there.

For fixed $m$, divide the positive terms in the first exponent by
$X_j$. Their sum tends to zero, by
\eqref{eq:limit-budget-asymptotic}, so it is eventually at most $b/2$.
For the second exponent, divide its positive terms by $X_j/j^2$:
\[
 \frac{m}{j^{3/2}}+C_m\frac{\log P_j}{X_j/j^2}\longrightarrow0.
\]
It is eventually at most $1$. These two observations prove
\eqref{eq:limit-increment-decay}.

Finally $X_{i+1}=i^2X_i$, and
\[
 \frac{X_{i+1}/(i+1)^2}{X_i/i^2}
       =\frac{i^4}{(i+1)^2}\ge2\qquad(i\ge3).
\]
Thus both positive decay exponents at least double at each later
stage. Once an exponent $A$ is at least $\log2$,
$\sum_{n\ge0}e^{-2^nA}\le\sum_{n\ge0}e^{-(n+1)A}\le2e^{-A}$.
Apply this separately to the two bounds, starting at $i=j+1$,
to obtain \eqref{eq:limit-initial-tail}.
The finitely many earlier increments are smooth, which proves the
full summability assertion.\hfill$\square$

\subsection{A single smooth, compactly supported datum}

Fix $R\ge2$ such that the support of the base datum
\eqref{eq:base-datum} lies in $\overline{B(0,R)}$.
The seed support and \eqref{eq:initial-packet-support} then place
every initial increment in this same ball.

\refstepcounter{equation}\label{prop:limiting-initial-datum}
\textbf{Proposition \theequation\ (The limiting initial velocity).}
The series
\begin{equation}\label{eq:limiting-initial-datum}
 u_0=U^{J-1}(0)+\sum_{j=J}^{\infty}w_{j,0}
       =\lim_{j\to\infty}U^j(0)
\end{equation}
converges in every $H^m(\R^3)$, $m\ge0$ an integer, to the same
velocity $u_0$. This velocity is smooth, odd, divergence free, and
supported in $\overline{B(0,R)}$. For every fixed $m$ and all
$j\ge j_m$,
\begin{equation}\label{eq:limiting-datum-error}
 \|u_0-U^j(0)\|_{H^m}
       \le2e^{-b x_j/2}+2e^{-x_j/(j+1)^2}\longrightarrow0.
\end{equation}
In particular, for every multi-index $I$,
\begin{equation}\label{eq:limiting-datum-uniform-derivatives}
 \|\partial^I(U^j(0)-u_0)\|_{L^\infty}
       \longrightarrow0.
\end{equation}

\emph{Proof.} Lemma~\ref{lem:initial-increment-summability} makes
the series absolutely convergent in each complete space $H^m$.
All these limits agree with its $L^2$ limit, since convergence in
$H^m$ implies convergence in $L^2$. Telescoping
\eqref{eq:limit-initial-increment} identifies its partial sums with
$U^j(0)$, and \eqref{eq:limit-initial-tail} proves
\eqref{eq:limiting-datum-error}.

For any multi-index $I$, choose an integer $m>|I|+3/2$.
Sobolev embedding $H^m(\R^3)\hookrightarrow C_b^{|I|}(\R^3)$ then
proves \eqref{eq:limiting-datum-uniform-derivatives} and supplies
compatible continuous derivatives of every order. Hence $u_0$ is
smooth. Every partial sum vanishes outside $\overline{B(0,R)}$;
its $L^2$ limit vanishes there almost everywhere, and continuity
makes the vanishing pointwise. Oddness passes to the limit in
$L^2$, and divergence freeness passes in $H^1$ (or distributions).
The smooth representative therefore has all the stated properties.
\hfill$\square$

The construction now provides one admissible smooth initial datum.
Section~\ref{sec:limiting-solution-stability} compares its Euler
solution with the finite stages, using
\eqref{eq:limiting-datum-error} at $m=3$ and the growing target
gradients in \eqref{eq:induction-accumulation}.

\section{Smooth comparison and finite-time blowup}
\label{sec:limiting-solution-stability}
\fancyhead[R]{\small Smooth comparison and finite-time blowup}

Let $u_0$ be the datum of Proposition~\ref{prop:limiting-initial-datum}.
The local-existence argument in the proof of
Lemma~\ref{lem:base-flow}, based on \eqref{eq:base-sobolev-energy},
applies to this smooth divergence-free datum. Denote its unique
maximal smooth Euler solution by $u$ and its lifespan by
$T_*=T_*(u_0)>0$. Here smoothness includes continuity in every fixed
Sobolev norm on compact time intervals before $T_*$.
Keep $T=\lim_jt_j$ for the accumulation time of
Corollary~\ref{cor:finite-stage-sequence}.

If $u$ remained smooth past $T$, the small initial differences would
remain small through every target time $t_j$. Since $H^3(\R^3)$
controls the velocity gradient pointwise, this would contradict
\eqref{eq:induction-accumulation}. The comparison below obtains its
constants from the smooth reference solution; it does not require
a uniform bound on the gradients of the finite-stage solutions.

\subsection{Comparison with a smooth reference solution}

\refstepcounter{equation}\label{lem:smooth-reference-stability}
\textbf{Lemma \theequation\ ($H^3$ stability near a smooth solution).}
Let $0<\sigma<\infty$ and let $u,v$ be smooth exact Euler solutions
on $[0,\sigma]\times\R^3$,
with finite Sobolev norms, and suppose
$\sup_{0\le t\le\sigma}\|u(t)\|_{H^4}\le A$.
There is an absolute constant $C$ such that, if
\begin{equation}\label{eq:stability-initial-smallness}
 \|v(0)-u(0)\|_{H^3}\,e^{C(A+1)\sigma}\le\tfrac12,
\end{equation}
then
\begin{equation}\label{eq:stability-comparison}
 \|v(t)-u(t)\|_{H^3}
 \le\|v(0)-u(0)\|_{H^3}\,e^{C(A+1)t}
 \qquad(0\le t\le\sigma).
\end{equation}
The constant is independent of any a priori Sobolev bound on $v$.

\emph{Proof.} Set $w=v-u$ and let $\pi$ be the pressure of $v$
minus that of $u$. Subtracting their exact equations gives
\begin{equation}\label{eq:stability-difference-equation}
 \partial_tw+(u+w)\cdot\nabla w+w\cdot\nabla u+\nabla\pi=0,
 \qquad \nabla\cdot w=0.
\end{equation}
For a multi-index $I$, $|I|\le3$, differentiate
\eqref{eq:stability-difference-equation} and take the
$L^2$ inner product with $\partial^Iw$. Incompressibility gives
the two cancellations
\[
 \langle\partial^Iw,(u+w)\cdot\nabla\partial^Iw\rangle=0,
 \qquad
 \langle\partial^Iw,\nabla\partial^I\pi\rangle=0.
\]
Thus only the differentiated transport commutators
$[\partial^I,u\cdot\nabla]w$, $[\partial^I,w\cdot\nabla]w$, and
$\partial^I(w\cdot\nabla u)$ remain.

For $a=u$ or $a=w$, the commutator $[\partial^I,a\cdot\nabla]w$
is a linear combination of terms
$(\partial^Ja)\cdot\nabla\partial^{I-J}w$, with $0<J\le I$
componentwise (and vanishes when $I=0$).
Leibniz' rule and the embeddings
$H^3\hookrightarrow W^{1,\infty}$ and $H^1\hookrightarrow L^6$
give
\begin{equation}\label{eq:stability-product-bounds}
 \begin{aligned}
 \sum_{|I|\le3}\|[\partial^I,u\cdot\nabla]w\|_{L^2}
       &\le C\|u\|_{H^3}\|w\|_{H^3},\\
 \sum_{|I|\le3}\|[\partial^I,w\cdot\nabla]w\|_{L^2}
       &\le C\|w\|_{H^3}^2,\\
 \|w\cdot\nabla u\|_{H^3}
       &\le C\|w\|_{H^3}\|u\|_{H^4}.
 \end{aligned}
\end{equation}
In the commutators, terms with only one derivative on $a$ or on $w$
are bounded by putting that factor in $L^\infty$ and the other in $L^2$.
The only middle case at top order is
$\partial^2a\,\partial^2w$: use $L^3\cdot L^6$, with $L^3$
obtained by interpolation between $L^2$ and $L^6$.
This proves the first two bounds. The last follows from the same
product estimates, or the algebra property of $H^3$ applied to
$w$ and $\nabla u$.

Summing the inner products and applying \eqref{eq:stability-product-bounds} yields
\begin{equation}\label{eq:stability-difference-energy}
 \frac12\frac{d}{dt}\|w\|_{H^3}^2
 \le C\|u\|_{H^4}\|w\|_{H^3}^2+C\|w\|_{H^3}^3.
\end{equation}
Equivalently, $y(t)=\|w(t)\|_{H^3}$ satisfies
\begin{equation}\label{eq:stability-scalar-inequality}
 y'\le CAy+Cy^2.
\end{equation}
At a zero of the norm, the corresponding integral inequality follows
by applying \eqref{eq:stability-difference-energy} to
$(\|w\|_{H^3}^2+\epsilon^2)^{1/2}$ and then letting $\epsilon\downarrow0$.
While $y\le1$, \eqref{eq:stability-scalar-inequality} and Gronwall give
$y(t)\le y(0)e^{C(A+1)t}\le1/2$ by
\eqref{eq:stability-initial-smallness}. Since $y$ is continuous and
$y(0)\le1/2$, it cannot have a first time at which $y=1$.
The estimate holds on the entire interval, proving
\eqref{eq:stability-comparison}.\hfill$\square$

\subsection{The lifespan bound for the limiting datum}

\refstepcounter{equation}\label{prop:limiting-finite-lifespan}
\textbf{Proposition \theequation\ (Blowup no later than the accumulation time).}
For the smooth compactly supported datum $u_0$ in
\eqref{eq:limiting-initial-datum}, its maximal smooth lifespan obeys
\begin{equation}\label{eq:limiting-lifespan-bound}
 0<T_*(u_0)\le T<S_{\rm base}<\infty.
\end{equation}

\textbf{Remark.} The conclusion gives only $T_*\le T$; the limiting
solution may cease to be smooth earlier. Comparison through all target
times is used only under the contradiction assumption $T_*>T$.
No stability of blowup under perturbations of the initial datum is
proved here.

\emph{Proof.} Local existence gives $T_*>0$.
Suppose $T_*>T$. Then $u$ is smooth on the compact interval $[0,T]$,
and
\[
 A:=\sup_{0\le t\le T}\|u(t)\|_{H^4}<\infty.
\]
For every $j$, the exact solution $U^j$ exists on $[0,t_j]$ by
Corollary~\ref{cor:finite-stage-sequence}. The initial convergence
\eqref{eq:limiting-datum-error} with $m=3$ implies, for all
sufficiently large $j$,
\[
 \|U^j(0)-u_0\|_{H^3}\,e^{C(A+1)T}\le\tfrac12.
\]
Apply Lemma~\ref{lem:smooth-reference-stability} with $v=U^j$
and $\sigma=t_j<T$, using the same $A$ for every stage. It gives
\begin{equation}\label{eq:limiting-solution-comparison}
 \sup_{0\le t\le t_j}\|U^j(t)-u(t)\|_{H^3}
 \le e^{C(A+1)T}\|U^j(0)-u_0\|_{H^3}\longrightarrow0.
\end{equation}
By Lemma~\ref{lem:smooth-reference-stability}, $C$ is universal, so the
factor $e^{C(A+1)T}$ depends only on $u$ and $T$, not on $U^j$,
and is therefore independent of the stage.

By Sobolev embedding and \eqref{eq:limiting-solution-comparison},
\begin{equation}\label{eq:limiting-target-comparison}
 \|M^j(t_j,0)-\nabla u(t_j,0)\|
 \le C\|U^j(t_j)-u(t_j)\|_{H^3}\longrightarrow0.
\end{equation}
But \eqref{eq:induction-accumulation} gives
$\|M^j(t_j,0)\|\ge h_j/2\to\infty$, whereas
$\|\nabla u(t_j,0)\|\le CA$.
The triangle inequality in \eqref{eq:limiting-target-comparison}
is incompatible with these two bounds. Hence $T_*\le T$.
Finally \eqref{eq:induction-accumulation} at the seed horizon gives
$T<S_{J-1}=S_{\rm base}$, proving
\eqref{eq:limiting-lifespan-bound}.\hfill$\square$

Section~\ref{sec:endpoint-conclusions} proves the endpoint conclusions
at the actual maximal time $T_*$, which need not equal $T$.

\section{Continuation criteria and endpoint divergence}
\label{sec:endpoint-conclusions}
\fancyhead[R]{\small The endpoint conclusions}

Proposition~\ref{prop:limiting-finite-lifespan} gives a finite maximal
smooth lifespan $T_*=T_*(u_0)$. We now identify two quantities that
must diverge there. The first conclusion follows from the Sobolev
energy estimate and a uniform local restart. For the vorticity
conclusion, a logarithmic estimate converts time-integrability of
$\|\operatorname{curl}u\|_{L^\infty}$ into a uniform $H^3$ bound,
which permits the same restart. This is the Beale--Kato--Majda
continuation argument \cite{beale-kato-majda1984}; we give the
estimates needed here explicitly.

\subsection{The continuation criterion}

\refstepcounter{equation}\label{lem:endpoint-restart}
\textbf{Lemma \theequation\ (Continuation from a bounded $H^3$ norm).}
Let $u$ be a smooth Euler solution on $[0,S)\times\R^3$, where
$0<S<\infty$, with $u(0)$ in every $H^m$ and
\begin{equation}\label{eq:endpoint-H3-input}
 \sup_{0\le t<S}\|u(t)\|_{H^3}\le M<\infty.
\end{equation}
Then $u$ extends smoothly beyond $S$. In particular, such an
extension exists if
\begin{equation}\label{eq:endpoint-gradient-integrability}
 \int_0^S\|\nabla u(t)\|_{L^\infty}\,dt<\infty.
\end{equation}

\emph{Proof.} Increase $M$ to at least $1$. The local-existence
argument used with \eqref{eq:base-sobolev-energy} gives a time
$\delta_0(M)>0$ for every divergence-free initial velocity of
$H^3$ norm at most $M$. This time depends on the $H^3$ bound alone.
Let $C_{\rm emb}$ satisfy
$\|\nabla f\|_{L^\infty}\le C_{\rm emb}\|f\|_{H^3}$, and set
\[
 \delta=\min\left\{\frac{\delta_0(M)}2,
                    \frac1{2C_3C_{\rm emb}M}\right\}>0,
\]
where $C_3$ is the constant in \eqref{eq:base-sobolev-energy}
at order three. For any $s<S$, restart from $u(s)$ and denote
the restarted solution by $v_s$. The same energy estimate gives
\begin{equation}\label{eq:endpoint-restart-bound}
 \|v_s(t)\|_{H^3}
 \le\frac{M}{1-C_3C_{\rm emb}M(t-s)}\le2M
 \qquad(s\le t\le s+\delta).
\end{equation}
For each integer $m\ge3$, apply
\eqref{eq:base-sobolev-energy} again to obtain
\begin{equation}\label{eq:endpoint-higher-restart}
 \|v_s(t)\|_{H^m}
 \le\|u(s)\|_{H^m}
       \exp\!\bigl(2C_mC_{\rm emb}M(t-s)\bigr).
\end{equation}
Thus the smooth higher-order solutions persist on this same
interval; its length does not shrink with $m$.
Choose $s\in(\max\{0,S-\delta/2\},S)$.
Uniqueness identifies $v_s$ with $u$ on $[s,S)$, while
$s+\delta>S$. This supplies the smooth extension.

Finally \eqref{eq:base-sobolev-energy} directly gives
\begin{equation}\label{eq:endpoint-integrated-energy}
 \|u(t)\|_{H^3}\le\|u(0)\|_{H^3}
       \exp\!\left(C_3\int_0^t\|\nabla u(\tau)\|_{L^\infty}\,d\tau\right).
\end{equation}
Condition \eqref{eq:endpoint-gradient-integrability} therefore
implies \eqref{eq:endpoint-H3-input}, proving the last assertion.
\hfill$\square$

\subsection{The logarithmic vorticity estimate}

The vorticity does not control the velocity gradient by a direct
$L^\infty$ estimate for the Biot--Savart operator. The following
logarithmic bound is sufficient for continuation.

\refstepcounter{equation}\label{lem:endpoint-logarithmic-bound}
\textbf{Lemma \theequation\ (Gradient bound with a logarithmic Sobolev cost).}
For every divergence-free $u\in H^3(\R^3)$, writing
$\omega=\operatorname{curl}u$, one has
\begin{equation}\label{eq:endpoint-logarithmic-bound}
 \|\nabla u\|_{L^\infty}
 \le C\left[1+\|u\|_{L^2}
       +\|\omega\|_{L^\infty}\log(e+\|u\|_{H^3})\right].
\end{equation}

\emph{Proof.} Since $\nabla\cdot u=0$,
$-\Delta u=\operatorname{curl}\omega$. The Fourier identity for
the whole-space Biot--Savart operator therefore gives
\begin{equation}\label{eq:endpoint-kernel-representation}
 \nabla u(x)=\mathsf L\omega(x)
       +\operatorname{p.v.}\!\int_{\R^3}\mathsf K(z)\omega(x-z)\,dz.
\end{equation}
Here $\mathsf L$ is a fixed linear map. The entries of
$\mathsf K$ are fixed linear combinations of the second derivatives
of $(4\pi|z|)^{-1}$ away from zero, namely
$(3z_az_b-\delta_{ab}|z|^2)/(4\pi|z|^5)$.
Consequently
\[
 |\mathsf K(z)|\le C|z|^{-3},\qquad
 |\nabla\mathsf K(z)|\le C|z|^{-4},\qquad
 \int_{\mathbb S^2}\mathsf K(\vartheta)\,d\vartheta=0.
\]
The local map $\mathsf L$ accounts for the distributional terms
at the origin when differentiating the Newtonian potential.

Fix a smooth radial cutoff $\chi$, equal to one on $|z|\le1$
and zero on $|z|\ge2$. For $0<d<1/2$, split the integral in
\eqref{eq:endpoint-kernel-representation} into $|z|<d$, the
remaining part with cutoff $\chi$, and the part with cutoff
$1-\chi$. The zero spherical mean permits subtraction of
$\omega(x)$ in the inner integral. Since
$H^3(\R^3)\hookrightarrow C^{1,1/4}(\R^3)$,
\begin{equation}\label{eq:endpoint-kernel-near-middle}
 \begin{aligned}
 \left|\int_{|z|<d}\mathsf K(z)
          [\omega(x-z)-\omega(x)]\,dz\right|
 &\le C\|u\|_{H^3}\int_0^d r^{-3}r^{1/4}r^2\,dr
 \le C d^{1/4}\|u\|_{H^3},\\
 \left|\int_{|z|\ge d}\chi(z)\mathsf K(z)\omega(x-z)\,dz\right|
 &\le C\|\omega\|_{L^\infty}\int_d^2\frac{dr}{r}
 =C\|\omega\|_{L^\infty}\log(2/d).
 \end{aligned}
\end{equation}
For the far part, use $\omega=\operatorname{curl}u$ and integrate
by parts, moving its derivative onto $(1-\chi)\mathsf K$.
That differentiated kernel belongs to $L^2$: it vanishes near zero
and decays as $|z|^{-4}$ at infinity. Cauchy--Schwarz gives
\begin{equation}\label{eq:endpoint-kernel-far}
 \left|\int_{\R^3}(1-\chi(z))\mathsf K(z)\omega(x-z)\,dz\right|
 \le C\|\nabla[(1-\chi)\mathsf K]\|_{L^2}\|u\|_{L^2}
 \le C\|u\|_{L^2}.
\end{equation}
The integration by parts is justified first with a distant cutoff;
its error tends to zero by the kernel decay and $u\in L^2$.

Choose $d=\tfrac14(e+\|u\|_{H^3})^{-4}$.
Then $d<1/2$, $d^{1/4}\|u\|_{H^3}\le1$, and
$\log(2/d)\le C\log(e+\|u\|_{H^3})$.
Combining \eqref{eq:endpoint-kernel-near-middle} and
\eqref{eq:endpoint-kernel-far} with the local term
$|\mathsf L\omega|\le C\|\omega\|_{L^\infty}$ proves
\eqref{eq:endpoint-logarithmic-bound}.
The argument can be performed on smooth approximations and passed
to the $H^3$ limit.\hfill$\square$

\subsection{Blowup at the maximal time}

\refstepcounter{equation}\label{prop:endpoint-divergence}
\textbf{Proposition \theequation\ (Endpoint gradient and vorticity divergence).}
Let $u_0$ be the smooth compactly supported datum of
Proposition~\ref{prop:limiting-initial-datum}, and let $u$ be its
maximal smooth Euler solution. At the finite time $T_*(u_0)$ of
Proposition~\ref{prop:limiting-finite-lifespan},
\begin{equation}\label{eq:endpoint-divergence}
 \limsup_{t\uparrow T_*}\|\nabla u(t)\|_{L^\infty}=\infty,
 \qquad
 \int_0^{T_*}\|\operatorname{curl}u(t)\|_{L^\infty}\,dt=\infty.
\end{equation}

\emph{Proof.} If the first limsup were finite, continuity on
compact subintervals of $[0,T_*)$ would give
$\sup_{t<T_*}\|\nabla u(t)\|_{L^\infty}<\infty$.
Since $T_*<\infty$, condition
\eqref{eq:endpoint-gradient-integrability} would hold with $S=T_*$.
Lemma~\ref{lem:endpoint-restart} would extend $u$ smoothly beyond
$T_*$, contradicting maximality.

For the second assertion put $\omega=\operatorname{curl}u$ and
suppose instead that $\int_0^{T_*}\|\omega(t)\|_{L^\infty}\,dt<\infty$.
Taking the $L^2$ inner product of Euler with $u$, the transport and
pressure terms cancel by incompressibility; hence
$\|u(t)\|_{L^2}=\|u_0\|_{L^2}$ for $t<T_*$.
Set $Z(t)=\log(e+\|u(t)\|_{H^3})$.
Dividing \eqref{eq:base-sobolev-energy} at order three by
$e+\|u\|_{H^3}$ and applying
Lemma~\ref{lem:endpoint-logarithmic-bound} gives
\begin{equation}\label{eq:endpoint-logarithmic-energy}
 Z'(t)\le C\|\nabla u(t)\|_{L^\infty}
       \le C(1+\|u_0\|_{L^2})+C\|\omega(t)\|_{L^\infty}Z(t).
\end{equation}
Gronwall's inequality yields, for every $t<T_*$,
\begin{equation}\label{eq:endpoint-vorticity-continuation}
 Z(t)\le\bigl[Z(0)+C(1+\|u_0\|_{L^2})t\bigr]
       \exp\!\left(C\int_0^t\|\omega(\tau)\|_{L^\infty}\,d\tau\right).
\end{equation}
The assumed finite integral bounds $Z$ uniformly on $[0,T_*)$.
Consequently $\|u(t)\|_{H^3}$ is uniformly bounded there.
Lemma~\ref{lem:endpoint-restart} again extends the solution beyond
$T_*$, which is impossible. This proves both assertions.
\hfill$\square$

Together, Propositions~\ref{prop:limiting-initial-datum},
\ref{prop:limiting-finite-lifespan}, and~\ref{prop:endpoint-divergence}
prove Theorem~\ref{thm:main}: a smooth compactly
supported divergence-free initial velocity with finite maximal smooth
lifespan and the two endpoint divergences
\eqref{eq:endpoint-divergence}.

\section*{Acknowledgments}
The author is supported by the Ram\'on y Cajal grant RYC2022-035363-I.

\end{document}